\documentclass[]{amsart}

\usepackage{amsmath,amsfonts,amssymb,amsthm,amscd,comment,euscript}
\usepackage{bm}
\usepackage{etoolbox}
\usepackage{stmaryrd}
\usepackage[all]{xy}
\usepackage{graphicx}
\usepackage{enumerate, enumitem} 

\usepackage[breaklinks, colorlinks=true, linktocpage]{hyperref} 
\usepackage[usenames,dvipsnames]{xcolor}
\usepackage{tikz-cd}
\usepackage{xcolor}
\usepackage{soul}
\usepackage{stmaryrd} 
\usepackage{environ}
\usepackage{mathdots} 
\usepackage{caption}
\usepackage{subcaption}
\usepackage{adjustbox}

\usepackage[capitalize]{cleveref}   

\usepackage[rightcaption]{sidecap}
\sidecaptionvpos{figure}{c}

\definecolor{UF2}{HTML}{FA4616}
\definecolor{UF}{HTML}{0021A5}
\definecolor{lightergray}{rgb}{0.8,0.85,0.85}

\newcommand{\bump}{\begin{tikzpicture}[scale = .4]
		\draw[lightergray] (0,0) -- (1,0) -- (1,1) -- (0,1) -- (0,0);
		\draw[very thick,UF!100!black] (0.5,0) to[out=90,in=180] (1,0.5);
		\draw[very thick,UF!100!black] (0,0.5) to[out=0,in=-90] (0.5,1);
\end{tikzpicture}}
\newcommand{\jay}{\begin{tikzpicture}[scale = .4]
		\draw[lightergray] (0,0) -- (1,0) -- (1,1) -- (0,1) -- (0,0);
		\draw[very thick,UF!100!black] (0,0.5) to[out=0,in=-90] (0.5,1);
\end{tikzpicture}}

\newcommand{\are}{\begin{tikzpicture}[scale = .4]
		\draw[lightergray] (0,0) -- (1,0) -- (1,1) -- (0,1) -- (0,0);
		\draw[very thick,UF!100!black] (0.5,0) to[out=90,in=180] (1,0.5);
\end{tikzpicture}}
\newcommand{\bcross}{\begin{tikzpicture}[scale = .4]
		\draw[lightergray] (0,0) -- (1,0) -- (1,1) -- (0,1) -- (0,0);
		\draw[very thick,UF!100!black] (.5,0) -- (.5,1);
		\draw[very thick,UF!100!black] (0,.5) -- (1,.5);
\end{tikzpicture}}

\NewEnviron{mosaic}[2][]{\begin{tikzpicture}[baseline]\def\tilesize{#2}\matrix[mosai,ampersand replacement=\&, #1]{\BODY};\end{tikzpicture}}

\newcommand{\cspace}{.1} 
\newcommand{\tileedge}{.05pt} 
\newcommand{\knotthickness}{1.5pt} 

\newcommand{\tileii}[1][{}]{\tikz {\draw[tilestyle] (0,0) rectangle (\tilesize,\tilesize); \draw[#1,knotstyle={\tilesize*\knotthickness}] (\tilesize,.5*\tilesize) to[out=180, in=90] (.5*\tilesize,0);}}

\newcommand{\tileiv}[1][{}]{\tikz {\draw[tilestyle] (0,0) rectangle (\tilesize,\tilesize); \draw[#1,knotstyle={\tilesize*\knotthickness}] (0,.5*\tilesize) to[out=0, in=-90] (.5*\tilesize,\tilesize);}}

\newcommand{\tileviii}[1][{}]{\tikz {\draw[tilestyle] (0,0) rectangle (\tilesize,\tilesize); \draw[#1,knotstyle={\tilesize*\knotthickness}] (\tilesize,.5*\tilesize) to[out=180, in=90] (.5*\tilesize,0); \draw[#1,knotstyle={\tilesize*\knotthickness}] (0,.5*\tilesize) to[out=0, in=-90] (.5*\tilesize,\tilesize);}}
\newcommand{\tileix}[1][{}]{\tikz {\draw[tilestyle] (0,0) rectangle (\tilesize,\tilesize); \draw[#1,knotstyle={\tilesize*\knotthickness}] (0,.5*\tilesize) to (\tilesize,.5*\tilesize); \draw[#1,knotstyle={\tilesize*\knotthickness}] (.5*\tilesize,0) to (.5*\tilesize,.5*\tilesize-\tilesize*\cspace); \draw[#1,knotstyle={\tilesize*\knotthickness}] (.5*\tilesize,.5*\tilesize+\tilesize*\cspace) to (.5*\tilesize,\tilesize);}}

\newcommand{\tileNoBot}{%
  \tikz \draw[tilestyle] (0,0) -- (0,\tilesize) -- (\tilesize,\tilesize) -- (\tilesize,0);
}
\newcommand{\tileNoTop}{%
  \tikz \draw[tilestyle] (0,\tilesize) -- (0,0) -- (\tilesize,0) -- (\tilesize,\tilesize);
}
\newcommand{\tileLR}{%
  \tikz \draw[tilestyle]
    (0,0) -- (0,\tilesize)
    (\tilesize,0) -- (\tilesize,\tilesize);
}
\newcommand{\tileLRtext}[2][{}]{%
  \tikz {
    \draw[tilestyle]
    (0,0) -- (0,\tilesize)
    (\tilesize,0) -- (\tilesize,\tilesize);
    \node[anchor=center] at (\tilesize/2,\tilesize/2) {#2};
  }%
}

\newcommand{\tilecdots}{%
  \tikz {
    \draw[tilestyle] (0,0) rectangle (\tilesize,\tilesize);
    \node at (\tilesize/2,\tilesize/2) {\(\cdots\)};
  }
}
\newcommand{\tilevdots}{%
  \tikz {
    \draw[tilestyle] (0,0) rectangle (\tilesize,\tilesize);
    \node at (\tilesize/2,\tilesize/2) {\(\vdots\)};
  }
}  
  \newcommand{\tileddots}{%
  \tikz {
    \draw[tilestyle] (0,0) rectangle (\tilesize,\tilesize);
    \node at (\tilesize/2,\tilesize/2) {\(\ddots\)};
  }
}
  \newcommand{\tileiddots}{%
  \tikz {
    \draw[tilestyle] (0,0) rectangle (\tilesize,\tilesize);
    \node at (\tilesize/2,\tilesize/2) {\(\iddots\)};
  }
}
\newcommand{\textile}[2][{}]{%
  \tikz {
    \draw[#1] (0,0) rectangle (\tilesize,\tilesize);
    \node[anchor=center] at (\tilesize/2,\tilesize/2) {#2};
  }%
}
\newcommand{\tilexvert}[1][{}]{%
  \tikz {
    \draw[tilestyle] (0,0) rectangle (\tilesize,\tilesize);
    \draw[red,knotstyle={\tilesize*\knotthickness}] (.5*\tilesize,0) to (.5*\tilesize,\tilesize);
    \draw[#1,knotstyle={\tilesize*\knotthickness}] (0,.5*\tilesize) to (.5*\tilesize-\tilesize*\cspace,.5*\tilesize);
    \draw[#1,knotstyle={\tilesize*\knotthickness}] (.5*\tilesize+\tilesize*\cspace,.5*\tilesize) to (\tilesize,.5*\tilesize);
  }
}
\newcommand{\tilebumpBotColor}[1][{}]{\tikz {\draw[tilestyle] (0,0) rectangle (\tilesize,\tilesize); \draw[red,knotstyle={\tilesize*\knotthickness}] (\tilesize,.5*\tilesize) to[out=180, in=90] (.5*\tilesize,0); \draw[#1,knotstyle={\tilesize*\knotthickness}] (0,.5*\tilesize) to[out=0, in=-90] (.5*\tilesize,\tilesize);}}
\newcommand{\tilebumpTopColor}[1][{}]{\tikz {\draw[tilestyle] (0,0) rectangle (\tilesize,\tilesize); \draw[#1,knotstyle={\tilesize*\knotthickness}] (\tilesize,.5*\tilesize) to[out=180, in=90] (.5*\tilesize,0); \draw[red,knotstyle={\tilesize*\knotthickness}] (0,.5*\tilesize) to[out=0, in=-90] (.5*\tilesize,\tilesize);}}

\tikzset{tilestyle/.style={use as bounding box, line width=\tileedge}} 
\tikzset{knotstyle/.style={line width=#1}} 
\tikzset{mosai/.style={execute at begin cell=\node\bgroup, execute at end cell=\egroup;, column sep={\tilesize cm,between origins}, row sep={\tilesize cm,between origins}, inner sep=0}}

\newtheorem{thm}{Theorem}[section]
\newtheorem{lem}[thm]{Lemma}
\newtheorem{cor}[thm]{Corollary}
\newtheorem{prop}[thm]{Proposition}

\newtheorem{egg}[thm]{Example}
\newtheorem{rmk}[thm]{Remark}
\newtheorem{defn}[thm]{Definition}
\newtheorem{nota}[thm]{Notation}

\newcommand{\precdot}{\prec\mathrel{\mkern-5mu}\mathrel{\cdot}}

\newcommand{\dem}{\mathrm{Dem}}

\newcommand{\mbb}{\mathbb}
\newcommand{\mcal}{\mathcal}
\newcommand{\mbf}{\mathbf}
\newcommand{\mrm}[1]{\mathrm{#1}}
\newcommand{\mfk}[1]{\mathfrak{#1}}
\newcommand{\ol}[1]{\overline{#1}}

\renewcommand{\P}{\mathbb{P}}

\newcommand{\C}{\mathbb{C}}

\newcommand{\Fl}{\mathrm{Fl}}
\newcommand{\R}{\mathbb{R}}
\newcommand{\Z}{\mathbb{Z}}
\newcommand{\B}{\mathcal{B}}

\newcommand{\mm}{m}
\newcommand{\q}{q}
\newcommand{\id}{e}

\newcommand{\Gr}{\mathrm{Gr}}
\DeclareMathOperator*{\conv}{conv}
\newcommand{\BS}{\mathrm{BS}}
\DeclareMathOperator*{\argmax}{argmax}
\newcommand{\Brick}{\mathrm{Brick}}
\newcommand{\BB}{\bm{\mathfrak{B}}}
\renewcommand{\b}{\ol{\mbf{b}}}

\usepackage{halloweenmath}
\newcommand{\atom}{\mathfrak{a}}
\newcommand{\boson}{\mathfrak{b}}
\newcommand{\conductivity}{\mathfrak{c}}
\newcommand{\dipole}{\mathfrak{d}}
\newcommand{\assumption}{\Upsilon}
\renewcommand{\pumpkin}{\mathcal{S}}
\newcommand{\LL}{\mathbb{L}}
\renewcommand{\mathcloud}{\mathcal{L}}

\makeatletter
\newcommand{\upsidedownwitch}{\mathord{\mathpalette\upsidedownwitch@\relax}}
\newcommand{\upsidedownwitch@}[2]{\raisebox{\depth}{\scalebox{1}[-1]{$\m@th#1\mathwitch$}}}
\makeatother

\renewcommand{\reversemathwitch}{\upsidedownwitch}

\renewcommand{\vec}[1]{\mathbf{#1}}

\newcommand\junk[1]{}

\newtoggle{comments}
\newtoggle{details}
\newtoggle{detailsnote}

\toggletrue{comments}   
\togglefalse{details}   

\toggletrue{detailsnote}   

\iftoggle{details}{%
	\newcommand{\details}[1]{
		\ \\
		{\color{OliveGreen}
			\textbf{Details:} #1
		}
		\\
	}
}{%
	\newcommand{\details}[1]{}
}

\usepackage[maxbibnames=9,backend=bibtex]{biblatex}
\begin{document}

\title{A cubulation of the associahedron using a \\ torus-equivariant degeneration of the brick variety}

\author[]{Raj Gandhi and Gabe Udell}

\address{
    Department of Mathematics \\
    Cornell University
}
\email{rg593@cornell.edu}

\address{
    Department of Mathematics \\
    Cornell University
}
\email{gru5@cornell.edu}

\keywords{brick polytope, associahedron, polyhedral subdivision, Richardson variety}

\begin{abstract}
We describe a mixed subdivision of the $(n-1)$-dimensional Loday associahedron into $(n-1)!$ maximal cells which are combinatorial cubes, and we show that the face poset of this subdivision is  isomorphic to the face poset of the subdivision of the $1\times 2\times \dotsb \times (n-1)$ box into unit cubes. More generally, we describe mixed subdivisions of generalized associahedra into pieces which are Minkowski sums of Bruhat interval polytopes. Our subdivisions arise from a torus-equivariant embedded (in a product of flag varieties) degeneration of  brick varieties into reduced unions of products of Richardson varieties. The subdivision of Loday's associahedron allows us to `embed' the 1-skeleton of the permutahedron in the boundary of the associahedron.
\end{abstract}

\maketitle

\tableofcontents

\section{Introduction}

The \textit{associahedron} $K_n$ is an $(n-1)$-dimensional polytope with: Catalan many vertices corresponding to all ways of completely parenthesizing a product of $n+1$ factors (e.g., the two vertices of $K_2$ correspond to $a\cdot (b\cdot c)$ and $(a\cdot b)\cdot c$), edges corresponding to single applications of the associative rule, and faces corresponding to partial parenthesizations (e.g. $a\cdot b\cdot (c\cdot d) \cdot e$). The face structure of the associahedron is then the poset of partial parenthesizations ordered by refinement (e.g. the vertex labeled by $(a\cdot (b\cdot (c\cdot d))) \cdot e$ would be contained in the face labeled by $a\cdot b\cdot (c\cdot d) \cdot e$). Tamari \cite{T51} studied the \textit{1-skeleton} of the associahedron and believed that the associahedron could be realized as a polytope, but an explicit polytopal realization was not discovered until much later.  Stasheff gave a realization of the associahedron as ``curvilinear polytopes'' in \cite{S63}, and the first polytopal realizations of $K_n$ were given in the 1980s independently by Haiman \cite{H84} and Lee \cite{L89}. We will focus in particular on the \textit{Loday realization} of the associahedron, described first by Shnider and Sternberg \cite{ShniderSternberg} and later by Loday \cite{L04}. Loday's description was by way of another Catalan object: ``full binary trees''. Just as Catalan objects appear throughout mathematics, so too does the associahedron. It is closely connected to homotopy theory, operads, Hopf algebras, cluster algebras, and even string theory.

A closely related polytope is the \textit{permutahedron}. The permutahedron $\mrm{Perm}_n$ is an $(n-1)$-dimensional lattice polytope in $\R^n$ constructed as the convex hull of the vectors $(\sigma(1),\dots, \sigma(n))$ for $\sigma\in S_n$. The ties between the permutahedron and the associahedron are numerous and deep. In this paper, we add new curiosities to the chronicle of their connections. One well-known connection is that the Loday realization is a ``generalized permutahedron'' \cite{postnikov}; that is, its \textit{normal fan} coarsens the normal fan of the permutahedron. As a consequence, one can identify collections of faces of the permutahedron to get the face structure of the associahedron. A strange byproduct of our  results is that we seem to be able to do the opposite: subdivide the faces of the Loday realization of the associahedron (with new vertices being inserted at integer coordinates) in such a way as to recover the face structure of the permutahedron. Doing so requires us to identify $n!-\frac{1}{n+1}\binom{2n}{n}$ new integer points on the boundary of the associahedron.
Just as the associahedron was a cell structure before it was realized explicitly as a polytope, this subdivision realizes an existing combinatorial story in an explicit geometric fashion.

Our main result is to explicitly construct a subdivision of the Loday realization of the associahedron $K_n$ into $(n-1)!$ pieces. The subdivision is combinatorially isomorphic to the unit cube subdivision of a $1\times 2 \times \cdots \times (n-1)$ box so each top-dimensional cell of the subdivision is combinatorially equivalent to a cube, and the 1-skeleton of the subdivision has $n!$ vertices. Each top-dimensional cell of the subdivision is a Minkowski sum of \textit{Bruhat interval polytopes}. We additionally introduce a new expression for Loday's associahedron as a Minkowski sum of Bruhat interval polytopes so that our cubulation of the associahedron is a  \textit{mixed subdivision} with respect to this Minkowski decomposition.

\begin{figure}
    \centering
    \begin{subfigure}{0.45\textwidth}
        \includegraphics[width=0.9\linewidth]{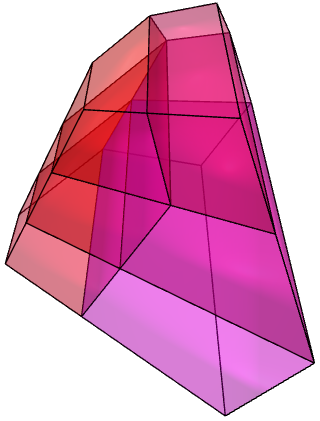}
    \end{subfigure}
    \hfill
\begin{subfigure}{0.45\textwidth}
    \includegraphics[width=0.9\linewidth]{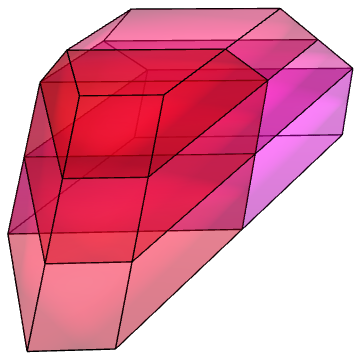}
\end{subfigure}
    \caption{Two views of the cubulation we construct of the 3D Loday associahedron}
\end{figure}

\begin{figure}
    \centering
    \includegraphics[width=0.4\linewidth]{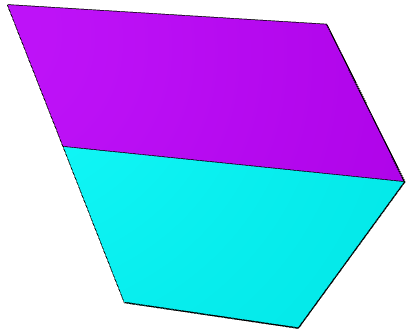}
    \caption{The cubulation we construct of the Loday realization of the 2D associahedron}
    \label{fig:placeholder}
\end{figure}

The subdivision of the permutahedron given by Harada--Horiguchi--Masuda--Park in \cite{HHMP} is \textit{also} isomorphic to the unit cube subdivision of a $1\times 2 \times \cdots \times (n-1)$ box (and therefore to the cubulation of the associahedron that we introduce). That subdivision, which we refer to as the \textit{HHMP subdivision},  has played a crucial role in recent work  \cite{quasiflag,bergeron_coxeter_2026,nadeau_geometry_2024} constructing and studying the ``quasisymmetric flag variety'', a union of toric varieties which has the ``quasisymmetric coinvariant ring'' as its cohomology ring. We are optimistic that there may be  an analog of the quasisymmetric flag variety that can be constructed using the subdivision of the Loday associahedron that we give in \cref{thm:main-result}, and that its cohomology ring will also be interesting.

We study Loday's associahedron through the more general framework of  \textit{brick polytopes} \cite{bricksort, PS15}.  Brick polytopes include not just the Loday realization of the associahedron, but all of the \textit{generalized associahedra} of Hohlweg, Lange and Thomas \cite{HLT}. We obtain similar subdivisions of all type A generalized associahedra which we conjecture to also all be composed of $(n-1)!$ cubes, but we do not undertake a detailed study of the combinatorics in these cases.

We discovered the subdivison of Loday's associahedron using toric geometry and our results about its combinatorics rely on geometric techniques. Every lattice polytope can be associated with a \textit{toric variety} embedded in projective space. In turn, every projective toric variety gives rise to a lattice polytope known as its \textit{moment polytope}. Sturmfels \cite{Sturmfels} showed that the ``regular'' subdivisions of a polytope correspond to ``Gr\"{o}bner degenerations''. Knutson, Sanchez, and Sherman-Bennett prove a more general incarnation of the same principle in \cite{AllenHammer} where they study generalizations of the HHMP subdivision of the permutahedron through degenerations of the permutahedral variety. They show that any \textit{torus-equivariant embedded degeneration} (\cref{defn:torus-equivariant-embedded-degeneration}) of a toric variety induces a subdivision of its moment polytope  into the moment polytopes of the irreducible components of the special fibre of the degeneration. Further, the face poset of the subdivision of the moment polytope is isomorphic to the poset of torus-orbit closures in the special fibre of the degeneration. Earlier work by Lian \cite{lian} studied the HHMP subdivision by constructing a different \textit{semitoric} degeneration of the permutahedral variety. In this paper we obtain and study subdivisions of brick polytopes by defining torus-equivariant embedded degenerations of their toric varieties.

Escobar \cite{E16} introduced \textit{brick varieties} and proved that 
their moment polytopes are brick polytopes. In particular, brick varieties include the toric varieties of generalized associahedra (including Loday's associahedron). Brick varieties are smooth subvarieties of products of Grassmannians which compactify ``braid varieties'', resolve \textit{Richardson varieties}, and are of interest in their own right.  
Traditionally, both brick varieties and brick polytopes are indexed by sequences of simple transpositions. We work with slightly more general versions of each that are indexed instead by sequences of permutations (that need not be simple transpositions). Unlike traditional brick varieties, these more general brick varieties are not necessarily smooth.

\begin{thm}\label{thm:main-result-2}(\cref{thm:degen})  
   Suppose $Q=(\q_1,\dots,\q_\mm)$ is a word in $S_n$ whose Demazure product is $w_0$ where $w_0$ is the permutation of $S_n$ with $w_0(i)=n+1-i$. Let \[L(Q)=\{(u_0,u_1,\dotsc,u_{\mm})\in (S_n)^{\mm+1}:u_0=\id,\text{ } u_{\mm}=w_0,\text{ }\ell(u_i\q_{i+1})=\ell(u_i)+\ell(\q_{i+1}),\text{ } u_{i+1}\leq u_i {\q_{i+1}},\text{ } i=0,\dotsc, \mm-1\}\] where $\ell(\sigma)$ denotes the length of a permutation $\sigma$. There is a torus-equivariant embedded (within $\mm+1$ copies of the complete flag variety $\Fl_n^{\mm+1}$) degeneration of the brick variety $\Brick^Q$ to the reduced union $\bigcup_{u_\bullet\in L(Q)}C_{u_\bullet}$, where \[C_{u_\bullet}=X_{u_0}^{u_0}\times X_{u_1}^{u_0 \q_1}\times X_{u_2}^{u_{1}\q_2}\times\dotsb\times X_{u_{\mm}}^{u_{\mm-1}\q_{\mm}}\] is a product of Richardson varieties in $\Fl_n$. 
\end{thm}

As a consequence of \cref{thm:main-result-2} we obtain the following equality of cohomology classes in $H^*(\Fl_n^{\mm+1})$:
\[[\mrm{Brick}^Q]=\sum_{u_{\bullet}\in L(Q)}[C_{u_\bullet}].\]

\subsection{Cubulating by thinking inside of the box}

The exact relationship between the HHMP subdivision and the subdivision described here is ultimately mysterious (and deeply fascinating) to us. Why do these two related sequences of polytopes have combinatorially isomorphic subdivisions? Our clearest understanding comes from \textit{Lehmer codes}, \textit{bracket vectors}, their connections, and their properties.

We define the \textbf{Lehmer code} of a permutation $\sigma\in S_n$ to be $\mbf{L}(\sigma)=(L_j(\sigma))_{j\in [n]}$ where $L_j(\sigma)=\#\{i<j:\sigma^{-1}(i)<\sigma^{-1}(j)\}$  \cite{Lehmer}. Lehmer codes give a bijection from $S_n$ to $\{0\}\times \{0,1\} \times \cdots \times \{0,1,\dots, n-1\}.$
A \textbf{full} (rooted) binary tree is a binary tree where every non-leaf vertex has two children. The \textbf{infix order} of the internal nodes of a binary tree is the total order on the vertices of the tree such that each node has a label which is greater than all its left descendants and less than all its right descendants. We define the \textbf{bracket vector} of a full binary tree $T$ with $n$ internal vertices to be $\mbf{b}(T)=(b_j(T))_{j\in [n]}$ with $b_j(T)$ being one less than the number of leaves in the left subtree of the $j$th vertex of $T$ in the infix order. See \cref{egg:bracket-vector} for an example of a full binary tree together with its bracket vector. Each bracket vector is contained in $\{0\}\times \{0,1\} \times \cdots \times \{0,1,\dots, n-1\}.$ 
    
Now is a convenient time to interject that the \textbf{Loday realization} of the associahedron is obtained by setting the $i$th coordinate of the vertex associated to a full binary tree to be $(a_i(T)+1)\cdot (b_i(T)+1)$, where $a_i(T)$ is one less than the number of leaves in the right subtree of the $i$th vertex of $T$ in the infix order.

When appropriately oriented, the edges of the permutahedron and the associahedron correspond to the cover relations of  the ``weak Bruhat order'' and the ``Tamari lattice''. We will not discuss these lattices beyond this section, but they help us to contextualize the cubulations we are interested in.

 We will now state two technical definitions from a paper of Pilaud and Poliakova \cite{PP25} that help explain how Lehmer codes and bracket vectors are related to the permutahedron and associahedron.
    \begin{defn}[\cite{PP25}, Definition 68]\label{defn:boxy-subdivision}
         A \textbf{box} $C$ is an axis-aligned parallelepiped in $\mbb{R}^n$. If $\mathbf{x},\mathbf{y}\in\mbb{R}^n$ are such that $x_i\leq y_i$ for all $i\in\{1,\dotsc,n\}$, then we will denote by $\mrm{Box}(\mathbf{x},\mathbf{y})$ the box $\prod_{i\in\{1,\dotsc,n\}}[x_i,y_i]$. A \textbf{subbox} of $C$ is any box included in $C$ whose vertices lie on the boundary of $C$. A \textbf{boxy subdivision} of $C$ is a collection $\mcal{D}$ of subboxes of $C$ such that
         \begin{itemize}
             \item the boundary of $C$ is the union of all subbox in $\mcal{D}$; and
             \item for any subbox $C',C''\in\mcal{D}$, the intersection $C'\cap C''$ is either: empty, equal to  $C'$, equal to $C''$, or is a subbox in $\mcal{D}$ with $\mrm{dim}(C'\cap C'')< \mrm{min}(\mrm{dim}(C'),\mrm{dim}(C''))$.
         \end{itemize}
         The \textbf{subbox poset} of the boxy subdivision $\mcal{D}$ is the poset of $\mcal{D}\cup\{C\}$ ordered by inclusion.
    \end{defn}

\begin{defn}[\cite{PP25}, Definition 69]\label{defn:boxy-realization}
A \textbf{boxy realization} of a poset $P$ is a map $\gamma:P\to \R^d$ such that $\gamma(P)$ lies on the boundary of $\mrm{Box}(\gamma(\min(P)),\gamma(\max(P)))$ and for any cover relation $p\lessdot q$ in $P$, $\gamma(p)-\gamma(q)$ is a positive multiple of some standard basis vector $e_i$.
\end{defn}

In \cite{PP25} Pilaud and Poliakova reproduce proofs from \cite{SU04} that: 
\begin{enumerate}
    \item Lehmer codes give a boxy realization of the weak order lattice,
    \item $\{\mrm{Box}(\mbf{L}(\sigma),\mbf{L}(\tau)):\sigma\leq \tau \text{ defines a face of the permutahedron}\}$ is a boxy subdivision with subbox poset isomorphic to the \textit{face lattice} of the permutahedron
    \item bracket vectors give a boxy realization of the Tamari lattice,
    \item $\{\mrm{Box}(\mbf{b}(S),\mbf{b}(T)):S\leq T \text{ defines a face of } K_n\}$ is a boxy subdivision with subbox poset isomorphic to the face lattice of the associahedron.
\end{enumerate}

Recently \cite{cubical}, Berggren and Chenevi\`ere found boxy realizations for all ``framing lattices'' which generalize weak order, the Tamari lattice, the Boolean lattice, type A ``Cambrian lattices'', the ``alt-$\nu$-Tamari lattices'', ``permutree lattices'', and more.

\begin{figure}
    \centering
    \includegraphics[scale=1.15]{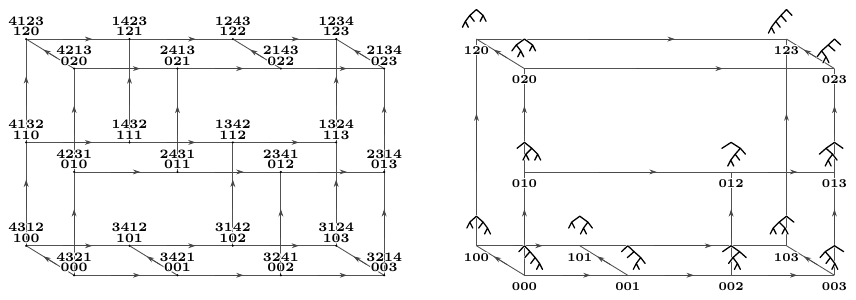}
    \caption{The boxy realizations of the weak order (on the left) and the Tamari lattice (on the right) for $S_4$. Figure is from \cite{PP25}, reproduced with permission.}
    \label{fig:placeholder0}
\end{figure}

\begin{figure}
    \centering
    \includegraphics[scale=1.7]{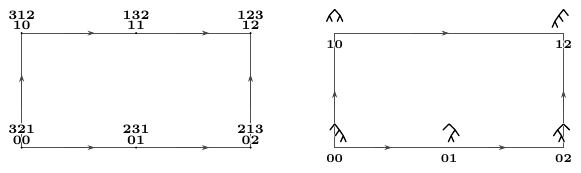}
    \caption{The boxy realizations of the weak order (on the left) and the Tamari lattice (on the right) for $S_3$. Figure is from \cite{PP25}, reproduced with permission.}
    \label{fig:placeholder2}
\end{figure}

The $1\times 2 \times \cdots \times (n-1)$ box containing the Lehmer codes and bracket vectors has a cubulation by $1\times \cdots \times 1$ unit (hyper)cubes, which we will call the \textbf{standard cubulation}. Since the subbox poset of the Lehmer code boxy subdivision matches the face lattice of the permutahedron, we might imagine perturbing the Lehmer codes by small amounts to obtain a legitimate convex realization of the permutahedron and  taking the cubulation along for the ride. As the set of Lehmer codes for $S_n$ is in bijection with $S_n$, and as vertices of the permutahedron correspond to elements of $S_n$, one can hope to construct the cubes of the desired subdivision of the permutahedron as convex hulls of vertices in the permutahedron that correspond to Lehmer codes defining the vertices of the standard cubulation. A priori, one should not expect this to work (e.g. the convex hull of four points might not be a combinatorial square, and the convex hull of eight points might not be a combinatorial cube), but \cite{HHMP} showed that this identification of vertices of the permutahedron with Lehmer codes does induce a cubulation of the permutahedron whose face lattice is isomorphic to the face lattice of the standard cubulation.

\begin{figure}
    \centering
    \includegraphics[scale=1.45]{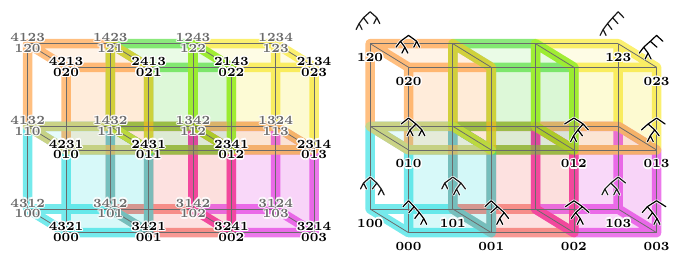}
    \caption{Standard cubulation overlayed with Lehmer code and bracket vector boxy subdivisions for $S_4$.}
    \label{fig:placeholder3}
\end{figure}

\begin{figure}
    \centering
    \includegraphics[scale=1.6]{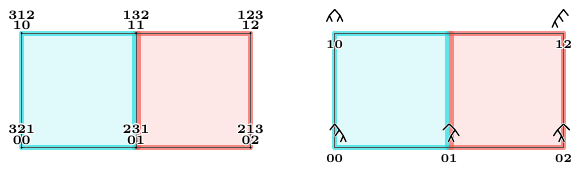}
     \caption{Standard cubulation overlayed with Lehmer code and bracket vector boxy subdivisions for $S_3$.}
    \label{fig:placeholder4} 
\end{figure}

We similarly might hope that the cubulation can tag along in the move from bracket vectors to Loday's realization of the associahedron. This is even more dubious because we need to identify $n!-\frac{1}{n1}\binom{2n}{n}$ integer points on the boundary which are not themselves vertices \textit{and} they need to have the appropriate convex structure to form nonoverlapping cubes which subdivide the associahedron. Surprisingly, our tools from toric geometry allow us to realize this cubulation.

To prove that our cubulation of the associahedron has the same combinatorics as the standard cubulation of the bracket vector boxy subdivision, we construct a bijection between the faces of the standard cubulation and the faces of our subdivision and prove that \textit{complemented bracket vectors} are mapped to the original vertices of the associahedron appropriately. Since the 1-skeleton of the permutahedron is a subgraph of the 1-skeleton of the standard cubulation, we can drop edges from the cubulation of the associahedron to obtain a drawing of the 1-skeleton of the permutahedron on the boundary of the associahedron. We conjecture this construction yields a dissection of the faces of the associahedron which has a face lattice isomorphic to that of the permutahedron. The 3D instance of this associahedron in permutahedron clothing is shown in \cref{fig:crazyweak}.

\begin{figure}
    \centering
    \includegraphics[width=0.49\linewidth]{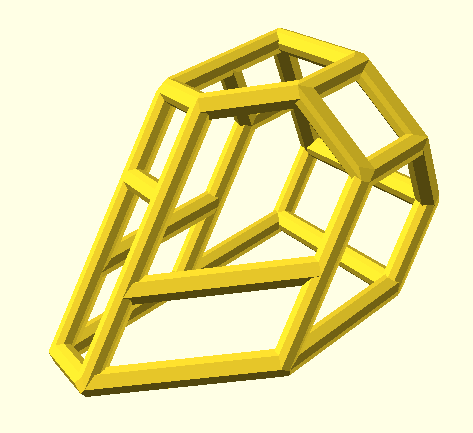}
    \includegraphics[width=0.49\linewidth]{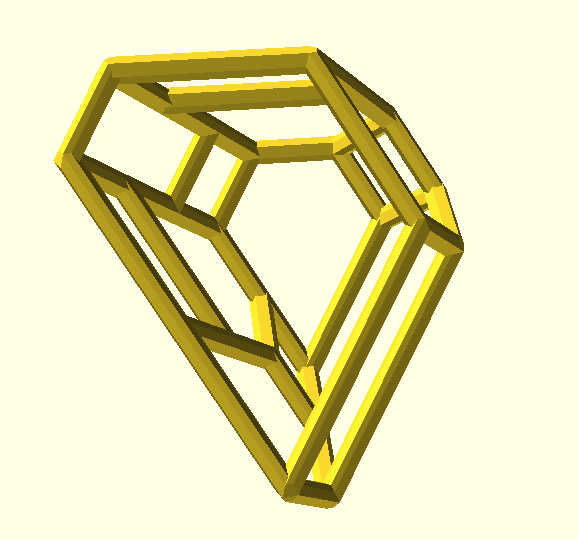}
    \caption{The weak order graph drawn on the boundary of the associahedron}
    \label{fig:crazyweak}
\end{figure}

\subsection{Concrete construction of the cubulation}

The HHMP subdivision of the permutahedron and the cubulation of the associahedron are both constructed using \textit{(twisted) Bruhat interval polytopes}. Bruhat interval polytopes were introduced in \cite{Kostant-Toda,TW15} and have since been studied by many authors \cite{benedetti2024lattice, LMP21, gorsky2026unexpected, gaetz2026poset, gaetz2023one}. For permutations $u,v\in S_n$, with $u\leq v$ in strong Bruhat order, the \textbf{twisted Bruhat interval polytope} $\tilde{P}_{u,v}$ is the convex hull of the points $(n+1-w^{-1}(1),n+1-w^{-1}(2),\dotsc,n+1-w^{-1}(n))$ over all $u\leq w\leq v$. 

We write $\id$ for the identity permutation, $s_i$ for the permutation transposing $i$ and $i+1$, and $w_0$ for the permutation transposing $i$ and $n+1-i$ for each $i$. Noting that the permutahedron is the twisted Bruhat interval polytope $\tilde{P}_{\id,w_0}$, the HHMP subdivision of the permutahedron can be written as
\[\tilde{P}_{\id,w_0} = \bigcup_{u\in S_{n}: u(n)=n} \tilde{P}_{u,us_{n-1}s_{n-2}\ldots s_{1}}.\]
Each of the $(n-1)!$ twisted Bruhat interval polytopes $\tilde{P}_{u,us_{n-1}s_{n-2}\ldots s_{1}}$ has the combinatorial structure of an $n-1$ dimensional cube, and they fit together to give a subdivision with the same combinatorial structure as the standard cubulation of the $1\times 2 \times \cdots \times (n-1)$ box.

\begin{thm}[\cref{lem:spackle-poset-iso}]\label{thm:main-result}
    Set $\sigma_1:=s_1s_2\dots s_{n-1}$, and $\sigma_i:=s_1s_2\dots s_{n-i+1}$ for $i\geq 2.$ Set $\alpha_0=\alpha_1:=e$, and $\alpha_i:=\sigma_2\sigma_3\dotsb \sigma_i$ for $i=2,\dotsc,n$. Let $L(Q)=\{(\id,v_1\alpha_1,\dots, v_{n}\alpha_{n}):v_i\in S_n, v_{i+1}\leq v_{i}, v_i\leq s_{i+1}s_{i+2}\dots s_{n-1} \text{ for } i\geq 1\}$ where $\leq$ denotes the strong Bruhat order relation. For $u_\bullet=(u_0,\dots, u_n)\in L,$ define the polytope $C_{u_\bullet}$ as the Minkowski sum of twisted Bruhat interval polytopes
    \[C_{u_\bullet}=\sum_{i=1}^{n}\tilde{P}_{u_i}^{u_{i-1}\sigma_i}.\]
     Then, after a translation and negation (i.e. multiplication by $-1$), the Loday realization of the associahedron $K_n$ has the subdivision 
    \[\bigcup_{u_\bullet \in L} C_{u_{\bullet}}. \]
    The polytopes $C_{u_\bullet}$ intersect at common faces, and each is combinatorially isomorphic to an $n-1$ dimensional hypercube. This subdivision overall is combinatorially isomorphic to the standard cubulation of the $1\times \cdots \times (n-1)$ box. In particular, the subdivision has $n!$ vertices and has $(n-1)!$ top dimensional cells.
\end{thm}

Along the path to proving \cref{thm:main-result}, we find that the terms of a well-known Minkowski sum representation for Loday's associahedron can be grouped together into twisted Bruhat interval polytopes. 

\begin{thm}[\cref{Loday mink sum}]\label{Loday mink sum intro}
    Let $\sigma_1=s_1s_2\dots s_{n-1}$ and $\sigma_i=s_1s_2\dots s_{n-i+1}$ for $i\geq 2.$  Let $\alpha_i=\sigma_2\sigma_3\cdots\sigma_i.$ 
    Then, after a translation and negation, the Loday associahedron has the following Minkowski decomposition into twisted Bruhat interval polytopes, each of which is combinatorially isomorphic to a hypercube:
    \[\sum \tilde{P}_{\alpha_i,\alpha_{i+1}}. \]
\end{thm}

We will later prove that the subdivision in \cref{thm:main-result} is a \textit{mixed subdivision} with respect to this Minkowski decomposition.

\subsection{Plan for this paper}\label{subsection:plan} In \cref{section:brick-polytope}, we recall preliminaries about polytopes, partially ordered sets, and permutations including: Bruhat interval polytopes, the Demazure product, subword complexes, and brick polytopes. The definition of the brick polytope we give is slightly more general than the one appearing in the literature, and we introduce a second realization of it which we call the \textit{bulky brick polytope}. In \cref{subsection:coincide-brick} we discuss how the generalized brick polytopes and bulky brick polytopes relate to each other and to traditional brick polytopes.

In \cref{section:pipe-dreams-bracket-vectors}, we introduce the specific sequences $Q$ of permutations such that the brick polytope associated with $Q$ is a generalized associahedron. For the case of the Loday associahedron we describe in \cref{subsection:pipe-dreams} how the facets of the subword complex can be visualized as special types of ``pipe dreams'' called $(1,n)$-twists, and we explain how (a translation and reflection of) Loday's associahedron can be computed directly from $(1,n)$-twists. 

With the connection between $(1,n)$-twists and the vertices of the associahedron in hand, we are able to explore how the cubulation in \cref{thm:main-result} relates to bracket vectors. In order to meaningfully claim that the cubulation in \cref{thm:main-result} realizes the relationship between the bracket vector boxy subdivision and the standard cubulation, it is not enough to show that the standard cubulation is isomorphic to the cubulation in \cref{thm:main-result}. We would need to further show that this isomorphism maps the bracket vector of a binary tree $T$ to the vertex of the  associahedron associated with $T$. We instead show that the isomorphism maps the \textit{complemented bracket vector} of a binary tree $T$ to the vertex of the associahedron associated with $T$ where the complemented bracket vector is related to the regular bracket vector by the antipodal symmetry of the $1\times 2\times \cdots \times (n-1)$ box. In \cref{subsection:bracket-vectors} we introduce the complemented bracket vector of a binary tree and the \textit{bump vector} of a $(1,n)$-twist, and show that the complemented bracket vector map and the bump vector map form a commuting triangle with a well known bijection between $(1,n)$-twists and binary trees. In \cref{subsection:cubulation}, we give an explicit description of the isomorphism between the vertices of the $1\times 2 \times \cdots \times (n-1)$ box and the vertices of the cubulation (although we will not have the tools to prove that these vertices define the subdivision of \cref{thm:main-result} until \cref{section:subdivision}) and show that this map takes the bump vector of a $(1,n)$-twist to the vertex of the associahedron associated with that $(1,n)$-twist.

In \cref{section:moment-polytopes}, we recall the moment polytope of an algebraic variety that is embedded torus-equivariantly in a projective space, and we will recall the Grassmannian, flag variety, Richardson variety, and brick variety, together with their moment polytopes with respect to certain projective embeddings. We prove a rather useful folk theorem (\cref{lem:MinkowskiMachine}) that says that the moment polytope of a subvariety of a product is the Minkowski sum of the moment polytopes of the projections onto the factors of the product. We use this to express the bulky brick polytope as a Minkowski sum of twisted Bruhat interval polytopes.  Finally, we will explain that in certain cases, the brick variety can be viewed as a resolution of singularities for a Richardson variety, and we will give a pattern-avoidance condition that implies when the brick resolution is isomorphic to the Richardson variety whose singularities it resolves.

In \cref{section:degeneration}, we describe a torus-equivariant embedded (in $\Fl_n^2$) degeneration of the variety of pairs of complete flags in $\C^n$ with \textit{relative position} less than or equal to $\sigma$ for $\sigma\in S_n$ (\cref{sigmadegen}). We use this to construct a torus-equivariant embedded (in $\Fl_n^{n+1}$) degeneration of the brick variety and prove \cref{thm:main-result-2}, which says that the brick variety degenerates into a reduced union of products of Richardson varieties. We provide more geometric background than is customary in the hopes that the arguments will be accessible to graduate students and to combinatorialists with some foundation in algebraic geometry.  

In \cref{section:subdivision}, we introduce the \textit{spackle poset} to capture the combinatorial structure of the degeneration of the brick variety. When the brick variety is toric, we show (\cref{subsubsection:spackle-poset}) this poset is isomorphic to the face poset of the subdivision of the bulky brick polytope induced by the degeneration of \cref{thm:main-result}, and also isomorphic to the poset of torus-orbit closures in the special fiber of the degeneration of the brick variety of \cref{thm:main-result-2}.
We also show that the subdivision of the bulky brick polytope is a mixed subdivision with respect to the Minkowski decomposition of the bulky brick polytope into twisted Bruhat interval polytopes. In \cref{subsection:volumes}  we compute the volumes of the full-dimensional cells of the subdivision of the brick polytope using a relationship between the Hilbert polynomial and Ehrhart polynomial of the brick variety with respect to its embedding in $\Fl_n^{m+1}$. 

In \cref{witch} we study the spackle poset associated with the subdivision of Loday's associahedron and give an explicit isomorphism between it and the \textit{face poset} of the standard cubulation of the $1\times 2\times \dotsb \times(n-1)$ box.

\section*{Acknowledgements}

The authors thank their advisor Allen Knutson for proposing the problem of indexing the components of $\Brick^Q(0)$, and for his ongoing insights and support. We thank Mario Sanchez for various crucial insights and note that the expression for certain brick polytopes as Minkowski sums of Bruhat interval polytopes was known to Mario Sanchez and Melissa Sherman-Bennett by other means. We thank Spencer Backman for suggesting we look for a Minkowski decomposition for the bulky brick polytope. We thank Vincent Pilaud for discussing this paper at multiple points and telling us about boxy subdivisions and boxy realizations. We also thank Ben Adenbaum, Laura Escobar,  Christian Gaetz, Karola M\'esz\'aros, Melissa Sherman-Bennett, and Vasu Tewari for helpful conversations. R.G. was partially supported by a Postgraduate Scholarship - Doctoral - from the Natural Sciences and Engineering Research Council of Canada. G.U. was supported by the National Science Foundation Graduate Research Fellowship under grant DGE~--~2139899.

\section*{AI statement}

The authors first started thinking about the ideas pertaining to this article in Spring of 2023. We presented early findings of this work (which can be found in this FPSAC 2025 abstract: \cite{GU25}) at the Canadian Mathematical Society meeting in December of 2023, and we presented later findings of this work at the Field's Institute in May of 2025 and also the Joint Mathematics Meetings in January of 2026.

In the preparation of this manuscript, the authors used generative AI in the following forms. We used OpenAI's GPT-6 Luna and Claude Opus 5.5 to generate the following figures: \cref{fig:placeholder0}, \cref{fig:placeholder2}, \cref{fig:placeholder3}, \cref{fig:placeholder4}, \cref{fig:PDBrick}, \cref{fig:lehmerEx}. Furthermore, we asked both OpenAI's GPT-6 Luna and Claude Opus 5.5 to generate referee reports for our paper, and we made modifications and corrections to this work based on those referee reports. Moreover, OpenAI's GPT-5.4 Mini and Google's Gemini were consulted while working out kinks in certain algebraic arguments.

\section{The brick polytope}\label{section:brick-polytope}
    In this section, we will introduce the main objects that we study in this paper: \textit{brick polytopes} and \textit{bulky brick polytopes}. We will give a brief overview on partially ordered sets and polyhedral geometry in \cref{subsection:polytopes}. Notions surrounding symmetric groups, Bruhat interval polytopes, the Demazure product, and subword complexes will be recalled and introduced in \cref{subsection:preliminaries} and \cref{subsection:subword}. In \cref{subsection:brick-polytopes}, we will recall brick polytopes and introduce bulky brick polytopes. We will introduce criteria that will ensure the brick polytope and bulky brick polytope coincide (up to translation) in certain cases in \cref{subsection:coincide-brick}.

    \subsection{Partially ordered set and polytopes}\label{subsection:polytopes}

   We will begin by discussing some preliminaries on partially ordered sets, before moving on to polytopes. We refer the reader to \cite{Ziegler} for further details.

    \subsubsection{Partially ordered sets}

    A \textbf{partially ordered set} or \textbf{poset} is a pair $(X,\preceq)$, where $X$ is a set and $\preceq$ is a partial order relation (that is, a reflexive, antisymmetric, transitive relation) on $X$. Posets arise in several ways in this paper, so we take a moment to give some basic definitions that will be helpful later on. 

    \begin{defn}
        Let $(A,\preceq_{A})$  and $(B,\preceq_{B})$ be posets and $f:A\to B$ some map.
        \begin{itemize}
            \item We call $f$ \textbf{monotone} when $x\preceq_A y$ implies $f(x) \preceq_B f(y)$.
            \item We call $f$ an \textbf{order embedding} when $x\preceq_A y$ if and only if $f(x) \preceq_B f(y)$. (Observe that order embeddings are always injective).
            \item We call $f$ an \textbf{order isomorphism}, and say that $(A,\preceq_{A})$  and $(B,\preceq_B)$ are \textbf{isomorphic as posets}, when $f$ is a bijective order embedding.
        \end{itemize}
    \end{defn}

    A \textbf{lattice} is a poset such that any two elements of the poset have a greatest lower bound and a least upper bound. We are not particularly focused on properties of lattices, but we will meet\footnote{no pun intended} several lattices later on. A \textbf{boolean lattice} of rank $n\geq 1$ is a poset which is isomorphic to the poset of subsets of $[n]:=\{1,\dotsc,n\}$, ordered by inclusion.

    \subsubsection{Polyhedral geometry}
    
    A \textbf{polyhedron} in $\mbb{R}^n$ is the intersection of finitely many closed half-spaces in $\mbb{R}^n$. 
    A \textbf{polytope} is a polyhedron which is bounded. Equivalently, a \textbf{polytope} in $\mbb{R}^n$ is the convex hull of finitely many points in $\mbb{R}^n$. We will denote the convex hull of a set $X$ by $\conv X$ throughout this article. 
    The \textbf{faces} of a polyhedron $P$ are the intersections of $P$ with subsets of the set of hyperplanes defining the closed half-spaces that bound $P$. The \textbf{face lattice} of a polyhedron is its set of faces ordered by inclusion (and it is indeed a lattice). Two polyhedra are \textbf{combinatorially equivalent} if they have isomorphic face lattices. 
   
   The \textbf{vertices} of a polytope are its faces of dimension $0$, its \textbf{edges}  are its faces of dimension $1$, and its \textbf{facets} are its faces of codimension $1$. The \textbf{one-skeleton} of a polytope $P$ is the graph $G$, whose vertices correspond to the vertices of $P$, and where two vertices of $G$ are connected by an edge whenever the corresponding vertices of $P$ lie on a common edge of $P$. The polytope $P$ is a \textbf{lattice polytope} if its vertices all lie on integer points in $\mbb{R}^n$. 
   \begin{defn}\label{defn:subdivision}
     Let $P$ be a polytope in $\mbb{R}^m$. A \textbf{subdivision} of $P$ is a collection of polytopes $C_1,\dots, C_k$ in $\mbb{R}^m$ such that $P=\bigcup C_i$, and for $i\neq j$, $C_i\cap C_j$ is a common face of $C_i$ and $C_j$. The \textbf{face poset} of the subdivision is the set of \textbf{cells} $\{F:F\text{ is a face of } C_i \text{ for some } i\}$ of the subdivision ordered by inclusion. We say that two subdivisions are \textbf{combinatorially isomorphic} if their face posets are isomorphic as posets. The \textbf{vertices} of the subdivision are the $0$-dimensional cells of the subdivision.
\end{defn}

\begin{defn}\label{defn:Minkowski-sum-mixed-subdivision}
    The \textbf{Minkowski sum} of two polytopes $P$ and $Q$ in $\mbb{R}^m$ is the polytope $P+Q:=\{p+q\mid p\in P,\text{ }q\in Q\}$ in $\mbb{R}^m$. A \textbf{full dimensional Minkowski cell} of a Minkowski sum $P=\sum_{i=1}^n P_i$ of polytopes $P_i$ in $\mbb{R}^m$ is a Minkowski sum $C=\sum_{i=1}^n C_i$ of polytopes $C_i$ in $\mbb{R}^m$ such that $\dim(C)=\dim(P)$ and $C_i$ is a face of $P_i$ for each $i$. A \textbf{mixed subdivision} of $P$ is a subdivision $P=\bigcup C_j$ of $P$ into polytopes $C_j$ in $\mbb{R}^m$ such that each $C_j$ is a full dimensional Minkowski cell of $P$. 
\end{defn}

     A \textbf{polyhedral complex} is a set $\mbb{S}$ of polyhedra so that if $P\in \mbb{S}$ then every face of $P$ is in $\mbb{S}$ and the intersection of any two polyhedra $P,Q \in \mbb{S}$ is a face of both $P$ and $Q$. Polyhedral complexes have a natural ordering by inclusion of faces, and we say that two polyhedral complexes are \textbf{combinatorially equivalent} if they are isomorphic as posets.

     A \textbf{cone} in $\mbb{R}^n$ is the $\mbb{R}_{\geq 0}$-span of finitely many vectors in $\mbb{R}^n$. A \textbf{fan} is a polyhedral complex where every polyhedron is a cone.
     Given a polytope $P$ in $\mbb{R}^n$, the \textbf{normal cone} to a vertex $v$ of $P$ is the  cone $C_v(P):=\{w\in \mbb{R}^n\mid v\in\mrm{argmax}_{x\in P}\text{ } (x\cdot w)\}$ in $\mbb{R}^n$, where $x\cdot w$ denotes the standard inner product of $x$ and $w$ in $\mbb{R}^n$. The \textbf{normal fan} of $P$ is the collection of normal cones to vertices, together with all faces of those cones.
     The \textbf{interior} of $C_v(P)$ is the set $C_v(P)^\circ=\{w\in \R^{n}: v= \argmax_{x\in P}x\cdot w \}$. Note that for $P$ a (not necessarily full-dimensional) polytope in $\R^n$ and $v$ a vertex of $P$, the normal cone $C_v(P)$ has dimension $n$.

    \subsection{Permutations and words}\label{subsection:preliminaries}
   Next, we will recall definitions and properties of the symmetric groups, the Bruhat order, Bruhat interval polytopes and the Demazure product.

   \subsubsection{Permutations and Bruhat interval polytopes} Throughout this paper we write $[n]$ for the set $\{1,2,\dots, n\}$ and $S_n$ for the symmetric group $S_n$ of permutations of $[n]$. We will write $\id$ to denote the identity element of $S_n$, $s_i$ to denote the simple transposition in $S_n$ that swaps $i$ and $i+1$ and fixes all other $j\in [n]$, and $w_0$ for the permutation $w_0\in S_n$ that sends $i$ to $n+1-i$ for all $i$. 
    The \textbf{one-line notation} for a permutation $v\in S_n$ is $i_1 i_2\dotsb i_n$, where $v(j)=i_j$. For example, the permutation $v=231$ in $S_3$ is the one-line notation for the permutation sending $1$ to $2$, $2$ to $3$, and $3$ to $1$. The one-line notation for $w_0$ is $w_0=n\quad n-1 \quad \cdots \quad 3\quad 2\quad 1$. 
    
    Any element $\sigma\in S_n$ can be expressed $\sigma=s_{i_1}\dotsb s_{i_{\mm}}$, where the $s_{i_j}$ are simple transpositions in $S_n$. The \textbf{length} $\ell(\sigma)$ of an element $\sigma\in S_n$ is the minimal $\mm$ over all expressions $\sigma=s_{i_1}\dotsb s_{i_{\mm}}$ of $\sigma$ as a product of simple transpositions $s_{i_j}$ in $S_n$. 
    
     A \textbf{word} in $S_n$ is a sequence of permutations $Q=(\q_{1},\dotsc,\q_{\mm})$ with $\q_{i}\in S_n$ for each $i$. We say that $Q$ is a \textbf{word in simple transpositions} if all $\q_{i}$ are simple transpositions.
    We will use the notation $Q_{(k)}:=\q_1\q_2\cdots\q_{k}$ for the product of the first $k$ entries of $Q$ with $Q_{(0)}:=\id$.
    For $v\in S_n$, we say that $Q$ is a \textbf{word for $v$} if $v=Q_{(\mm)}$.
    We say $Q$ is a \textbf{reduced word} if $\ell(\q_{1}\dotsb \q_{\mm})=\ell(\q_1)+\cdots+\ell(\q_{\mm})$.
    The \textbf{Bruhat order} $\leq$ on $S_n$ is the partial order on $S_n$ in which $u\leq v$ if and only if, for some (equivalently, any) reduced word in simple transpositions $Q$ for $v$, there is a subsequence of $Q$ that is a reduced word for $u$. 
     We define a \textbf{subword} of a word $Q=(\q_{1},\dotsc,\q_{\mm})$ to be a word $(\q_1',\q_2',\dots, \q_{\mm}')$ such that $\q_i'\leq \q_i$ for each $i$. 
     When $Q$ is a sequence of simple transpositions, this is equivalent to defining a subword of $Q$ to be a subsequence of $Q$.

    For permutations $u\leq v\in S_n$, the \textbf{Bruhat interval} between $u$ and $v$ is $[u,v]=\{w\in S_n:u\leq w\leq v\}.$ When $u\not \leq v,$ we let $[u,v]$ be the empty set. All permutations are less than or equal to $w_0$ and greater than or equal to $\id$,  so $S_n=[e,w_0].$

    We will make use of the \textit{lifting property} for the Bruhat order on $S_n$ later in this section, which we record below. 

\begin{prop}[Lifting property in $S_n$](\cite[Proposition 2.2.7]{BB05})\label{prop:lifting-property}
    Let $v,w\in S_n$ be such that $v\leq w$, and suppose $s_i$ is a simple transposition in $S_n$ such that $\ell(w s_i)<\ell(w)$ and $\ell(v s_i)> \ell(v)$. Then $v\leq w s_i$ and $v s_i\leq w$.
\end{prop}

 The \textit{Bruhat interval polytope} and the \textit{twisted Bruhat interval polytope} are lattice polytopes that contain information about the Bruhat order. We will now define the (twisted) Bruhat interval polytope in \cref{defn:bruhat-interval-polytope} and explain the interpretation of (twisted) Bruhat intervals as \textit{moment polytopes of Richardson varieties} in \cref{rmk:why-toric} below.

    \begin{defn}[\cite{flagPositroids, Kostant-Toda, TW15}]
    \label{defn:bruhat-interval-polytope}
    For permutations $u$ and $v$ in $S_n$, the \textbf{Bruhat interval polytope} $P_{u,v}$ is the convex hull of the points $(w(1),w(2),\dotsc,w(n))$ over all $u\leq w\leq v$. The \textbf{twisted Bruhat interval polytope} $\tilde{P}_{u,v}$ is the convex hull of the points $(n+1-w^{-1}(1),n+1-w^{-1}(2),\dotsc,n+1-w^{-1}(n))$ over all $u\leq w\leq v$. A Bruhat interval polytope $P_{u,v}$ (resp. twisted Bruhat interval polytope $\tilde{P}_{u,v}$) is \textbf{toric} if $\dim P_{u,v}=\ell(v)-\ell(u)$ (resp. $\dim \tilde{P}_{u,v}=\ell(v)-\ell(u)$).
    \end{defn}

Note that $\tilde{P}_{u,v}=P_{w_0v^{-1},w_0u^{-1}}$ so the set of Bruhat interval polytopes is the same as the set of twisted Bruhat interval polytopes. When $P_{u,v}$ is toric, so are $P_{u^{-1},v^{-1}}$ and $\tilde{P}_{u,v}$, and vice versa \cite{LMP21}. The \textbf{permutahedron} is the (twisted) Bruhat interval polytope $P_{e,w_0}=\tilde{P}_{e,w_0}$.

\begin{rmk}\label{rmk:why-toric}
    As we will explain in \cref{background}, we can associate a lattice polytope, called a moment polytope, to every algebraic variety with a torus action that is torus-equivariantly embedded in a projective space. If the dimension of such a variety equals the dimension of its moment polytope, then the variety is a toric variety and its moment polytope is called toric. We will see in \cref{lem:Richardson sum} (cf. \cite{Kostant-Toda}) that the twisted Bruhat interval polytope $\tilde{P}_{u,v}$ is the \textit{moment polytope} of the Richardson variety $X_u^v$ when $v\geq u$ in $S_n$. As $\mrm{dim}(X_u^v)=\ell(v)-\ell(u)$ when $v\geq u$, it follows that $\tilde{P}_{u,v}$ is toric when $\dim \tilde{P}_{u,v}=\ell(v)-\ell(u)$
\end{rmk}

    \subsubsection{Demazure products} We first define the \textbf{Demazure product} of a permutation $w\in S_n$ and a simple transposition $s_i\in S_n$ as
    \[w\star s_i = \begin{cases}
            w \cdot s_i  & \text{if $\ell(w\cdot s_i)>\ell(w)$,}\\
            w & \text{otherwise.}
            \end{cases}
    \]
    This product respects the ``braid relations'' in the sense that  $(w\star s_i) \star s_j = (w\star s_j) \star s_i$ for $|i-j|>1$, and $((w\star s_i) \star s_{i+1}) \star s_i = ((w\star s_{i+1}) \star s_{i}) \star s_{i+1}$ so it can be extended to a product  $S_n\times S_n\to S_n$ by defining $w\star v= ((\cdots ((w\star s_{i_1})\star s_{i_2}) \cdots)\star s_{i_k}) $ where $s_{i_1}\cdots s_{i_k}$ is any reduced expression for $v$. The Demazure product is an associative product. One nice way to think about it is that $w\star v$ is through the fact that the Bruhat interval $[\id,w\star v]$ is equal to $\{a\cdot b: a\in [\id,w],\ b\in [\id,v]\}$ (see e.g. \cite{LOYY24,He}). For $Q=(q_1,\dots,q_{\mm})$ a word in $S_n$, we define
    $\dem(Q)=q_1\star q_2 \star \cdots \star q_{\mm}.$

We will often make use of two lemmas, \cref{lem:dem-assoc} and \cref{lem: demlem}, concerning Demazure products, which we will record and prove below.

\begin{lem}\label{lem:dem-assoc}
    For all $a,b,c\in S_n$, we have $(ab)\star c\leq a \star b\star c$.
\end{lem}
\begin{proof}
    This follows from the fact that $(ab)\star c\leq (a \star b) \star c=a\star b\star c$ for all $a,b,c\in S_n$.
\end{proof}

\begin{lem}\label{lem: demlem}
    Suppose $a,w\in S_n$ such that $a\leq w$. Then for any $x\in S_n$, we have $a \star x\leq w \star x$ and $x \star a\leq x \star w.$ 
\end{lem}
\begin{proof}
    First assume $x$ is a simple transposition $s_i$, and suppose $a\leq w$.
If $a \star s_i=a$ then $a s_i\leq a \leq w \leq w\star s_i.$ If $a \star s_i=as_i$ and $w \star s_i=ws_i$ then $a \star s_i=as_i\leq ws_i=w \star s_i.$ Finally, if $a \star s_i=as_i$ and $w \star s_i=w$ then by the lifting property of $S_n$ (\cref{prop:lifting-property}), $a \star s_i=as_i\leq w=w \star s_i.$ Thus, if $a\leq w$, then $a \star s_i\leq w \star s_i.$

Now let $s_{i_1}s_{i_2}\cdots s_{i_k}$ be a reduced word for $x$, and suppose $a\leq w$. Then $a \star s_{i_1}\leq w \star s_{i_1}.$ So $a \star s_{i_1}\star s_{i_2}=(a \star s_{i_1})\star s_{i_2}\leq (w\star s_{i_1})\star s_{i_2}=w\star s_{i_1}\star s_{i_2}.$ Repeating this logic $k$ times gives $a\star x\leq w\star x.$ Thus, for any $x,a,w\in S_n$, if $a\leq w,$ then $a\star x\leq w\star x$.

An analogous proof shows that for any $x,a,w\in S_n$, if $a\leq w,$ then $x\star a\leq x\star w$.
\end{proof}

\subsection{Subword complexes and subword sets}\label{subsection:subword}
We will now recall the definition of the \textit{subword complex} from \cite{KnutsonMiller} and introduce  a generalization of subword complexes that we call \textit{subword sets}.

 Fix a word $Q=(\q_1,\dotsc, \q_{\mm})$ in simple transpositions. A \textbf{subword} $J$ of $Q$ is a sequence $(\q'_{1},\dotsc,q'_{\mm})$ obtained from $Q$ by replacing some entries by $\id$ and leaving others the same. Denote by $Q\setminus J$ the sequence of length $\mm$ whose $i$th entry equals $1$ if $\q'_i\neq \id$ and equals $\q_{i}$ otherwise.
 \begin{egg}
     There are $8$ subwords of the word $Q=(s_1,s_2,s_1)$ in $S_3$:  
 \[(e,e,e),\quad (s_1,e,e),\quad (e,s_2,e),\quad (e,e,s_1),\quad (s_1,s_2,e),\quad (s_1,e,s_1),\quad (e,s_2,s_1),\quad (s_1,s_2,s_1).\]
 \end{egg}

    \begin{defn}(cf. \cite{KM04, KnutsonMiller})
 Let $Q=(s_{i_1},\dotsc,s_{i_\mm})$ be a word in simple transpositions in $S_n$ and $w\in S_n$ a permutation. The \textbf{subword complex} $\Delta(Q,w)$ is a simplicial complex on the vertex set $Q$ whose facets (resp. faces) are the subwords $Q\setminus J$ of $Q$, where $J$ is itself a subword of $Q$, such that  $J$ is a reduced word for $w$ (resp. the product $J$ contains a reduced subword for $w$.) 
\end{defn}

\begin{rmk}
        Assume $Q=(s_{i_1},\dotsc,s_{i_{\mm}})$ is a word in simple transpositions in $S_n$, and let $w\in S_n$.
        The subword complex $\Delta(Q,w)$ is homeomorphic to a sphere if $\dem(Q)=w$ and to a ball otherwise \cite[Corollary 3.8]{KM04}. We only focus on non-empty spherical subword complexes $\Delta(Q,w_0)$ in this work, a choice partly justified by the fact that every spherical subword complex is isomorphic to $\Delta(Q',w_0)$ for some $Q'$ such that $\dem(Q')=w_0$ \cite[Theorem 3.7]{CJP14}.
    \end{rmk}

Instead of considering subword complexes, we will look at \textit{subword sets}.

 \begin{defn}\label{defn:subword-set}
        Let $Q=(\q_1,\dotsc,\q_\mm)$ be a word in $S_n$ and $w\in S_n$ be a permutation.
        The \textbf{subword set} $\nabla(Q,w)$ of $Q$ with respect to $w$, is the set of all subwords of $Q$ with product $w$:
         \[\nabla(Q,w)=\{J=(v_1,\dotsc,v_\mm)\mid v_i\leq \q_i \text{ and } J_{(\mm)}=w\}.\]
    \end{defn}

When $Q$ is a sequence of simple transpositions, the subword set will contain the complement of facets of the corresponding subword complex. From here on, we will only ever talk about $\nabla(Q,w_0)$ rather than more general subword sets.
The primary reason for defining the elements of a subword complex in terms of complements of subwords was the simplicity of viewing the subword complex as a simplicial complex when defined this way. Since we do not define a simplicial complex structure on the subword set, we prefer not to take complements in our definition of the subword set.

\subsection{Introducing the brick polytope and bulky brick polytope}\label{subsection:brick-polytopes}

 We recall the \textit{brick polytope} of \cite{bricksort} and \cite{PS15} (which is the moment polytope of the \textit{brick variety} with respect to a torus-equivariant embedding of the brick variety into a product of Grassmannians; see \cref{prop: unbulky moment}), largely following the exposition in \cite{E16}, and we introduce the \textit{bulky brick polytope} (which is the moment polytope of the brick variety with respect to a torus-equivariant embedding of the brick variety into a product of flag varieties; see \cref{prop: bulky moment}), laying the foundations for this paper. In \cref{subsection:coincide-brick}, we will give sufficient conditions for the brick polytope and bulky brick polytope to coincide.

 Consider the standard basis vectors $p_i$ of $\mbb{R}^n$, with $1$ appearing in entry $i$ of $p_i$ and $0$ appearing in all other entries.  Define the vectors $\omega_i=p_1+\cdots+p_i$ for $i\in [n]$. We let $S_n$ act on $\R^n$ by setting $\sigma \cdot p_i=p_{\sigma(i)}$ for $\sigma\in S_n$ and extending linearly.

 Let $Q=(\q_1,\dotsc,\q_{\mm})$ be any word in $S_n$ and let $J\in \nabla(Q,w_0)$ be a subword of $Q$. 
   The \textbf{weight function} $w(J,\cdot)\colon [\mm] \to \mbb{R}^n$ is 
 
   \[w(J,k):=J_{(k)}\left(\sum_{j\in[n-1]:s_j\leq \q_k}\omega_j\right),\]
    the \textbf{brick vector} is $B(J):=\sum_{k=1}^{\mm}w(J,k)$, and the \textbf{brick polytope} of $Q$ is 
        \[\B(Q):=\mrm{conv}\{B(J)\mid J\in\nabla(Q,w_0)\}.\]
   We also define the \textbf{bulky weight function} 
    $\mbf{w}(J,\cdot)\colon [\mm]\to \mbb{R}^n$ with the formula
    \[\mbf{w}(J,k):=J_{(k)}\left(\sum_{i=1}^{n}\omega_{i}\right).\]

    \noindent The \textbf{bulky brick vector} is $\mbf{B}(J)=\sum_{k=1}^{\mm}\mbf{w}(J,k)$, and the \textbf{bulky brick polytope} of $Q$ is
  \[\BB(Q):=\conv\{\mbf B(J)\mid J\in\nabla(Q,w_0)\}.\]

\begin{rmk}
    The brick polytope was originally defined in \cite{bricksort} in the special case that all $\q_i$ are simple transpositions. In that case, the definitions of $B(J)$ and $\mathcal{B}(Q)$ come from \cite{PS15}. In \cite{bricksort,PS15} they take the convex hull of all reduced subwords $J$ of $Q$ with product $w_0$. We follow \cite{E16} in taking the convex hull over all subwords $J$ of $Q$ that have product $w_0$ rather than only reduced $J$ (and when the $Q$ is a word in simple transpositions, \cite{E16} showed that this does not make a difference). Our definitions generalize brick polytopes to allow $Q$ to be any word in $S_n$ with $\dem(Q)=w_0$ and introduce the bulky versions $\mbf{w}(J,k),\mbf{B}(J),$ and $\BB(Q)$ to match our treatment of the \textit{brick variety}.
    \end{rmk}

When $\dem(Q)\neq w_0$, the subword set $\nabla(Q,w_0)$ will be empty so from now on, we will always assume that $\dem(Q)=w_0$.

\subsection{Equalities between brick polytopes and bulky brick polytopes}\label{subsection:coincide-brick}

Now we will describe conditions that will ensure that brick polytopes and bulky brick polytopes coincide or coincide up to translation.

\begin{lem}\label{lem: smashbrick}
    Let $\q_1,\dots, \q_{\mm} \in S_n.$ Let \[Q=(\q_1, \dots, \q_{\mm}) \quad \text{ and }\quad Q'=(\q_{1},\dots, \q_{i-1},\q_i\q_{i+1},\q_{i+2},\dots, \q_{\mm}).\]
    If $[\id,\q_i]\cap [\id,\q_{i+1}]=\{\id\}$ then $\B(Q)=\B(Q').$  
\end{lem}
\begin{proof}
    Define a map $f:\nabla(Q,w_0)\to \nabla(Q',w_0)$ by \[(v_1,\dots,v_{i-1},v_i,v_{i+1},v_{i+2},\dots v_{\mm})\mapsto (v_1,\dots,v_{i-1},v_iv_{i+1},v_{i+2},\dots v_{\mm}).\]

    The condition $[\id,\q_i]\cap [\id,\q_{i+1}]=\{\id\}$ is equivalent to saying that the set of simple transpositions appearing in reduced words for $\q_i$ is disjoint from the set of simple transpositions appearing in reduced words for $\q_{i+1}$. This precludes any cancellation in the product $\q_i\q_{i+1}$ and thus implies that $\q_i\q_{i+1}=\q_i\star\q_{i+1}.$ Then $[\id, \q_i\q_{i+1}]=[\id,\q_i\star\q_{i+1}]=\{v_{i}\cdot v_{i+1}:v_i\leq q_i, v_{i+1}\leq q_{i+1}\}$ which shows that $f$ is a surjection\footnote{It is not too hard to see that $f$ is in fact a bijection, but doing so is not necessary for this proof} from $\nabla(Q,w_0)$ to $\nabla(Q',w_0)$.

   Let $J=(v_1,\dots,v_{i-1},v_i,v_{i+1},v_{i+2},\dots v_{\mm})\in \nabla(Q,w_0)$. 
   For $k<i,$ $J_{(k)}=f(J)_{(k)}$ so $w(J,k)=w(f(J),k).$ For $k\geq i,$ $J_{(k+1)}=v_1v_{2}\dots v_{k+1}=
   f(J)_{(k)}$. This tells us that for $k>i,$ $w(J,k+1)=w(f(J),k)$ (but does not tell us that $w(J,i+1)=w(f(J),i)$ since the $i+1$st permutation of $Q$ differs from the $i$th permutation of $Q'$).
   
   For $j\neq j',$ $s_{j'} \omega_{j}=\omega_{j}.$ Thus for $s_j\not \leq \sigma,$ $\sigma \omega_j = \omega_j.$ Since $[1,\q_i]\cap [1,\q_{i+1}]=\{1\},$  $v_{i+1}\sum_{j\in [n-1]:s_j\leq \q_{i}}\omega_i = \sum_{j\in [n-1]:s_j\leq \q_{i}}\omega_i.$ We make use this fact in the following calculation.

\begin{align*}
    w(f(J),i)&=f(J)_{i}\left(\sum_{j\in [n-1]:s_j\leq \q_i \q_{i+1}}\right)\\
    &=v_1\cdots v_{i}v_{i+1}\left(\sum_{j\in [n-1]:s_j\leq \q_i }\omega_j+\sum_{j\in [n-1]:s_j\leq  \q_{i+1}}\omega_j\right)\\
    &=v_1\cdots v_{i}v_{i+1}\left(\sum_{j\in [n-1]:s_j\leq \q_i }\omega_j\right)+v_1\cdots v_{i}v_{i+1}\left(\sum_{j\in [n-1]:s_j\leq  \q_{i+1}}\omega_j\right)\\
    &=v_1\cdots v_{i}\left(\sum_{j\in [n-1]:s_j\leq \q_i }\omega_j\right)+v_1\cdots v_{i}v_{i+1}\left(\sum_{j\in [n-1]:s_j\leq  \q_{i+1}}\omega_j\right)\\
    &=J_{(i)}\left(\sum_{j\in [n-1]:s_j\leq \q_i }\omega_j\right)+J_{(i+1)}\left(\sum_{j\in [n-1]:s_j\leq  \q_{i+1}}\omega_j\right)\\
    &=w(J,i)+w(J,i+1).
\end{align*}   

We can now show that $B(J)=B(f(J))$ as follows:
\begin{align*}
    B(f(J))&=\sum_{k=1}^{\mm-1}w(f(J),k)\\
    &=\left(\sum_{k=1}^{i-1}w(f(J),k)\right)+\left(\sum_{k=i+1}^{\mm-1}w(f(J),k)\right)+w(f(J),i)\\
    &=\left(\sum_{k=1}^{i-1}w(f(J),k)\right)+\left(\sum_{k=i+1}^{\mm-1}w(J,k+1)\right)+w(J,i)+w(J,i+1)\\
    &=\left(\sum_{k=1}^{i-1}w(f(J),k)\right)+\left(\sum_{k=i+2}^{\mm}w(J,k)\right)+w(J,i)+w(J,i+1)\\
    &=B(J).
\end{align*}
We conclude that $\B(Q)=\B(Q')$.\qedhere
    
\end{proof}

\begin{prop}\label{prop:same-normal-fan}
    Assume $Q=(\q_1,\dotsc,\q_{\mm})$ is a word in $S_n$ such that each $\q_i$  is a product of \textit{distinct} simple transpositions $h_{1}^{i}\dotsb h_{\ell(\q_i)}^{i}$. Let $Q'$ be the word $(h_{1}^{1},\dotsc,h_{\ell(\q_1)}^{1},\dotsc,h_{1}^{\mm},\dotsc,h_{\ell(\q_\mm)}^{\mm})$. Then $\B(Q)=\B(Q')$. 
    \end{prop}
    \begin{proof}
        This follows from  \cref{lem: smashbrick}.
    \end{proof}

\begin{rmk}
       In the notation of \cref{prop:same-normal-fan}, although the brick polytope $\mcal{B}(Q')$ is independent of the reduced decompositions of the $\q_i$, the bulky brick polytope $\BB(Q')$ does depend on the reduced decompositions of the $\q_i$. 
\end{rmk}

\begin{lem}\label{lem:bulk up}
     Let $\q_1,\dots, \q_{\mm} \in S_n.$ Set $Q=(\q_1, \dots, \q_{\mm})$ and $H_r=\{j\in [n-1]:s_j\leq \q_r\}.$ If there exists some $r$ such that $H_1\subseteq H_2 \subseteq \cdots \subseteq H_r \supseteq H_{r+1} \cdots \supseteq H_{\mm},$ then $\B(Q)$ and $\BB(Q)$ are identical up to translation. 
\end{lem}
\begin{proof}
    We show that $B(J)$ and $\mbf{B}(J)$ differ by a translation by showing that $\mbf{w}(J,k)$ and $w(J,k)$ differ by a translation which doesn't depend on $J$.
    
  The nesting property means that for $k\leq r$, if a simple reflection doesn't show up in $\q_k$, then it doesn't show up in $\q_{k'}$ for $k'<k$ and therefore also doesn't show up in $J_{(k)}.$ So for $k\leq r$, $j\in [n]\setminus H_k$ implies $s_j \not \leq J_{(k)}$. So $\mbf{w}(J,k) - w(J,k) = J_{(k)} \sum_{j\in [n]\setminus H_k} \omega_j = \sum_{j\in [n]\setminus H_k} \omega_j.$

  For $k \geq r,$ $j\notin H_k$ implies $j\notin H_{k'}$ for $k'>k.$ Let $J=(v_1,\dots, v_{\mm})\in \nabla(Q,w_0)$. For $k\geq r$ and $j\notin H_k$, 
    $w_0\omega_j=J_{(m)}\omega_{j}=J_{(k)}v_{k+1}\cdots v_{\mm}\omega_j = J_{(k)}\omega_j.$ Therefore, for $k\geq r,$  $\mbf{w}(J,k) - w(J,k) = J_{(k)} \sum_{j\in [n]\setminus H_k} \omega_j = w_0\sum_{j\in [n]\setminus H_k} \omega_j$.
    
    We have shown that  $\mbf{w}(J,k) - w(J,k)$ does not depend on $J$, so we conclude that $\mbf{B}(J)-B(J)$ does not depend on $J$. Therefore, $\B(Q)$ and $\BB(Q)$ are identical up to translation. 
\end{proof}

\section{Pipe dreams, bracket vectors, and the cubulation of the associahedron}\label{section:pipe-dreams-bracket-vectors}

In this section, we will describe some connections between various combinatorial objects: certain ``pipe dreams" called $(1,n)$-twists, rooted binary trees, and Lehmer codes. Describing these connections constitutes the bulk of \cref{subsection:Loday-associehedron}, \cref{subsection:pipe-dreams}, and \cref{subsection:bracket-vectors}. In \cref{subsection:cubulation}, we will see how $(1,n)$-twists correspond to vertices of the Loday associahedron, and suggest which points on the boundary of the Loday associahedron are vertices of the subdivision of \cref{thm:main-result}.

\subsection{Loday associahedron and generalized associahedra}\label{subsection:Loday-associehedron}

We begin by explaining how the Loday realization of the associahedron arises as a brick polytope, and then connect this understanding to a special type of ``pipe dream'' called a $(1,n)$-twist.

The following result describes when a brick polytope, bulky brick polytope, and Loday associahedron all coincide up to translation.
\begin{prop}\label{prop:is Loday}
      Let $\sigma_i =s_1s_2\cdots s_{n-i+1}\in S_n$ for $i\geq 2$ and set $\sigma_1=\sigma_2$. Let $Q=(\sigma_1,\sigma_2,\dots, \sigma_{n}).$ Let $Q_i=(s_1,s_2,\dots, s_{n-i+1})$ for $i\geq 2$ and set $Q_1=Q_2$. Set $Q'$ to be the concatenation $Q_1+Q_2+\dots+Q_n$. Then $\nabla(Q,w_0)$ is in bijection with $\nabla(Q',w_0)$; and $\B(Q')$, $\B(Q)$, $\BB(Q)$, and the Loday realization of the associahedron all coincide up to translation and multiplication by $-1$.
\end{prop}
\begin{proof}
    The result that $\B(Q')$ coincides with the Loday associahedron up to translation can be attributed to \cite{PS15}. The weights we sum are in some sense negative of the ones they use so the resulting brick polytopes are translations of the result of multiplying their brick polytopes by $-1$. The word $Q$ satisfies the hypotheses of both \cref{prop:same-normal-fan} and \cref{lem:bulk up} so $\B(Q')=\B(Q)$, and $\BB(Q)$ is a translation of $\B(Q')$. We have that $\nabla(Q,w_0)$ is in bijection with $\nabla(Q',w_0)$ by \cref{lem: smashbrick}.
\end{proof}

The word $Q$ described above is a special case of a more general construction. We give a brief and partial history of generalized associahedra.

A \textbf{Coxeter element} in $S_n$ is an element of $S_n$ obtained by taking the product of all the simple reflections in some order with each appearing exactly once. 
Let $C$ be a reduced word in simple transpositions for a Coxeter element $c$, i.e., $C$ is  some ordering of the simple reflections of $S_n$. Denote by $C^\infty$ the concatenation of $C$ to itself ad infinitum. For $w\in S_n,$ the \textbf{$C$-sorting word} for $w$ is the lexicographically first subword (in terms of positions) of $C^\infty$ that is a reduced word for $w$. Letting $q_i$ be the product of the simple reflections coming from the $i$th copy of $C$ in the $C$-sorting word for $w$, we can view $Q=(\q_1, \q_1,\q_2,\dots, \q_k)$ as a subword of $(c,c,\dots, c).$ A permutation $w\in S_n$ is called \textbf{$c$-sortable} if $q_1\geq q_2 \geq \cdots \geq q_k$. The longest permutation $w_0$ is $c$-sortable for all $c$ and we denote its $C$-sorting word by $w_0(C)$. Since $(\q_1,\q_1,\dots, \q_k)$ satisfies the conditions of both \cref{prop:same-normal-fan} and \cref{lem:bulk up}, it follows from a result of \cite{PS15}, that $\BB(Q)$ is the ``$c$-generalized associahedron'' of \cite{HLT} up to translation (and negation). We give some brief notes about the context/history of these objects before discussing one application of this perspective.

When $W$ is a finite Coxeter group and $c$ is a Coxeter element of $W$, one can define the sets of ``non-crossing partitions''  and ``$c$-clusters'' associated to $W$ and $c$. The sizes of these two sets are the same and do not depend on $c$. These sizes are called \textbf{$W$-Catalan numbers} and in type $A$ they are the regular Catalan numbers. 

Associated to each integer $n$, we have: the Loday realization of the associahedron $K_n$, the normal fan of the Loday realization of the associahedron, the Tamari lattice $T_n$, and the $n$th Catalan number. The ``Hasse diagram'' of the Tamari lattice can be obtained as an orientation of the 1-skeleton of the Loday associahedron. Reading generalized the Tamari lattice and the normal fan of the Loday associahedron to the $c$-Cambrian lattice and the $c$-Cambrian fan associated to any Coxeter element $c$ \cite{Cambrian}. Reading conjectured the $c$-Cambrian fan to be the normal fan of some polytope. Hohlweg, Lange, and Thomas introduced the \textit{$c$-generalized associahedra} $\mathrm{Asso}_c(W)$ which have $c$-Cambrian fans as their normal fans \cite{HLT}.

We will make use of a special case of the following technical result.

\begin{lem}\label{lem:easydem}
    Let $C$ be a reduced word for a Coxeter element $c$ in $S_n$. Let $a_i$ be the product of the simple reflections in $w_0(C)$ coming from the $i$th copy of $C$ in $C^{\infty}$. Then 
    \begin{align*}
        \dem(C^i)&=a_1a_2\cdots a_i\\
        &=\dem(a_1,a_2,\dots, a_i)\\
        &=\dem(a_1,a_1,a_2,a_3,\dots, a_{i-1})\\
        &=w_0\dem(a_{i+1},a_{i+2},\dots, a_{n-1})^{-1}.
    \end{align*}
    Additionally, $a_i=\dem(C^{i-1})^{-1}\dem(C^i).$

\end{lem}
\begin{proof}
    Since the $C$-sorting word of $w_0$ is defined as the lexicographically first subword of $C^{\infty}$ which has product $w_0$, it can be built by reading through $C^{\infty}$ from left to right and taking every letter that increases the length. That is exactly the Demazure product. This shows that $a_1\dots a_i = \dem(C^i)$ and consequently $a_i=\dem(C^{i-1})^{-1}\dem(C^i).$ The fact that letters are only taken when they increase the length means that $\dem(a_1,\dots,a_i)=a_1a_2\cdots a_i$ and $\dem(a_i,a_{i+1},\dots, a_{n-1})=a_ia_{i+1}\cdots a_{n-1}.$ Since $w_0=a_1a_2\cdots a_{n-1},$
     \begin{align*}
        w_0\dem(a_{i+1},\dots, a_{n-1})^{-1}&=w_0\cdot (a_{i+1}a_{i+2}\cdots a_{n-1})^{-1}\\
        &=a_1a_2\cdots a_{n-1} \cdot (a_{i+1}a_{i+2}\cdots a_{n-1})^{-1}\\
        &=a_1a_2\cdots a_{i}.
    \end{align*}
    Finally, 
      \begin{align*}
    \dem(a_1,a_1,a_2,a_3,\dots, a_{i-1})
    &=a_1\star \dem(a_1,\dots, a_{i-1})\\ 
    &=\dem(C)\star (a_1a_2\cdots a_{i-1})\\
    &=\dem(C) \star \dem(C^{i-1})\\
    &=\dem(C^i).
\end{align*}
\end{proof}
In particular, when $c=s_1s_2\dots s_{n-1},$ we have $a_i=s_1s_2\dots s_{n-i}$ and we get the Loday realization of the associahedron.

\subsection{Pipe dreams and twists}\label{subsection:pipe-dreams}

Subword complexes were originally defined by Knutson and Miller \cite{KnutsonMiller} to describe ``pipe dreams''. Here we only care about pipe dreams ``for the permutation'' written  $1~~n+1~~n~\cdots~4~3~~2$ in one-line notation. These pipe dreams were first shown to be Catalan many by Woo \cite{woo}. We will introduce these objects directly and will refer to them as $(1,n)$-twists.

A $(1,n)$-twist is a filling of the NW triangle (including the antidiagonal) of an $(n+1)\times (n+1)$ square with $\bcross$, $\bump$, $\are$, and $\jay$ such that the `pipes' passing through the tiles form continuous curves, there are pipes starting in the left of every row but the topmost one and ending at the top of every column but the leftmost one,  the pipe starting in the $i$th row from the bottom ends in the $(i+1)$st column from the left, and any pair of pipes crosses in at most one location.  We will say the tile in the $i$th row from the bottom and the $j$th column from the left is in position $(i,j)$. We will refer to the pipe starting in the $i$th row from the bottom as the $i$ pipe.

\begin{figure}
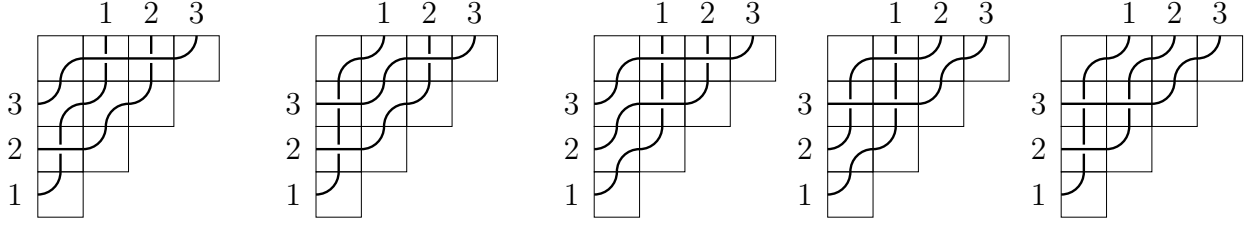

     \centering
     \begin{subfigure}[b]{0.18\textwidth}
         \centering
         \[
\begin{mosaic}[font=\Large\sffamily]{0.6}
  \&  \& \text{$1$}\& \text{$2$} \& \text{$3$}  \\
 \& \tileii \& \tileix \& \tileix \& \tileiv \\
\text{$3$} \& \tileviii \& \tileviii \& \tileiv \\
\text{$2$} \& \tileix \& \tileiv\\
\text{$1$} \& \tileiv \\
\end{mosaic}
\]
        
         \label{fig:y equals x}
     \end{subfigure}
     \hfill
     \begin{subfigure}[b]{0.18\textwidth}
         \centering
         \[
\begin{mosaic}[font=\Large\sffamily]{0.6}
  \&  \& \text{$1$}\& \text{$2$} \& \text{$3$}  \\
 \& \tileii \& \tileviii \& \tileix \& \tileiv \\
\text{$3$} \& \tileix \& \tileviii \& \tileiv \\
\text{$2$} \& \tileix \& \tileiv\\
\text{$1$} \& \tileiv \\
\end{mosaic}
\]
        
         \label{fig:three sin x}
     \end{subfigure}
     \hfill
     \begin{subfigure}[b]{0.18\textwidth}
         \centering
         \[
\begin{mosaic}[font=\Large\sffamily]{0.6}
  \&  \& \text{$1$}\& \text{$2$} \& \text{$3$}  \\
 \& \tileii \& \tileix \& \tileix \& \tileiv \\
\text{$3$} \& \tileviii \& \tileix \& \tileiv \\
\text{$2$} \& \tileviii \& \tileiv\\
\text{$1$} \& \tileiv \\
\end{mosaic}
\]
       
         \label{fig:five over x}
     \end{subfigure}
     \begin{subfigure}[b]{0.18\textwidth}
         \centering
         \[
\begin{mosaic}[font=\Large\sffamily]{0.6}
  \&  \& \text{$1$}\& \text{$2$} \& \text{$3$}  \\
 \& \tileii \& \tileix \& \tileviii \& \tileiv \\
\text{$3$} \& \tileix \& \tileix \& \tileiv \\
\text{$2$} \& \tileviii \& \tileiv\\
\text{$1$} \& \tileiv \\
\end{mosaic}
\]
     
         \label{fig:five over x}
     \end{subfigure}
      \begin{subfigure}[b]{0.18\textwidth}
         \centering
         \[
\begin{mosaic}[font=\Large\sffamily]{0.6}
  \&  \& \text{$1$}\& \text{$2$} \& \text{$3$}  \\
 \& \tileii \& \tileviii \& \tileviii \& \tileiv \\
\text{$3$} \& \tileix \& \tileix \& \tileiv \\
\text{$2$} \& \tileix \& \tileiv\\
\text{$1$} \& \tileiv \\
\end{mosaic}
\]
       
         \label{fig:five over x}
     \end{subfigure}
        \caption{The five $(1,3)$-twists}
        \label{fig:three graphs}
\end{figure}

The elements of $\nabla(Q',w_0)$ with $Q'$ as in \cref{prop:is Loday} are in bijection with $(1,n)$-twists. We associate the square in position $(i,j)$ of a $(1,n)$-twist with the simple transposition $s_{i-j}$ for $i>j$ and $(i,j)\neq (n+1,1)$. We read the squares column by column, starting from the leftmost column and proceeding right, and reading from bottom to top within each column. We skip over the squares on and below the antidiagonal ($i\leq j$) as well as the top left corner $(n+1,1)$.
Writing down the associated transpositions as we go gives the word $Q'=(s_1,s_2\dots s_{n-1})+(s_1,s_2,\dots s_{n-1})+(s_1,s_2,\dots, s_{n-2})+\dots +(s_1,s_2,s_3)+(s_1,s_2)+(s_1)$ where $+$ is concatenation. Define $f(\ell)$ to be the $\ell$th position $(i,j)$ according to this reading order of the squares of a $(1,n)$-twist.
A specific $(1,n)$-twist $T$ is mapped (bijectively) to the subword $\xi(T)=(v_1,\dots,v_{n(n+1)/2-1}) \in \nabla(Q,w_0)$ with \[v_\ell=\begin{cases} q_{\ell} & \text{there is a cross tile in position} f(\ell) \text{ of } T\\ \id & \text{else} \end{cases}\] 
where $q_\ell$ is the $\ell$'th letter of $Q'$.
We now set $\sigma_i=s_1s_2\cdots s_{n-i+1}$ for $i\geq 2$, $\sigma_1=\sigma_2$, and $Q=(\sigma_1,\dots, \sigma_n)$ as in \cref{prop:is Loday}. By \cref{prop:is Loday}, $\nabla(Q',w_0)$ is in bijection with $\nabla(Q,w_0)$. Here $\sigma_i$ is the product of the simple transposition coming from the $i$th column of a $(1,n)$-twist. Following the bijection between $\nabla(Q,w_0)$ and $\nabla(Q',w_0),$ a specific $(1,n)$-twist $T$ is mapped (bijectively) to the subword $\Xi(T):=(v_1,\dots,v_{n}) \in \nabla(Q,w_0)$ with \[v_j=\prod_{i=1}^{\min(n-j+1,n-1)}\begin{cases} s_{i} & \text{there is a cross tile in position} (i+j,j) \text{ of } T\\ \id & \text{else.} \end{cases}\] 
If we view $(1,n)$-twists as wiring diagrams, we can understand the product $\Xi(T)_{(k)}=v_1v_2\cdots v_k$ as the ordering of the pipes after the first $k$ columns. An example of this is visualized in figure \cref{fig:PDBrick}. With the notation $\pi_{T,k}:=\Xi(T)_{(k)}$, the following lemma formalizes this intuition.

\begin{figure}
    \centering
    \includegraphics[width=0.8\linewidth]{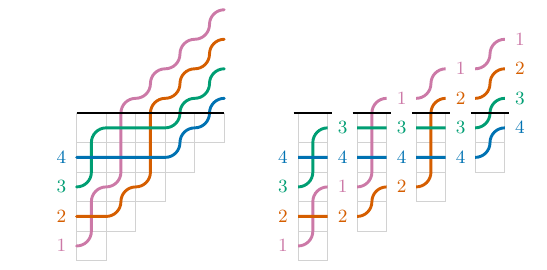}
    \caption{Visualization of the column permutations $\pi_{T,k}$ of a $(1,n)$-twist. On the left is a $(1,n)$-twist $T$. On the right, the 4 columns of $T$ are separated out to show the permutations of the pipes that appear between them. In this instance $(\pi_{T,1},\pi_{T,2},\pi_{T,3},\pi_{T,4})=(2143,2431,4321,4321)$. The pipes are shown as continuing above the $(1,n)$-twist and its columns to visualize how the column permutations incorporate pipes that have already exited the $(1,n)$-twist.  }
    \label{fig:PDBrick}
\end{figure}

\begin{lem}\label{lem:aaaah}
    Let $p$ be a path in a $(1,n)$-twist $T$ that starts below the antidiagonal between the $k$th and $k+1$st columns, travels up until it reaches the top of $T$, and then travels left until it reaches the top left corner of $T$. The path $p$ will intersect every pipe of $T$ and the $j$th pipe it intersects will be $\pi_{T,k}(j)$.
\end{lem}

Combining \cref{lem:aaaah} with the definition of the bulky brick polytope, we arrive at the following proposition.

\begin{prop}\label{lem:coordinates-Loday}
Set $\sigma_i=s_1s_2\cdots s_{n-i+1}$ for $i\geq 2$, $\sigma_1=\sigma_2$, and $Q=(\sigma_1,\dots, \sigma_n)$ as in \cref{prop:is Loday}.
The elements of $\nabla(Q,w_0)$ are 
     \[\nabla(Q,w_0)=\{(\pi_{T,1},\pi_{T,1}^{-1}\pi_{T,2},\pi_{T,2}^{-1}\pi_{T,3},\dots, \pi_{T,n}^{-1}\pi_{T,n+1}): T \text{ is a } (1,n)\text{-twist}\}. \]

    Up to translation and reflection, the Loday associahedron is
    \[\conv_{T \text{ a } (1,n) \text{-twist}}\sum_{j=1}^{n+1}\pi_{T,j}(\sum_{i=1}^{n}\omega_{i}).\]
\end{prop}

For any $\pi\in S_n,$ $\pi(\sum_{i=1}^{n}\omega_{i})=\pi(n,n-1,\dots, 2,1)=(n+1-\pi^{-1}(1),n+1-\pi^{-1}(2),\dots, n+1-\pi^{-1}(n))$. Thus, up to translation, the vertex of the Loday associahedron associated with a $(1,n)$-twist $T$ can be obtained by looking at the permutations of pipes between columns of $T$, taking their inverses, and adding their one-line representations as vectors. The interpretation of the Loday construction that views $(1,n)$-twists as elements of $\nabla(Q',w_0)$ is well known. As far as we are aware the interpretation of the Loday construction through the column permutations $\pi_{T,k}$ of $(1,n)$-twists has not appeared before in the literature. Through \cref{thm:Minkowski} and \cref{Loday mink sum} we show that the convex hull and sum in \cref{lem:coordinates-Loday} can be exchanged. That is, up to translation and reflection, the Loday associahedron has the following Minkowski decomposition:
 \begin{align*}
     \BB(Q)&=\conv_{T \text{ a } (1,n) \text{-twist}}\sum_{j=1}^{n+1}\pi_{T,j}(\sum_{i=1}^{n}\omega_{i})\\
     &=\sum_{j=1}^{n+1} \conv_{T \text{ a } (1,n) \text{-twist}} \pi_{T,j}(\sum_{i=1}^{n}\omega_{i}).
 \end{align*}

\begin{figure}
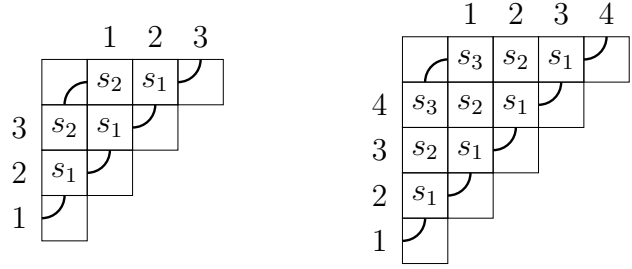
\label{figure:cambrian}
    \centering
   \begin{subfigure}[b]{0.3\textwidth}
       \centering
                \[
\begin{mosaic}[font=\Large\sffamily]{0.6}
  \&  \& \text{$1$}\& \text{$2$} \& \text{$3$}  \\
     \& \tileii \& \textile{$s_2$} \& \textile{$s_1$} \& \tileiv \\
\text{$3$} \& \textile{$s_2$} \& \textile{$s_1$} \& \tileiv \\
\text{$2$} \& \textile{$s_1$} \& \tileiv\\
\text{$1$} \& \tileiv \\
\end{mosaic}
\]
   \end{subfigure}
\begin{subfigure}[b]{0.3\textwidth}
   \centering
       \[
\begin{mosaic}[font=\Large\sffamily]{0.6}
  \&  \& \text{$1$}\& \text{$2$} \& \text{$3$} \&\text{$4$} \\
 \&  \tileii \& \textile{$s_3$} \& \textile{$s_2$} \& \textile{$s_1$} \& \tileiv \\
\text{$4$} \& \textile{$s_3$} \& \textile{$s_2$} \& \textile{$s_1$} \& \tileiv \\
\text{$3$} \& \textile{$s_2$} \& \textile{$s_1$} \& \tileiv \\
\text{$2$} \& \textile{$s_1$} \& \tileiv\\
\text{$1$} \& \tileiv \\
\end{mosaic}
\]
   \end{subfigure}
    \caption{Visualization of the simple transpositions corresponding to each cell of a $(1,n)$-twist}
    \label{fig:placeholder}
\end{figure}

\subsection{Trees, $(1,n)$-twists, and bracket vectors}\label{subsection:bracket-vectors}

We will now describe the relationship between $(1,n)$-twists, rooted binary trees, and bracket vectors. A bijection between $(1,n)$-twists and rooted binary trees is described in \cref{thm:bijection-with-rooted-binary-tree}.

In \cref{defn:binary-tree} and \cref{defn:bracket-vector} below, we will define rooted binary trees and bracket vectors.

\begin{defn}\label{defn:binary-tree}
    A \textbf{full} binary tree is a rooted tree so that every non-leaf vertex has two children. Let $T$ be a full binary tree with $n$ internal vertices. The \textbf{infix numbering} on $T$ is a labeling of the internal vertices of $T$ with numbers $\{1,\dotsc,n\}$ so that the label of the $i$th vertex is greater than the labels of all vertices in the left sub-tree of vertex $i$ and less than the labels of all vertices in the right sub-tree of vertex $i$. 
    \end{defn}

\begin{defn}\label{defn:bracket-vector}
    Consider a rooted binary tree $T$ with $n$ internal vertices, and number the vertices of the sub-tree of $T$ consisting of internal vertices in infix numbering. The \textbf{bracket vector} of $T$ is  $\mbf{b}(T)=(b_1,b_2,\dotsc,b_n)$, where $b_i$ is one less than the number of leaves in the left sub-tree of the internal vertex with infix number $i$. Define the \textbf{complemented bracket vector} of $T$ is $\ol{\mbf{b}}(T)=(\ol b_1,\ol b_2,\dotsc,\ol b_n)$ with $\ol b_i=i-1-b_i$
\end{defn}

Below, we give an example of a binary tree, together with its complemented bracket vector.

\begin{egg}\label{egg:bracket-vector}
    Consider the tree $T$ of \cref{egg:infix-numbering}. The bracket vector of $T$ is
$\mbf{b}(T)=(b_1,b_2,b_3,b_4)=(0,0,2,0)$. The complemented bracket vector of $T$ is $\ol{\mbf{b}}(T)=(0,1,0,3)$. 
\end{egg}

\begin{figure}
     \centering
     \[
    \begin{tikzcd}[row sep = 1.1em, column sep = 1.7em]
	    & &  & 3 & & &\\
	      & 1 & &  & & 4 &\\
	\bullet & &  2 & & \bullet & & \bullet  \\
	& \bullet & & \bullet &  & & & \\
 	\arrow[no head, from=1-4, to=2-2]
 	\arrow[no head, from=2-2, to=3-1]
 	\arrow[no head, from=2-2, to=3-3]
 	\arrow[no head, from=3-3, to=4-2] 	\arrow[no head, from=3-3, to=4-4]
 	\arrow[no head, from=1-4, to=2-6] 	\arrow[no head, from=2-6, to=3-5]
    \arrow[no head, from=2-6, to=3-7]
\end{tikzcd}
\]
     \caption{A full rooted binary tree with internal vertices labeled in infix order}
     \label{egg:infix-numbering}
 \end{figure}
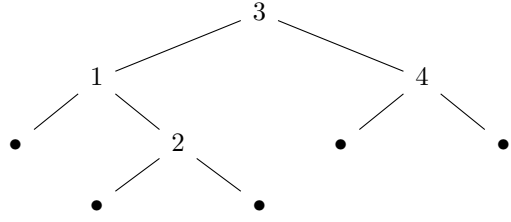

In \cref{defn:bump-tile} below, we will define the \textit{bump vector} of a $(1,n)$-twist, and in \cref{egg:twist}, we give an example of a $(1,n)$-twist together with its bump vector.

\begin{defn}\label{defn:bump-tile}
    Define the \textbf{bump vector} of a $(1,n)$-twist $D$ to be $\b(D)=(\ol b_1,\dots, \ol b_{n})$ with $\ol b_i+1$ being the index of the leftmost column of $D$ that has a $\bump$ or a $\jay$ in the $i$th row. I.e. $\ol b_i$ is the smallest number such that $(i,\ol b_i+1)$ is a $\bump$ tile or a $\jay$ tile. 
\end{defn}

\begin{egg}\label{egg:twist}
    Consider the tree $T$ from \cref{egg:infix-numbering}. Under the bijection of \cref{thm:bijection-with-rooted-binary-tree}, the tree $T$  corresponds to the following $(1,4)$-twist:  
\[
\begin{mosaic}[font=\Large\sffamily]{1}
  \& \text{$5$} \& \text{$1$}\& \text{$2$} \& \text{$3$} \&\text{$4$} \\
\text{$5$} \&  \tileii \& \tileix \& \tileix \& \tileviii \& \tileiv \\
\text{$4$} \& \tileix \& \tileix \& \tileix \& \tileiv \\
\text{$3$} \& \tileviii \& \tileviii \& \tileiv \\
\text{$2$} \& \tileix \& \tileiv\\
\text{$1$} \& \tileiv \\
\end{mosaic}
\]
The bump vector is $(0,1,0,3)$.
\end{egg}

The theorem below gives a bijection between $(1,n)$-twist and rooted binary trees.

\begin{thm}\label{thm:bijection-with-rooted-binary-tree}
    There is a bijection $\mfk{T}$ between $(1,n)$-twists and rooted binary trees with $n$ internal nodes. The tree associated with a $(1,n)$-twist is constructed as follows:
    \begin{itemize}
        \item Place the root at the $\are$ tile in the northwest corner of the $(1,n)$-twist
        \item Place an internal vertex at the position of each $\bump$ tile.
        \item Place a leaf at each $\jay$ tile.
        \item If two vertices (including leafs and the root) lie in the same row or the same column with no vertex lying directly between the two, we connect them by an edge.
    \end{itemize}
   For any  $(1,n)$-twist $D$, $\b(D)=\b(\mfk{T}(D))$. That is, the complemented bracket vector of the image of $D$ is the bump vector of $D$.
\end{thm}
\begin{proof}
The bijection itself is well known and appears in \cite{Karola}. To prove the relation to the bracket vector we end up mostly reproving the bijection. 

We will proceed by induction on $n$. The base case for our induction will be the case $n=1$. In this case, there is a single $(1,1)$-twist
\[
\begin{mosaic}[font=\Large\sffamily]{1}
  \&  \& \text{$1$} \\
 \&  \tileii \& \tileviii \\
\text{$1$} \& \tileviii \\
\end{mosaic}
\]
The tree associated to this pipe dream is
\[
    \begin{tikzcd}[row sep = 1.1em, column sep = 1.7em]
	    & \text{$1$} & \\
        \bullet & & \bullet 
 	\arrow[no head, from=1-2, to=2-1]
 	\arrow[no head, from=1-2, to=2-3]
\end{tikzcd}
\]
The bracket vector associated with this tree is $(0)$ and the complemented bracket vector is also $(0)$. The leftmost bump in row $1$ of the $(1,1)$-twist is in column 1, so the bump vector is also $(0).$ This proves the base case.

Now we will proceed with the induction step. Our inductive hypothesis will be that for $m<n$, the map $\mfk{T}$ from $(1,m)$-twists to binary trees with $m$ internal vertices is a bijection and that for $D$ a $(1,m)$-twist, the bracket vector for $\mfk{T}(D)$ is $(b_1,\dots,b_m)$ where the leftmost $\bump$ tile in the $i$th row of $D$ is in column $i-b_i$.

Let us analyze the set of $(1,n)$ twists such that the highest bump tile that occurs in the first column of the $(1,n)$-twist, occurs in the $i$th row from the bottom. We first show that the set of such $(1,n)$-twists has the following shape, and denote an arbitrary example  of such a $(1,n)$-twist by $R_i$: 
    \[
\begin{mosaic}[font=\Large\sffamily]{1}
  \&  \& \text{$1$}\& \text{$2$} \& \text{$\cdots$} \& \text{i-3} \& \text{i-2} \& \text{i-1} \& \text{i} \& \text{$\cdots$} \& \text{n-1} \& \text{n}\\
 \&  \tileii \& \tileix   \& \tileix    \& \tilecdots \& \tileix   \& \tileix \& \tileix  \& \tileviii \& \begin{tikzpicture}
    \draw[white] (0,0) -- (0,1);
    \draw[very thick,red] (0,0) -- (0,1) -- (1,1);
\end{tikzpicture} \& \begin{tikzpicture}
    \draw[white] (0,0) -- (0,1);
    \draw[very thick,red] (0,1) -- (1,1);
\end{tikzpicture} \& \begin{tikzpicture}
    \draw[very thick,red] (0,0) -- (1,0) -- (1,1) -- (0,1) ;
\end{tikzpicture} \\
\text{$n$} \& \tileix \& \tileix  \& \tileix    \& \tilecdots \& \tileix   \& \tileix \& \tileix  \& \begin{tikzpicture}
    \draw[very thick,red] (0,0) -- (0,1) -- (1,1);
    \draw[white] (0,1) -- (1,1);
\end{tikzpicture}  \& \begin{tikzpicture}
    \draw[white] (0,0) -- (0,1);
    \draw[white] (0,1) -- (1,1); \node at (0.5,0.5) {\textcolor{red}{\textbf{$Q_i$}}};
\end{tikzpicture} \& \begin{tikzpicture}
    \draw[very thick,red] (0,0) -- (1,0) -- (1,1) ;
\end{tikzpicture} \\
\text{$\vdots$} \&\tilevdots  \&\tilevdots\&\tilevdots\&\tileddots \&\tilevdots \&\tilevdots\&\tilevdots\& \begin{tikzpicture}
    \draw[very thick,red] (0,0) -- (0,1);
    \draw[white] (0,1) -- (1,1);
\end{tikzpicture}  \& \begin{tikzpicture}
    \draw[very thick,red] (0,0) -- (1,0) -- (1,1) ;
\end{tikzpicture} \\
\text{$i+1\,\,$ } \& \tileix  \& \tileix \& \tileix \& \tilecdots\& \tileix   \& \tileix \& \tileix \& \begin{tikzpicture}
    \draw[very thick,red] (0,1) -- (0,0) -- (1,0) -- (1,1) ;
    \draw[white] (0,1) -- (1,1);
\end{tikzpicture} \\
\text{$i$} \& \tileviii \& \begin{tikzpicture}
    \draw[white] (1,0) -- (1,1);
    \draw[very thick, white] (0,0) -- (0,1);
    \draw[very thick, blue] (0,0) -- (0,1) -- (1,1);
\end{tikzpicture}  \& \begin{tikzpicture}
    \draw[white] (1,0) -- (1,1);
    \draw[very thick, white] (0,0) -- (0,1);
    \draw[very thick, blue] (0,1) -- (1,1);
\end{tikzpicture}\& \begin{tikzpicture}
    \draw[white] (1,0) -- (1,1);
    \draw[very thick, white] (0,0) -- (0,1);
    \draw[very thick, blue] (0,1) -- (1,1);
\end{tikzpicture}\&\begin{tikzpicture}
    \draw[white] (1,0) -- (1,1);
    \draw[very thick, white] (0,0) -- (0,1);
    \draw[very thick, blue] (0,1) -- (1,1);
\end{tikzpicture}  \&\begin{tikzpicture}
    \draw[white] (1,0) -- (1,1);
    \draw[very thick, white] (0,0) -- (0,1);
    \draw[very thick, blue] (0,1) -- (1,1);
\end{tikzpicture}  \& \begin{tikzpicture}
    \draw[very thick, blue] (0,0) -- (1,0) -- (1,1) ;
    \draw[very thick, blue] (0,1) -- (1,1);
\end{tikzpicture}\\
\text{$i-1\,\,$ } \& \begin{tikzpicture}
    \draw[white] (1,0) -- (1,1);
    \draw[very thick, blue] (0,0) -- (0,1) -- (1,1);
\end{tikzpicture}  \& \begin{tikzpicture}
    \draw[white] (1,0) -- (1,1);
    \draw[very thick, white] (0,0) -- (0,1);
\end{tikzpicture} \&\begin{tikzpicture}
    \draw[white] (1,0) -- (1,1);
    \draw[very thick, white] (0,0) -- (0,1);
\end{tikzpicture} \& \begin{tikzpicture}
    \draw[white] (1,0) -- (1,1);
    \draw[very thick, white] (0,0) -- (0,1);
\end{tikzpicture}\&\begin{tikzpicture}
    \draw[white] (1,0) -- (1,1);
    \draw[very thick, white] (0,0) -- (0,1);
\end{tikzpicture}    \& \begin{tikzpicture}
    \draw[very thick, blue] (0,0) -- (1,0) -- (1,1) ;
\end{tikzpicture} \\
\text{$\vdots$} \& \begin{tikzpicture}
    \draw[white] (1,0) -- (1,1);
    \draw[very thick, blue] (0,0) -- (0,1);
\end{tikzpicture} \& \begin{tikzpicture}
    \draw[white] (1,0) -- (1,1);
    \draw[very thick, white] (0,0) -- (0,1);
\end{tikzpicture} \&\begin{tikzpicture}
    \draw[white] (1,0) -- (1,1);
    \draw[very thick, white] (0,0) -- (0,1);
    \node at (0.5,0.5) {\textcolor{blue}{\textbf{$P_i$}}};
\end{tikzpicture} \& \begin{tikzpicture}
    \draw[white] (1,0) -- (1,1);
    \draw[very thick, white] (0,0) -- (0,1);
\end{tikzpicture} \& \begin{tikzpicture}
    \draw[very thick, blue] (0,0) -- (1,0) -- (1,1) ;
\end{tikzpicture}\\
\text{$4$} \& \begin{tikzpicture}
    \draw[white] (1,0) -- (1,1);
    \draw[very thick, blue] (0,0) -- (0,1);
\end{tikzpicture} \& \begin{tikzpicture}
    \draw[white] (1,0) -- (1,1);
    \draw[very thick, white] (0,0) -- (0,1);
\end{tikzpicture}  \& \begin{tikzpicture}
    \draw[white] (1,0) -- (1,1);
    \draw[very thick, white] (0,0) -- (0,1);
\end{tikzpicture}  \& \begin{tikzpicture}
    \draw[very thick, blue] (0,0) -- (1,0) -- (1,1) ;
\end{tikzpicture} \& \\
\text{$3$} \&\begin{tikzpicture}
    \draw[white] (1,0) -- (1,1);
    \draw[very thick, blue] (0,0) -- (0,1);
\end{tikzpicture} \& \begin{tikzpicture}
    \draw[white] (1,0) -- (1,1);
    \draw[very thick, white] (0,0) -- (0,1);
\end{tikzpicture}   \& \begin{tikzpicture}
    \draw[very thick, blue] (0,0) -- (1,0) -- (1,1) ;
\end{tikzpicture} \& \\
\text{$2$} \& \begin{tikzpicture}
    \draw[white] (1,0) -- (1,1);
    \draw[very thick, blue] (0,0) -- (0,1);
\end{tikzpicture}  \& \begin{tikzpicture}
    \draw[very thick, blue] (0,0) -- (1,0) -- (1,1) ;
\end{tikzpicture}  \& \\
\text{$1$} \& \begin{tikzpicture}
    \draw[very thick, blue] (0,1) -- (0,0) -- (1,0) -- (1,1) ;
\end{tikzpicture}\\
\end{mosaic}
\]
Since the $i$ pipe emerges from row $i$ in column 1, the $1$ pipe must emerge from row $i$ in column 2 and go straight up until it reaches the top. This forces the $2$ pipe to emerge from row $i$ in column 3 and move straight up. This continues up to $i-1$, forcing only $\bcross$ tiles in rows $n+1, n,n-1,\dots, i+1$ and columns $1,\dots, i$. The labels around the boundary of the shape labeled $P_i$ are then exactly the boundary labels of a $(1,i-1)$-twist, and the labels around $Q_i$ are all $i$ plus the corresponding label on a $(1,n-i)$-twist. Thus, $P_i$ is a $(1,i-1)$-twist and $Q_i$ is a $(1,n-i)$-twist.

The map $\mfk{T}$ described in the statement will then send $R_i$ to the binary tree with left subtree $\mfk{T}(P_i)$ and right subtree $\mfk{T}(Q_i)$. By induction we may assume $\mfk{T}$ is well-defined and bijective on $(1,i-1)$-twists  and $(1,n-i)$-twists so it must  be a bijection between $(1,n)$-twists and binary trees with $n=(n-i)+(i-1)+1$ internal vertices.

By the induction hypothesis, the bracket vector $\mfk{T}(P_i)$ is $(c_1,\dotsc,c_{i-1})$, where the leftmost $\bump$ tile in the $j$th row of $\mfk{T}(P_i)$ (counted from bottom to top), occurs in the $(j-c_{j})$-th column (counted from left to right). Viewing  $Q_i$ as a $(1,n-i)$-twist and leveraging the inductive hypothesis, the bracket vector for $\mfk{T}(Q_i)$  is $(d_1,\dotsc,d_{n-i-1})$, where the leftmost bump tile in the $j$th row of $Q_i$ (counted from bottom to top) occurs in the $(j-d_{j})$-th column of $Q_i$ (counted from left to right). Let $(b_1,\dotsc,b_{n})$ be the bracket vector of $\mfk{T}(R_i)$.

For $j<i$, the $j$th vertex is in $\mfk{T}(P_i)$ so $b_j=c_j$ and by induction, the leftmost $\bump$ in the $j$th row of $\mfk{T}(R_i)$ is in the $(j-b_j)$th column. 

For $j>i$, the $j$th vertex of $R_i$ is the $(j-i)$th vertex of $\mfk{T}(Q_i)$. So $b_j = d_{j-i}$. By induction, the leftmost $\bump$ in row $j-i$ of $Q_i$ is in the $(j-i-d_{j-i})$th column of $Q_i$. So the leftmost bump in row $j$ of $R_i$ is in column $i+(j-i-d_{j-i}) = j-b_j$. 

Finally, the root node of $R_i$ will have label $i$, so $b_i=i-1.$ One can clearly see that the leftmost $\bump$ in the $i$th row of $R_i$ is in column $1=i-(i-1)$. Thus, for all $j,$ the leftmost $\bump$ in row $i$ of $R_i$ is in row $i-b_i$. 

We conclude that $\b(D)=\b(\mfk{T}(D))$ for all $(1,n)$-twists $D$.
\end{proof}

\subsection{Bracket vectors and the cubulation}\label{subsection:cubulation}

At this point we have sufficiently rich notation to be able to state and study the formula for the vertices of the cubulation of the associahedron in \cref{thm:main-result}. 

In the introduction we claimed informally that the way the cubulation of the associahedron in \cref{thm:main-result} interacts with the faces of the associahedron is the same as the way the standard cubulation interacts with the subboxes of the boxy subdivision of the $1\times 2 \times \cdots \times n-1$ box which has subbox poset isomorphic to the face lattice of the associahedron. In this section we show that this is true on the level of vertices: when we label the vertices of the associahedron by trees, the isomorphism between the cubulation of the associahedron and the standard cubulation will send the vertex labeled by a tree $T$ to the \textit{complemented bracket vector} of $T$.

Let $\sigma_i=s_1s_2\cdots s_{n-i+1}$ for $i\geq 2$, set $\sigma_1=\sigma_2$. Further define $\alpha_i:=\sigma_2\sigma_3\cdots \sigma_i$ with $\alpha_1:=\id$. Set $\LL^n_0:=[0,1]\times [0,2]\times \cdots \times [0,n-1]\subseteq \mbb{N}^{n-1}$, and define a map $\reversemathwitch:\LL^n_0 \to (S_n)^n$ (which we will define again in \cref{witch} once we have established further notation) that sends $\daleth=(\daleth_1,\daleth_2,\dots, \daleth_{n-1})\in \LL^n_0$ to $(x_1\alpha_1, x_2\alpha_2,\dots, x_n\alpha_n)$, with 
 \[x_j:= \prod_{i=1}^{n} \left(\begin{cases}
        s_i & j\leq \daleth_i\\
        \id & \text{otherwise.}
    \end{cases}\right). \]
The vertex corresponding to $\daleth$ in the cubulation of the Loday associahedron of \cref{thm:main-result} which we will construct in \cref{section:subdivision} will be 
\begin{equation}\label{eqn:vertices-of-subdivision}\sum_{i=1}^n\left(  (x_i\alpha_i) (\omega_1+\dots+\omega_n)\right). \end{equation}
It is no coincidence that this equation looks almost exactly like the definition of the bulky brick vector. 
If we define a map $\mu:S_n^\mm \to \Z^n$ given by 
\[(\sigma_1,\cdots, \sigma_\mm)\mapsto \sum_{i=1}^\mm (\sigma_i)(\omega_1+\cdots \omega_n).\]
Then \cref{eqn:vertices-of-subdivision} is equal to $\mu\circ \reversemathwitch$ and $\mu(\nabla(Q,w_0))$ is the composition of $\mu$ with the map $J\mapsto (J_{(1)},J_{(2)},\dots, J_{(m)})$.

In \cref{section:moment-polytopes}, we will see that this map is related to the \textit{moment map} of $\Fl_n^\mm$, and we will not use $\mu$ in this way again after this section. 

\begin{rmk}
    We write $\reversemathwitch$ as an upside down witch because it is magical. Prononcuiation options include ``inverse witch'', ``reverse witch'', or ``shtew'' which has the sounds of ``witch'' in reverse order.
\end{rmk}

In \cref{section:subdivision}, we will show that the subdivision of the Loday associahedron in \cref{thm:main-result} has vertices $\mu(\reversemathwitch(\LL_n^0))$. When we show that the subdivision is isomorphic to the standard cubulation of the $1\times 2\times \cdots \times (n-1)$ box by unit cubes, we will show that $\mu\circ \reversemathwitch$ gives an isomorphism between the vertex sets of the subdivisions and then we will build this up to an isomorphism of face posets. We never show directly that $\mu \circ \reversemathwitch$ is injective, but it comes ``for free'' as an implicit consequence of the main result.

Because the maps to $\Z^d$ from  $(1,n)$-twists and $\LL_n^0$ both factor through $S_n^n$, we can study the relationship between subdivision vertices corresponding to $(1,n)$-twists and subdivision vertices corresponding to $\LL_n^0$ by studying both maps to $S_n^n$ (i.e., we only need to look at $\reversemathwitch$, and not $\mu \circ \reversemathwitch$).

Recall that \cref{thm:bijection-with-rooted-binary-tree} allows us to identify the subword complex $\nabla(Q,w_0)$ with the set of $(1,n)$-twists where $Q$ is as in \cref{prop:is Loday}. 
Bump vectors give us a way to map from $\nabla(Q,w_0)$ to $\LL_n^0$. We can view $\nabla(Q,w_0)$ as a subset of $S_n^n$ and $\reversemathwitch$ maps from $\LL_n^0$ to $S_n^n$. We now show that these maps form a commuting triangle. Recalling \cref{thm:bijection-with-rooted-binary-tree}, this means that given an element of $\mbb{L}_0^n$ corresponding to a binary tree via its complemented bracket vector, its image under $\mu\circ\reversemathwitch$ is the vertex of Lodays associahedron that corresponds to that tree.

\begin{prop}\label{prop:combinatorial-shadow}
Let $\sigma_i=s_1s_2\cdots s_{n-i+1}$ for $i\geq 2$, $\sigma_1=\sigma_2$ and let $Q=(\sigma_1,\dots, \sigma_n)$ as in \cref{prop:is Loday}. Additionally define $\alpha_i=\sigma_2\sigma_3\cdots \sigma_i$ with $\alpha_1=\id.$ Both triangles in the following diagram commute.

\[
\begin{tikzcd}
     {\nabla(Q,w_0)} \arrow[rrd, "\b", hook] \arrow[d, "{J\mapsto (J_{(1)},\dots,J_{(n)})}",swap, hook] \arrow[rr, "\mfk{T}", two heads, hook] & & \text{binary trees} \arrow[d, "\b"',hook]  \\
     (S_n)^n  &  & \mathbb{L}_0^n \arrow[ll, "\reversemathwitch"' two heads, hook]
\end{tikzcd}
\]

\end{prop}
\begin{proof}
    Identifying $\nabla(Q,w_0)$ with $(1,n)$-twists, \cref{thm:bijection-with-rooted-binary-tree} is precisely the statement that the upper right triangle commutes. We now show that the lower left triangle commutes.

    Let $D$ be a $(1,n)$-twist and $J\in \nabla(Q,w_0)$ be the subword identified with $D$ by $\Xi$. Set $\b(D)=(b_1,\dots, b_r)$.
    We have $\mathwitch(\b(D))=(y_1\alpha_1,\dots, y_n\alpha_n)$ with $y_\iota\leq s_1\dots s_{n-1}$ and $s_{r-1}\leq y_\iota$ if and only if $\iota\leq b_r$ .

      For $\iota=1,$ $y_{1}$ is exactly $v_1=J_{(1)}$ since it collects (in the correct order) the simple reflections corresponding to positions of cross tiles in the first column. Since $\alpha_1=1,$ we note that $y_1\alpha_1=J_{(1)}$.

    Suppose the inductive hypothesis that $y_{i}\alpha_{i}=J_{(i)}.$ Then $y_i\alpha_i v_{i+1} =J_{(i+1)}.$

\noindent
\begin{minipage}{0.41\textwidth}
    \noindent
    We can represent $y_{i}\alpha_{i}v_{i+1}$ as the first $i+2$ columns of a $(1,n)$-twist $T_i$. The resulting picture is shown in Figure \ref{fig:wow}. Note that the labels above the columns are not the column positions. Column 2 is labeled with a 1, column 3 with a 2, etc. Since $y_i\alpha_iv_{i+1}=J_{(i+1)}$, the $i+1$st column permutation of $T_i$ is $J_{(i+1)}.$
    
    Since $y_i\leq s_{i}s_{i+1}\dots s_{n-1},$  column 1 of $T_i$ can only have crosses in rows $i+1$ through $n$. Multiplying by $\alpha_i$ corresponds to filling columns 2 through $i$ with crosses. 
\end{minipage}
\hfill
\begin{minipage}{0.52\textwidth}

   \noindent
    \[
\begin{mosaic}[font=\Large\sffamily]{0.77}
  \&  \& 1\& 2 \&  \& i-3 \& i-2 \& i-1 \& i\\
 \&  \tileii \& \tileix   \& \tileix    \& \tilecdots \& \tileix   \& \tileix \& \tileix  \&\tileNoBot \& $\cdots$\\
n \& \tileNoBot\& \tileix  \& \tileix    \& \tilecdots \& \tileix   \& \tileix \& \tileix  \&\tileLRtext{$\beta$} \& $\cdots$\\
 \&\tileLRtext{$\gamma$} \&\tilevdots\&\tilevdots\&\tileddots \&\tilevdots \&\tilevdots\&\tilevdots\&\tileNoTop \& $\iddots$\\
i+1\, \& \tileNoTop \& \tileix \& \tileix \& \tilecdots\& \tileix   \& \tileix \& \tileix \& \tileiv \\
i \& \tileviii \& \tileix \& \tileix \& \tilecdots\& \tileix \& \tileix   \& \tileiv\\
i-1\, \& \tileviii \& \tileix \& \tileix \& \tilecdots\& \tileix   \& \tileiv\\
 \& \tilevdots \& \tilevdots \& \tilevdots \& \tileiddots \& \tileiddots\\
4 \& \tileviii \& \tileix \& \tileix \& \tileiddots \& \\
3 \& \tileviii \& \tileix  \& \tileiv \& \\
2 \& \tileviii  \& \tileiv \& \\
1 \& \tileiv \\
\end{mosaic}
\]
    
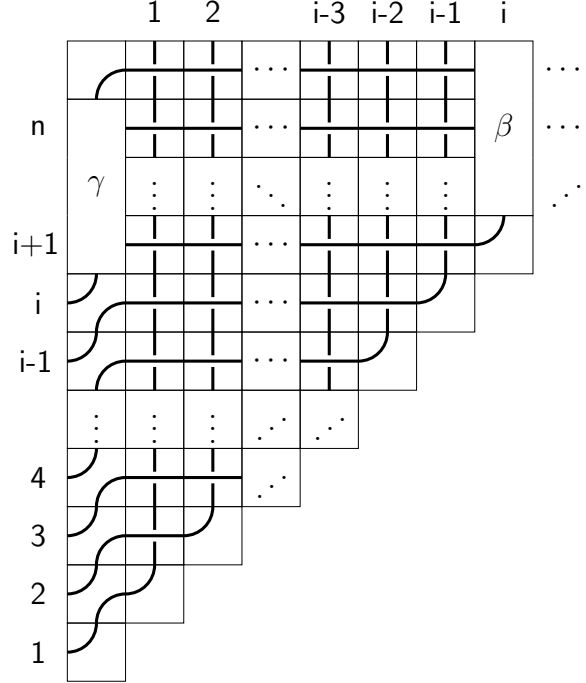
\captionof{figure}{The $(1,n)$-twist  $T_i$ representing the expression $y_{i-1}\alpha_{i-1}v_i$
    \vspace{10mm}}
    \label{fig:wow}

\end{minipage}

   Pipes starting below row $i$ end before column $i-1$ and pipes that leave column 1 between rows $i+1$ and $n+1$ move horizontally until column $i$, so we can ignore columns $2$ through $i-1$ and ignore pipes $1$ though $i-1$. The resulting picture, where we connect column 1 directly to column $i$ is shown in Figure \ref{fig:two cols short}.

\noindent
\begin{minipage}{0.36\textwidth}
         \centering
         \[
\begin{mosaic}[font=\Large\sffamily]{0.75}
  \&  \& i\& \\
 \&  \tileii   \&\tileNoBot \\
n \& \tileNoBot  \&\tileLRtext{$\beta$} \\
 \&\tileLRtext{$\gamma$}\&\tileNoTop \\
i+1 \,\& \tileNoTop  \& \tileiv \\
i \& \tileiv \\
\end{mosaic}\]
         
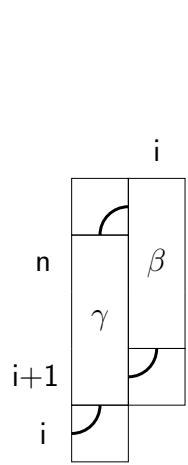
\captionof{figure}{Columns 1 and $i+1$ from figure \ref{fig:wow}}
         \label{fig:two cols short}
\end{minipage}
    \begin{minipage}{0.43\textwidth}
                 \[\begin{mosaic}[font=\Large\sffamily]{0.75}
  \&  \& i\& \\
 \&  \tileii   \&\tileNoBot\\
n \& \tileNoBot  \&\tileLR \\
 \&\tileLR\&\tileLR \\
  \&\tileLR{}\&\tileLRtext{$\beta$} \\
    \&\tileLR \&\tileLR \\
      \&\tileLRtext{$\gamma$}\&\tileLR \\
      \&\tileLR\&\tileLR \\
      \&\tileLR\&\tileLR \\
i+1 \, \& \tileLR  \& \tileiv \\
i \& \tileiv \\
\end{mosaic}\]
        
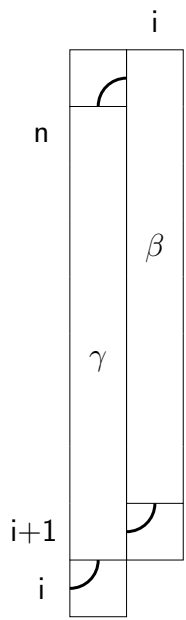
\captionof{figure}{Same as Figure \ref{fig:two cols short}}
         \label{fig:two cols long}
    \end{minipage}
\begin{minipage}{0.2\textwidth}
       We redraw Figure \ref{fig:two cols short} in Figure \ref{fig:two cols long} to give ourselves more room to show what is happening. In Figures \ref{fig:wow},\ref{fig:two cols short}, and \ref{fig:two cols long}, we use $\gamma$ to denote the unknown configuration of pipes in column 1 that encodes $y_i$ and $\beta$ to denote the unknown configuration of pipes in column $i+1$ that encodes $v_{i+1}$.
\end{minipage}

 \noindent
    \begin{minipage}{0.49\textwidth}
    \noindent
         \[\begin{mosaic}[font=\Large\sffamily]{0.9}
  \&  \& i\& \\
 \&  \tileii   \&\tilexvert \& $\gamma'_{n+1}$\\
n \& \tileNoBot  \&\tilexvert  \& $\gamma'_n$\\
 \&\tileLRtext{$\gamma'$}\&\tilevdots \\
  j+1\,\&\tileNoTop\&\tilexvert \& $\gamma'_{j+1}$\\
    j \&\tilebumpBotColor \&\tilebumpTopColor \& $\beta'_{j}$\\
     j-1 \,\&\tilexvert\&\tileNoBot \& $\beta'_{j-1}$ \\
      \&\tilevdots\&\tileLRtext{$\beta'$}  \\
     i+2\, \&\tilexvert\&\tileNoTop \& $\beta'_{i+2}$ \\
i+1\, \& \tilexvert  \& \tileiv \\
i \& \tileiv[red] \\
\end{mosaic}\]
\captionof{figure}{Same as Figure \ref{fig:two cols long} but pipe $i$ must turn right in some row so we assume it is the $j$th row. The unknown portions of $\gamma$ and $\beta$ are now $\gamma'$ and $\beta'$ respectively 
\vspace{10mm}}
         \label{fig:turn at j}
    \end{minipage}
    \hfill
\begin{minipage}{0.46\textwidth}
     The pipe labeled $i$ starts in row $i$ of column 1 and ends at the top of column $i+1$. Therefore, there must be some row $j$ where it turns right to leave column 1. Figure \ref{fig:turn at j} visualizes this. In positions $(r,1)$ for $i<r<j$ and $(r,i+1)$ for $j+1\leq r\leq n+1$, there must be cross tiles. This reduces the unknown portion of $\gamma$ to the positions $(r,1)$ with $j+1\leq r\leq n$ which we call $\gamma'$ and reduces the unknown portion of $\beta$ to the positions $(r,i+1)$ with $j+1\leq r\leq n$ which we call $\beta'$. We denote by $\beta'_r$ the label of the pipe which exits $\beta'$ on the right in row $r$ for $i+2\leq r\leq j-1$ and by $\beta'_j$ we denote the label of the pipe which exits $\beta'$ from the top. for $i+2\leq r\leq j,$ $J_{(i+1)}(r)=\beta'_r.$ 
\end{minipage}

We obtain Figure \ref{fig:switched} by switching $\beta'$ with the cross cells left of it in column 1, and switching the bump tile in position $(j,i+1)$ with the cross tile in position $(i+1,1).$ 

\noindent
  \begin{minipage}{0.48\textwidth}
         
         The definition of $y_{\iota}$ tells us that there is a bump in  position $(r,1)$ of Figure \ref{fig:turn at j} precisely when $b_r\leq i$. Every bump tile in column 1 of Figure \ref{fig:turn at j} is still a bump tile of column 1 in Figure \ref{fig:turn at j}. A cross tile in position $(r,1)$ of Figure \ref{fig:turn at j} becomes a bump tile in Figure \ref{fig:switched} precisely when there is a non-cross tile in position $(r,i+1),$ and that happens precisely when $b_r=i+1.$ Hence, there is a bump in  position $(r,1)$ of Figure \ref{fig:switched} precisely when $b_r\leq i+1$.
         This means that the permutation obtained from column 1 of Figure \ref{fig:switched} is $y_{i+1}.$

    \end{minipage}
   \hfill
\begin{minipage}{0.47\textwidth}
       \[\begin{mosaic}[font=\Large\sffamily]{0.9}
  \&  \& i\& \\
 \&  \tileii   \&\tilexvert \& \, $\gamma'_{n+1}$\\
n \& \tileNoBot  \&\tilexvert  \&\,  $\gamma'_n$\\
 \&\tileLRtext{$\gamma'$}\&\tilevdots \\
  j+1\&\tileNoTop\&\tilexvert \& $\alpha'_{j+1}$\\
    j\&\tileviii \&\tilexvert \& \, $\beta'_{j}$\\
     j-1 \&\tileNoBot \&\tilexvert\&\,  $\beta'_{j-1}$ \\
      \&\tileLRtext{$\beta'$} \&\tilevdots  \\
     i+2 \&\tileNoTop \&\tilexvert \& \, $\beta'_{i+2}$ \\
i+1 \& \tilebumpBotColor \& \tileiv[red] \\
i \& \tileiv[red] \\
\end{mosaic}\] 
 
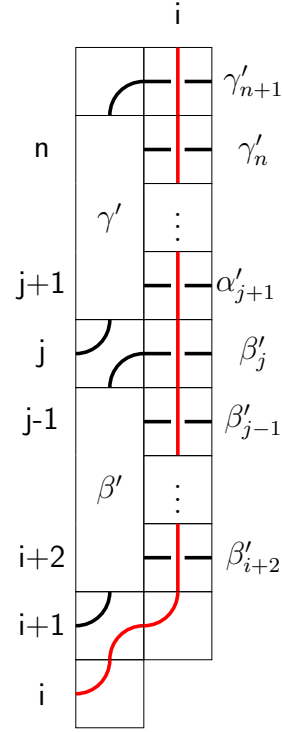
\captionof{figure}{Obtained from Figure \ref{fig:turn at j} by switching $\beta'$ with the crosses to the left of it (in column 1) and swapping the bump in position $(j,i+1)$ with the cross in position $(i+1,1)$.
 \vspace{10mm}}
         \label{fig:switched}
       
    \end{minipage}
  
  It is evident from the figures that the pipes entering $\beta'$ from the left and on the bottom are the same (and in the same places) between Figure \ref{fig:turn at j} and Figure \ref{fig:switched}. Thus, the same is true for the pipes exiting $\beta'$ from the top and the right. One can further see from these figures that this means that the pipes exiting column $i+1$ on the right are the same between Figure \ref{fig:turn at j} and Figure \ref{fig:switched}.

    When we put columns $1,i+1$ of Figure \ref{fig:switched} back in the context of Figure \ref{fig:wow}, the $i+1$st column permutation will then remain equal to $J_{(i+1)}$. At the same time, we know that the first column corresponds to $y_{i+1},$ columns $2$ through $i$ correspond to $\alpha_i$, and now column $i+1$ gives $s_1s_2\dots s_{n-i}=\sigma_{i+1}.$ Hence $J_{(i+1)}=y_{i+1}\alpha_i\sigma_{i+1}=y_{i+1}\alpha_{i+1}.$ We conclude that for all $\iota$, $J_{(\iota)}=y_\iota \alpha_{\iota}$.\qedhere
   
\end{proof}

\section{Brick polytopes as moment polytopes of brick varieties}\label{section:moment-polytopes}

To every projectively embedded variety with a torus action, one can associate its \textit{moment polytope}. The focus of this section is to analyze the moment polytopes of brick varieties.

We will recall the definition of the moment polytope of a projectively embedded variety with a torus action in \cref{background}. We will describe Grassmannians, complete flag varieties, Schubert varieties, Bott-Samelson varieties, and brick varieties in  \cref{subsection:Schubert-varieties} and \cref{subsection:brick-varieti}. In particular, we will show that brick polytopes and bulky brick polytopes are both the moment polytopes of brick varieties, albeit with respect to two different projective embeddedings (one into a product of Grassmannians and the other into a product of complete flag varieties). In \cref{subsection:Minkowski}, we will describe Minkowski decompositions of the brick polytope and bulky brick polytope, and we will note a characterization of toric Richardson varieties in terms of brick varieties. In certain cases, the brick variety can be viewed as a resolution of singularities for a Richardson variety; in \cref{section:fibres}, we give a sufficient pattern-avoidance condition for when such a \textit{brick resolution} is isomorphic to the Richardson variety whose singularities it resolves.

\subsection{Moment polytopes (of subvarieties of products)}\label{background}

A \textbf{torus} is an algebraic group which is isomorphic to $(\C^{\times})^k$ for some $k\geq 0$, where $\C^{\times}:=\C\setminus\{0\}.$ When $T$ is a torus which acts on a variety $X$ we will write $X^T$ to denote the set of fixed points of the action. In subsequent sections, the torus we work with will always be the set of diagonal matrices in $\mrm{GL}_n(\C)$. 

We now recall the notion of the \textit{moment polytope} of a projectively embedded variety with a torus action.

    \begin{defn}
        Let $T=(\C^\times)^k$. An algebraic action of $T$ on the projective space $\mbb{P}^{n}$ is, up to change of coordinates, always of the form $(t_1,\dots,t_k)\cdot (x_1:\dots:x_{n+1})=(x_1\prod_{i=1}^k t_i^{w_{1,i}}:\dots:x_{n+1}\prod_{i=1}^k t_i^{w_{n+1,i}})$ with $w_{j,i}\in \mbb{Z}$. We call each vector $(w_{j,1},\dotsc,w_{j,k})\in\mbb{Z}^k\subseteq\mbb{R}^n$ a \textbf{weight vector}. In such a case, we say the $\textbf{$T$-weight}$ of a fixed point of the form $(0:\dots:0:x_j:0:\dots:0)$ is the weight vector $(w_{j,1},\dots,  w_{j,k})$. In general a fixed point may have more than one nonzero coordinate, in which case each nonzero coordinate is acted upon with the same weight vector, which we will call the \textbf{$T$-weight} of the fixed point. Suppose $X$ is a variety with an algebraic action of $T$.
        \begin{itemize}  
            \item When we have a $T$-equivariant embedding $f\colon X\hookrightarrow \mbb{P}^n$ for some $n$, and have chosen such coordinates for $\mbb{P}^n$, then we define the \textbf{moment map} with respect to $f$ as the map $\mu_{\mbb{P}^n}:X^{T}\to \Z^k$  sending $\gamma \in X^{T}$ to the $T$-weight of $f(\gamma)$.  We define the \textbf{moment polytope} of $X$ with respect to $f$ as the convex hull of the image of $\mu_{\mbb{P}^n}$ and denote it $\mu_{\mbb{P}^n}(X)$\footnote{The moment polytope $\mu_{\mbb{P}^n}(X)$ does not depend on the choice of coordinates of $\mbb{P}^n$.}.
            \item The \textbf{effective torus} is $T/\mrm{St}_{T}(x)$, where $\mrm{St}_{T}(x)$ is the $T$-stabilizer of a \textit{general}\footnote{\textit{general} in the sense that there is an open subset $U\subseteq X$ such that $\mrm{St}_{T}(x)$ is constant for all $x\in U$.}  $x\in X$.
            \item $X$ is a \textbf{toric variety} with respect to the action of $T$ if $X$ contains a dense $T$-orbit (which must be open in $X$ by \cref{lem:dense-torus-orbit-is-open}). Equivalently, by the orbit-stabilizer theorem, $X$ is toric if and only if $\dim(X)=\dim(T/\mrm{St}_{T}(x)).$  Note that we do not require toric varieties to be normal, as is often the case in other sources.
        \end{itemize}
    \end{defn}

\begin{rmk}
    The term ``moment map'' originates in symplectic geometry where the moment map is a map from a manifold with a Hamiltonian action of a Lie group $G$ to the dual of the Lie algebra of $G$. In the broader context of this paper it makes some sense for the moment map of a variety $X$ acted on by a torus $T$ to be formulated as a function $X\to \R \otimes_{\Z} \mrm{Hom}(T,\C^{\times})\cong \R^k$. Our formulation of the moment map as a function from $X^{T}$ to $\R^k$ fits with the algebraic approach of \cite{BrionProcesi}.
\end{rmk}

The following lemma will tell us that, for any variety with the algebraic action of a torus $T$, any dense $T$-orbit in the variety is open. In particular, in a variety that is toric with respect to $T$, every dense orbit of a torus $T$  is open. 

\begin{lem}\label{lem:dense-torus-orbit-is-open}
    Let $X$ be an algebraic variety with the algebraic action of a torus $T$. Any dense orbit of $T$ in $X$ is a smooth, open subvariety of $X$.
\end{lem}
\begin{proof}
    A dense torus orbit $\mcal{O}:=T\cdot x$ in $X$ is a smooth locally closed algebraic variety by \cite[Chapter I, Proposition 1.8]{B91}. Thus, $\mcal{O}=U\cap Z$, where $U$ is an open subvariety of $X$ and $Z$ is a closed subvariety of $X$. As $\mcal{O}$ is dense in $X$, it follows that $X=\overline{\mcal{O}}=\overline{U\cap Z}$, which is contained in $Z$ as $Z$ is closed. Therefore, $Z=X$, and hence $\mcal{O}=U\cap X=U$ is open.
\end{proof}

Given a variety $X$ with the algebraic action of a torus $T$ that is $T$-equivariantly embedded in $\P^n$, the dimension of the moment polytope $\mu_{\P^n}(X)$ is equal to the dimension of the effective torus acting on $X$. In particular, $X$ is toric with respect to $T$ if and only if $\dim_{\C}(X)=\dim_{\R}(\mu_{\P^n}(X))$ (c.f. \cite{LauraThesis}[Lemma 3.4.7].)

Recall the Segre embedding $\P^a \times \P^b \hookrightarrow \P^{(a+1)(b+1)-1}$ which maps $((x_1:\dots: x_a:x_{a+1}),(y_1:\dots: y_b:y_{b+1}))$ to $(x_1y_1:x_1y_2:\cdots: x_1y_{a+1}: x_2y_1:\cdots, x_iy_j:\cdots: x_ay_{b+1}).$ Suppose $T\cong (\C^\times)^k$ acts on $\P^a$ and $\P^b$ algebraically. Defining the Segre embedding to be $T$-equivariant induces an action on $\P^{(a+1)(b+1)-1}$. If the $x_i$ coordinate has $T$-weight $\alpha_i$ and the $y_j$ coordinate has $T$-weight $\beta_j$, then the $x_iy_j$ coordinate of $\P^{(a+1)(b+1)-1}$ has $T$-weight $\alpha_i+\beta_j$. When we take products of more than two projective spaces we can iterate the Segre embedding and find that weights still add in this manner. Equation \ref{eq:aaa1} below formalizes this conclusion. A more detailed explanation can be found in \cite{E19}.

Let $T$ be an algebraic torus acting algebraically on $X\subseteq Y_1\times Y_2\times \dots Y_\mm$ with $T$-equivariant projection maps $\pi_i:X\to Y_i$ and $T$-equivariant embeddings $Y_i\hookrightarrow \P^{n_i}$ (with projective coordinates chosen to be $T$-fixed points). This setup provides us with a moment map $\mu_{\P^{\prod(n_i+1)-1}}:X^T\to \R^k$ with respect to the Segre embedding. The discussion above tells us that 
\begin{equation}
    \mu_{\P^{\prod(n_i+1)-1}}(x)= \sum_{i=1}^\mm \mu_{\P^{n_i}}(\pi_i(x))\label{eq:aaa1}
\end{equation}
for $x\in X^T$.
We will soon see that 
\begin{align}
\mu_{\P^{\prod(n_i+1)-1}}(X)&=\conv_{x\in X^T} \left( \sum_{i=1}^\mm \mu_{\P^{n_i}}(\pi_i(x)) \right) \label{eq:aaa1new}\\
&=\sum_{i=1}^\mm \conv_{x\in X^T} \mu_{\P^{n_i}}(\pi_i(x)) \label{eq:aaa2}\\
&=\sum_{i=1}^\mm \mu_{\P^{n_i}}(\pi_i(X)).\label{eq:aaa3}
\end{align}
We might summarize this observation abstractly as \[\conv_{a\in A} \sum_{b\in B} f(a,b) = \sum_{b\in B} \conv_{a\in A} f(a,b)\] (where $A=X^T$, $B=[\mm]$, and $f(a,b)=\mu_{\P^{n_b}}(\pi_b(a))$). This observation does not hold for arbitrary $A,B,f$. When it does hold, it produces very useful Minkowski decompositions. We can already see \cref{eq:aaa1new} from \cref{eq:aaa1}. To see that  \cref{eq:aaa2} and \cref{eq:aaa3} are equal, we show in \cref{pre- fix of proj} that $\pi_i(X)^T=\pi_i(X^T).$  \cref{lem:MinkowskiMachine} proves all of the above equalities, using \cref{lim optimizes functional} as a stepping stone towards \cref{eq:aaa2}.

    \begin{lem}\label{pre- fix of proj}
    Let $T=(\mbb{C}^\times)^k$ be a torus, and suppose $X$ and $Y$ are projective with algebraic actions of $T$. If $f\colon X\to Y$ is a $T$-equivariant regular map, then $f(X^T)=f(X)^T$.
\end{lem}
\begin{proof}
    If $x\in X^T$, then $f(x)\in (f(X))^T$, as $f$ is $T$-equivariant.
    Now let $x\in (f(X))^T$. Since $x$ is a point, it is closed in the Zariski topology on $f(X)$. The map $f$ being continuous, we have that $f^{-1}(x)$ is closed in $X$. Since $X$ is projective and $f^{-1}(x)$ is closed in $X$, it follows that $f^{-1}(x)$ is projective. Moreover, the set $f^{-1}(x)$ is fixed by the action of $T$ (otherwise, $x=f(t\cdot y)=t\cdot f(y)\notin  f^{-1}(x)$ for some $t\in T$ and $y\in f^{-1}(x)$, which is a contradiction).
    As $f^{-1}(x)$ is nonempty and projective, it follows from Borel's fixed point theorem that $f^{-1}(x)$ has a $T$-fixed point.
\end{proof}

\begin{lem}\label{lim optimizes functional}
    Let $T=(\C^\times)^k$ act on $\P^{n-1}$ by $t\cdot (x_1:\dots:x_{n})=(t^{\alpha_1}x_1,\dots, t^{\alpha_n}x_n)$ with $\alpha_i\in \Z^k$. Let $X\subseteq \P^{n-1}$ be some subvariety of $\P^{n-1}$ which is closed under the action of $T$. Let $\beta=(\beta_1,\dots, \beta_k)\in \mathbb{Z}^k$ and define a map $\gamma:\C^{\times}\to T$ by $\gamma(s)=(s^{\beta_1},s^{\beta_2},\dots, s^{\beta_k}).$
    If $\alpha_i\cdot \beta \neq \alpha_j \cdot \beta$ for $\alpha_i\neq \alpha_j,$ and $z=(z_1:\dots :z_{n})$ then 
    $\mu(\lim_{s\to \infty} \gamma(s)\cdot z)$ is the $\alpha_i$ which maximizes $\alpha_i\cdot \beta$ and satisfies $z_i\neq 0$ (in other words, $\mu(\lim_{s\to \infty} \gamma(s)\cdot z)=\alpha_i$ with $i\in \mrm{argmax}_{j \text{ s.t. } z_j\neq 0}(\alpha_j\cdot \beta)$).
\end{lem}

\begin{proof}
    We can compute $\gamma(s)\cdot z = (s^{\alpha_1\cdot \beta}z_1:s^{\alpha_2\cdot \beta}z_2:\cdots:s^{\alpha_n\cdot \beta}z_n)$. 
    
    Suppose without loss of generality that $1\in \mrm{argmax}_{j \text{ s.t. } z_j\neq 0}(\alpha_j\cdot \beta)$. That is, $z_1\neq 0$ and 
    \begin{equation}\label{lilineq}
        \alpha_1\cdot \beta\geq \alpha_i\cdot \beta \qquad  \text{for all $i$ with $z_i\neq 0$.}
    \end{equation} Our goal is now to show that $\mu(\lim_{s\to \infty}\gamma(s)\cdot z)=\alpha_1$. 
    
   We can rewrite $\gamma(s)\cdot z$ as
$\gamma(s)\cdot z = (z_1:s^{\alpha_2\cdot \beta-\alpha_1\cdot \beta}z_2:\cdots:s^{\alpha_n\cdot \beta-\alpha_1\cdot \beta}z_n)$. First note that whenever $z_i=0,$ the $i$th coordinate of $\lim_{s\to \infty}(\gamma(s)\cdot z)$ will be 0. Next observe that the $i$th coordinate of  $\lim_{s\to \infty}(\gamma(s)\cdot z)$  will also be 0 whenever $\alpha_i\cdot \beta-\alpha_1\cdot \beta< 0$. For all $i$ such that $z_i\neq 0$, we have $\alpha_i\cdot \beta-\alpha_1\cdot \beta\leq 0$ by \cref{lilineq}. Our assumption that $\alpha_i\cdot \beta \neq \alpha_j \cdot \beta$ for $\alpha_i\neq \alpha_j$  means that $\alpha_1\neq \alpha_i$ implies either $z_i=0$ or $\alpha_i\cdot \beta-\alpha_1\cdot \beta< 0$. Hence, the $i$th coordinate of $\lim_{s\to \infty}(\gamma(s)\cdot z)$ is 0 unless $\alpha_1=\alpha_i$. 
Thus,
    $\lim_{s\to \infty}(\gamma(s)\cdot z) = (z_1':z_2':\dots: z_n')$ with 
    \[z_i'=\begin{cases}
        z_i & \text{if } \alpha_i = \alpha_1\\
        0 & \text{else.}
    \end{cases}\]
Since the coordinates that survive in the limit all have weight $\alpha_1$, the limit $\lim_{s\to \infty}(\gamma(s)\cdot z)$ is a $T$-fixed point with weight $\alpha_1$, and $\mu(\lim_{s\to \infty}(\gamma(s)\cdot z))=\alpha_1$. 
\end{proof}

In the following proposition, we explain how the moment polytope of a variety torus-equivariantly embedded in a product can be written as the Minkowski sum of the moment polytopes of the factors in the product. 
We find this result to be very useful for producing interesting Minkowski decompositions. The result itself is known to some experts, but we were unable to find a satisfactory statement of it in the literature. As such, we have included the statement, together with a complete proof, below.

\begin{prop}\label{lem:MinkowskiMachine} Let $X$ be a closed and irreducible subvariety of a product of closed irreducible varieties $Y_1\times Y_2\times \cdots \times Y_{\mm}$  with projection maps $\pi_i:X\to Y_i$ and embedding maps $\rho_i:Y_i\hookrightarrow \P^{n_i}.$ Suppose a torus $T=(\C^{\times})^k$ acts on $X$, each $Y_i$, and each $\P^{n_i}$ so that each $\pi_i, \rho_i$ is $T$-equivariant. Embedding each $Y_i$ in projective space and applying the Segre embedding endows $X$ with a moment map $\mu_{\P^{\prod(n_i+1)-1}}$ such that 
\begin{align}
\mu_{\P^{\prod(n_i+1)-1}}(X)&=\conv_{x\in X^T} \left( \sum_{i=1}^\mm \mu_{\P^{n_i}}(\pi_i(x)) \right) \label{eq:bbb1}\\
&=\sum_{i=1}^\mm \conv_{x\in X^T} \mu_{\P^{n_i}}(\pi_i(x)) \label{eq:bbb2}\\
&=\sum_{i=1}^\mm \mu_{\P^{n_i}}(\pi_i(X)).\label{eq:bbb3}
\end{align}

\end{prop}
\begin{proof}
    We recover \cref{eq:bbb1} from \cref{eq:aaa1}. The equivalence of \cref{eq:bbb2} and \cref{eq:bbb3} follows from \cref{pre- fix of proj} and our definition of the moment polytope. It remains to prove that $\mu_{\P^{\prod(n_i+1)-1}}(X) = \sum_{i=1}^\mm \mu_{\P^{n_i}}(\pi_i(X)).$
     
    For any $x\in X^T,$ $\pi_i(x)\in \pi_i(X)^T$ so $\mu_{\P^{n_i}}(\pi_i(x))\in \mu_{\P^{n_i}}\pi_i(X)$. Hence \[\mu_{\P^{\prod(n_i+1)-1}}(X)=\conv_{x\in X^T}\sum_{i=1}^\mm \mu_{\P^{n_i}}(\pi_i(x)) \subseteq \sum_{i=1}^\mm \mu_{\pi_i(X)}.\]  We therefore just need to show that $\sum_{i=1}^\mm \mu_{\P^{n_i}}\pi_i(X) \subseteq \mu_{\P^{\prod(n_i+1)-1}}(X).$

    Let $v$ be an arbitrary vertex of $Q=\sum_{i=1}^\mm \mu_{\P^{n_i}}\pi_i(X)$. The normal cone to $v$ is the set $N_v=\{w\in \R^{k}: v\in \argmax_{x\in Q}x\cdot w \}.$ The interior of $N_v$ is $N_v^\circ=\{w\in \R^{k}: v= \argmax_{x\in Q}x\cdot w \}.$ For $p,q,$ vertices of $\mu(X)$ (equivalently, they are weights of fixed points), $\{x\in \R^k:x\cdot p=x\cdot q\}$ is a hyperplane of $\R^k$ while $N_v$ is an $n$-dimensional cone. Thus, there will exist some $w\in N_v^\circ \cap \Z^k$ such that there is a unique vertex $p$ of $\mu(X)$ maximizing $p\cdot w.$ Since $w\in N_v^{\circ},$ we will have $p=v.$

    Since $v=\argmax_{x\in Q}x\cdot w,$ we will also have that $v=\sum_{i=1}^\mm \argmax_{x\in \mu_{\P^{n_i}}\pi_i(X)}x\cdot w.$ Note that the fact that $\argmax_{x\in Q}x\cdot w$ is the single point $v$ implies that for each $i$, $\argmax_{x\in \mu_{\P^{n_i}}\pi_i(X)}x\cdot w $ will also be a single point. 

    Now define $\gamma:\C^\times\to (\C^\times)^k$ by $\gamma(s)=(s^{w_1},s^{w_2},\dots, s^{w_k})$ and let $z$ be a general point of $X$. Since $w$ was chosen to induce a linear functional that is maximized by a unique vertex of $\mu(X),$ $\mu(\lim_{s\to \infty}\gamma(s)\cdot z)$ will pick out that unique vertex (note that we may choose coordinates on each $\P^{n_i}$ compatible with applying \cref{lim optimizes functional}). Moreover, we will have
    \begin{align*}
\mu(\lim_{s\to \infty} \gamma(s) \cdot z)&=\sum_{i=1}^\mm \mu_{\P^{n_i}}\pi_i(\lim_{s\to \infty} \gamma(s) \cdot z))\\
&=\sum_{i=1}^\mm \mu_{\P^{n_i}}(\lim_{s\to \infty} \gamma(s)\cdot \pi_i(z))\\
&=\sum_{i=1}^\mm \argmax_{x\in \mu_{\P^{n_i}}\pi_i(X)} x\cdot w\\
&=v.
\end{align*}

This proves that every vertex of $Q=\sum_{i=1}^\mm \mu_{\P^{n_i}}\pi_i(X)$ is contained in $\mu(X)$ so $\sum_{i=1}^\mm \mu_{\P^{n_i}}\pi_i(X) \subseteq \mu(X)$ and therefore $\sum_{i=1}^\mm \mu_{\P^{n_i}}\pi_i(X) = \mu(X)$   
\end{proof}

A number of corollaries of \cref{lem:MinkowskiMachine} have been proven elsewhere by other means. These include \cref{lem:Richardson sum} (originally appearing as Corollary 6.11 in \cite{Kostant-Toda}, and later quoted by Proposition 2.7 of \cite{TW15} and Theorem 5.6 of \cite{MarioMV}) and
\cref{Brick Gr sum} (originally appearing as Proposition 3.29 in \cite{bricksort} and generalized by Proposition 5.17 in \cite{PS15}). We use \cref{lem:MinkowskiMachine} to get a Minkowski decomposition for bulky brick polytopes (and thus a new Minkowski decomposition for certain (non bulky) brick polytopes) in \cref{thm:Minkowski}.

\subsection{Grassmannians and flag varieties}\label{subsection:Schubert-varieties}

 Next, we will recall Grassmannians, complete flag varieties, relative position, Schubert varieties, and Richardson varieties, and discuss their moment maps.
    
	Fix a basis $\{e_1,\dotsc,e_n\}$ for $\mbb{C}^n$. Let $G=\mrm{GL}_n(\mbb{C})$ be the set of $n\times n$ invertible matrices over $\mbb{C}$. Let $B_+$ (resp. $B_-$) denote the set of upper (resp. lower) triangular matrices in $G$, and let $T$ denote
    the set of diagonal matrices in $G$. For any $I\subseteq [n]$, define $E_{I}:=\mrm{Span}( \{e_i:i\in I\})$.

 \subsubsection{Grassmannians}

    Let $\mrm{Gr}(k,n)$ be the \textbf{Grassmannian} of $k$-planes in $\mbb{C}^n$, and observe that there is a natural action of $G$ on $\mrm{Gr}(k,n)$.
   The $T$-fixed points $\Gr(k,n)^{T}$ of $\Gr(k,n)$ are exactly the coordinate subspaces $E_{I}$ with $I\subseteq [n]$ such that $|I|=k$.
    
    We view $S_n$ as a subgroup of $G$ by associating the permutation $\sigma$ with the matrix that sends $e_i$ to $e_{\sigma(i)}$ for each $i$. In particular the action of $S_n$ on $\Gr(k,n)$ is such that $\sigma\cdot E_{[k]}=E_{\sigma\cdot [k]}=E_{\{\sigma(1),\dotsc,\sigma(k)\}}$.  For $I\subseteq [n]$ with $|I|=k$, we can always write $E_{I}=\sigma\cdot E_{[k]}$ for some $\sigma\in S_n$.

    The Grassmannian $\Gr(k,n)$ can be embedded in projective space via the  ``Pl\"ucker embedding" $\psi_k\colon \mrm{Gr}(k,n)\hookrightarrow  \mbb{P}^{{n\choose k}-1}$ which sends the column span of a $k\times n$ matrix $M$ to a point in $\mbb{P}^{{n\choose k}-1}$ with coordinates given by the $k\times k$ minors of $M$.
The $T$-weight of the image of the fixed point $E_{\{i_1,\dotsc,i_k\}}$ of $\mrm{Gr}(k,n)$ under $\psi_k$ is $p_{i_1}+\cdots +p_{i_k}$ where $p_i$ is the $i$th standard basis vector (as in \cref{subsection:brick-polytopes}). The $T$-weight of $\sigma\cdot E_{[k]}$ is $\sigma(\omega_k)$  where $\omega_k=p_1+\cdots+ p_k$ (as in \cref{subsection:brick-polytopes}).

We will suppose that $T$ acts diagonally on any product of Grassmannians $\prod_{i=1}^\mm \Gr(k_i,n)$ and will embed such a product in projective space by first applying $\psi_{k_i}$ to each $\Gr(k_i,n)$ and then applying the Segre embedding. 

\begin{nota}
    Instead of writing $\mu_{\mbb{P}^{{n\choose k}-1}}$ for the moment map on $\Gr(k,n)$ with respect to the Pl\"ucker embedding, we will write $\mu_{\Gr(k,n)}$. Similarly, when $Z$ is any product of Grassmannians, we will write $\mu_{Z}$ for the moment map obtained by applying Pl\"ucker embeddings to each Grassmannian and then applying the Segre embedding. 
\end{nota}

With this notation,
$\mu_{\Gr(k,n)}(\sigma\cdot E_{[k]})= \sigma( \omega_k)$.

\subsubsection{The flag variety}\label{flags}
   
   The \textbf{complete flag variety} $\mrm{Fl}_n$ is the variety of complete flags of subspaces in $\mbb{C}^n$:
	\[\mrm{Fl}_n=\{F_\bullet=(F_1\subseteq F_2\subseteq \dotsb \subseteq F_{n-1}\subseteq F_n=\mbb{C}^n) : \mrm{  dim}(F_i)=i\}.\]
    The standard flag in $\mbb{C}^n$ is $E_\bullet^{\text{st}}=\left(E_{[1]}\subseteq E_{[2]}\subseteq \dotsb\subseteq E_{[n-1]}\subseteq \mbb{C}^n\right)$, and the opposite flag in $\mbb{C}^n$ is $E_\bullet^{\text{op}}=
    \left(E_{\{n\}}\subseteq E_{\{n,n-1\}}\subseteq \dotsb\subseteq E_{\{n,n-1,n-2,\dotsc,2\}}\subseteq \mbb{C}^n\right)$.
    The group $G$ acts on $\mrm{Fl}_{n}$ by the action
	\[g\cdot F_\bullet:=((g\cdot F_1)\subseteq (g\cdot F_2)\subseteq \dotsb\subseteq (g\cdot F_{n-1})\subseteq\mbb{C}^n),\quad g\in G,\text{ }F_\bullet\in\mrm{Fl}_n.\]
 
 The flag variety $\Fl_n$ may be ($G$-equivariantly) identified  with the ``homogeneous space'' $G/B_{+}$ by noting that $G$ acts transitively on $\Fl_n$ and $B_{+}$ is the stabilizer of $E_\bullet^{\mrm{st}}$.

 The $T$-fixed points of $\Fl_n$ are flags $F_{\bullet}$ such that each $F_i$ is a $T$-fixed point of $\Gr(k,n)$. We can describe the set of fixed $T$-fixed points as
$(\Fl_n)^{T}=\{\sigma\cdot E_\bullet^{\mrm{st}}: \sigma\in S_n\}$.

Embedding $\Fl_n$ in $\prod_{k=1}^{n}\mrm{Gr}(k,n)$ via $\phi$ means that the restriction of the moment map $\mu_{\prod_{k=1}^{n}\mrm{Gr}(k,n)}$ to $\phi(\Fl_n)$ is a moment map for $\Fl_n$. Under this moment map, $\mu_{\prod_{k=1}^{n}\mrm{Gr}(k,n)}(\phi(\sigma\cdot E_\bullet^{\text{st}}))=\sigma(\omega_1+\dotsb+\omega_n)$.

We have a $G$-equivariant (and thus $T$-equivariant) embedding $\phi\colon \mrm{Fl}_n\hookrightarrow \prod_{k=1}^{n}\mrm{Gr}(k,n)$, where $F_\bullet=(F_1\subseteq F_2\subseteq \dotsb \subseteq F_{n}\subseteq \mbb{C}^n)$ is sent to $(F_1, F_2, \dotsc, F_{n})$ (and where $G$ acts diagonally on $\prod_{k=1}^{n}\mrm{Gr}(k,n)$). We will denote the projection $\mrm{Fl}_n\to \Gr(k,n)$ by $\phi_k$ so that $\phi(F_{\bullet})=(\phi_1(F_{\bullet}),\dots, \phi_n(F_\bullet)).$

\begin{nota}
    Instead of writing $\mu_{\prod_{k=1}^n \mrm{Gr}(k,n)}\circ \phi$ for the moment map on $\Fl_n,$ we will write $\mu_{\Fl_n}$. We will write $\mu_{(\Fl_n)^m}$ for the moment map on $(\Fl_n)^m$ (or any $T$-invariant subvariety thereof) with respect to the embedding obtained by applying $\phi$ then applying Pl\"ucker embeddings and finally Segre embeddings. 
\end{nota}
With this notation, \[\mu_{\Fl_n}(\sigma \cdot E_\bullet^{\text{st}}) = \sigma(\omega_1+\dotsb+\omega_n).\]

\subsubsection{Flags in relative position}

 The group $G\times G$ acts on $\mrm{Fl}_n\times \mrm{Fl}_n$ by $(f,g)\cdot (F_\bullet,G_\bullet):=(f\cdot F_\bullet,g\cdot G_\bullet)$ for all $f,g\in G$, $F_\bullet,G_\bullet\in\mrm{Fl}_n$. Define $G_\Delta:=\{(g,g)\mid g\in G\}\subseteq G\times G$. For $\sigma\in S_n$, the set of pairs of flags in \textbf{relative position $\sigma$} is the $G_\Delta$-orbit of $(E_\bullet^{\text{st}},\sigma\cdot E_\bullet^{\text{st}})$ in $\mrm{Fl}_n\times \mrm{Fl}_n$, and we will denote this set by $D_\sigma^\circ$. There is a well-known alternative characterization of $D_\sigma^\circ$, which we record below.
 \begin{lem}
     The set of pairs of flags in relative position $\sigma$ can equivalently be characterized as 
 \[D_{\sigma}^{\circ} = \{(F_\bullet,G_\bullet)\in (\Fl_n)^2: \mrm{dim}(F_i\cap G_j)=|[i]\cap\sigma([j])| \text{ for all } i,j\}.\]
 \end{lem}
 \begin{proof}
     Note that the action of $G_{\Delta}$ preserves intersection dimensions (for vector spaces $V_1,V_2\in\mbb{C}^n$ and $g\in G$, we have $\dim(gV_1\cap gV_2)=\dim(V_1\cap V_2)$ and for flags $F_\bullet$, $G_\bullet\in\Fl_n$, we have  $\dim((gF_\bullet)_i\cap (gG_\bullet)_j)=\dim(F_i\cap G_j)$.) Thus, if $(F_\bullet,G_\bullet)\in D_\sigma^\circ$, then $\dim(F_i\cap G_j)=\dim(E_{[i]}\cap E_{\sigma([j])})=|[i]\cap \sigma([j])|$. This shows one containment. To get the other containment note that the $D_\sigma^\circ$ partition $\Fl_n^2$, and the sets $\{(x\cdot E_\bullet^\text{st},y\cdot E_\bullet^\text{st}): x^{-1}y\in \overline{B_+\sigma B_+}\}$ don't overlap with each other. So if $(F_\bullet, G_\bullet)\in \{(x\cdot E_\bullet^\text{st},y\cdot E_\bullet^\text{st}): x^{-1}y\in \overline{B_+\sigma B_+}\}$, then $(F_\bullet, G_\bullet)$ cannot be in $D_\tau^\circ$ for any $\tau\neq \sigma$. That achieves the other containment.
 \end{proof}
  The closure $D_{\sigma}$ of  $D_{\sigma}^\circ$ in $\mrm{Fl}_n\times \mrm{Fl}_n$ is $D_{\sigma}=\sqcup_{\sigma'\leq \sigma}D_{\sigma'}^\circ$ (see, e.g., \cite[Lemma 2.1]{Bilen}). Zariski closures of orbits of connected algebraic groups are irreducible, so $D_{\sigma}$ is a irreducible.

  Since $D_\sigma^{\circ}=G\cdot(E_\bullet^\text{st},\sigma E_\bullet^\text{st}),$ for $x,y\in G$ we have $(x\cdot E_\bullet^\text{st},y \cdot E_\bullet^\text{st})\in D_\sigma^\circ$ if and only if there is some $b\in \mrm{Stab}_G(E_\bullet^\text{st})=B_+$ such that $bx^{-1}y E_\bullet^\text{st}  = \sigma E_\bullet^\text{st}$. From this we conclude the well known alternate definition that $(xE_\bullet^\text{st},yE_\bullet^\text{st})\in D_\sigma^\circ$ if and only if $x^{-1}y \in B_+\sigma B_+.$ By the Bruhat decomposition, this means that $D_\sigma=\{(x\cdot E_\bullet^\text{st},y\cdot E_\bullet^\text{st}): x^{-1}y\in \overline{B_+\sigma B_+}\}$.

\noindent
  \begin{minipage}{0.49\textwidth}
    \begin{egg}

       Relative position is relatively simple to understand for simple reflections. 
       \[|[j] \cap s_i([j])|=\begin{cases} j & i\neq j\\ j-1 &i=j \end{cases}\]
        so $D_{s_i}^{\circ}$ is the set of pairs of flags $(F_\bullet, G_\bullet)$ so that $F_i\neq G_i$  and     $F_j=G_j$ when $j\neq i$. Taking the closure, $D_{s_i}=\{(F_\bullet, G_\bullet)\in (\Fl_n)^2 : F_j=G_j \text{ for } j\neq i\}.$ We visualize this situation in \cref{fig: relpos i}.
        \end{egg}

    \end{minipage}
    \hfill
    \begin{minipage}{0.47\textwidth}
    
\adjustbox{scale=0.75,center}{
\begin{tikzcd}[row sep = 1.1em]
               & \mathbb{C}^n                          &                \\
               & \vdots \arrow[u]                      &                \\
               & F_{i+1}=G_{i+1} \arrow[u]             &                \\
F_i \arrow[ru] &                                       & G_i \arrow[lu] \\
               & F_{i-1}=G_{i-1} \arrow[lu] \arrow[ru] &                \\
               & \vdots \arrow[u]                      &                \\
               & F_1=G_1 \arrow[u]                     &                \\
               & (0) \arrow[u]                         &               
\end{tikzcd}}

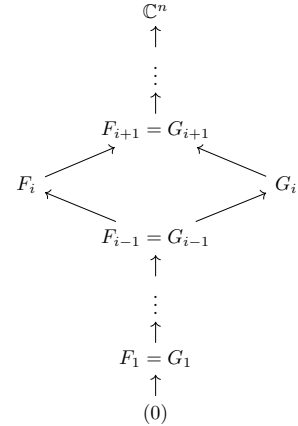
\captionof{figure}{Illustration of $(F_{\bullet},G_{\bullet})\in D_{s_i}$ with arrows representing inclusions of vector spaces}\label{fig: relpos i}

    \end{minipage}

\begin{lem}\label{lem: no change}
      If $s_i \not \leq \sigma,$ then for $(F_\bullet, G_\bullet) \in D_{\sigma},$ $F_i = G_i$.
\end{lem}
\begin{proof}
    Suppose $s_i\not \leq \sigma$ and let $(F_\bullet, G_\bullet) \in D_{\sigma}.$ For simple reflections $s_j\neq s_i,$ $s_j([i])=[i]$. So $\sigma([i])=[i].$ So $\dim(F_i\cap G_i) = |[i]\cap [i]|=i.$ Therefore, $F_i\cap G_i=F_i=G_i.$
    \end{proof}

\begin{minipage}{0.31\textwidth}
\center
    
\begin{tikzcd}[row sep = 1.1em, column sep = 0.82em]
               & \mathbb{C}^5            &                           \\
F_4 \arrow[ru] &                         & G_4 \arrow[lu]            \\
F_3 \arrow[u]  &                         & G_3 \arrow[u] \arrow[llu] \\
F_2 \arrow[u]  &                         & G_2 \arrow[u] \arrow[llu] \\
F_1 \arrow[u]  &                         & G_1 \arrow[u] \arrow[llu] \\
               & (0) \arrow[lu] \arrow[ru] &                          
\end{tikzcd}

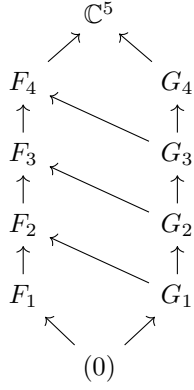
\captionof{figure}{Illustration of $D_{s_1s_2s_3s_4}$}

\vspace{10mm}
\label{fig:1234}
\end{minipage}
\hfill
\begin{minipage}{0.31\textwidth}
\centering
    
\begin{tikzcd}[row sep = 1.1em, column sep = 0.82em]
                          & \mathbb{C}^5            &                \\
F_4 \arrow[ru]            &                         & G_4 \arrow[lu] \\
F_3 \arrow[u] \arrow[rru] &                         & G_3 \arrow[u]  \\
F_2 \arrow[u] \arrow[rru] &                         & G_2 \arrow[u]  \\
F_1 \arrow[u] \arrow[rru] &                         & G_1 \arrow[u]  \\
                          & (0) \arrow[lu] \arrow[ru] &               
\end{tikzcd}
\captionof{figure}{$(F_{\bullet},G_{\bullet})$ in $D_{s_4s_3s_2s_1}$ 
\vspace{10mm}}
\label{fig:4321}
\end{minipage}
\hfill
\begin{minipage}{0.31\textwidth}
   
\begin{tikzcd}[row sep = 1.1em, column sep = 0.82em]
                          & \mathbb{C}^5            &                           \\
F_4 \arrow[ru]            &                         & G_4 \arrow[lu]            \\
F_3 \arrow[u]             &                         & G_3 \arrow[u] \arrow[llu] \\
F_2 \arrow[u] \arrow[rru] &                         & G_2 \arrow[u]             \\
F_1 \arrow[u]             &                         & G_1 \arrow[u] \arrow[llu] \\
                          & (0) \arrow[lu] \arrow[ru] &                          
\end{tikzcd}
\captionof{figure}{$(F_{\bullet},G_{\bullet})$ in $D_{s_1s_3s_2s_4}$}
\vspace{10mm}
\label{fig:1324}
\end{minipage}

The torus $T$ acts on $\Fl_n\times \Fl_n$ diagonally via $T\cong T_\Delta:=\{(t,t)\mid t\in T\} \hookrightarrow G_{\Delta} \hookrightarrow G\times G.$ 
    The fixed points of this action are just $(\Fl_n)^T \times (\Fl_n)^T.$ Since $D_{\sigma}$ is a $G$-orbit closure, it will be $T$-invariant. 
    Since $(w_1\cdot E_\bullet^{\text{st}}, w_2\cdot E_\bullet^{\text{st}}) = w_1\cdot (E_\bullet^{\text{st}}, w_1^{-1}w_2\cdot E_\bullet^{\text{st}})$, we find that 
    \begin{align*}
         D_{\sigma}^T &= \{(w_1\cdot E_\bullet^{\text{st}}, w_2\cdot E_\bullet^{\text{st}})\mid w_1,w_2\in S_n, w_1^{-1}w_2 \leq \sigma\}\\
         &=\{(w\cdot E_\bullet^{\text{st}},  w\cdot \sigma'\cdot E_\bullet^{\text{st}})\mid w\in S_n, \sigma'\leq \sigma\}.
    \end{align*}

\subsubsection{Schubert and Richardson varieties}

For $w\in S_n$ we define the \textbf{Schubert variety} $X_w=\overline{B_{-}wB_+/B_+}=\{F_\bullet:(E_\bullet^{\text{op}},F_\bullet)\in D_{w_0 \cdot w}\}$ and \textbf{opposite Schubert variety} $X^w=\overline{B_{+}wB_+/B_+}=\{F_{\bullet}:(E_{\bullet}^{\text{st}},F_{\bullet})\in D_{w}\}$\footnote{Our characterization of the Schubert variety using relative positions of flags follows from our characterization of the opposite Schubert variety in terms of relative positions. To see this, observe \begin{align*}
X_w=w_0\cdot X^{w_0\cdot w}=\{w_0\cdot F_\bullet:(E_\bullet^{\text{st}},F_\bullet)\in D_{w_0\cdot w}\}&=\{w_0\cdot F_\bullet:w_0\cdot (E_\bullet^{\text{st}},F_\bullet)\in  D_{w_0\cdot w}\}\\&=\{w_0\cdot F_\bullet:(E_\bullet^{\text{op}},w_0\cdot F_\bullet)\in D_{w_0\cdot w}\}\\&=\{G_\bullet: (E_\bullet^{\text{op}},G_\bullet)\in D_{w_0\cdot w}\}
.
\end{align*}},
which have the properties that $X_w\subseteq X_v$ iff $w\geq v$, and $X^w \subseteq X^v$ iff $w\leq v.$ Intersections of these are known as \textbf{Richardson varieties} and denoted $X_u^v:=X_u\cap X^v.$ We note that $X_u^v\neq \emptyset$ if and only if $u\leq v$.

\begin{rmk}
    In \cref{fig: relpos i,fig:1234,fig:4321,fig:1324} we visualize relative position using inclusions. The reader may wonder if this is always possible. It is not! The question of whether $D_{\sigma}$ can be described by inclusions is equivalent to asking if $X_{\sigma}$ can be described by inclusions. Such Schubert varieties are aptly called Schubert varieties \textbf{defined by inclusions} and were classified by Gasharov and Reiner in \cite{inclusions}. They showed that $X_{\sigma}$ is defined by inclusions if and only if $\sigma$ ``avoids'' the permutations 4231, 35142, 42513 and 351624.
\end{rmk}

\begin{egg}
    The Schubert variety $\sigma=4231$ is the set of all flags $\{0\}\subset F_1\subset F_2\subset F_3\subset \C^4$ such that $F_2\cap \langle e_1, e_2\rangle \neq \{0\}$.
\end{egg}

 Since $T=B_{-}\cap B_{+}$, Schubert varieties, opposite Schubert varieties, and hence Richardson varieties, are all $T$-invariant. The set of $T$-fixed points of $X_u^v$ is $\{\sigma\cdot E_\bullet^{\text{st}}\mid \sigma\in S_n, u\leq \sigma\leq v \}$.

        Observe that $\sigma(\omega_1+\dots +\omega_n)=\sigma(n,n-1,\dotsc,2,1)=(n+1-\sigma^{-1}(1),n+1-\sigma^{-1}(2),\dotsc,n+1-\sigma^{-1}(n))$. In particular,
        \[\mu_{\Fl_n}(\sigma \cdot  E_\bullet^{\text{st}} ) = (n+1-\sigma^{-1}(1),n+1-\sigma^{-1}(2),\dotsc,n+1-\sigma^{-1}(n)).\]
        Since $(X_u^v)^T = \{w\cdot E_\bullet^{\text{st}}: u\leq w \leq v\},$ the moment polytope of $X_u^v$ is exactly the twisted Bruhat interval polytope $\tilde{P}_{u,v}.$
    
        The dimension of a $T$-invariant subvariety matches the dimension of its moment polytope precisely when the subvariety is toric, so $\tilde{P}_{u,v}$ is toric exactly when $X_u^v$ is. The latter observation also appears in \cite{LMP21}.

Since $X_u^v$ is a $T$-invariant subvariety of a product of Grassmannians, this gives us a first corollary (albeit not a novel one) of \cref{lem:MinkowskiMachine}.

\begin{cor}[Essentially  Corollary 6.11 in \cite{Kostant-Toda}]\label{lem:Richardson sum}
    For permutations $u\leq v$ in $S_n$, \[\tilde{P}_{u,v} = \mu_{\Fl_n}(X_u^v)=\sum_{i=1}^n \mu_{\Gr(i,n)}(\phi_i(X_u^v))\]
\end{cor}

\begin{rmk}
    The polytopes $\mu_{\Gr(i,n)}(\phi_i(X_u^v))$ are the moment polytopes of projected Richardson varieties, and thus are ``positroid polytopes''. 
\end{rmk}

\subsection{Bott-Samelson varieties and brick varieties}\label{subsection:brick-varieti}

Let $Q=(\q_1,\dots, \q_{\mm})$ be a word in $S_n$. We define the \textbf{generalized Bott-Samelson} variety (see \cite{Perrin, silv}) to be
\[\BS^Q:=\{ (F_\bullet^{0},F_\bullet^1,\dotsc,F_\bullet^{\mm})\mid F_\bullet^0=E_\bullet^{\text{st}},\, (F_\bullet^{j-1},F_\bullet^{j})\in D_{\q_{j}}\text{ for }  j=1,\dotsc,\mm\}.
\]

Generalized Bott-Samelson varieties are irreducible, normal, and Cohen-Macaulay \cite{silv}. These properties hold because generalized Bott-Samelson varieties are \textit{locally trivial iterated Schubert bundles}. 

The fibre of the map $\BS^Q\to \BS^{Q'}$ with $Q'=(\q_1,\dots, \q_{\mm-1})$ over a point $(F_\bullet^0,\dotsc,F_\bullet^{m-1})\in \BS^{Q'}$ is equal to $(F_\bullet^0,\dotsc,F_\bullet^{m-1})$ times the set of flags in relative position $\q_\mm$ to $F_\bullet^{m-1}$. That set is isomorphic to a Schubert variety so $\BS^Q$ is an iterated Schubert bundle. It is in fact a locally trivial iterated Schubert bundle. Since Schubert varieties are irreducible, normal, and Cohen-Macaulay, the irreducibility, normality, and Cohen-Macaulay property pass inductively through the iterated bundle construction.

The iterated bundle construction also allows us to observe that
\begin{align*}
\dim(\BS^Q)&=\dim(X_{\q_1})+\cdots+\dim(X_{\q_{\mm}})\\
    &=\ell(\q_1)+\cdots + \ell(\q_{\mm}).
\end{align*}
When $Q$ is a word in simple reflections, $\BS^Q$ is a traditional Bott-Samelson variety. Traditional Bott-Samelson varieties are iterated locally trivial $\mbb{P}^1$ bundles and are therefore smooth. We will explain momentarily  a way to see the irreducibility of generalized Bott-Samelson varieties as a consequence of the irreducibility of traditional Bott-Samelson varieties.

We now define the \textbf{brick variety} as
\begin{align*}
    \Brick^Q &=\BS^Q \cap ((\Fl_n)^{\mm-1}\times \{ E_\bullet^{\text{op}}\}).
\end{align*}
Equivalently, the brick variety is the fibre over $E_\bullet^{\text{op}}$ of the map $\BS^Q\to \Fl_n$ sending $(F^0_{\bullet},\dots, F^\mm_\bullet)$ to $F^\mm_\bullet$.
Brick varieties were defined by Escobar in \cite{E16} in the case where $Q$ is a word in simple transpositions. In that setting $\Brick^Q$ is a smooth irreducible variety with dimension $\mm-\ell(\dem(Q))$
\cite[Theorem 20]{E16}.

We introduce a quick technical lemma that helps us to understand Bott-Samelson and brick varieties.

 \begin{lem}\label{lem:technical} 
    Let $\sigma_{i}$ and $\sigma_{i+1}$ be permutations in $S_n$, and choose some $\sigma\leq \mrm{Dem}(\sigma_i,\sigma_{i+1})$.
    \begin{enumerate}
        \item\label{itm:1} There are $\sigma',\sigma''$ in $S_n$ satisfying $\sigma'\leq \sigma_i$, $\sigma''\leq \sigma_{i+1}$, $\sigma=\sigma'\sigma''$, and $\ell(\sigma)=\ell(\sigma')+\ell(\sigma'')$.
        \item\label{itm:2} If $[1,\sigma_i]\cap [1,\sigma_{i+1}]=\{1\}$, then the $\sigma'$, $\sigma''$ satisfying the conditions in Part (\ref{itm:1}) are unique.
    \end{enumerate}
    \end{lem}
    \begin{proof}
        First we will prove (1). Consider a word $D:=(s_{i_1},\dotsc,s_{i_m})$ for $\mrm{Dem}(\sigma_i,\sigma_{i+1})$ in simple transpositions in $S_n$. By definition of the Demazure product, there is some integer $j\in[m]$ such that we can write $D=C_{i}+C_{i+1}$, where $C_i:=(s_{i_1},\dotsc,s_{i_{j}})$ is a subword of a word for $\sigma_i$ and $C_{i+1}:=(s_{i_{j+1}},\dotsc,s_{i_m})$ is a subword of a word for $\sigma_{i+1}$. Here, we have used `$+$' to mean concatenation of words. Since $\sigma\leq \mrm{Dem}(\sigma_i,\sigma_{i+1})$, there is a reduced word $C$ for $\sigma$ which is also a subword of $D$. In particular, $C$ is a reduced subword of $C_{i}+C_{i+1}$. So $C$ can be expressed $C'+C''$, where $C'$ is a reduced subword of $C_{i}$, and $C''$ is a reduced subword of $C_{i+1}$.
        Taking $\sigma'$ and $\sigma''$ to be the permutations corresponding to the reduced words $C'$ and $C''$ respectively completes the proof of Part (1).

        Next, we will prove Part (2). The condition $[1,\sigma_i]\cap [1,\sigma_{i+1}]=\{1\}$ says that we can write reduced expressions $\sigma_i=s^{(i)}_{j_1}s^{(i)}_{j_2}\dotsb s^{(i)}_{j_k}$ and $\sigma_{i+1}=s^{(i+1)}_{\ell_1}s^{(i+1)}_{\ell_2}\dotsb s^{(i+1)}_{\ell m}$, where the words $C_{i}:=(s^{(i)}_{j_1},s^{(i)}_{j_2},\dotsb, s^{(i)}_{j_k})$ and $C_{i+1}:=(s^{(i+1)}_{\ell_1},s^{(i+1)}_{\ell_2},\dotsb, s^{(i+1)}_{\ell m})$ are disjoint. That $C_i\cap C_{i+1}=\emptyset$ implies that the concatenation $C_i+C_{i+1}$ is reduced, and hence $\mrm{Dem}(C_i+C_{i+1})=\sigma_i\sigma_{i+1}$. As $\sigma\leq \sigma_i\sigma_{i+1}$, we can choose a reduced word $C$ for $\sigma$ such that $C=C'+C''$, where $C'$ is a reduced subword of $C_i$ and $C''$ is a reduced subword of $C_{i+1}$. We can take $\sigma'$ and $\sigma''$ to be the permutations corresponding to the words $C'$ and $C''$, respectively. To see that $\sigma'$ and $\sigma''$ are unique, it is enough to show that $C'$ and $C''$ are unique. This follows from the well-known property \cite[Proposition 1.2]{LOYY24}.
    \end{proof}

\begin{prop}\label{lem: surject}
Let $Q=(\q_1,\dots, \q_\mm)$ and let $Q'=(\q_1,\dots, \q_{i-1},\q_{i}\star\q_{i+1},\q_{i+2},\dots, \q_\mm)$ for some $i<\mm.$
Define $\varrho_i:(\Fl_n)^{\mm+1}\to (\Fl_n)^{\mm}$ by 
\[(F^0_\bullet,\dotsc,F^\mm_\bullet)\mapsto (F^0_\bullet,\dotsc,F^{i-1}_\bullet,F^{i+1}_\bullet,\dotsc,F^\mm_\bullet).\]
Then $\varrho_i$ is a $T$-equivariant regular map of projective varieties. The restriction of $\varrho_i$ to $\BS^Q$ is a surjection onto $\BS^{Q'}$ and the restriction to $\Brick^Q$ is a surjection onto $\Brick^{Q'}$.
\end{prop}
\begin{proof}
   Let $\mfk{F}=(F^0_\bullet,\dotsc,F^\mm_\bullet)\in \BS^Q.$ Without loss of generality suppose that $F^i_\bullet=g_i\cdot E_\bullet^\text{st}$  with $g_i\in G.$ Then $g_{i-1}^{-1}g_{i}\in \overline{B_+ \q_i B_+}$  and $g_{i}^{-1}g_{i+1}\in \overline{B_+ \q_{i+1} B_+}$. Thus $g_{i-1}^{-1}g_{i+1}\in \overline{B_+ \q_i B_+}\overline{B_+ \q_{i+1} B_+} = \overline{B_+ \q_i\q_{i+1} B_+}$ by \cite[Section 1]{smallResolutions}. Hence $\varrho_i(\mfk{F})\in \BS^{Q'} $. Since $i< \mm, $ the last flag of $\varrho_i$ remains unchanged, so if $\mfk{F}\in \Brick^Q,$ $\varrho(\mfk{F})\in \Brick^{Q'}.$

    Now let  $\mfk{F}'=(F^0,\dotsc,F^{i-1}_\bullet,F^{i+1},F_\mm)\in \BS^{Q'}.$ 
 Then $F^{i-1}$ and $F^{i+1}$ are in relative position $\sigma$ for a \textit{unique} $\sigma\leq \mrm{Dem}(\q_i,\q_{i+1})$. By (\ref{itm:1}) of \cref{lem:technical}, there  are $\q_i',\q_{i+1}'$ in $S_n$ satisfying $\q_i'\leq \q_i$, $\q_{i+1}''\leq \q_{i+1}$, $\sigma=\q_i'\q_{i+1}'$, and $\ell(\sigma)=\ell(\q_{i}')+\ell(\q_{i+1}')$. We wish to find $F_i$ such that $F_{i-1}$ is in relative position $\sigma'$ with $F_{i}$ and such that $F_{i}$ is in relative position $\sigma''$ with $F_{i+1}$. Such an $F_i$ exists, and it is unique, by the argument in \cite[Lemma 3.2]{CGGLSS25}, which proves surjectivity. The argument applies word for word if we replace $\BS^Q$ with $\Brick^Q$.
\end{proof}

When $Q$ is a word in simple reflections, both $\BS^Q$ and $\Brick^Q$ are projective varieties \cite{E16}. By replacing a word $Q'=(\q_1,\dots, \q_m)$ with the concatenation of reduced words in simple transpositions for the $\q_i$ and repeatedly applying \cref{lem: surject}, we can conclude that both $\BS^{Q'}$ and $\Brick^{Q}$ are the images of projective varieties under regular maps. This proves that generalized Bott-Samelson varieties and generalized brick varieties are in fact  projective varieties.

\subsubsection{Moment polytopes}

Since the projection of $\Brick^Q$ onto its first $\Fl_n$ factor (indexed by 0) is constant, we can view $\Brick^Q$ as embedded in $(\Fl_n)^\mm$ (rather than $(\Fl_n)^{\mm+1}$). The moment polytope with respect to this embedding is the bulky brick polytope $\BB(Q)$, as we now show.

\begin{prop}\label{prop: bulky moment}
     Let $Q=(\q_1,\dotsc,\q_{\mm})$ be a word in $S_n$ with $\dem(Q)=w_0.$ Define
          \[p_J:=\left(E_\bullet^{\text{st}},J_{(1)}\cdot E_\bullet^{\text{st}},\dotsc,J_{(\mm-1)}\cdot E_\bullet^{\text{st}},E_\bullet^{\text{op}}\right)\] for $J\in\nabla(Q,w_0)$.
      \begin{itemize}
          \item The map $J\mapsto p_J$ is a bijection $\nabla(Q,w_0)\to (\Brick^Q)^T.$
          \item Defining $\pi_i:\Fl_n^{\mm+1}\to \Fl_n$ as the projection of $(\Fl_n)^{\mm+1}$ onto its $i$th $\Fl_n$ factor, the $T$-weight of $\pi_i(p_J)$ is the bulky weight function $\mbf{w}(J,i).$ That is, \[\mu_{\Fl_n}(\pi_i(p_J))=\mbf{w}(J,i)\] for $J\in \nabla(Q,w_0).$ 
          \item The fixed point $p_J$ has $T$-weight \[\mu_{(\Fl_n)^{\mm}}(p_J) = \mbf{B}(J)\] for $J\in \nabla(Q,w_0).$
          \item The moment polytope $\mu_{\Fl_n}(\Brick^Q)$ of $\Brick^Q$ with respect to the projective embedding resulting from the inclusion $\Brick^Q \subseteq (\Fl_n)^{\mm}$ (ignoring the 0th factor) is the bulky brick polytope $\BB(Q)$. That is,
          \[\mu_{(\Fl_n)^{\mm}}(\Brick^Q) = \BB(Q).\]
      \end{itemize}
\end{prop}
\begin{proof}
    The $T$-fixed points of $\mrm{Brick}^Q$  (with $T$ acting diagonally on $(\Fl_n)^{\mm+1} \supseteq \mathrm{Brick}^Q$) are the sequences $F_\bullet=(F_\bullet^{0}=E_\bullet^{\text{st}},F_\bullet^{1},\dotsc,F_{\bullet}^{\mm-1},F_\bullet^{\mm}=E_{\bullet}^{\text{op}})$, where $F_\bullet^j$ has the form $w_j \cdot E_\bullet^{\text{st}}$ for some $w_j\in S_n$ and $(F_\bullet^{j-1},F_\bullet^{j})\in D_{\q_{j}}$ for each $j=1,\dotsc,\mm$. Letting $u_j = w_{j-1}^{-1}w_{j}$, we have $u_j\leq \q_j$ and $w_j = u_1\cdots u_{j}$. In particular,  every $T$-fixed point of $\mrm{Brick}^Q$ is uniquely determined by a sequence $(u_1,u_2,\dotsc,u_{\mm})$, such that $u_j\leq \q_j$; and each such sequence determines a $T$-fixed point of $\Brick^Q$. 
   
   By definition, any $J\in\nabla(Q,w_0)$ is of the form $J=(u_1,\dotsc,u_m)$, where $u_j\leq \q_j$. Hence the map $J\mapsto p_J$ is a bijection $\nabla(Q,w_0)\to(\Brick^Q)^T$ (see also \cite[pp. 7]{E16}).
    We can now compute
    \begin{align*}
        \mu_{\Fl_n}(\pi_i(p_J)) &= \mu_{\Fl_n}(J_{(i)}\cdot E_\bullet^{\text{st}})\\
        &=J_{(i)}(\omega_1+\dots +\omega_n)\\
        &=\mbf{w}(J,i).
    \end{align*}
Hence,
\begin{align*}
      \mu_{\Fl_n^{\mm}}(p_J)&=\sum_{i=1}^{\mm}  \mu_{\Fl_n}(\pi_i(p_J))  \\
      &=\sum_{i=1}^{\mm} \mbf{w}(J,i)\\
      &=\mbf{B}(J).
\end{align*}
And so,
\begin{align*}
        \mu_{(\Fl_n)^{\mm}}(\Brick^Q)&=\conv_{x\in (\Brick^Q)^T} \mu_{\Fl_n}(x)\\
        &=\conv_{J\in \nabla(Q,w_0)} \mu_{\Fl_n}(p_J)\\
        &=\conv_{J\in \nabla(Q,w_0)} \mbf{B}(J)\\
        &=\BB(Q).\qedhere
\end{align*}

\end{proof}

The (unbulky) brick polytope is the moment polytope of the brick variety with respect to a different projective embedding. The projective embedding coming from the inclusion $\Brick^Q\subseteq (\Fl_n)^{\mm}$ turns out to be very redundant (bulky, if you will). Recall from \cref{lem: no change} that for $(F_\bullet, G_\bullet) \in D_{\sigma},$ if $s_j \not \leq \sigma,$ $F_j = G_j$. On the other hand, for $(F_\bullet, G_\bullet) \in D_{\sigma},$ the subspaces which might change between $F_{\bullet}$ and $G_\bullet$ are in dimensions $\{j: s_j\leq \sigma\}.$ So for a less bulky projective embedding of $\Brick^Q$, we should project to $\prod_{i=1}^{\mm}\prod_{j\in[n-1]:s_j\leq \q_{i}}\Gr(j,n)$.

When a word $Q=(\q_1,\dots, \q_{\mm})$ in $S_n$ has been specified,
we will define
\[\rho_i:\Brick^Q \to \prod_{j\in[n-1]:s_j\leq \q_{i}}\Gr(j,n)  \]
and
    \[\rho\colon \Brick^Q\hookrightarrow \prod_{i=1}^{\mm}\prod_{j\in[n-1]:s_j\leq \q_{i}}\Gr(j,n),\quad  \quad F_\bullet=(F_\bullet^0,\dotsc,F_\bullet^{|Q|})\mapsto (\rho_1(F_\bullet),\dotsc,\rho_{|Q|}(F_\bullet))\] 
in the natural manner. We will also define
\[\nu_i:\Fl_n \to \prod_{j\in[n-1]:s_j\leq \q_{i}}\Gr(j,n) \]
so that $\rho_i=\nu_i \circ \pi_i$ for $\pi_i$ as in \cref{prop: bulky moment}.

\begin{nota}
    When $Q$ is clear from context, we will write $\mu_{\prod\prod\Gr}$ instead of $\mu_{\prod_{i=1}^{\mm}\prod_{j\in[n-1]:s_j\leq \q_{i}}\Gr(j,n)} \circ \rho$ and $\mu^i_{\prod\Gr}$ instead of $\mu_{\prod_{j\in[n-1]:s_j\leq \q_{i}}\Gr(j,n)}$. 
\end{nota}

\begin{egg}[\cite{E16}, Example 12]
        Consider the word $Q=(s_1,s_2,s_1,s_2,s_1)$ in $S_3$. The brick variety $\Brick^Q$ can be visualized with a \textbf{Magyar diagram}:
        \[\adjustbox{center}{
        \begin{tikzcd}[row sep = 1.1em, column sep = 1.7em]
	   \mbb{C}^3 \\
	\langle e_1,e_2\rangle & V_2 & \langle e_2,e_3\rangle \\
	\langle e_1\rangle & V_1 & V_3 & \langle e_3\rangle \\
	(0)
	\arrow[no head, from=1-1, to=2-2]
	\arrow[no head, from=1-1, to=2-3]
	\arrow[no head, from=2-1, to=1-1]
	\arrow[no head, from=2-1, to=3-2]
	\arrow[no head, from=2-2, to=3-3]
	\arrow[no head, from=2-3, to=3-4]
	\arrow[no head, from=3-1, to=2-1]
	\arrow[no head, from=3-2, to=2-2]
	\arrow[no head, from=3-3, to=2-3]
	\arrow[no head, from=4-1, to=3-1]
	\arrow[no head, from=4-1, to=3-2]
	\arrow[no head, from=4-1, to=3-3]
	\arrow[no head, from=4-1, to=3-4]
\end{tikzcd}
}\]
In particular, The points $(E_\bullet^{\text{st}},F_\bullet^{1},\dotsc,F_\bullet^{4},E_\bullet^{\text{op}})$ of $\mrm{Brick}^Q$ between $E_\bullet^{\text{st}}$ and $E_\bullet^{\text{op}}$ are of the form $F_\bullet^{1}=(V_1\subseteq \langle e_1,e_2\rangle\subseteq \mbb{C}^3)$, $F_\bullet^{2}=(V_1\subseteq V_2\subseteq\mbb{C}^3)$, $F_\bullet^{3}=(V_3\subseteq V_2\subseteq \mbb{C}^3)$, and $F_\bullet^{4}=(V_3\subseteq \langle e_2,e_3\rangle\subseteq\mbb{C}^3)$.
The point $p_J$ corresponding to $J=(s_1,s_2,-,-,s_1)$ has $V_1= \langle e_2\rangle$, $V_2=\langle e_2,e_3\rangle$, and $V_3=\langle e_2\rangle$. 
\end{egg}

\begin{prop}[\cite{E16}, Theorem 21]\label{prop: unbulky moment}
      Let $Q=(\q_1,\dotsc,\q_{\mm})$ be a word in $S_n$ with $\dem(Q)=w_0.$ When we embed $\Brick^Q$ into projective space by applying $\rho$ followed by Pl\"ucker and Segre embeddings,
      \begin{itemize}
          \item The $T$-weight coming from the projection by $\rho_k$ of the fixed point $p_J$ is the $k$th weight function $w(J,k)$. That is,
          \[ \mu^k_{\prod \Gr}(\rho_k(p_J)) = w(J,k).\]
          \item The $T$-weight of $p_J$ is the brick vector $B(J)$, i.e.,
          \[\mu_{\prod\prod \Gr}(p_J) = B(J).\]
          \item And the moment polytope of $\Brick^Q$ is the brick polytope $\B(Q)$, i.e.,
          \[\mu_{\prod\prod \Gr}(\Brick^Q) = \B(Q).\]
      \end{itemize}
\end{prop}
\begin{proof}
   The $T$-weight of $\rho_k(p_J)$ is
   \begin{align*}
       \mu^k_{\prod \Gr}(\rho_k(p_J)) &= \sum_{j\in[n-1]:s_j\leq \q_k} \mu_{\Gr(j,n)}( \phi_{j}(\pi_k(p_J))) \\
       &= \sum_{j\in[n-1]:s_j\leq \q_k} \mu_{\Gr(j,n)}( \phi_{j}(J_{(k)}\cdot E_\bullet^{\text{st}}) ) \\
       &= \sum_{j\in[n-1]:s_j\leq \q_k} J_{(k)}(\omega_j)\\
       &= J_{(k)} \left( \sum_{j\in[n-1]:s_j\leq \q_k} \omega_j \right)\\
       &=w(J,k).
   \end{align*}
    So when we projectively embed $\Brick^Q$ by first applying $\rho$, the $T$-weight of $p_J$ is
    \begin{align*}
        \mu_{\prod\prod \Gr}(p_J) &= \sum_{k=1}^{\mm}  \mu^k_{\prod \Gr}(\rho_k(p_J))\\
        &=\sum_{k=1}^{\mm} w(J,k)\\
        &=B(J).
    \end{align*}
    Finally,
    \begin{align*}
        \mu_{\prod\prod \Gr}(\Brick^Q) &= \conv_{x\in (\Brick^Q)^T} \mu_{\prod\prod \Gr}(x)\\
        &=\conv_{J\in \nabla(Q,w_0)} \mu_{\prod\prod \Gr}(p_J)\\
        &=\conv_{J\in \nabla(Q,w_0)} B(J)\\
        &=\B(Q).\qedhere
    \end{align*}
\end{proof}

\begin{rmk}
    We have now shown that the brick polytope and bulky brick polytope are moment polytopes of the brick variety with respect to two different $T$-equivariant projective embeddings. As a consequence of this, $\B(Q)$ and $\BB(Q)$ have the same normal fan.
\end{rmk}

\subsection{Minkowski decompositions of brick polytopes}\label{subsection:Minkowski}

    In \cref{thm:Minkowski} and \cref{Brick Gr sum}, we apply \cref{lem:MinkowskiMachine} to produce Minkowski decompositions for $\BB(Q)$ and $\B(Q)$.

The result below says that, for certain words in $S_n$, the restriction of the $i$th projection map $(\Fl_n)^m\to \Fl_n$ to the brick variety has as its image a Richardson variety. Furthermore, in certain cases, we can deduce that the brick variety is a toric variety if the Richardson variety is a toric variety, and vice versa.

\begin{prop}\label{proj is Richardson}
    Let $\pi_i:(\Fl_n)^{\mm}\to \Fl_n$ be the projection of $(\Fl_n)^{\mm}$ onto its $i$-th factor (with indexing of $i$ starting at 0). Let $R,S$ be words in $S_n$. Set $i:=|R|$, $u:=w_0\dem(S)^{-1}$, and $v:=\dem(R)$.  
    \begin{enumerate}
        \item We have $\pi_{i}(\Brick^{R+S})=X_u^v$, where $+$ denotes concatenation of words.
        \item If $\Brick^{R+S}$ is toric, then $X_u^v$ is toric.
        \item If $\mrm{dim}_{\mbb{C}}(\mrm{Brick}^{R+S})=\mrm{dim}_{\mbb{C}}(X_u^v)$ and $X_u^v$ is toric, then $\mrm{Brick}^{R+S}$ is toric.  
   \end{enumerate}      
\end{prop}
\begin{proof}
    (1) follows from slight modifications of the proof of \cite[Theorem 26]{E16}. For (2), assume $\mrm{Brick}^{R+S}$ is toric. Then it contains a dense, open $T$-orbit $\mcal{O}$. As $\pi_i$ is $T$-equivariant, it follows that $\pi_i(\mcal{O})$ is a $T$-orbit in $X_u^v$, and as $\pi_i$ is surjective onto $X_u^v$ and continuous, it follows that $\pi_i(\mcal{O})$ is dense in $X_u^v$. To show that $X_{u}^v$ is a toric variety, it is therefore enough to show that $\pi_i(\mcal{O})$ is open in $X_u^v$. As $\mcal{O}$ is open in $\mrm{Brick}^{R+S}$, it follows that $\pi_i(\mcal{O})$ is a constructible subvariety of $X_u^v$. By \cite[\S AG.11, Corollary 10.2]{B91}, the constructible set $\pi_i(\mcal{O})$ contains an open subvariety $\mcal{U}$ of $X_u^v$. We will show that $\pi_i(\mcal{O})$ is open in $X_u^v$ by showing that it equals $\mcal{U}$. As $\pi_i(\mcal{O})$ is a $T$-orbit in $X_u^v$, we can write $\pi_i(\mcal{O})=T\cdot x$ for some $x\in \pi_i(\mcal{O})$. Given any $y\in \pi_i(\mcal{O})$ and any $u\in \mcal{U}\subseteq \pi_i(\mcal{O})$, there exist $r,s\in T$ such that $y=r\cdot x$ and $u=s\cdot x$. Thus, $y=(rs^{-1})\cdot u\in (rs^{-1})\cdot \mcal{U}$. As the action of $t\in T$ on $X_u^v$ induces an automorphism of $X_u^v$, it follows that $t\cdot \mcal{U}$ is open in $X_u^v$ for all $t\in T$. In particular, $(rs^{-1})\cdot \mcal{U}$ is an open subset of $X_u^v$ contained in $\pi_i(\mcal{O})$ which contains $y$. As $y\in \pi_i(\mcal{O})$ was chosen arbitrarily, we have shown that $\cup_{t\in T} (t\cdot \mcal{U})=\pi_i(\mcal{O})$. As $\pi_i(\mcal{O})$ is a union of open subsets of $X_u^v$, it follows that $\pi_i(\mcal{O})$ is itself open in $X_u^v$. Therefore, $X_u^v$ is a toric variety, which proves (2).

    Now we will prove (3). For the rest of this proof, we will abuse notation and denote the restriction $\pi_i|_{\mrm{Brick^{R+S}}}$ by $\pi_i$ itself. By (1), the restriction $\pi_i\colon \mrm{Brick}^{R+S}\to X_u^v$ is surjective, so it is dominant.  The result \cite[Theorem 3.7]{P08} implies that any dominant morphism of varieties of the same dimension is generically finite. Therefore, there is a dense open subset $U'$ of $X_u^v$ such that $\pi_i^{-1}(x)$ is a finite set for any $x\in U'$.
    
    Let $U$ be a dense, open $T$-orbit in $X_u^v$. Then the preimage $V:=\pi_i^{-1}(U)$ is dense, open, and $T$-invariant in $\mrm{Brick}^{R+S}$ by $T$-equivariance and continuity of $\pi_i$. Let $x\in U\cap U'\neq \emptyset$. As $U$ is a $T$-orbit in $X_{u}^v$, we have $T\cdot x=U$, and as $x\in U'$, we have $(\pi_i)|_{V}^{-1}(x)=\{y_1,\dotsc,y_k\}$ for some $y_i\in V$. Moreover, we have
    \[V=((\pi_i)|_{V})^{-1}(U)=((\pi_i)|_{V})^{-1}(T\cdot x)=T\cdot (((\pi_i)|_{V})^{-1}(x))=(T\cdot y_1)\cup\dotsb\cup(T\cdot y_k).\]As $V$ is dense in $\mrm{Brick}^{R+S}$, it follows that $\mrm{Brick}^{R+S}=\overline{V}=(\overline{T\cdot y_1})\cup\dotsb\cup(\overline{T\cdot y_k})$. If $\overline{T\cdot y_i}=\mrm{Brick}^{R+S}$ for some index $i$, then $T\cdot y_i$ is a dense torus orbit in $\mrm{Brick}^{R+S}$; by \cref{lem:dense-torus-orbit-is-open}, a dense torus orbit in $\mrm{Brick}^{R+S}$ is open, so $\mrm{Brick}^{R+S}$ is toric. Otherwise 
     $\mrm{Brick}^{R+S}$ is the disjoint union of multiple proper closed subsets, contradicting irreducibility of $\mrm{Brick}^{R+S}$. 
    \end{proof}

\begin{rmk}
    In the notation of \cref{proj is Richardson}, if $R$ and $S$ are reduced words in simple transpositions in $S_n$, 
    then it is shown in \cite[Theorem 26]{E16} that $\pi_i|_{\mrm{Brick}^{R+S}}\colon \mrm{Brick}^{R+S}\to X_u^v$ is a surjective, birational map. Thus, if $X_u^v$ is toric and $R$ and $S$ are reduced words in simple transpositions in $S_n$, then it follows from part (3) of \cref{proj is Richardson} that $\mrm{Brick}^{R+S}$ is toric.
    
    The map $\pi_i|_{\mrm{Brick}^{R+S}}$ will be analyzed further in \cref{subsubsection:brick-resolution}. In particular, we will give a sufficient condition for $\pi_i|_{\mrm{Brick}^{R+S}}$ to be an isomorphism in \cref{thm:patterns}.  
\end{rmk}

The following corollary is a direct consequence of \cref{proj is Richardson}.

\begin{cor}\label{toric-iff}
    Let $\pi_i:(\Fl_n)^{\mm}\to \Fl_n$ be the projection of $(\Fl_n)^{\mm}$ onto its $i$-th factor (with indexing of $i$ starting at 0). Let $R,S$ be words in $S_n$. Set $i:=|R|$, $u:=w_0\dem(S)^{-1}$, and $v:=\dem(R)$. Assume that $\mrm{dim}_{\mbb{C}}(\mrm{Brick}^{R+S})=\mrm{dim}_{\mbb{C}}(X_u^v)$. Then $\mrm{Brick}^{R+S}$ is toric if and only if $X_u^v$ is toric. 
\end{cor}

Recall that the brick polytope $\B(Q)$ is defined as the convex hull over all $J\in \nabla(Q,w_0)$ of the brick vectors $B(J)$ and the brick vector $B(J)$ is defined as $B(J)=\sum_{i=1}^{\mm}w(J,k).$ When $Q$ is a word in reduced reflections, Proposition 3.29 of \cite{bricksort} and Proposition 5.17 of \cite{PS15} explain that the summation and convex hull in the definition of the brick polytope can be exchanged. That is, \[\conv_{J\in \nabla(Q,w_0)} \sum_{i=1}^{\mm}w(J,i) = \sum_{i=1}^{\mm} \conv_{J\in \nabla(Q,w_0)} w(J,i). \]
This result is proved combinatorially in \cite{bricksort} and \cite{PS15}. In the following corollary we show that the same observation holds when $Q$ is a sequence of permutations as a natural consequence of the geometry of brick varieties and \cref{lem:MinkowskiMachine}. Additionally, when we write $w(J,i)$ as a sum, this sum can also be exchanged with the convex hull operation.

\begin{cor}[Proposition 3.29 of \cite{bricksort}, Proposition 5.17 of \cite{PS15}]\label{Brick Gr sum}
   Let $Q=(\q_1,\dots, \q_\mm)$ be a word in $S_n$ such that $\dem(Q)=w_0$.
  \begin{align}
      \B(Q)&=\conv_{J\in \nabla(Q,w_0)}\sum_{i=1}^{\mm}w(J,i) \nonumber\\ 
      &= \sum_{i=1}^{\mm}\conv_{J\in \nabla(Q,w_0)} w(J,i) \nonumber\\
      &=\sum_{i=1}^{\mm} \mu^i_{\prod \Gr}(\rho_i(\Brick^Q)) \nonumber\\
      &=\sum_{i=1}^{\mm} \sum_{j\in[n-1]:s_j\leq \q_{i}} \conv_{J\in \nabla(Q,w_0)} J_{(i)}\omega_j \label{eq:positroidSum1}
  \end{align}
  where $\rho_k$ is as in Section  \ref{background}.
\end{cor}
\begin{proof}
The equality in the first line of this corollary is by definition.

The brick polytope is defined as \[\B(Q)=\mrm{conv}_{J\in\nabla(Q,w_0)} B(J) = \mrm{conv}_{J\in\nabla(Q,w_0)}\sum_{k=1}^{\mm}w(J,k).\]
\cref{prop: unbulky moment} tells us that $\B(Q)=\mu_{\prod\prod\Gr}\Brick^Q$. Viewing the embedding $\rho\colon \Brick^Q\hookrightarrow \prod_{i=1}^{\mm}\prod_{j\in[n-1]:s_j\leq \q_{i}}\Gr(j,n)$ as the product of its constituent maps $\rho_i:\Brick^Q \to \prod_{j\in[n-1]:s_j\leq \q_{i}}\Gr(j,n)$ allows us to apply \cref{lem:MinkowskiMachine} to get 
\begin{align*}
    \B(Q)&=\mu_{\prod\prod\Gr}\Brick^Q\\
    &=\sum_{i=1}^{\mm} \conv_{x\in (\Brick^Q)^T}
    \mu^i_{\prod \Gr} (\rho_i(x))\\
    &=\sum_{i=1}^{\mm} \mu^i_{\prod \Gr}(\rho_i(\Brick^Q)).
\end{align*}
We can then apply \cref{prop: unbulky moment} to find \[\B(Q)= \sum_{i=1}^{\mm} \conv_{x\in (\Brick^Q)^T}
    \mu^i_{\prod \Gr} (\rho_i(x)) =  \sum_{i=1}^{\mm} \conv_{J\in \nabla(Q,w_0)} w(J,i) \]
    as well.

Since $\rho_i(\Brick^Q)$ is a $T$-invariant subvariety of $\prod_{j\in[n-1]:s_j\leq \q_{i}}\Gr(j,n),$ we could apply \cref{lem:MinkowskiMachine} again to further expand $\mu^i_{\prod \Gr}(\rho_i(\Brick^Q))$. To avoid introducing notation for maps $\prod_{j\in[n-1]:s_j\leq \q_{i}}\Gr(j,n) \to \Gr(r,n),$ we note that the following diagram commutes for any $i,r$ such that $s_r \leq \q_i$.

\begin{center}
    \begin{tikzcd}
 \Brick^Q\subseteq (\Fl_n)^{\mm+1} \arrow[d,"\pi_i"]  \arrow[r,"\rho_i"] & \prod_{j\in[n-1]:s_j\leq \q_{i}}\Gr(j,n) \arrow[d] \\
 \Fl_n \arrow[r,"\phi_r"]        & \Gr(r,n)          
\end{tikzcd}
\end{center}

\begin{align*}
    \mu^i_{\prod \Gr}(\rho_i(\Brick^Q))&=\mu_{\prod_{j\in[n-1]:s_j\leq \q_{i}}\Gr(j,n)}\rho_i(\Brick^Q)\\
    &=\sum_{j\in[n-1]:s_j\leq \q_{i}}\mu_{\Gr(j,n)}(\phi_j(\pi_i(\Brick^Q)))\\
    &=\sum_{j\in[n-1]:s_j\leq \q_{i}} \conv_{x\in \phi_j(\pi_i(\Brick^Q))^T} \mu_{\Gr(j,n)}(x)\\
    &=\sum_{j\in[n-1]:s_j\leq \q_{i}} \conv_{x\in (\Brick^Q)^T} \mu_{\Gr(j,n)}(\phi_j(\pi_i(x)))\\
    &=\sum_{j\in[n-1]:s_j\leq \q_{i}} \conv_{J\in \nabla(Q,w_0)}\mu_{\Gr(j,n)}(\phi_j(\pi_i(p_J)))\\
    &=\sum_{j\in[n-1]:s_j\leq \q_{i}} \conv_{J\in \nabla(Q,w_0)}\mu_{\Gr(j,n)}(\phi_j(  J_{(i)}\cdot E_\bullet^{\text{st}}  ))\\
    &=\sum_{j\in[n-1]:s_j\leq \q_{i}} \conv_{J\in \nabla(Q,w_0)} J_{(i)}\cdot \omega_j.\qedhere
\end{align*}
\end{proof}

\begin{rmk}
    The polytopes $\conv_{J\in \nabla(Q,w_0)} J_{(i)}w_j$ appearing in \cref{eq:positroidSum1} are associated to projected Richardson varieties and are therefore positroid polytopes.
\end{rmk}

The situation of bulky brick polytopes closely parallels that of (unbulky) brick polytopes. The bulky brick polytopes are convex hulls of bulky brick vectors, and the bulky brick vectors are sums of bulky weight vectors. Again, the underlying geometry allows the use of \cref{lem:MinkowskiMachine} to commute the convex hull and sum operations. This yields \cref{eq:pieces}, which expresses bulky brick polytopes as Minkowski sums of twisted Bruhat interval polytopes. \cref{eq:positroidSum2} shows that the sum of Bruhat interval polytopes can be expanded and reorganized to reveal the exact same positroid polytope summands appearing in \cref{eq:positroidSum1}, but now with multiplicities.

\begin{thm}\label{thm:Minkowski}
Let $Q=(\q_1,\dots, \q_{\mm})$ be a word in $S_n$ with $\dem(Q)=w_0.$
Set $\lambda_{i,j}=\min(\{k:k>i, s_j\leq \q_k \} \cup \{\mm+1\})-i$ for $i>0$ and set $\lambda_{0,j}=\min(\{k-1: s_j\leq \q_k \} \cup \{\mm\})$. Then the bulky brick polytope has the following Minkowski decompositions into twisted Bruhat interval polytopes and into matroid polytopes:
    \begin{align}
 \BB(Q)
 &=\sum_{i=1}^{\mm} \tilde{P}_{w_0\dem(\q_{i+1},\dots, \q_{\mm})^{-1},\dem(\q_{1},\dots, \q_{i})}\label{eq:pieces}\\
&=\left(\sum_{j=1}^{n}\lambda_{0,j}\omega_j \right)+\sum_{i=1}^{\mm}\sum_{j\in [n-1]:s_j\leq \q_i} \lambda_{i,j} \conv_{J\in \nabla(Q,w_0)}J_{(i)}(\omega_j). \label{eq:positroidSum2}   
    \end{align}
\end{thm}
\begin{proof}
    Viewing $\Brick^Q$ as a subvariety of $(\Fl_n)^{\mm},$ we can apply \cref{prop: bulky moment}, \cref{lem:MinkowskiMachine}, and \cref{proj is Richardson} to get 
    \begin{align*}
        \BB(Q)&=\mu_{(\Fl_n)^{\mm}}(\mrm{Brick}^Q)\\
        &=\sum_{i=1}^{\mm} \mu_{\Fl_n}(\pi_i(\Brick^Q))\\
        &=\sum_{i=1}^{\mm} \mu_{\Fl_n}\left(X_{w_0\dem(\q_{i+1},\dots, \q_{\mm})^{-1}}^{\dem(\q_{1},\dots, \q_{i})}\right)\\
        &=\sum_{i=1}^{\mm} \tilde{P}_{w_0\dem(\q_{i+1},\dots, \q_{\mm})^{-1},\dem(\q_{1},\dots, \q_{i})},
    \end{align*}
     proving \cref{eq:pieces}. 

    We could further expand each $\tilde{P}_{w_0\dem(\q_{i+1},\dots, \q_{\mm})^{-1},\dem(\q_{1},\dots, \q_{i})}$ via \cref{lem:Richardson sum}.
    However, it helps to instead think about expanding $\mu_{\Fl_n}(\pi_i(\Brick^Q))$ (even though the expansion will be the same). This expansion (via \cref{lem:MinkowskiMachine} as usual) gives us
\begin{equation}
    \BB(Q)=\sum_{i=1}^{\mm} \sum_{j=1}^{n} \mu_{\Gr(j,n)}(\phi_j(\pi_i(\Brick^Q))) \label{eq:anothersum}
\end{equation}

This Minkowski decomposition of $\BB(Q)$ is about embedding $\Brick^Q$ in $\prod_{i=1}^{\mm} \prod_{j=1}^n \Gr(j, n)$ and getting a Minkowski summand for each $\Gr(j,n)$. However, we know that this embedding is very redundant since we were also able to embed $\Brick^Q$ in $\prod_{i=1}^{\mm}\prod_{j\in[n-1]:s_j\leq \q_{i}}\Gr(j,n)$ (and composing the embedding into $\prod_{i=1}^{\mm} \prod_{j=1}^n \Gr(j, n)$ with the projection $\prod_{i=1}^{\mm} \prod_{j=1}^n \Gr(j, n)\to \prod_{i=1}^{\mm}\prod_{j\in[n-1]:s_j\leq \q_{i}}\Gr(j,n)$ gives the embedding of $\Brick^Q$ into $\prod_{i=1}^{\mm}\prod_{j\in[n-1]:s_j\leq \q_{i}}\Gr(j,n)$). We can thus rewrite \cref{eq:anothersum} by only summing terms $\mu_{\Gr(j,n)}(\phi_j(\pi_i(\Brick^Q)))$ where $i=0$ or $s_j\leq \q_i$ and attaching multiplicities to these terms recording the number of times they are repeated in \cref{eq:anothersum}. The $i=0$ terms don't appear to be present in \cref{eq:anothersum}, but we need to sum them anyway to account for copies of subspaces that appear first in $\pi_0(\Brick^Q)$.

By \cref{lem: no change}, when $s_j \not \leq \q_{i+1},$ $\phi_j\circ \pi_i = \phi_j \circ \pi_{i+1}.$ Moreover,  $\phi_j\circ \pi_i = \phi_j \circ \pi_{i'}$  for $i\leq i' \leq \min(\{k:k>i, s_j\leq \q_k \} \cup \{\mm\})$. So when $s_j\leq \q_{i}$, the number of terms of \cref{eq:anothersum} that are equal to $\phi_j \circ \pi_{i}$ is $\min(\{k:k>i, s_j\leq \q_k \} \cup \{\mm+1\})-i = \lambda_{i,j}.$ For $i=0,$ the number of terms equal to $\phi_j \circ \pi_{i}$ is $\min(\{k-1: s_j\leq \q_k \} \cup \{\mm\}) = \lambda_{0,j}$.

Since the simple transposition $s_n$ does not exist inside $S_n$, we have $\lambda_{0,n}=\mm$ and this is indeed the number of times we should sum $\phi_n(\pi_0(\Brick^Q)).$ We further note that $\phi_j(\pi_0(\Brick^Q)) = \{\omega_j\}$ for all $j$.
So \cref{eq:anothersum} may be manipulated to become
\begin{align*}
    \BB(Q)&=
    \left(\sum_{j=1}^{n}\lambda_{0,j}\omega_j \right)+\sum_{i=1}^{\mm}\sum_{j\in [n-1]:s_j\leq \q_i} \lambda_{i,j} \mu_{\Gr(j,n)}(\phi_j(\pi_i(\Brick^Q)))\\
    &=\left(\sum_{j=1}^{n}\lambda_{0,j}\omega_j \right)+\sum_{i=1}^{\mm}\sum_{j\in [n-1]:s_j\leq \q_i} \lambda_{i,j} \conv_{J\in \nabla(Q,w_0)}\mu_{\Gr(j,n)}(\phi_j(\pi_i(p_J)))\\
    &=\left(\sum_{j=1}^{n}\lambda_{0,j}\omega_j \right)+\sum_{i=1}^{\mm}\sum_{j\in [n-1]:s_j\leq \q_i} \lambda_{i,j} \conv_{J\in \nabla(Q,w_0)}\mu_{\Gr(j,n)}(\phi_j(J_{(i)}\cdot F_{\bullet}^{\text{st}}    ))\\
    &=\left(\sum_{j=1}^{n}\lambda_{0,j}\omega_j \right)+\sum_{i=1}^{\mm}\sum_{j\in [n-1]:s_j\leq \q_i} \lambda_{i,j} \conv_{J\in \nabla(Q,w_0)}J_{(i)}(\omega_j).\qedhere
\end{align*}
\end{proof}

 \begin{lem}[Special case of \cite{LMP21} Theorem 7.1]
        If $uv^{-1}=s_1s_2\cdots s_{k}$ then $\tilde{P}_{u,v}$ is combinatorially a cube 
\end{lem}

\begin{cor}\label{Loday mink sum}
    Let $\sigma_1=s_1s_2\dots s_{n-1}$ and $\sigma_i=s_1s_2\dots s_{n-i+1}$ for $i\geq 2.$  Let $\alpha_i=\sigma_2\sigma_3\cdots\sigma_i.$ 
    Then, up to translation and multiplication by $-1$, the Loday associahedron has the following Minkowski decomposition into twisted Bruhat interval polytopes, each of which is combinatorially a cube:
    \[\sum \tilde{P}_{\alpha_i,\alpha_{i+1}} \]
\end{cor}
\begin{proof}
Recall from \cref{prop:is Loday} that $\BB((\sigma_1,\dots, \sigma_n))$ is a translation of $-1$ times the Loday associahedron. We are thus reduced to showing that $\BB((\sigma_1,\dots, \sigma_n)) = \sum \tilde{P}_{\alpha_i,\alpha_{i+1}}.$

\cref{thm:Minkowski} thus tells us that \[\BB((\sigma_1,\dots, \sigma_n)) = \sum_{i=1}^{n} \tilde{P}_{w_0\dem(\sigma_{i+1},\dots, \sigma_n)^{-1},\dem(\sigma_1,\dots, \sigma_i)}.\]
Recall that for $c=s_1s_2\cdots s_{n-1}$ and $C$ a reduced word for $c$, $\sigma_{i+1}$ is the product of the simple reflections in $w_0(C)$ that come from the $i$th copy of $C$ in $C^{\infty}$. Thus,  \cref{lem:easydem} implies that we have $w_0\dem(\sigma_{i+1},\dots, \sigma_n)^{-1} = \sigma_{2}\sigma_3\cdots \sigma_{i} = \alpha_{i}$ and $\dem(\sigma_1,\dots, \sigma_i) = \sigma_{2}\sigma_{3}\cdots \sigma_{i+1}=\alpha_{i+1}.$ Hence \[\BB((\sigma_1,\dots, \sigma_n)) = \sum \tilde{P}_{\alpha_i,\alpha_{i+1}}.\qedhere\]

A special case of \cite{LMP21}[Theorem 7.1] is that if $uv^{-1}=s_1s_2\cdots s_{k}$ then $\tilde{P}_{u,v}$ is combinatorially a cube. Thus, we have an expression for Loday's associahedron as a Minkowski sum of twisted Bruhat interval polytopes, each of which is combinatorially a cube.
\end{proof}

Twisted Bruhat interval polytopes which are combinatorially cubes are especially interesting because they correspond to smooth toric Richardson varieties.

\begin{thm}(\cite[Theorem 5.7]{LMP21})\label{thm:smooth-toric}
    Let $X_u^v$ be a Richardson variety in $\Fl_n$. The following are equivalent:
    \begin{enumerate}
        \item The variety $X_u^v$ is smooth and toric.
        \item The moment polytope $\tilde{P}_{u,v}$ of $X_u^v$ is combinatorially equivalent to a cube.
        \item The variety $X_u^v$ is toric and the Bruhat interval $[u,v]$ is a Boolean lattice.
    \end{enumerate}
\end{thm}

In light of this, the proof of \cref{Loday mink sum} tells us that the toric variety of Loday's associahedron is a subvariety of a product of smooth toric Richardson varieties.

\subsection{The brick resolution}\label{section:fibres}

    In \cite[\S 4.4]{E16}, Escobar showed that brick varieties with respect to certain words in simple transpositions are resolutions of singularities of Richardson varieties. In this section, we will give a pattern-avoidance criterion for when a brick resolution is isomorphic to the Richardson variety whose singularities it resolves (\cref{thm:patterns}). This result generalizes the result of Tenner \cite{T07}, who gives a pattern-avoidance criterion for when a Bott-Samelson resolution is isomorphic to the opposite Schubert variety whose singularities it resolves.

    \subsubsection{Algebraic geometry results}

    We will record three standard results from algebraic geometry that will be used to prove \cref{prop:fibres-iso}.
    \begin{lem}(\cite[Chapter II, \S 3]{H77})\label{lem:fibres-projective}
        If $X$ and $Y$ are projective varieties and $f\colon X\to Y$ is a morphism, then the fibres of $f$ are projective subschemes of $X$. 
    \end{lem}

    \begin{thm}(Zariski's connectedness theorem (cf. \cite[Chapter III, Corollary 11.4]{H77}))\label{thm:ZMT}
        Suppose $f\colon X\to Y$ is a proper, birational morphism of algebraic varieties, and that $Y$ is normal. Then the fibres of $f$ are connected.
    \end{thm}

    \begin{thm}(Upper-semicontinuity of fibre dimension (\cite[Chapter III, \S 12]{H77}))\label{thm:upper-sc}
        Assume that $f\colon X\to Y$ is a proper morphism of varieties.
        Define a function 
			\[d\colon Y\to \mbb{N},\quad y\mapsto \mrm{dim}_{\mbb{C}}(f^{-1}(y)),\]
            where we view the fibre $f^{-1}(y)$ as a closed subscheme of $X$ in the natural way. Then $d$ is upper-semicontinuous on $Y$.
    \end{thm}

\subsubsection{Algebraic varieties with torus actions}
    We will record and prove standard results concerning algebraic varieties with algebraic torus actions, which will be used in the proof of \cref{prop:fibres-iso}.
        
    \begin{prop}\label{prop:one-parameter-subgroup}(\cite[page 80]{K82})
        Let $X$ be an algebraic variety with the algebraic action of a torus $T$, such that $X$ has finitely many $T$-fixed points. There exists a regular element $\varrho\in \mrm{Hom}_\mbb{Z}(\mbb{C}^\times,T)$ such that action of $\varrho(\mbb{C}^\times)$ on $X$ has the same set of fixed points as the action of $T$ on $X$.
    \end{prop}

    The following lemma follows from \cite[Lemma 2.4.1]{CG10}, noting  \cref{prop:one-parameter-subgroup}.
    
    \begin{lem}\label{lem:closure-orbit}(\cite[Lemma 2.4.1]{CG10})
        Suppose $X$ is a projective variety with the algebraic action of a torus $T$. If $\mcal{O}_x$ is the orbit of $x\in X$ under $T$, then the closure $\overline{\mcal{O}_x}$ of $\mcal{O}_x$ in $X$ contains a $T$-fixed point of $X$.  
    \end{lem}

    We will now prove \cref{prop:fibres-iso}, which is the  main result of this subsection. The statement of this proposition was inspired by, and generalizes, the statement of \cite[Lemma 4.3]{DS19}.
    \begin{prop}\label{prop:fibres-iso}
        Suppose $X$ and $Y$ are projective varieties with algebraic actions of a torus $T$, such that $X$ and $Y$ both have finitely many $T$-fixed points. Let $f\colon X\to Y$ be a birational $T$-equivariant morphism of varieties. Assume $X$ is smooth and $Y$ is normal. Then $f$ is an isomorphism if and only if $f$ is a bijection on $T$-fixed points.
    \end{prop}
    \begin{proof}
       One direction is clear: if $f$ is an isomorphism, then $f$ is a bijection on the $T$-fixed points. We will now prove the other direction, and assume that $f$ is a bijection on $T$-fixed points. 
        Let $z$ be a point in $Y$ that is \textit{not} a $T$-fixed point. We will show that the fibre of $f$ over $z$ consists of a single point as well. To do this, we will use \cref{prop:one-parameter-subgroup} to choose a regular element $\varrho$ of $\mrm{Hom}_\mbb{Z}(\mbb{C}^\times,T)$ such that the action of $\varrho(\mbb{C}^\times)$ on $Y$ has the same set of fixed points at the action of $T$ on $Y$. Consider the orbit $\mcal{O}_z:=\{\varrho(t)\cdot z\mid t\in\mbb{C}^\times\}$. As the map $f$ is $T$-equivariant, the fibres of $f$ over any two points in $\mcal{O}_z$ are isomorphic. Define a function
        \[d\colon Y\to\mbb{N},\quad y\mapsto \mrm{dim}_{\mbb{C}}(f^{-1}(y)).\]
        As the fibres of $f$ over any two points in $\mcal{O}_z$ are isomorphic, it follows that the restriction of $d$ to $\mcal{O}_z$ is constant. Thus, we can assume there is $C\in \mbb{N}$ such that $d(z')=C$ for all $z'\in\mcal{O}_z$.
        By \cref{prop:one-parameter-subgroup} and \cref{lem:closure-orbit}, the closure $\overline{\mcal{O}}_z$ of $\mcal{O}_z$ in $X$ contains a $T$-fixed point. As the fibre of $f$ over any $T$-fixed point of $Y$ is a single point, it follows from the upper-semicontinuity of fibre dimension (\cref{thm:upper-sc}) that $C\leq 0$, which ensures $C=0$. As the dimension of the fibre of $f$ over any point of $Y$ is zero-dimensional, $f^{-1}(z)$ is the union of finitely many closed points. By \cref{lem:fibres-projective} and \cref{thm:ZMT}, $f^{-1}(z)$ is connected. As a singleton set is the only type of connected finite union of closed points in the Zariski topology, it follows that $f^{-1}(z)$ consists of a single point. This shows that the fibre of $f$ over any point of $Y$ consists of a single point. In conclusion, the map $f$ is bijective. A bijective morphism of varieties whose source is smooth and whose target is normal is an isomorphism, which completes the argument.
    \end{proof}

    \subsubsection{The Bott-Samelson resolution}

    In certain cases, a Bott-Samelson variety is a resolution of singularities for an opposite Schubert variety. We will state a well-known pattern-avoidance criterion for deciding when a Bott-Samelson resolution is isomorphic to the opposite Schubert variety whose singularities it resolves.

    Let $Q$ be a word in simple transpositions in $S_n$, and set $w(Q):=\mrm{Dem}(Q)$. In this situtation, there is a morphism 
    \[m_Q\colon \mrm{BS}^Q\to X^{w(Q)}, \quad m_Q(F^0_\bullet,F^1_\bullet,\dotsc,F^{|Q|}_\bullet):=F^{|Q|}_\bullet.\] 
    When $Q$ is a reduced word, the morphism $m_Q$ is proper and birational. Therefore, if $Q$ is a reduced word in simple transpositions and $X^{w(Q)}$ is singular, then $\mrm{BS}^Q$ is a resolution of singularities for $X^{w(Q)}$.

    \begin{defn}
        If $Q$ is a reduced word in simple transpositions in $S_n$, then the Bott-Samelson variety $\mrm{BS}^Q$ will also be called the \textbf{Bott-Samelson resolution}.
    \end{defn}

    Write $w\in S_n$ in one-line notation: $w=w_1 w_2\dotsb w_n$. Let $v\in S_k$ be another permutation, with one-line notation $v=v_1 v_2\dotsb v_k$. We say that $w$ \textbf{avoids} the pattern $v$ if $w_{i_1}\dotsb w_{i_k}$ does not appear in the same relative order as $v_1\dotsb v_k$ for any $1\leq  i_1< i_2<\dotsb < i_k\leq n$. The following result is a pattern-avoidance criterion that determines when a Bott-Samelson variety is isomorphic to the opposite Schubert variety whose singularities it resolves.
    
    \begin{thm}(\cite{T07}; cf. \cite[Pattern Avoidance Property 7]{BA16})\label{lem:useful-facts}
		Suppose $Q$ is a reduced word in simple transpositions in $S_n$. The following are equivalent.
		\begin{itemize}
			\item The morphism $m_Q\colon \mrm{BS}^Q\to X^{w(Q)}$ is an isomorphism.
			\item $w(Q)$ avoids the patterns $321$ and $3412$.
		\end{itemize}
	\end{thm}

    \subsubsection{The brick resolution}\label{subsubsection:brick-resolution}

    Similar to the Bott-Samelson variety, in certain cases a brick variety is a resolution of singularities for a Richardson variety. In this subsection, we will generalize \cref{lem:useful-facts} and give a pattern-avoidance criterion for when a brick resolution is isomorphic to the Richardson variety whose singularities it resolves.

    Let $u,v\in S_n$, with $u\leq v$. Suppose $R$ and $S$ are words in simple transpositions in $S_n$ such that $v=\mrm{Dem}(R)$ and $u^{-1}w_0=\mrm{Dem}(S)$. In this situtation, there is a morphism 
    \[m_S^R\colon \mrm{Brick}^{R+S}\to X_u^v,\quad m_S^R\left(F^0_\bullet,F^1_\bullet,\dotsc,F^{|R|-1}_\bullet, F^{|R|}_\bullet,\dotsc,F^{|R|+|S|}_\bullet\right):=F^{|R|}_\bullet.\]
 	When $R$ and $S$ are both reduced words, the morphism $m_S^R$ is proper and birational (\cite[Theorem 26]{E16}). Therefore, if $R$ and $S$ are reduced words in simple transpositions and $X_u^v$ is singular, then $\mrm{Brick}^{R+S}$ is a resolution of singularities\footnote{As explained in \cite[\S 4.4]{E16}, when $S$ and $R$ are both reduced words in simple transpositions in $S_n$, the brick variety $\mrm{Brick}^{R+S}$ is isomorphic to the resolution of singularities of the Richardson variety introduced by Brion \cite[Theorem 4.2.1]{B05}.} for $X_u^v$. 
    
    \begin{defn}
        If $u,v\in S_n$ with $u\leq v$, and $R$ and $S$ are reduced words in simple transpositions in $S_n$ such that $v=\mrm{Dem}(R)$ and $u^{-1}w_0=\mrm{Dem}(S)$, then the brick variety $\mrm{Brick}^{R+S}$ will also be called the \textbf{brick resolution}. 
    \end{defn}

    The following result gives a comparison between the fibres of the brick resolution over $T$-fixed points of $X_u^v$ and fibres of the Bott-Samelson resolution over fibres of $T$-fixed points of opposite Schubert varieties. For a reduced word $Q$ in simple transpositions in $S_n$, recall the Bott-Samelson resolution $m_Q\colon \mrm{BS}^Q\to X^{w(Q)}$.
    
    \begin{lem}\label{fibre-of-brick}
        Let $u,v\in S_n$ with $u\leq v$, and  suppose $R$ and $S$ are reduced words in simple transpositions in $S_n$ such that $v=\mrm{Dem}(R)$ and $u^{-1}w_0=\mrm{Dem}(S)$. For any $w\in [u,v]$, there is a set-theoretic isomorphism
        \[(m_S^R)^{-1}(w\cdot E^{\text{st}}_\bullet)\xrightarrow{\simeq} m_R^{-1}(w\cdot E^{\text{st}}_\bullet)\times m_S^{-1}(w^{-1}w_0\cdot E^{\text{st}}_\bullet),\]
        \begin{align*}
&\left(F^0_\bullet,F^1_\bullet,\dotsc,F^{|R|-1}_\bullet, F^{|R|}_\bullet,F_\bullet^{|R|+1},\dotsc,F^{|R|+|S|}_\bullet\right) \\& \mapsto \left(\left(F^0_\bullet,F^1_\bullet,\dotsc,F^{|R|-1}_\bullet, F^{|R|}_\bullet\right),\left(w^{-1}\cdot F_\bullet^{|R|+1},\dotsc,w^{-1}\cdot F_\bullet^{|R|+|S|}\right)\right)
        \end{align*}
    \end{lem}
    \begin{proof}
        Write $R=(q_{1},q_2,\dotsc,q_{|R|})$ and $S=(q_{|R|+1},q_{|R|+2},\dotsc,q_{|R|+|S|})$, where $q_i$ are simple transpositions in $S_n$. By definition, we have set-theoretic equalities:
        \begin{align*}
        &m_R^{-1}(w\cdot E^{\text{st}}_\bullet)=
        \left\{
\begin{aligned}
\bigl(&F_\bullet^0,F^1_\bullet,\dotsc,F_\bullet^{|R|}\bigr)\in \mrm{BS}^{R}\mid\, F_\bullet^{|R|}=w\cdot E_\bullet^{\text{st}}
\end{aligned}
\right\};
\end{align*}
\begin{align*}
    m_S^{-1}(w^{-1}w_0\cdot E^{\text{st}}_\bullet)=\left\{
\begin{aligned}
\bigl(&F_\bullet^{0},F^{1}_\bullet,\dotsc,F_\bullet^{|S|}\bigr)\in \mrm{BS}^{S}\mid
\,
F_\bullet^{|S|}=w^{-1}w_0\cdot E_\bullet^{\text{st}}
\end{aligned}
\right\}.
\end{align*}
As elements of $S_n$ induce automorphisms of $\Fl_n$, it follows that $w\cdot m_S^{-1}(w^{-1}w_0\cdot E^{\text{st}}_\bullet)$ is isomorphic to $m_S^{-1}(w^{-1}w_0\cdot E^{\text{st}}_\bullet)$. Explicitly, $w\cdot m_S^{-1}(w^{-1}w_0\cdot E^{\text{st}}_\bullet)$ equals
\begin{align*}
    w\cdot m_S^{-1}(w^{-1}w_0\cdot E^{\text{st}}_\bullet)=\left\{
\begin{aligned}
\bigl(&F_\bullet^{0},F^{1}_\bullet,\dotsc,F_\bullet^{|S|}\bigr)\in w\cdot \mrm{BS}^{S}\mid
\,
F_\bullet^{|S|}=w_0\cdot E_\bullet^{\text{st}}
\end{aligned}
\right\}.
\end{align*}
The set underlying the fibre $(m_S^R)^{-1}(wB/B)$ is 
        \begin{align*}
&(m_S^R)^{-1}(wB/B)\\
&=
\left\{
\begin{aligned}
\bigl(&F_\bullet^0,F^1_\bullet,\dotsc,
F_\bullet^{|R|+|S|}\bigr)\in \mrm{BS}^{R+S}\mid
\,
F_\bullet^{|R|}=w\cdot E_\bullet^{\text{st}},\text{ } F_\bullet^{|R|+|S|}=w_0\cdot E_\bullet^{\text{st}}
\end{aligned}
\right\}
\\&=
\left\{
\begin{aligned}
\bigl(&F_\bullet^0,F^1_\bullet,\dotsc,F_\bullet^{|R|}\bigr)\in\mrm{BS}^{R}\mid
\,
F_\bullet^{|R|}=w\cdot E_\bullet^{\text{st}}
\end{aligned}
\right\}\times
\\&
\,  \left\{
\begin{aligned}
\bigl(&F_\bullet^{|R|}, F^{|R|+1}_\bullet,\dotsc,F_\bullet^{|R|+|S|}\bigr)\in w\cdot \mrm{BS}^{S}\mid \,
\text{ }F_\bullet^{|R|+|S|}=w_0\cdot E_\bullet^{\text{st}}
\end{aligned}
\right\}
\\&= m_R^{-1}(w\cdot E_\bullet^{\text{st}})\times w\cdot m_S^{-1}(w^{-1}w_0\cdot E_\bullet^{\text{st}}). 
\end{align*}
Therefore, $(m_S^R)^{-1}(wB/B)$ is isomorphic to $m_R^{-1}(w\cdot E_\bullet^{\text{st}})\times  m_S^{-1}(w^{-1}w_0\cdot E_\bullet^{\text{st}})$ via the function defined in the lemma statement.
    \end{proof}

    \begin{prop}\label{thm:patterns}
        Let $u,v\in S_n$ with $u\leq v$, and  suppose $R$ and $S$ are reduced words in simple transpositions in $S_n$ such that $v=\mrm{Dem}(R)$ and $u^{-1}w_0=\mrm{Dem}(S)$. Then $m_S^R\colon\mrm{Brick}^{R+S}\to X_u^v$ is an isomorphism if and only if, for each $w\in [u,v]$, the fibres $m_R^{-1}(w\cdot E_\bullet^\text{st})$ and $m_S^{-1}(w^{-1}w_0\cdot E_\bullet^{\text{st}})$ each consist of a single point.

        In particular, if $v$ and $u^{-1}w_0$ both avoid the patterns $321$ and $3412$, then $m_S^R$ is an isomorphism, and in particular, $X_u^v$ is smooth.
    \end{prop}
    \begin{proof}
    The first statement follows from \cref{prop:fibres-iso}, which says that $m_S^R$ is an isomorphism if and only if it is a bijection on $T$-fixed points. By \cref{fibre-of-brick}, $m_S^R$ is a bijection on $T$-fixed points if and only if, for each $w\in[u,v]$, the fibres $m_R^{-1}(w\cdot E_\bullet^\text{st})$ and $m_S^{-1}(w^{-1}w_0\cdot E_\bullet^{\text{st}})$  each consist of a single point.
    
    The second statement follows from the first, noting \cref{lem:useful-facts}.
    \end{proof}

    \begin{rmk}
        In \cite{G01}, Gaussent shows that the fibre of a Bott-Samelson resolution $\mrm{BS}^Q$ over a torus-fixed point of $X^{w(Q)}$ admits a finite cellular decomposition, where each cell is an affine subvariety of an open subset of $\mrm{BS}^Q$. It follows from \cref{fibre-of-brick} that the fibre of a brick resolution over a torus-fixed point also admits a finite cellular decomposition.
    \end{rmk}

    \begin{rmk}
        The Poincar\'{e} polynomials of the fibres of Bott-Samelson resolutions were studied by Deodhar in \cite{D90}. It follows from \cref{fibre-of-brick} that the Poincar\'{e} polynomial of the fibre of a brick resolution is the product of Poincar\'{e} polynomials of the fibres of two Bott-Samelson resolutions.
    \end{rmk}

    \begin{rmk}
        In \cite[Theorem 1]{BW01}, the authors characterized when a Bott-Samelson resolution is a ``small resolution" in terms of a pattern-avoidance criterion that they called ``$321$-hexagon avoiding". It would be interesting to determine a pattern-avoidance criterion that characterizes when a brick resolution is a small resolution.
    \end{rmk}

\section{The degeneration of the brick variety}\label{section:degeneration}

In this section, we will describe a $T$-equivariant (embedded in a product of flag varieties) degeneration of the brick variety into unions of products of Richardson varieties. The subdivision of the bulky brick polytope and Loday associahedron induced by this degeneration is discussed in \cref{section:subdivision}.

\cref{subsection:preliminaries2} is dedicated to providing the background from algebraic geometry required to define and analyze our degeneration of the brick variety. In \cref{subsection:the-degeneration}, we will describe a flat degeneration of the closure $D_\sigma$ of the variety consisting of pairs of flags in $\mrm{Fl}_n^2$ that are in relative position $\sigma\in S_n$. In \cref{subsection:degeneration-of-brick}, we will use the flat degeneration of $D_\sigma$ to define and analyze our degeneration of the brick variety.

\begin{nota}
    In this section and in \cref{section:subdivision}, we will always assume that a \textit{variety} is a reduced, irreducible, separated scheme of finite type over $\mbb{C}$. 
\end{nota}

\subsection{Preliminaries}\label{subsection:preliminaries2}

In this subsection, we will compile relevant results from algebraic geometry.

\subsubsection{Intersecting varieties and the Chow ring}

We will recall notions such as proper and transverse intersections of varieties, and state Kleiman's transversality theorem which says that any two subvarieties of a homogeneous space for an algebraic group can be translated by a group element in such a way that their translates intersect generically transversely. We will also discuss the Chow ring of a smooth variety.

\begin{defn}(\cite{EH16}, page 18)
Suppose $A$ and $B$ are subvarieties of a variety $X$. 
\begin{enumerate}
    \item $A$ and $B$ intersect \textbf{transversely} at a point $p\in X$ if $A$, $B$, and $X$ are all smooth at $p$, and 
    $T_pA+T_pB=T_pX$, i.e., the tangent spaces of $A$ and $B$ at $p$ span the tangent space of $X$ at $p$.
    \item $A$ and $B$ intersect \textbf{generically transversely} if they meet transversely at a general point in each component $C$ of $A\cap B$.
    \item $A$ and $B$ intersect \textbf{properly}\footnote{In \cite{EH16}, the authors use the term ``dimensionally transverse" instead of ``proper".} if for every irreducible component $C$ of $A\cap B$, we have $\mrm{codim}_X(C)=\mrm{codim}_X(A)+\mrm{codim}_X(B)$. 
\end{enumerate}
\end{defn}

\begin{rmk}\label{rmk:finite-points}
    If a generically transverse intersection is a finite number of points then the intersection is actually transverse (as explained in \cite[page 18]{EH16}).
\end{rmk}

\begin{prop}[\cite{EH16}, Proposition 1.28]\label{prop:dimensionally-transverse}
    Subschemes $A$ and $B$ of a variety $X$ are generically transverse if and only if they intersect properly, and each irreducible component of $A\cap B$ contains a point where $X$ is smooth and $A\cap B$ is reduced.
\end{prop}

\begin{thm}[Kleiman's transversality theorem]\label{thm:Kleiman}
    Suppose that an algebraic group $G$ acts transitively on a (characteristic $0$) algebraic variety $X$, and suppose $A$ is a subvariety of $X$. If $B$ is another subvariety of $X$, then there is an open dense subset $U$ of $G$ such that $g\cdot A$ is generically transverse to $B$ for all $g\in U$.
\end{thm}

Let $X$ be an algebraic variety. We will recall the \textit{Chow group} of $X$, following the exposition in \cite{EH16}, \cite{F98}. The \textbf{group of $k$-cycles} $Z_k(X)$ on $X$ is the free abelian group generated by isomorphism classes of reduced, irreducible subvarieties of $X$ of dimension $k$. The \textbf{$k$th Chow group} $\mrm{CH}_k(X)$ of $X$ is the quotient of the group $Z_k(X)$ by the relation ``rational equivalence." The \textbf{Chow group} of $X$ is defined to be the graded group $\mrm{CH}_*(X):=\oplus_{k\in\mbb{Z}}\mrm{CH}_k(X)$. Observe that $\mrm{CH}_k(X)=(0)$ for all $k>\mrm{dim}_\mbb{C}(X)$, since $X$ does not have any subvarieties of dimension greater than $\mrm{dim}_{\mbb{C}}(X)$. Any subscheme $Y$ of $X$ defines an element in the Chow group $\mrm{CH}_*(X)$ in the following way: let $Y_1,\dotsc,Y_m$ be the top dimensional irreducible components of $Y$, and let $m(Y_i,Y)$ be the ``intersection multiplicity" of $Y_i$ along $Y$. The class in $\mrm{CH}_*(X)$ defined by $Y$ is $[Y]:=\sum_{i=1}^m m(Y_i,Y)[Y_i]$. 
If $X$ is smooth, then define the abelian group $\mrm{CH}^k(X):=\mrm{CH}_{\mrm{dim}_{\mbb{C}}(X)-k}$, and further, define the graded abelian group $\mrm{CH}^*(X):=\oplus_{k\in\mbb{N}}\mrm{CH}^k(X)$. The group $\mrm{CH}^*(X)$ is graded by \textit{codimensions} of cycles, whereas $\mrm{CH}_*(X)$ is graded by the \textit{dimensions} of cycles.

\begin{thm}[\cite{F98}, cf. Theorem-Definition 1.5 of \cite{EH16}]\label{theo:unique-product-structure}
Let $X$ be a smooth variety. There is a unique product structure on $\mrm{CH}^*(X)$ satisfying the following condition:
\[\text{If two subvarieties $A,B$ of $X$ are generically transverse, then}\]
\[[A]\cdot [B]=[A\cap B].\]
This product structure makes $\mrm{CH}^*(X)$ into an associative, commutative, unital ring, graded by codimension.
\end{thm}

\begin{defn}
    Let $X$ be a smooth variety. The ring $\mrm{CH}^*(X)$ is called the \textbf{Chow ring} of $X$, and the product on $\mrm{CH}^*(X)$ is called the \textbf{intersection product}.
\end{defn}

\begin{egg}\label{egg:CH-Pn}
    Consider the projective space $\mbb{P}^n$. For $0\leq k\leq n$, we will view the projective space $\mbb{P}^k$ as a closed subvariety of $\mbb{P}^n$ via the embedding
    \[\mbb{P}^k\hookrightarrow \mbb{P}^n,\quad [x_0:x_1:\dotsb:x_k]\mapsto [x_0:x_1:\dotsb:x_k:0:\dotsb :0].\]
    The Chow ring of $\mbb{P}^n$ has the following ring presentation:
    \[\mrm{CH}^*(\mbb{P}^n)\simeq \mbb{Z}[x]/(x^{n+1}).\]
    The ring generator $x$ corresponds to the class $[\mbb{P}^{n-1}]$ of the subvariety $\mbb{P}^{n-1}$ of $\mbb{P}^n$. The elements $\{x^0,x^1,x^2,\dotsc,x^n\}$ form a $\mbb{Z}$-basis for $\mrm{CH}^*(\mbb{P}^n)$, where $x^k$ is identified with $[\mbb{P}^{n-k}]$. 
\end{egg}

\begin{thm}[\cite{EH16}, Theorem 1.26]\label{thm:multiplicity1}
    If $A$ and $B$ are subvarieties of a smooth variety $X$ that intersect generically transversely, and $C_1,\dotsc,C_n$ are the components of the intersection $A\cap B$, then 
    \[[A]\cdot [B]=\sum_{i=1}^n [C_i].\]
    That is, the intersection multiplicity of each $C_i$ along $A\cap B$ is $m_{C_i}(A,B)=1$.
\end{thm}

 We will now recall the \textit{K\"{u}nneth isomorphism} for the Chow rings of certain smooth, projective varieties. 
 \begin{thm}(K\"{u}nneth isomorphism)\label{thm:Kunneth}
     Assume that $X$ and $Y$ are smooth projective varieties, each with an algebraic action of a torus $S=(\mbb{C}^\times)^m$ with finitely many fixed points. Then there is a ring isomorphism
     \[\mrm{CH}^*(X)\otimes_{\mbb{Z}}\mrm{CH}^*(Y)\xrightarrow{\simeq} \mrm{CH}^*(X\times Y),\]
     which is defined on generators by sending $[A]\otimes [B]$ to $[A\times B]$, where $A$ and $B$ are subvarieties of $X$ and $Y$, respectively. 
 \end{thm}

Let $X$ be a smooth, projective variety.
We will now compare the Chow ring $\mrm{CH}^*(X)$ and singular cohomology ring $H^*(X)$ of $X$. There is a natural ring homomorphism $c_X\colon \mrm{CH}^*(X)\to H^*(X)$ called the \textbf{cycle class map}. The following theorem gives a sufficient condition for the cycle class map to be an isomorphism.
\begin{thm}\label{thm:cycle-class-iso}
    Let $X$ be a smooth, projective variety, with the algebraic action of a torus $S=(\mbb{C}^\times)^m$ with finitely many fixed points. Then the cycle class map $c_X\colon \mrm{CH}^*(X)\to H^*(X)$ is a ring isomorphism.
\end{thm}
 Flag varieties are examples of varieties that meet the conditions of \cref{thm:cycle-class-iso}, and hence the cycle class map is a ring isomorphism from the Chow ring of a flag variety to its cohomology ring. The Schubert varieties define natural classes in the Chow ring of the flag variety, which form a $\mbb{Z}$-basis for the Chow ring, and the images of these classes in the singular cohomology ring of the flag variety define a natural $\mbb{Z}$-basis as well. This basis is called the \textbf{Schubert basis} of $\mrm{CH}^*(\Fl_n)\simeq H^*(\Fl_n)$. The classes of opposite Schubert varieties also form a $\mbb{Z}$-basis of $\mrm{CH}^*(\Fl_n)\simeq H^*(\Fl_n)$.

\subsubsection{Grothendieck group of coherent sheaves}\label{subsub:Grothendieck-group}

Let $X$ be a variety. The \textbf{Grothendieck group $G_0(X)$ of coherent sheaves} on $X$ is the quotient of the free abelian group generated by isomorphism classes of coherent sheaves on $X$ by the subgroup generated by the elements $[\mcal{F}]-[\mcal{G}]-[\mcal{H}]$ whenever there is a short exact sequence of coherent sheaves on $X$ of the form
\begin{equation}\label{eqn:SES}
0\to \mcal{G}\to \mcal{F}\to \mcal{H}\to 0.
\end{equation}
These relations in $G_0(X)$ can be seen as formally imposing the relation that all short exact sequences of coherent sheaves on $X$ should split. Although the tensor product $\mcal{M}\otimes_{\mcal{O}_X}\mcal{N}$ of two coherent sheaves $\mcal{M}$ and $\mcal{N}$ on $X$ is a coherent sheaf on $X$, the tensor product $-\otimes_{\mcal{O}_X}\mcal{M}$ does \textit{not} necessarily preserve a short exact sequence of the form (\ref{eqn:SES}). In particular, the tensor product of coherent sheaves does not induce a product structure on $G_0(X)$. The story changes considerably if we instead consider the Grothendieck group of vector bundles, i.e. \textit{locally free} coherent sheaves, on $X$.

The \textbf{Grothendieck group $K_0(X)$ of vector bundles} on $X$ is the quotient of the free abelian group generated by isomorphism classes of vector bundles on $X$ by the subgroup generated by the elements $[\mcal{F}]-[\mcal{G}]-[\mcal{H}]$ whenever there is a short exact sequence of vector bundles on $X$ of the form (\ref{eqn:SES}). Given a vector bundle $\mcal{M}$ on $X$, the tensor product $-\otimes_{\mcal{O}_X}\mcal{M}$ \textit{does} preserve a short exact sequence of the form (\ref{eqn:SES}), for \textit{all} coherent sheaves $\mcal{F}$, $\mcal{G}$, $\mcal{H}$ that fit into the exact sequence. Therefore, unlike in $G_0(X)$, the tensor product of vector bundles induces a product structure on $K_0(X)$,
\[[\mcal{F}]\cdot [\mcal{G}]:=[\mcal{F}\otimes_{\mcal{O}_X}\mcal{G}],\quad [\mcal{F}],[\mcal{G}]\in K_0(X),\]
which makes $K_0(X)$ into a commutative, unital ring, and moreover, the tensor product of coherent sheaves induces a $K_0(X)$-module structure on $G_0(X)$,
\[[\mcal{F}]\cdot [\mcal{G}]:=[\mcal{F}\otimes_{\mcal{O}_X}\mcal{G}],\quad [\mcal{F}]\in K_0(X), \text{ }[\mcal{G}]\in G_0(X).\]
We will call $K_0(X)$ the \textbf{Grothendieck ring of vector bundles} on $X$. As every vector bundle on $X$ is a coherent sheaf on $X$, there is always a group homomorphism $s\colon K_0(X)\to G_0(X)$. When $X$ is smooth, the group homomorphism $s$ is an isomorphism.

\begin{egg}\label{egg:K_0(P)}
    Consider the projective space $\mbb{P}^n$. For $0\leq k\leq n$, we will view the projective space $\mbb{P}^k$ as a closed subvariety of $\mbb{P}^n$ in the same way as in \cref{egg:CH-Pn}.
    The Grothendieck ring of vector bundles on $\mbb{P}^n$ has the following ring presentation:
    \[K_0(\mbb{P}^n)\simeq \mbb{Z}[x]/(x^{n+1}).\]
    The ring generator $x$ corresponds to the class $[\mcal{O}_{\mbb{P}^{n-1}}]$ of the structure sheaf of $\mbb{P}^{n-1}$ in $\mbb{P}^n$. The elements $\{x^0,x^1,x^2,\dotsc,x^{n}\}$ form a $\mbb{Z}$-basis for $K_0(\mbb{P}^n)$, where $x^k$ is identified with $[\mcal{O}_{\mbb{P}^{n-k}}]$.
\end{egg}

Recall that, given any two coherent sheaves $\mcal{F},\mcal{G}$ on a variety $X$, the $i$th ``Tor-sheaf" $\mrm{Tor}_i^{\mcal{O}_X}(\mcal{F},\mcal{G})$ on $X$ associated with $\mcal{F}$ and $\mcal{G}$ is a coherent sheaf on $X$ (\cite{H77}).

\begin{thm}\label{thm:free-resolution}
    Assume that $X$ is a smooth variety, and let $\mcal{G}$ be a coherent sheaf on $X$. Then $\mcal{G}$ has a finite resolution by vector bundles $\mcal{F}_i$ on $X$:
    \[0\to \mcal{F}_n\to \mcal{F}_{n-1}\to \dotsb\to \mcal{F}_1\to \mcal{F}_0\to \mcal{G}\to 0,\quad n\leq \mrm{dim}_{\mbb{C}}(X).\]
    There is a well-defined group homomorphism $s'\colon G_0(X)\to K_0(X)$ defined on the generators of $G_0(X)$ by 
    \[s'([\mcal{G}]):=\sum_{i=0}^n(-1)^i[\mcal{F}_i],\quad [\mcal{G}]\in G_0(X).\]
    The homomorphism $s'$ is inverse to $s$, so the groups $K_0(X)$ and $G_0(X)$ are isomorphic. The product on $G_0(X)$ induced by that of $K_0(X)$ is given by
    \[[\mcal{F}]\cdot [\mcal{G}]=\sum_{i=0}^{\mrm{dim}(X)}(-1)^i[\mrm{Tor}_i^{\mcal{O}_X}(\mcal{F},\mcal{G})],\quad [\mcal{F}],[\mcal{G}]\in G_0(X).\]
\end{thm}

\subsubsection{Euler characteristic map}
We now define  the Euler characteristic map. Let $X$ be a projective variety (or, more generally, a ``proper variety'') of dimension $n$, and $\mcal{F}$ a coherent sheaf on $X$. The \textbf{Euler characteristic} of $\mcal{F}$ is the integer $\chi(X;\mcal{F}):=\sum_{i=0}^{n}(-1)^i\mrm{dim}(H^i(X;\mcal{F}))$, where $H^i(X;\mcal{F})$ is the $i$th ``sheaf cohomology group" of $\mcal{F}$. Given a short exact sequence of coherent sheaves of the form (\ref{eqn:SES}), the Euler characteristic has the property $\chi(X;\mcal{F})=\chi(X;\mcal{G})+\chi(X;\mcal{H})$. In particular, the Euler characteristic induces a well-defined group homomorphism $\chi\colon G_0(X)\to \mbb{Z}$. We record the definition of this group homomorphism below.
\begin{defn}\label{defn:Euler-char}
    Let $X$ be a projective variety of dimension $n$. There is a natural group homomorphism $\chi\colon G_0(X)\to\mbb{Z}$ called the \textbf{Euler characteristic map}, defined on generators by
    \[\chi(X;[\mcal{F}]):=\sum_{i=0}^n(-1)^i \mrm{dim}(H^i(X;\mcal{F})),\quad [\mcal{F}]\in G_0(X).\]
\end{defn}

\subsubsection{Hilbert polynomials and subschemes of projective space}\label{sub-sub:Hilbert}

Let $X$ be a closed subscheme of a projective space $\mbb{P}^m$, and consider the very ample line bundle $\mcal{O}_X(1):=\mcal{O}_{\mbb{P}^m}(1)|_{X}$ on $X$ which induces the closed embedding $X\hookrightarrow \mbb{P}^m$. Let $\mcal{F}$ be a coherent sheaf on $X$, and consider the $d$th \textbf{Serre twisting sheaf} $\mcal{F}(d):=\mcal{F}\otimes_{\mcal{O}_X}\mcal{O}_X(d)$ of $\mcal{F}$. The \textbf{Hilbert function} $h_{\mcal{F}}$ of $\mcal{F}$ is the function defined by
\[h_{\mcal{F}}\colon \mbb{N}\to\mbb{N},\quad h_{\mcal{F}}(d):=\mrm{dim}(H^0(X;\mcal{F}(d))).\]
There is a polynomial $H_{\mcal{F}}(x)\in\mbb{Q}[x]$ of degree at most $\mrm{dim}_{\mbb{C}}(X)$ called the \textbf{Hilbert polynomial} of $\mcal{F}$, with the property that $H_{\mcal{F}}(d)=h_{\mcal{F}}(d)$ for $d$ sufficiently large.

The ``Serre vanishing theorem" implies that for $d$ sufficiently large, $H^i(X;\mcal{F}(d))=0$ for all $i>0$. In particular, for $d$ sufficiently large, $\chi(X;\mcal{F}(d))=\mrm{dim}(H^0(X;\mcal{F}(d))=h_\mcal{F}(d)=P_\mcal{F}(d)$. As the Euler characteristic map is well-defined on $G_0(X)$, we thus obtain the following lemma.

\begin{lem}[cf. \cite{V}, Section 18.6]\label{lem:Hilbert-poly-SES}
    Let $X$ be a projective variety, together with a projective embedding $X\hookrightarrow \mbb{P}^m$. Given any two coherent sheaves $\mcal{G}$ and $\mcal{G}'$ on $X$ such that $[\mcal{G}]=[\mcal{G}']$ in $G_0(X)$, there is an equality of Hilbert polynomials $H_{\mcal{G}}(x)=H_{\mcal{G}'}(x)$. Furthermore, $H_{\mcal{F}}(x)=H_{\mcal{G}}(x)+H_{\mcal{H}}(x)$ whenever there is a short exact sequence of coherent sheaves on $X$ of the form (\ref{eqn:SES}). 
\end{lem}

  We will now state and prove an important application of the Hilbert polynomial. Recall that if $X$ is a projective scheme, together with a closed embedding $\iota:X\hookrightarrow \mbb{P}^m$, then $\iota_*\mcal{O}_X$ is isomorphic to $\mcal{O}_{\mbb{P}^m}/\mcal{I}_{X}$, where $\mcal{I}_X$ is the ``ideal sheaf" of $X$ in $\mbb{P}^m$. In this case, the ring of global sections $H^0(\bigoplus_{d\geq 0}(\mcal{O}_{\mbb{P}^m}/\mcal{I}_{X})(d))$ is the homogeneous coordinate ring of $X$ with respect to the embedding $\iota$.

\begin{lem}\label{lem:equal-subschemes}
    Let $X$ and $Y$ be closed subschemes of $\mbb{P}^m$. If $X\hookrightarrow Y$ is an inclusion of schemes and the Hilbert polynomials of $X$ and $Y$ with respect to the closed embeddings $X\hookrightarrow \mbb{P}^m$ and $Y\hookrightarrow \mbb{P}^m$ are equal, then $X=Y$.
\end{lem}
\begin{proof}
    The inclusion of schemes $X\hookrightarrow Y$ yields the following short exact sequence of sheaves on $\mbb{P}^m$,
    \[0\to \mcal{I}_X/\mcal{I}_Y\to \mcal{O}_Y\to\mcal{O}_X\to 0,\]
    where $\mcal{I}_X$ and $\mcal{I}_Y$ are the ideal sheaves of $X$ and $Y$ with respect to the closed embeddings $X\hookrightarrow \mbb{P}^m$ and $Y\hookrightarrow \mbb{P}^m$. It follows from \cref{lem:Hilbert-poly-SES} that
    \[h_{\mcal{I}_X/\mcal{I}_Y}(x)=h_{\mcal{O}_Y}(x)-h_{\mcal{O}_X}(x).\]
    By assumption, $h_{\mcal{O}_Y}(x)=h_{\mcal{O}_X}(x)$. Therefore, $h_{\mcal{I}_X/\mcal{I}_Y}(x)=0$. As explained in \cite[Theorem 18.6.3]{V}, the only coherent sheaf with Hilbert polynomial $0$ is the zero-sheaf, so $\mcal{I}_X/\mcal{I}_Y$ is the zero-sheaf. Thus, the ideal sheaves $\mcal{I}_X$ and $\mcal{I}_Y$ are equal, so $X=Y$ as subschemes of $\mbb{P}^n$. 
\end{proof}

\subsubsection{Integer-valued polynomials and $K$-groups of projective space}
We will now define the group of integer-valued polynomial functions in $\mbb{Q}[x]$, and we will explain that Hilbert polynomials lie in this group of functions. The Grothendieck group $G_0(\P^m)$ of coherent sheaves on $\mbb{P}^m$ is isomorphic to this group of functions, as we will explain below.

\begin{defn}
    Define the free abelian group \[\mrm{Int}_{x,n}:=\{f\in \mbb{Q}[x]:\deg(f)\leq n,\,\, f(\Z)\subseteq \Z\}.\]
\end{defn}

Let $X$ be a projective variety, together with a closed embedding $X\hookrightarrow \mbb{P}^m$, and suppose $\mcal{F}$ is a coherent sheaf on $X$. Set $n:=\mrm{dim}_{\mbb{C}}(X)$, and recall that $H_{\mcal{F}}(x)\in \mbb{Q}[x]$ and $\mrm{deg}(H_{\mcal{F}}(x))\leq n$. As $H_{\mcal{F}}(d)\in \mbb{N}$ for all sufficiently large integers $d$, it follows that $H_{\mcal{F}}(d)\in\mbb{Z}$ for all $d\in\mbb{Z}$. Therefore, $H_{\mcal{F}}(x)\in \mrm{Int}_{x,n}$. It follows from \cref{lem:Hilbert-poly-SES} that there is a group homomorphism $H\colon K_0(\mbb{P}^m)\to \mrm{Int}_{x,m}$, defined on generators by
\begin{equation}\label{eqn:Hilbert-poly}
    H\colon K_0(\mbb{P}^m)\to \mrm{Int}_{x,m},\quad H([\mcal{F}]):=H_{\mcal{F}}(x).
\end{equation}
Recall \cref{egg:K_0(P)}, where we showed there is a ring isomorphism $K_0(\mbb{P}^m)\simeq \mbb{Z}[x]/(x^{m+1})$. This identification shows that $K_0(\mbb{P}^m)$ is a free $\mbb{Z}$-module with basis $\{x^0,x^1,x^2,\dotsc,x^{m}\}$, where $x^k$ is identified with $[\mcal{O}_{\mbb{P}^{m-k}}]$. The function $H$ of (\ref{eqn:Hilbert-poly}) sends the basis element $[\mcal{O}_{\mbb{P}^{k}}]$ to \[H([\mcal{O}_{\mbb{P}^{k}}])=H_{\mcal{O}_{\mbb{P}^{k}}}(x)=\frac{(x+k)(x+k-1)\dotsb (x+1)}{k!}\in \mrm{Int}_{x,m}.\]
As shown in \cite[Proposition 1]{Ch14}, the elements ${x+k\choose k}:=\frac{(x+k)(x+k-1)\dotsb (x+1)}{k!}\in \mrm{Int}_{x,m}$ as $k$ ranges from $0$ to $m$ form a $\mbb{Z}$-basis for $\mrm{Int}_{x,m}$. As both $K_0(\mbb{P}^m)$ and $\mrm{Int}_{x,m}$ are free $\mbb{Z}$-modules of rank $m+1$ and the group homomorphism $H$ sends the $\mbb{Z}$-basis $\{[\mcal{O}_{\mbb{P}^k}]\}_{k=0}^{m}$ of $K_0(\mbb{P}^m)$ to the $\mbb{Z}$-basis $\{{x+k\choose k}\}_{k=0}^{m}$ of $\mrm{Int}_{x,m}$, it follows that $H$ is a group isomorphism. We record this below.

\begin{lem}\label{lem:K-ring-basis}
    The group homomorphism (\ref{eqn:Hilbert-poly}) is a group isomorphism.
\end{lem}

\subsubsection{Flat morphisms}

We will now state several properties possessed by flat morphisms of varieties which will be relevant for this work. Recall that a morphism of varieties $f\colon X\to Y$ is called \textbf{flat} if, for every point $p\in X$, the induced map on the stalk $f_\#\colon \mcal{O}_{Y,f(p)}\to \mcal{O}_{X,p}$ is a flat morphism of rings.

Let $f\colon X\to\mbb{A}^1$ be a flat morphism of varieties, and let $X_y$ be the fibre of $f$ over a point $y\in \mbb{A}^1$. Assume the fibres $X_{y_1}:=f^{-1}(y_1)$ and $X_{y_2}:=f^{-1}(y_2)$ are isomorphic as schemes for all  $y_1,y_2\in\mbb{A}^1\setminus \{0\}$. In this setting, we will call $f$ a \textbf{flat degeneration} of $X_{1}$ to $X_{0}$. If $f\colon X\to\mbb{A}^1$ is a flat degeneration of $X_1$ to $X_0$, then it follows from \cref{lem:special-equals-general} that $[X_{0}]=[X_{1}]$ in $\mrm{CH}^*(X)$.

Next, we will define what we mean by a \textit{$T$-equivariant embedded degeneration} of a variety. This definition will be used frequently in the forthcoming sections.

\begin{defn}\label{defn:torus-equivariant-embedded-degeneration}
    Let $Y$ be a projective variety, $f:Y\times \mbb{A}^1\to \mbb{A}^1$ be the natural projection, and $X$ be a closed subscheme of $Y\times \mbb{A}^1$ such that the map $f|_X:X\to \mbb{A}^1$ is flat and the fibers $(f|_X)^{-1}(t)$ are isomorphic to each other for $t\neq 0$. Set $X_t:=(f|_X)^{-1}(t)$. In this case we say we have a \textbf{degeneration}  $X_1\leadsto X_0$ which is \textbf{embedded} in $Y$. The scheme $X_0$ (which is often reducible) is called the \textbf{special fibre} or \textbf{central fibre} of the degeneration, while any $X_t$, $t\neq 0$, is called a \textbf{general fibre} of the degeneration.

    If additionally $Y$ is a $T$-variety, and the action of $T$ preserves each fiber $f^{-1}(t)$, we call the degeneration \textbf{$T$-equivariant}. If $f|_X\colon X\to \mbb{A}^1$ is a $T$-equivariant, embedded (in $Y$) degeneration $X_1\leadsto X_0$ and $X_0$ has finitely many $T$-orbits\footnote{A (not necessarily normal) algebraic variety with an algebraic $T$-action has finitely many $T$-orbits if and only if it is toric.}, then we call the degeneration \textbf{semitoric}.
\end{defn}

We will now describe several useful properties of flat morphisms of varieties. First, we will state that a flat morphism of varieties extends to the closure.

\begin{thm}[\cite{H77}, Chapter III, Proposition 9.8]\label{thm:closure-of-flat}
    Let $Y$ be a smooth variety of dimension $1$. Let $p\in Y$, and assume that $X$ is a closed subvariety of $\mbb{P}^N\times (Y\setminus p)$ which is flat over $Y\setminus p$. Then there is a unique closed subscheme $\overline{X}$ in $\mbb{P}^N\times Y$, which is flat over $Y$ and whose intersection with $\mbb{P}^N\times (Y\setminus p)$ is $X$. 
\end{thm}

We will now state that the property of being reduced is open for the collection of fibres of a flat, proper morphisms of varieties.

\begin{lem}[\cite{stacks-project}, \href{https://stacks.math.columbia.edu/tag/0C0E}{Tag 0C0E}] \label{lem:reducedness-open-property}
    Let $f\colon X\to Y$ be a flat, proper morphism of complex varieties. Then the set $\{y\in Y\mid \text{$X_y$ is reduced}\}$ is open in $Y$, where $X_y$ is the fibre of $f$ over $y\in Y$.
\end{lem}

We will now state a characterization of flat morphisms of varieties in terms of Hilbert polynomials.

\begin{thm}[\cite{H77}, Theorem 9.9 and Corollary 9.10]\label{thm:Hilbert-polynomial-flatness}
Let $Y$ be any algebraic variety, and let $X$ be a closed subscheme of $\mbb{P}^N\times Y$ for some $N>0$. Consider the restriction $\pi|_{X}\colon X\to Y$ of the projection map $\pi_Y\colon \mbb{P}^N\times Y\to Y$. As $X$ is a closed subvariety of $\mbb{P}^N\times Y$, we will view the fibre $X_y$ of $\pi|_{Y}$ over any point $y\in Y$ as a closed subscheme of $\mbb{P}^N\times \{y\}$ in the natural way, and we will denote the Hilbert polynomial of $X_y$ under this projective embedding by $H_y\in\mbb{Q}[t]$. 

The morphism $\pi|_X\colon X\to Y$ is a flat morphism of varieties if and only if the Hilbert polynomial $H_y$ is independent of $y$. Furthermore, as the leading term of a Hilbert polynomial depends only on the degree and dimension of a projectively embedded variety, it follows that the degree and dimension of $X_y$ are both also independent of $y$.
\end{thm}

We will now state a sufficient condition, commonly referred to as ``miracle flatness", which implies that a morphism of varieties is flat. This condition will be useful for proving \cref{lem:D-flat} in the upcoming subsection. Recall that a variety is \textbf{Cohen-Macaulay} if the local ring at every point of the variety is Cohen-Macaulay. Any smooth variety is Cohen-Macaulay.

\begin{thm}(\cite[ Chapter III, Exercise 10.9]{H77})\label{thm:miracle-flatness}
    Let $f\colon X\to Y$ be a morphism of varieties. Assume that $X$ is Cohen-Macaulay, $Y$ is smooth, and all fibres of $f$ have the same dimension. Then $f$ is a flat morphism of varieties.  
\end{thm}

The following result is a useful property of flatness, which says fibres of a flat morphism of varieties $X\to \mbb{A}^1$ all have the same class in the Chow ring of $X$. This result follows from \cite[Chapter I, Theorem 1.7]{F98}, noting that any two points on $\mbb{A}^1$ are rationally equivalent, so the Chow ring of $\mbb{A}^1$ is isomorphic to $\mbb{Z}$.

\begin{thm}[\cite{F98}, Chapter I, Theorem 1.7]\label{lem:special-equals-general}
    Let $f:X\to \mbb{A}^1$ be a flat morphism of varieties. Given any two points $y_1,y_2\in  \mbb{A}^1$, the classes of the fibres $[f^{-1}(y_1)]$ and $[f^{-1}(y_2)]$ are equal in $\mrm{CH}^*(X)$.
\end{thm}

\subsection{Degeneration of $D_\sigma$}\label{subsection:the-degeneration}
In this subsection, we will define a $T$-equivariant degeneration of $D_\sigma$ to the union $\bigcup_{w\in S_n\text{ s.t. }\ell(w\sigma)=\ell(w)+\ell(\sigma)}(X_w\times X^{w \sigma})$. The main result of this subsection is \cref{sigmadegen}.

       Consider the embedding $\rho\colon \mbb{C}^\times \to T$, given by
	\[\rho(t)=\begin{bmatrix}
		1 &  &  & \\ 
		& t^{-1} &  & \\ 
		&  &  \ddots & \\ 
		&  &   & t^{-n+1}
	\end{bmatrix}.\]
    Define $f:\Fl^2 \times \C^\times \to \Fl^2 \times \C^\times$ by $(F_\bullet, G_\bullet,t)\mapsto (F_\bullet, \rho(t)\cdot G_\bullet,t)$. Now set $\mcal{D}_\sigma^{\circ} := f(D_\sigma \times \C^\times)$ and define $\mcal{D}_\sigma $ to be the closure of $\mcal{D}_\sigma^{\circ}$ in $\Fl_n^2 \times \mbb{A}^1$. Since $\mcal{D}_\sigma^{\circ} $ is the image of a variety under a regular map, it too is a variety. Finally, define $D_\sigma(t)$ to be the fiber of $\mcal{D}_\sigma \to \mbb{A}^1$ over $t$.

     The following result shows that $D_\sigma$ is Cohen-Macaulay and is used to apply miracle flatness (\cref{thm:miracle-flatness}) to prove that $\mcal{D}_\sigma\to\mbb{A}^1$ is a flat degeneration from $D_\sigma$ to $D_{\sigma}(0)$ in \cref{lem:D-flat}. In \cref{sigmadegen}, we will show that $\mcal{D}_\sigma\to\mbb{A}^1$ is in fact a $T$-equivariant degeneration from $D_\sigma$ to a reduced union of varieties of the form $X_w\times X^v$ with $w,v\in S_n$. 

    \begin{lem}\label{lem:CM}
        The variety $D_\sigma$ is Cohen-Macaulay of dimension $\mrm{dim}(\mrm{Fl}_n)+\ell(\sigma)$.
    \end{lem}
\begin{proof}
    Let $Q=(w_0,\sigma)$ and consider $\BS^Q$. Since $D_{w_0}=(\Fl_n)^2,$ we have
    $\BS^Q=\{E_\bullet^\text{st}\}\times D_{\sigma}\cong D_{\sigma}.$
    Generalized Bott-Samelson varieties are Cohen-Macaulay by \cite{silv} so $D_{\sigma}$ is always Cohen-Macaulay. The dimension of $\BS^Q$ is $\ell(w_0)+\ell(\sigma)=\mrm{dim}(\mrm{Fl}_n)+\ell(\sigma).$
\end{proof}

\begin{lem}\label{lem:D-flat}
    The morphism $\mcal{D}_\sigma\to\mbb{A}^1$ is a flat degeneration from $D_\sigma$ to $D_{\sigma}(0)$.
    \end{lem}       
    \begin{proof}
        First, we will show that the morphism $\mcal{D}_\sigma\to(\mbb{A}^1\setminus \{0\})$ is flat. The fibre of $\mcal{D}_\sigma\to(\mbb{A}^1\setminus \{0\})$ over $t\in\mbb{A}^1\setminus \{0\}$ is $D_{\sigma}(t)$, which is isomorphic to $ D_\sigma$. Thus, every fibre of the family $\mcal{D}_\sigma^\circ\to\mbb{A}^1\setminus\{0\}$ has the same dimension, and by \cref{lem:CM}, every fibre is Cohen-Macaulay. As $\mbb{A}^1\setminus \{0\}$ is smooth, it follows from \cref{thm:miracle-flatness} that the morphism $\mcal{D}_\sigma^\circ\to(\mbb{A}^1\setminus\{0\})$ is flat. As $\mcal{D}_\sigma^\circ\subseteq \Fl_n^2\times \mbb{A}^1$, it follows from \cref{thm:closure-of-flat} that the scheme theoretic closure $\mcal{D}_\sigma=\overline{\mcal{D}_\sigma^\circ}$ in $\Fl_n^2\times \mbb{A}^1$ defines a flat morphism $\mcal{D}_\sigma\to\mbb{A}^1$ and satisfies $\mcal{D}_\sigma \cap (\mbb{A}^1\setminus \{0\}) = \mcal{D}_\sigma^\circ$.
\end{proof}

    Suppose $G\times G$ has matrix coordinates $((a_{i,j})_{i,j=1,\dotsc,n},(b_{i,j})_{i,j=1,\dotsc,n})$.
Let $\mcal{G}$ be the subscheme of $G\times G \times \mbb{A}^1$ cut out by the equations
	\[
		a_{i,j}=t^{2j-2i}b_{i,j},\text{ }\quad i\leq j; \qquad \qquad 
		b_{j,i}=t^{2j-2i}a_{j,i},\text{ }\quad i\leq j
	\]
    with $t$ the coordinate of $\mbb{A}^1$.  We let $G_t$ denote the fiber of $\mcal{G} \to \mbb{A}^1$ over $t$. 
    With $T$ acting on $G\times G\times \mbb{A}^1$ by acting diagonally on $G\times G$ and trivially on $\mbb{A}^1$, we have that $\mathcal{G}\to \mbb{A}^1$ is $T$-equivariant and flat.

    \begin{lem}[\cite{KMS06}, Proposition 4.2]
    The morphism $\mathcal{G}\to \mbb{A}^1$ is a $T$-equivariant degeneration of $G_1$ to $G_0$. The fiber over $t\neq 0$ is $G_t=\{(g,\rho(t)g\rho(t)^{-1},t) : g\in G\}$. The special fiber is $G_0=B_-\times_T B_+$ (i.e., $G_0$ is the set of elements in the product $B_-\times B_+$ whose diagonals are equal).
    \end{lem}

    For $\sigma\in S_n$ and $t\in\mbb{C}^\times$, the variety $D_{\sigma}(t):=(1,\rho(t))\cdot D_{\sigma}$ is invariant under the action of $G_t$ on $\mrm{Fl}_n\times \mrm{Fl}_n$: if $(F_\bullet, G_\bullet)\in D_\sigma$ and $g\in G$, then
	\[(g,\rho(t)g(\rho(t))^{-1})\cdot (F_\bullet,\rho(t)(G_\bullet))=((g\cdot F_\bullet),\rho(t)(g\cdot G_\bullet))\in D_{\sigma}(t).\]

    \begin{lem}\label{lem:induced-action-on-special-fibre}
        The action of $G_t$ on $D_{\sigma}(t)$ for each $t\in \mbb{C}^\times$ induces an action of $G_0$ on $D_{\sigma}(0)$. In particular, $D_{\sigma}(0)$ is a union of $G_0$-orbits in $\mrm{Fl}_n^2$.
    \end{lem}
    \begin{proof}
        This proof is the same as the proof of \cite[Proposition 4.1]{KMS06}\footnote{The orbit degeneration technique of degenerating a group alongside a variety so that the degeneration of the variety is a union of orbit closures of the degeneration of the group comes from \cite{KMS06}, which employed the same family of groups $G_t\subseteq G\times G.$}.
    \end{proof}

The following definition is used in \cref{thm:multiplicity-free}, which is used to prove \cref{sigmadegen}.
  \begin{defn}[\cite{B03}]
    A subvariety $V$ of $\Fl_n^2$ is \textbf{multiplicity-free} if its class $[V]\in H^*(\Fl_n^2)$ has the form
    \[[V]=\sum_{v,w\in S_n}c_v^w[X_v\times X^w],\quad c_v^w\in\{0,1\}.\]
\end{defn}
\begin{thm}[\cite{B03}, Theorem 1]\label{thm:multiplicity-free}
    Any multiplicity-free subvariety of $\Fl_n^2$ admits a flat degeneration to a reduced and Cohen-Macaulay union of varieties of the form $X_u\times X^v$.
\end{thm}

    The set of $T$-fixed points of $D_{\sigma}$ is $D_{\sigma}^T=\{(F_{\{{u(i_1),\dotsc,u(i_n)}\}},F_{\{{u\sigma'(i_1),\dotsc,u\sigma'(i_n)}\}})\mid u\in S_n,\sigma'\leq \sigma\}$. In fact $D_{\sigma}(t)^T=D_{\sigma}^T$ for all $t\in\mbb{C}^\times$. By continuity of the action of $T$ on $\mathcal{D}_\sigma$, it follows that $D_{\sigma}(0)^T\supseteq D_\sigma^T$ as well. \cref{lem:induced-action-on-special-fibre} implies that  $D_{\sigma}(0)$ is invariant under $\mcal{G}_0=B_-\times_T B_+$. In particular, the $(B_-\times_T B_+)$-orbit of each fixed point in $D_\sigma^T$ must lie in $D_{\sigma}(0)$. Thus, there is a set-theoretic inclusion $\cup_{u\in S_n\text{ s.t.  } \ell(u\sigma) = \ell(u)+\ell(\sigma)}(X_u\times X^{u\sigma})\subseteq D_{\sigma}(0)$. We record this below. 

    \begin{lem}\label{lem:B-orbit-inclusion}
        There is a set-theoretic inclusion 
        \[\cup_{u\in S_n\text{ s.t.  } \ell(u\sigma) = \ell(u)+\ell(\sigma)}(X_u\times X^{u\sigma})\subseteq D_{\sigma}(0).\]
    \end{lem}

 Now we will show $\cup_{u\in S_n\text{ s.t.  } \ell(u\sigma) = \ell(u)+\ell(\sigma)}(X_u\times X^{u\sigma})= D_{\sigma}(0)$ as schemes, where $\cup_{w\in S_n\text{ s.t.  } \ell(w\sigma) = \ell(w)+\ell(\sigma)}(X_u\times X^{u\sigma})$ has the reduced scheme structure.

\begin{thm}\label{sigmadegen}
      There is a scheme-theoretic equality
      \[D_{\sigma}(0)=\bigcup\limits_{w\in S_n\text{ s.t.  } \ell(w\sigma) = \ell(w)+\ell(\sigma)} X_w\times X^{w\sigma }.\]
    The scheme $D_{\sigma}(0)$ is reduced and Cohen-Macaulay, and each component $ X_w\times X^{w\sigma }$ appearing in the union above is Cohen-Macaulay. Moreover, the flat morphism $\mcal{D}_\sigma\to\mbb{A}^1$ is a $T$-equivariant degeneration of $D_\sigma$ to $D_{\sigma}(0)$.
        \end{thm}
        \begin{proof}
        Consider the singular cohomology ring $\mcal{H}:=H^*(\mrm{Fl_n})\otimes_{\mbb{Z}}H^*(\Fl_n)$. By \cref{thm:cycle-class-iso}, there is a ring isomorphism $H^*(\Fl_n)\simeq \mrm{CH}^*(\Fl_n)$, so we can identify the cup product on $H^*(\Fl_n)$ with the intersection product on $\mrm{CH}^*(\Fl_n)$. The sets $\{[X_w]\}_{w\in S_n}$ and $\{[X^w]\}_{w\in S_n}$ are both bases of $H^*(\mrm{Fl_n})$ over $\mbb{Z}$. Identifying the rings $\mathcal{H}$ and $H^*(\Fl_n)\otimes_{\mbb{Z}}H^*(\Fl_n)$ via the K\"{u}nneth isomorphism (\cref{thm:Kunneth}), we see that $\{[X^u]\otimes [X_v]\}_{u,v\in S_n}$ forms a $\mbb{Z}$-basis for $\mathcal{H}$. By \cref{lem:D-flat}, the morphism $\mcal{D}_\sigma\to\mbb{A}^1$ is flat, so by \cref{lem:special-equals-general} $[D_{
		\sigma}]=[D_{\sigma}(0)]$ in $\mcal{H}$. As $\{[X^u]\otimes [X_v]\}_{u,v\in S_n}$ is a $\mbb{Z}$-basis for $\mathcal{H}$, there are $ c_u^v\in\mbb{Z}$ such that 
	\[[D_{\sigma}(0)]=[D_{\sigma}]=\sum_{u,v\in S_n}c_u^v[X_u]\otimes [X^v].\]
    The bases $\{[X^w]\}_{w\in S_n}$ and $\{[X_v]\}_{v\in S_n}$ have the property that $[X_w]\cdot [X^v]=\delta_{u,v}$ under the intersection product in $\mrm{CH}^*(\Fl_n)$, where $\delta_{u,v}$ is $1$ when $u=v$ and $0$ otherwise. Thus, we can multiply $[D_{\sigma}]$ by $[X^u]\otimes [X_v]$ to extract the coefficient $c_u^v$:
	\[c_{u}^v[\mrm{pt}]=[D_\sigma]\cdot ([X^u]\otimes [X_v])=[D_\sigma]\cdot[X^u\times X_v].\]

	By \cref{thm:Kleiman}, there is an open subset $U$ of $G$ such that for all $g_1,g_2\in U$, the varieties $D_\sigma$ and $(g_1,g_2)\cdot (X^u\times X_v)$ are generically transverse in $\Fl_n^2$. As $D_\sigma$ is a $\Delta G$-orbit in $\Fl_n^2$, we see that $(g,g)\cdot D_\sigma=D_\sigma$ for any $g\in G$. Thus, 
        \[(g_1^{-1},g_1^{-1})\cdot (D_{\sigma}\cap ((g_1,g_2)\cdot (X^u\times X_v)))=D_\sigma\cap ((1,g_1^{-1}g_2)\cdot (X^u\times X_v)),\] 
        and the varieties $D_\sigma$ and $((1,g_1^{-1}g_2)\cdot (X^u\times X_v))$ are  generically transverse because $D_\sigma$ and $(X^u\times X_v)$ are generically transverse and the action of $(g_1^{-1},g_1^{-1})$ on $\Fl_n^2$ is an automorphism.
        As the pair $(g_1,g_2)$ was chosen generically, we can assume $g_1^{-1}g_2$ lies in the \textit{big cell} of $G$ and hence has an \textit{LU decomposition}\footnote{Recall that $G=\mrm{GL}_n(\mbb{C})$ is the disjoint union over the subsets $B_+ w B_-$, where $w\in S_n$. When $w=e$ is the identity, the subset $B_+ wB_-=B_+B_-$ is dense and open in $G$, and we call it the \textbf{big cell} of $G$. Every element $g\in B_+B_-$ by definition can be written (non-uniquely) in the form $g=b_2 b_1$, where $b_2\in B_+$ and $b_1\in B_-$. Such a decomposition of an element $g\in G$ into an element of $B_+$ times an element of $B_-$ is called an \textbf{LU decomposition} of $g$.}, i.e., we can assume $g_1^{-1}g_2=b_2b_1$ for some $b_1\in B_-$ and $b_2\in B_+$. Note that $b_1\cdot X_v=X_v$ and $b_2\cdot X^u=X^u$. Thus,
        \[(b_2^{-1},b_2^{-1})\cdot (D_{\sigma}\cap ((1,b_2b_1)\cdot (X^u\times X_v)))=D_\sigma\cap ((b_2^{-1}\cdot X^u)\times (b_1\cdot X_v))=D_\sigma\cap  (X^u\times X_v),\]
        and the intersection of $D_\sigma$ and $(X^u\times X_v)$ is generically transverse because the intersection of $((1,g_1^{-1}g_2)$ and $(X^u\times X_v))$ is generically transverse and the action of $(b_2^{-1},b_2^{-1})$ on $\Fl_n^2$ is an automorphism.

        Our argument above shows that the varieties $D_\sigma$ and $(X^u\times X_v)$ intersect generically transversely, and moreover that  
        \begin{align*}
            c_u^v[\mrm{pt}]=[D_\sigma]\cdot[X^u\times X_v]&=[D_{\sigma}\cap ((g_1,g_2)\cdot (X^u\times X_v))]
            \\&=[(g_1^{-1},g_1^{-1})\cdot (D_{\sigma}\cap ((g_1,g_2)\cdot (X^u\times X_v)))]\\&=[D_\sigma\cap ((1,g_1^{-1}g_2)\cdot (X^u\times X_v))]\\&=[D_\sigma\cap ((1,b_2b_1)\cdot (X^u\times X_v))]\\&
            =[(b_2^{-1},b_2^{-1})\cdot (D_{\sigma}\cap ((1,b_2b_1)\cdot (X^u\times X_v)))]\\&=[D_\sigma\cap ((b_2^{-1}\cdot X^u)\times (b_1\cdot X_v))]\\&=[D_\sigma\cap  (X^u\times X_v)].
        \end{align*}        

        As the intersection $D_\sigma\cap (X^u\times X_v)$ is generically transverse, it is proper by \cref{prop:dimensionally-transverse}, so the intersection either consists of a (nonzero) finite number of points, or the intersection product $[D_\sigma]\cdot[ (X^u\times X_v)]$ is $0$. A generically transverse intersection consisting of finitely many points is transverse (\cref{rmk:finite-points}); in this case, when $D_\sigma\cap (X^u\times X_v)$ consists of finitely many points, each point appearing in the intersection $D_\sigma\cap (X^u\times X_v)$ appears with multiplicity $1$ (\cref{thm:multiplicity1}), and $c_u^v$ represents the number of (distinct) points in the intersection $D_\sigma\cap (X^u\times X_v)$. 

        We will next determine the pairs $(u,v)$ such that the (proper) intersection $D_\sigma\cap (X^u\times X_v)$ is nonempty. The coefficient $c_u^v$ is nonzero only if any of the following equivalent conditions hold: 
		\begin{align*}
			&\mrm{dim}(\mrm{Fl_n}\times \mrm{Fl_n})=(\mrm{dim}(\mrm{Fl_n}\times \mrm{Fl_n})-\mrm{dim}(D_\sigma))+(\mrm{dim}(\mrm{Fl_n}\times \mrm{Fl_n})-\mrm{dim}(X^u\times X_v))\\
			\iff &0=(\mrm{dim}(\mrm{Fl_n}\times \mrm{Fl_n}))-(\mrm{dim}(D_\sigma))-(\mrm{dim}(X^u\times X_v))\\
			\iff &0=(2\cdot \mrm{dim}(\mrm{Fl_n}))-(\mrm{dim}(\mrm{Fl_n})+l(\sigma))-(l(u)+l(w_0v))\\
			\iff &0=-\mrm{dim}(\mrm{Fl_n})+l(\sigma)+l(u)+l(w_0v)\\
			\iff &0=-l(w_0)+l(\sigma)+l(u)+l(w_0v)\\
			\iff &l(v)=l(\sigma)+l(u).
		\end{align*}
        
        Therefore, the coefficient $c_u^v$ can be nonzero only if $l(v)=l(\sigma)+l(u)$. We will now show that, when $l(v)=l(\sigma)+l(u)$, there is \textit{exactly} one point in the intersection $D_\sigma\cap (X^u\times X_v)$, i.e., that $c_u^v=1$. The diagonal subgroup $\Delta T:=\{(t,t)\mid t\in T\}$ of $T\times T$ acts on both $D_\sigma$ and $X^u\times X_v$, and hence $\Delta T$ acts on the intersection $D_\sigma\cap (X^u\times X_v)$. As $D_\sigma\cap (X^u\times X_v)$ consists of finitely many points, it follows that all points in $D_\sigma\cap (X^u\times X_v)$ are $\Delta T$-fixed points. The $\Delta T$-fixed points of $D_\sigma\cap (X^u\times X_v)$ are precisely the points $(wB/B,w'B/B)\in \mrm{Fl}_n\times \mrm{Fl}_n$, where $w\leq u$ and $w'\geq v$ and $(w,w')=(w,w\sigma')$ for some $\sigma'\leq \sigma$ in $S_n$. As $\sigma'\leq \sigma$, we have $l(\sigma')\leq l(\sigma)$. As $w\leq u$, we have $l(w)\leq l(u)$. As $w\sigma'\geq v$, we have $l(w\sigma')\geq l(v)$. Combining these inequalities, we get
	 	\[l(u)+l(\sigma)=l(v)\leq l(w\sigma')\leq l(w)+l(\sigma')\leq l(w)+l(\sigma)\leq l(u)+l(\sigma).\]
	 	The upper and lower bounds in the string of inequalities above both equal $l(u)+l(\sigma)$. Therefore, all inequalities in between must be equalities. In particular, $l(u)+l(\sigma)=l(w)+l(\sigma)$ and $l(w)+l(\sigma')=l(w)+l(\sigma)$. These two equations imply that $l(w)=l(u)$ and $l(\sigma')=l(\sigma)$. As $w\leq u$ and $l(w)=l(u)$, we must have $w=u$, and as $\sigma'\leq \sigma$ and $l(\sigma')=l(\sigma)$, we must have $\sigma'=\sigma$. In conclusion, the only $\Delta T$-fixed point in $D_\sigma\cap (X^u\times X_v)$ is $(uB/B,u\sigma B/B)$. Therefore, $c_u^v=1$.   

        We have now shown that in $\mcal{H}$,
        \begin{equation}\label{eqn:multiplcity-free}
            [D_\sigma]=[D_{\sigma}(0)]=\sum_{w\in S_n\text{ s.t. }\ell(w\sigma)=\ell(w)+\ell(\sigma)}[X_w\times X^{w \sigma}].
        \end{equation}
        As $D_{\sigma}(0)$ is invariant under $\mcal{G}_0=B_-\times_T B_+$ and $\{[X_v\times X^w]\}_{v,w\in S_n}$ is a $\mbb{Z}$-basis for $\mcal{H}$, we therefore have, set-theoretically, that 
         \[D_{\sigma}(0)=\bigcup_{w\in S_n\text{ s.t. }\ell(w\sigma)=\ell(w)+\ell(\sigma)}(X_w\times X^{w \sigma}).\]

         We will now show that $D_{\sigma}(0)$ is reduced and Cohen-Macaulay.
        The equation (\ref{eqn:multiplcity-free}) implies that $D_\sigma$ is a multiplicity-free subvariety of $\Fl_n^2$, so \cref{thm:multiplicity-free} says that $D_\sigma$ admits a flat degeneration to a reduced Cohen-Macaulay union of varieties of the form $X_v\times X^w$, $u,v\in S_n$. Let $Z$ be this reduced union of varieties, and observe that $Z$ has an algebraic torus action induced by the action of $T$ on $D_\sigma$. It follows from \cref{lem:special-equals-general} that $[Z]=[D_\sigma]=[D_{\sigma}(0)]$ in $\mcal{H}$. As $\{[X_v\times X^w]\}_{v,w\in S_n}$ is a $\mbb{Z}$-basis for $\mcal{H}$, it follows from (\ref{eqn:multiplcity-free}) that the $X_v\times X^w$ appearing as components of $Z$ must be the same 
        ones that appear as components of $D_{\sigma}(0)$. Therefore,
        $Z=(D_{\sigma}(0))_{\text{red}}$, where $(D_{\sigma}(0))_{\text{red}}$ is the scheme $D_{\sigma}(0)$ with its reduced structure. As the degenerations of $D_\sigma$ to both $Z$ and to $D_{\sigma}(0)$ are flat, the Hilbert polynomials $h_{D_\sigma}$, $h_{Z}$, and $h_{D_{\sigma}(0)}$ are all equal (\cref{thm:Hilbert-polynomial-flatness}). 
        In particular, the scheme $Z$ is a closed subscheme of $D_{\sigma}(0)$, and $Z$ and $D_{\sigma}(0)$ have the same Hilbert polynomial. Therefore, by \cref{lem:equal-subschemes}, the schemes $Z=D_{\sigma}(0)$ are equal. This shows that $D_{\sigma}(0)$ is reduced and Cohen-Macaulay. 

        The varieties $X_w\times X^{w\sigma}$ appearing as components of $D_\sigma(0)$ are reduced because they are irreducible components of the reduced variety $D_{\sigma}(0)$, and they are Cohen-Macaulay because $X_w$ and $X^{w\sigma}$ are both Cohen-Macaulay and products of Cohen-Macaulay varieties are Cohen-Macaulay. 

        That the flat degeneration of $D_\sigma$ to $D_{\sigma}(0)$ is a $T$-equivariant degeneration now follows from the fact that the $T$-action on $\mcal{D}_\sigma$ induces the natural  diagonal action of $T$ on each component $X_w\times X^{w\sigma}$. 
         
    \end{proof}

\begin{egg}
    In $S_3$, we have $D_{s_1}(0)=(X_{1}\times X^{s_1})\cup(X_{s_2}\times X^{s_2s_1})\cup(X_{s_1s_2}\times X^{s_1s_2s_1}) $. 
\end{egg}

\subsection{Degeneration of the brick variety}\label{subsection:degeneration-of-brick}

In this subsection, we will define a $T$-equivariant degeneration of the brick variety to a reduced union of products of Richardson varieties (\cref{thm:degen}). The products of Richardson varieties that appear as components of this union are indexed by \textit{brickly sequences}, as defined in \cref{defn:brickly} below. The consequences of our degeneration will be explored in \cref{section:subdivision}. 

Let $Q=(\q_1,\dotsc,\q_\mm)$ be a sequence of permutations in $S_n$.

\begin{defn}\label{defn:brickly}
         A sequence $(u_0,u_1,\dotsc,u_{\mm})$ in $S_n$ is a \textbf{brickly sequence} with respect to $Q$ if $u_0=\id$, $u_{\mm}=w_0$, $\ell(u_i\q_{i+1})=\ell(u_i)+\ell(\q_{i+1})$, and $u_{i+1}\leq u_i {\q_{i+1}}$ for all $i=0,\dotsc,\mm-1$. The set of brickly sequences for $Q$ will be denoted $L(Q)$.
    \end{defn}
    \begin{egg}\label{egg:brickly}
        In $S_3$, set $Q=(s_1,s_2,s_1,s_2,s_1)$. Then $u=(\id,\id,s_2,s_1,s_1s_2,s_1s_2s_1)$ is a brickly sequence with respect to $Q$. 
    \end{egg}
    \begin{egg}
        In $S_3$, let $Q=(s_1s_2, s_1s_2,s_1)$. The two brickly sequences are $(u_0,u_1,u_2,u_3)=(\id,s_2, s_1s_2, s_1s_2s_1)$ and $(u_0,u_1,u_2,u_3)=(\id,\id, s_1s_2, s_1s_2s_1)$.
    \end{egg}
    \begin{egg}
        In $S_4$, let $Q=(s_1s_2s_3, s_1s_2s_3,s_1s_2,s_1)$. The six brickly sequences are:
        \begin{itemize}
            \item $(u_0,u_1,u_2,u_3,u_4)=(\id,\id, s_1s_2s_3, s_1s_2s_3s_1s_2,s_1s_2s_3s_1s_2s_1)$,
            \item  $(u_0,u_1,u_2,u_3,u_4)=(\id,s_2, s_1s_2s_3, s_1s_2s_3s_1s_2,s_1s_2s_3s_1s_2s_1)$,
            \item  $(u_0,u_1,u_2,u_3,u_4)=(\id,s_3, s_1s_2s_3, s_1s_2s_3s_1s_2,s_1s_2s_3s_1s_2s_1)$,
            \item $(u_0,u_1,u_2,u_3,u_4)=(\id,s_2s_3, s_1s_2s_3, s_1s_2s_3s_1s_2,s_1s_2s_3s_1s_2s_1)$,
           \item  $(u_0,u_1,u_2,u_3,u_4)=(\id,s_3, s_3s_1s_2s_3, s_1s_2s_3s_1s_2,s_1s_2s_3s_1s_2s_1)$,
           \item $(u_0,u_1,u_2,u_3,u_4)=(\id,s_2s_3, s_3s_1s_2s_3, s_1s_2s_3s_1s_2,s_1s_2s_3s_1s_2s_1)$,
        \end{itemize}
        The reader is encouraged to observe that $6=3!$.
    \end{egg}

\begin{defn}
   Let $Q=(\q_1,\dots, \q_{\mm})$ be a word in $S_n$.  
    We introduce the following definitions:
    \begin{align*}
        \mcal{I}_j&:=\Fl_n^j \times \mcal{D}_{\q_j} \times \Fl_n^{\mm-j-1} \subseteq \Fl_n^{\mm+1}\times \mbb{A}^1 && \text{for $j\leq m$, }  \\
        \mcal{I}_{\mm+1}&:=X^{\id}\times \Fl_{n}^{\mm-1} \times X_{w_0}\times \mbb{A}^1, && \\
        \mcal{I}_{[j]}&:=\bigcap_{i\in [j]}\mcal{I}_i &&\text{for $j\leq m+1$. }  
    \end{align*}
    Note that the intersection in the definition of $\mcal{I}_{[j]}$ is a scheme-theoretic one. 
\end{defn}

\begin{defn}
    Let $Q=(\q_1,\dots, \q_{\mm})$ be a word in $S_n$. For $t\in \mbb{A}^1$, we define the following scheme-theoretic intersections:
    \begin{align*}
        I_j(t)&=\mcal{I}_j\cap ((\Fl_n)^{m+1}\times \{t\}) \subseteq \Fl_n^{\mm+1}&&\text{ for $j\leq m$, }\\
        I_{\mm+1}(t) &:= \mcal{I}_{\mm+1} \cap ((\Fl_n)^{m+1}\times \{t\})\subseteq \Fl_n^{\mm+1}, &&\\
        I_{[j]}(t)&:=\mcal{I}_{[j]}\cap ((\Fl_n)^{m+1}\times \{t\})\subseteq \Fl_n^{\mm+1} &&\text{for $j\leq m+1$. }
    \end{align*}
\end{defn}

Note that for any $t\in \mbb{A}^1,$ we have the scheme-theoretic equalities
\begin{align*}
    I_j(t)&=\Fl_n^j \times D_{\q_j}(t) \times \Fl_n^{\mm-j-1}\\
    I_{\mm+1}(t)&=X^{\id}\times \Fl_{n}^{\mm-1} \times X_{w_0},\\
    I_{[j]}(t)&= \bigcap_{i\in [j]} I_i(t).
\end{align*}
Set theoretically, $I_{[\mm+1]}(1)=\Brick^Q$ by definition. However, $I_{[\mm+1]}(1)$ is defined as a scheme-theoretic intersection, and we defined $\Brick^Q$ as having reduced structure.
To see that $I_{[\mm+1]}(1)$ and $\Brick^Q$ are equal as schemes, we will first show that $\mcal{I}_{[\mm+1]}$ is flat over $\mbb{A}^1$ and that $I_{[\mm+1]}(0)$ is reduced. We will then use these facts to conclude that $I_{[\mm+1]}(1)$ is reduced.

\begin{prop}\label{degen of brick}
    For $u=(u_0,u_1,\dotsc,u_{\mm})\in L(Q)$, set \[C_u:=X_{u_0}^{u_0}\times X_{u_1}^{u_0 \q_1}\times X_{u_2}^{u_{1}\q_2}\times\dotsb\times X_{u_{\mm}}^{u_{\mm-1}\q_{\mm}}.\]
    Then $I_{[\mm+1]}(0)$ is $\bigcup_{u\in L(Q)}C_u$. In particular, $I_{[m+1]}(0)$ is reduced.
    \end{prop}
\begin{proof}
Applying \cref{sigmadegen}, we find\footnote{Note that there is a Frobenius splitting on $\Fl_n$ that compatibly splits all Schubert varieties and opposite Schubert varieties. This Frobenius splitting extends to the product $\Fl_n^{m+1}$, and compatibly splits all of the $I_j(0)$. As intersections and unions of compatibly split varieties are Frobenius split, and as compatibly split varieties are all reduced, it follows that every intersection and union in this sequence of equalities is reduced.}
            \begin{align*}
                I_{[\mm+1]}(0)&=\bigcap_{j\in [\mm+1]} I_{j}(0)\\
                &=I_{\mm+1}(0) \cap \bigcap_{j\in [\mm]} \left(  \bigcup_{u\in S_n\text{ s.t. }\ell(u\q_j)=\ell(u)+\ell(\q_j)} (\Fl_n)^j \times X_u \times X^{u\q_j}  \times (\Fl_n)^{\mm-j-1} \right)\\
                &=\bigcup_{\substack{\{u_j\in S_n\}_{j=0}^{\mm-1}\\ \text{ s.t. }  \ell(u_j\q_{j+1})=\ell(u_j)+\ell(\q_{j+1})}}  (X^\id \cap X_{u_0} )\times \left(\prod X^{u_{j}\q_{j+1}}\cap X_{u_{j+1}}\right) \times (X^{u_{\mm-1}\q_{\mm}}\cap X_{w_0})\\
                &=\bigcup_{\substack{\{u_j\in S_n\}_{j=0}^{\mm-1}\\ \text{ s.t. }  \ell(u_j\q_{j+1})=\ell(u_j)+\ell(\q_{j+1})}}  X_{u_0}^e\times \left(\prod X^{u_{j}\q_{j+1}}_{u_{j+1}}\right) \times X^{u_{\mm-1}\q_{\mm}}_{w_0}.
            \end{align*}
            
        Intersections of Schubert varieties and opposite Schubert varieties are reduced, and, over $\mbb{C}$, the scheme-theoretic intersection of a product of varieties is the product of the scheme-theoretic intersections. As each Schubert variety and opposite Schubert variety appearing in \cref{sigmadegen} is reduced, it follows that every non-empty component appearing in the above union is reduced.
        Since a Richardson variety $X_u^v$ is nonempty if and only if $u\leq v$, the nonempty terms of the above union are exactly when $u_0=\id$, $u_{i+1}\leq u_iq_{i+1}$ for $i\leq m-2$, and $w_0\leq u_{\mm-1}\q_{\mm}$.       
        That is, 
        \[ I_{[\mm+1]}(0) = \bigcup_{(u_0,\dots,u_{\mm})\in L(Q)} X_{\id}^{\id} \times \prod_{i=1}^{\mm} X_{u_i}^{u_{i-1}\q_i}.\qedhere \]

    \end{proof}

As $D_\sigma$ is fixed by the action of $\Delta G$, given any $t\in\mbb{C}^\times$ and $i\geq 0$ there is an equality of schemes
    \[D_{\sigma}(t)=(1,\rho(t))\cdot D_\sigma=(\rho(t)^i,\rho(t)^{i+1})\cdot D_\sigma.\]
Therefore, for any $t\in\mbb{C}^\times$ and any $j\leq m+1$ we have equalities 
\[I_{j}(t)=(1,\rho(t),\rho(t)^2,\dotsc,\rho(t)^{\mm})\cdot I_{j}(1),\]
and
\[I_{[j]}(t)=(1,\rho(t),\rho(t)^2,\dotsc,\rho(t)^{\mm})\cdot I_{[j]}(1).\]

    The following lemma will be used in \cref{lem:proper-and-CM} to prove that 
     $[\mcal{O}_{I_{[j]}(t)}]=[\mcal{O}_{I_{[j-1]}(t)}]\cdot [\mcal{O}_{I_{j}(t)}]$.

\begin{lem}[\cite{B02}, Lemma 1]\label{lem:Cohen-Macaulay}
    Let $A$ and $B$ be closed subschemes of a smooth variety $X$. If $A$ and $B$ are equidimensional, Cohen-Macaulay, and intersect properly in $X$, then their scheme-theoretic intersection $A\cap B$ is equidimensional and Cohen-Macaulay. Furthermore, 
    \[[\mcal{O}_{A\cap B}]=[\mcal{O}_{A}]\cdot[\mcal{O}_B]\]
\end{lem}

\begin{lem}\label{lem:proper-and-CM}
    For every $t\in\mbb{C}$ and $j\in [\mm+1],$ the schemes $I_{[j-1]}(t)$ and $I_j(t)$ intersect properly,  $I_{[j]}(t)$ is equidimensional and Cohen-Macaulay, and 
      \[[\mcal{O}_{I_{[j]}(t)}]=\prod_{k=1}^j [\mcal{O}_{I_{\{k\}}(t)}].\]
\end{lem}
\begin{proof}
First, we consider the case $t=0$. In this case, we note that by \cref{sigmadegen}, $D_{\q_j}(0)$ is a reduced union of varieties of the form $(\Fl_n)^{j}\times X_u\times X^w \times (\Fl_n)^{\mm-j-1}$ with $\ell(w)=\ell(u)+\ell(\q_j)$. Thus, each $I_j(0)$ is equidimensional, and, as Schubert varieties and opposite Schubert varieties are Cohen-Macaulay and products of Cohen-Macaulay varieties are Cohen-Macaulay, it follows that $I_j(0)$ is Cohen-Macaulay. As a nonempty intersection of a reduced Schubert variety with a reduced opposite Schubert variety is proper (\cite{R92}\footnote{In \cite[Theorem 3.2]{R92}, the author shows that the intersection of a Schubert cell $X_u^\circ$ and opposite Schubert cell $X^v_\circ$ is proper, provided the intersection is nonempty, i.e., provided that $u\leq v$. From this it follows that the intersection of the Schubert variety $X_u$ and opposite Schubert variety $X^v$ is proper (since $X_u\cap X^v$ contains the dense open subset $X_u^\circ\cap X^v_\circ$).}), it follows that a nonempty intersection of $I_1(0)$ and $I_2(0)$ is proper. 
Therefore, \cref{lem:Cohen-Macaulay} implies that the intersection $I_1(0)\cap I_2(0)$ is equidimensional and Cohen-Macaulay and $[\mcal{O}_{I_{[2]}(0)}] = [\mcal{O}_{I_1(0)}]\cdot [\mcal{O}_{I_2(0)}]$. By repeating this argument inductively, we see that the proper intersection $I_{[j]}(0)=I_{[j-1]}(0)\cap I_j(0)$ is equidimensional and Cohen-Macaulay and $[\mcal{O}_{I_{[j]}(0)}]=\prod_{k=1}^j [\mcal{O}_{I_{k}(0)}].$ When we get to $I_{\mm+1}(0)$ it is also equidimensional and Cohen-Macaulay and the properness argument remains unchanged. 

Next, we will consider the case $t\in\mbb{C}^\times$. In this case, we have that $I_{j}(t)=(1,\rho(t),\rho(t)^2,\dotsc,\rho(t)^{\mm})\cdot I_{j}(1)$ and 
$I_{[j]}(t)=(1,\rho(t),\rho(t)^2,\dotsc,\rho(t)^{\mm})\cdot I_{[j]}(1)$. Thus, it suffices to consider the case $t=1$.

For each $j$, $I_j(1)$ is a product of equidimensional Cohen-Macaulay varieties by \cref{lem:CM}, so we have that $I_j(1)$ is equidimensional and Cohen-Macaulay. 

We next show that $I_{[j-1]}(1)$ and $I_j(1)$ intersect properly by computing dimensions.
If we take\footnote{When we view $I_{[\{1,\dotsc,j\}]}$ as a scheme-theoretic intersection as we are doing, it actually \textit{does} have a reduced structure sheaf by ``Frobenius splitting'' arguments in \cite{BrionKumar}. This means that we don't really \textit{need} to ``take'' the reduced structure.} the reduced structure on $I_{[j]}$, it is isomorphic to the Bott-Samelson variety\footnote{When $\q_1,\dots, \q_m$ are simple transpositions, \cite{AndersonFulton} calls $I_{[\mm]}(1)$ a ``big Bott-Samelson variety."}  $\BS^{(w_0,\q_1,\dots, \q_j,w_0,\dots, w_0)}$. Our earlier discussion of generalized Bott-Samelson varieties tells us that $I_{[j]}(1)$ is irreducible with dimension $(\mm-j)\cdot \dim(\Fl_n)+\ell(\q_1)+\dots+\ell(\q_j).$ Thus, $\mrm{codim}(I_{[j]}(1))=(\mm+1) \cdot \dim(\Fl_n)-((\mm-j)\cdot \dim(\Fl_n)+\ell(\q_1)+\dots+\ell(\q_j)) = (j+1)\cdot \dim(\Fl_n) -(\ell(\q_1)+\dots+\ell(\q_j)).$ We can also see that $(I_j(1))_{\mrm{red}}\cong\BS^{(w_0,\dots,w_0,\q_j,w_0,\dots,w_0)}$ so $\mrm{codim}(I_j(1))=\dim(\Fl_n)-\ell(\q_j)$. Hence, $\mrm{codim}(I_{[j]}(1))=\mrm{codim}(I_{[j-1]}(1))+\mrm{codim}(I_j(1)).$ That is, $I_{\{1,\dotsc,j-1\}}$ and $I_j$ intersect properly for $j\leq \mm$. 

To see that $I_{[\mm]}(1)$ and $I_{\mm+1}(1)$ intersect properly, we will use a different argument. The group $G^{\mm+1}$ acts transitively on $\Fl_n^{\mm+1}$. So by \cref{thm:Kleiman}, there is an open subset of $\vec{g}=(g_0,\dots, g_\mm)\in G^{\mm+1}$ such that $\vec{g}\cdot I_{\mm+1}(1)$ is generically transverse to $I_{[\mm] }(1)$. Let $\vec{g}'=(g_0,\id,\id,\dots,\id, g_{\mm})$. Since $I_{\mm+1}(1)=X^e\times \Fl_n^{\mm-1}\times X_{w_0},$ $\vec{g}'\cdot I_{\mm+1}(1)=\vec{g}\cdot I_{\mm+1}(1)$ so $\vec{g}'\cdot I_{\mm+1}(1)$ will also be generically transverse to $I_{[\mm]}(1)$. Hence, $(g_0^{-1},g_0^{-1},\dots, g_0^{-1})\cdot \vec{g}'\cdot I_{\mm+1}(1) = (\id, \dots, \id, g_0^{-1}g_{\mm})\cdot I_{\mm+1}(1)$ will be generically transverse to $(g_0^{-1},g_0^{-1},\dots, g_0^{-1})\cdot I_{[\mm] }(1)= I_{[\mm] }(1)$. Now, since there was an open set of choices of $\vec{g}$ that would work, we can retroactively pick a $\vec{g}$ such that $g_0^{-1}g_{\mm}$ lies in the \textit{big cell} $B_{+}B_{-}\subseteq G$. Writing $g_0^{-1}g_{\mm}=b^{-1}b'$ with $b\in B_+$ and $b'\in B_{-}$ and recalling that $B_{-}$ is the stabilizer of $X_{w_0}=\{E_\bullet^{\mrm{op}}\},$ we have that $(\id, \dots, \id, b^{-1}b')\cdot  X^e\times \Fl_n^{\mm-1}\times X_{w_0}=(\id, \dots, \id, b^{-1})\cdot I_{\mm+1}(1)$ is generically transverse to $I_{[\mm] }(1)$. We now act on each with $(b,b,\dots,b)$ to get that $(b,b,\dots,b)\cdot (\id, \dots, \id, b^{-1})\cdot I_{\mm+1}(1) = (b,\id,\dots, \id)\cdot I_{\mm+1}(1) = I_{\mm+1}(1) $ is generically transverse to $I_{[\mm] }(1)$.

Now we note that each $I_j(1)$ is Cohen-Macaulay so by \cref{lem:Cohen-Macaulay}, the Cohen-Macaulayness inductively passes to each $I_{[j]}(1)$ and we have that each $I_{[j]}(1)$ is equidimensional and Cohen-Macaulay and \[[\mcal{O}_{I_{[j]}(1)}]=\prod_{k=1}^j [\mcal{O}_{I_{\{k\}}(1)}]\]
for all $j.$ By our earlier reduction, the same holds for all $t\in \C^\times$.
 \end{proof}

Specializing \cref{lem:proper-and-CM} to $j=\mm+1$, and recalling that $\Brick^Q$ and $I_{[\mm+1]}(1)$ have the same underlying topological spaces, we can compute the dimension of $\Brick^Q$.
\begin{prop}\label{prop:dim of brick}
    For $Q=(\q_1,\dots, \q_\mm)$ a word in $S_n$ with $\dem(Q)=w_0,$ we have $\dim(\Brick^Q)=\ell(\q_1)+\dots+\ell(\q_\mm)-\ell(w_0).$ 
\end{prop}

\begin{thm}\label{thm:degen}
    Let $Q=(\q_1,\dots,\q_\mm)$ be a word in $S_n$ with $\dem(Q)=w_0$.  The scheme  $\mcal{I}_{[\mm+1]}$ is flat over $\mbb{A}^1$ with general fiber $I_{[\mm+1]}(1)$ equal to the brick variety $\Brick^Q$.
   This gives a $T$-equivariant embedded (within $\Fl_n^{\mm+1}$) degeneration $\Brick^Q\leadsto \bigcup_{u\in L(Q)}C_u$ with $C_u=X_{u_0}^{u_0}\times X_{u_1}^{u_0 \q_1}\times X_{u_2}^{u_{1}\q_2}\times\dotsb\times X_{u_{\mm}}^{u_{\mm-1}\q_{\mm}}.$ 
\end{thm}
\begin{proof}

      By \cref{lem:proper-and-CM}, we have that 
      $[\mcal{O}_{I_{[\mm+1]}(t)}]=\prod_{k=1}^{\mm+1} [\mcal{O}_{I_{\{k\}}(t)}]$. As $\mcal{I}_j$ is a flat degeneration from $I_j(1)$ to $I_j(0)$, it follows from \cref{thm:Hilbert-polynomial-flatness} that $H_{I_j(0)}(x)=H_{I_j(t)}(x)=H_{I_j(1)}(x)$ for all $j\in [\mm+1],t\in \C$. By \cref{lem:K-ring-basis}, we can translate that equality of Hilbert polynomials to an equality of classes in $G_0(\Fl_n^{\mm+1})$: $[\mcal{O}_{I_{\{j\}}(0)}]=[\mcal{O}_{I_{\{j\}}(t)}]=[\mcal{O}_{I_{\{j\}}(1)}]$ for all $j\in [\mm+1],t\in \C$.
      
      Therefore, we get the equalities
\[
    [\mcal{O}_{I_{[\mm+1]}(t)}]=[\mcal{O}_{I_{\{1\}}(t)}] \dotsb [\mcal{O}_{I_{\{\mm+1\}}(t)}]=[\mcal{O}_{I_{\{1\}}(1)}] \dotsb [\mcal{O}_{I_{\{\mm+1\}}(1)}]= [\mcal{O}_{I_{[\mm+1]}(1)}]
\]
for all $t\in \C$. By \cref{thm:Hilbert-polynomial-flatness}, this shows that $\mcal{I}_{[\mm+1]}\to \mbb A^1$ is a flat family. Since the action of $T$  on $\Fl_n^{\mm+1}$ gives an algebraic action on $\mcal{I}_{[\mm+1]}$ that restricts to fibers, we have a $T$-equivariant embedded (within $\Fl_n^{\mm+1}$) degeneration $I_{[\mm+1]}(1)\leadsto I_{[\mm+1]}(0)$.

We know $I_{[\mm+1]}(0)$ is reduced by \cref{degen of brick}. Thus, by \cref{lem:reducedness-open-property}, there is some $t\in \C^\times$ such that $I_{[\mm+1]}(t)$ is reduced. For $t\in \C^\times,$ $I_{[\mm+1]}(t)\cong I_{[\mm+1]}(1)$, so we find that $I_{[\mm+1]}(1)$ is reduced. Recalling that $I_{[\mm+1]}(1)$ and $\Brick^Q$ are equal as sets, and $\Brick^Q$ was defined to have reduced structure, we conclude that  $I_{[\mm+1]}(1)=\Brick^Q$.

Applying \cref{degen of brick}, we conclude that $\mcal{I}_{[\mm+1]}$ gives a $T$-equivariant embedded (within $\Fl_n^{\mm+1}$) degeneration $\Brick^Q\leadsto \bigcup_{u\in L(Q)}C_u$.
\end{proof}

\section{The subdivision of the brick polytope}\label{section:subdivision}

We now have all the tools we need to prove \cref{thm:main-result}. Owing to \cite[Theorem 3.1]{AllenHammer}, when $\Brick^Q$  has finitely many $T$-orbits, the $T$-equivariant embedded degeneration of $\Brick^Q$ from \cref{thm:degen} induces a subdivision of the bulky brick polytope whose face poset matches the poset of torus-orbit closures in the special fibre of the degeneration. We call this poset the \textit{spackle poset}, and this result is recorded as  \cref{thm:poset-isos}, which is the main result of \cref{subsection:bricks-tricks}. Formulas for the volumes of the full-dimensional cells of the subdivision are computed in \cref{subsection:volumes}. In \cref{subsection:Lehmer-complex}, we focus on the induced subdivision of the Loday associahedron, and show that the  spackle poset in this case is isomorphic to the face poset of the unit cube subdivision of the $1\times 2\times \dotsb \times (n-1)$ box, ultimately proving \cref{thm:main-result}.

 \subsection{The spackle poset and a  mixed subdivision of the bulky brick polytope }\label{subsection:bricks-tricks}
In \cref{subsubsection:Richardson}, we will prove relevant results concerning containments of Richardson varieties inside other Richardson varieties, and in \cref{subsubsection:spackle}, we will prove \cref{thm:poset-isos}.

\subsubsection{Richardson varieties and torus-orbit closures in $\mrm{Brick}^Q(0)$}\label{subsubsection:Richardson} 

To begin, we will prove a technical result, \cref{subrichardson}, involving brickly sequences. A consequence of this result is \cref{subrichardson2}, which concerns containments of Richardson varieties within other Richardson varieties. \cref{subrichardson2} will ultimately be used in the proof of \cref{thm:poset-isos}.

\begin{lem}\label{subrichardson}
Let $Q=(\q_1, \q_2,\dots, \q_{\mm})$ be a word in $S_n$ with $\dem(Q)=w_0$.
Suppose $1=u_0,u_1,\dots, u_{\mm}=w_0$ is a brickly sequence in $S_n$.

Then for $x_i=w_0\dem(\q_{i+1},\q_{i+2},\dots, \q_{\mm})^{-1}$ and $y_i=\dem(\q_{1},\q_{2},\dots \q_{i})$, the Bruhat interval $[u_{i},u_{i-1}\q_i]$ is contained in $[x_i,y_i]$.
\end{lem}
\begin{proof}
	For $i=0$ we can see $u_0 = \prod_{j=1}^0 \q_j=1.$ Suppose for induction sake that $u_i\leq \dem(\q_1,\q_2,\dots, \q_i)$. Then   $u_{i+1}\leq u_{i}\star\q_{i+1}\leq \dem(\q_1,\q_2,\dots, \q_{i+1})=u_{i}\q_{i+1}\leq \dem(\q_1,\q_2,\dots, \q_{i+1})$ by \cref{lem: demlem}. Therefore $u_i\leq \dem(\q_1,\dots, \q_{i})$ for all $i$. In particular $u_{i-1}\q_i \leq y_i$ for all $i$.
    	
	For $i=\mm,$ we have that $u_{\mm}^{-1}w_0=1$. Suppose for (backwards) induction that $u_{i+1}^{-1}w_0\leq \dem(\q_{i+2},\dots, \q_{\mm})$. 
	Well, $u_{i+1}\leq u_{i}\q_{i+1}$ thus $u_{i+1}^{-1}\leq \q_{i+1}^{-1}u_i^{-1}$ and $\q_{i+1}^{-1} u_i^{-1}w_0 \leq u_{i+1}^{-1}w_0\leq \dem(\q_{i+2},\dots, \q_{\mm})$. Hence, by Lemma \ref{lem: demlem}, $u_i^{-1}w_0=\q_{i+1}\q_{i+1}^{-1} u_i^{-1}w_0\leq \dem(\q_{i+1},\q_{i+2},\dots, \q_{|Q|}).$  By induction, $u_i^{-1}w_0\leq \dem(\q_{i+1},\q_{i+2},\dots, \q_{\mm})$ for all $i.$ Hence, $u_i\geq w_0 \dem(\q_{i+1},\q_{i+2},\dots, \q_{\mm})^{-1} = x_i$ for all $i.$ We conclude that $[u_i,u_{i-1}\q_i]\subseteq [x_i,y_i]$ for all $i$.
\end{proof}

\begin{lem}\label{subrichardson2}
Suppose $(u_0,u_1,\dots, u_{\mm})$, where $u_0=\id$ and $u_{\mm}=w_0$, is a brickly sequence in $S_n$ for $Q=(\q_1, \q_2,\dots, \q_{\mm})$. 
Then for $x_i=w_0\dem(\q_{i+1},\q_{i+2},\dots, \q_{\mm})^{-1}$ and $y_i=\dem(\q_{1},\q_{2},\dots \q_{i})$, we have $X_{u_{i}}^{u_{i-1}\q_{i}} \subseteq X_{x_i}^{y_i}=\pi_i(\Brick^Q)$ where $\pi_i$ is the projection of $\Fl_n^{\mm+1}$ to the $i$th $\Fl_n$ factor (with indexing starting at 0).
 
\end{lem}
\begin{proof}
    This follows from \cref{subrichardson} with the equality $X_{x_i}^{y_i}=\pi_i(\Brick^Q)$ coming from \cref{proj is Richardson}.
\end{proof}

We will next prove a result, \cref{lem: prod the rich}, characterizing $T$-orbit closures in the special fibre $\mrm{Brick}^Q(0)$ of the degeneration of the brick variety from \cref{degen of brick}. This result will also be used in the proof of \cref{thm:poset-isos}.

\begin{lem}\label{lem: prod the rich}
    Let $c_1,\dots, c_{\mm},d_1,\dots, d_{\mm}$ be permutations in $S_n$.
    Suppose the product of Richardson varieties $Y=\prod_{i=1}^{\mm} X_{c_i}^{d_i}$ is toric with respect to the diagonal $T$ action on $(\Fl_n)^{\mm}$ with $T$ being the torus of invertible diagonal matrices. Then the $T$-orbit closures contained in $\prod_{i=1}^{\mm} X_{c_i}^{d_i}$ are exactly the products of sub-Richardson varieties, i.e. $\prod_{i=1}^{\mm} X_{a_i}^{b_i}$ with $c_i\leq a_i\leq b_i\leq d_i$.
\end{lem}
\begin{proof}
We will make repeated use of the fact that the dimension of a $T$-variety $Z$ matches the dimension of its moment polytope if and only if $Z$ is toric with respect to $T$.

We will first show that each product of sub-Richardson varieties is a $T$-orbit closure.

 Since $Y$ is a product, its dimension is the sum of the dimensions of the terms in the product. Since $Y$ is toric,  each $X_{c_i}^{d_{i}}$ must also be toric and thus has the same dimension as its moment polytope.

    Since $\mu_{\Fl_n^\mm}(Y)=\sum_{i=1}^{\mm}\tilde{P}_{c_i}^{d_{i}},$ we conclude via \cref{lem:MinkowskiMachine} that
    \begin{align*}
        \dim\left( \sum_{i=1}^{\mm}    \tilde{P}_{c_i}^{d_i}\right) &=\dim(\mu_{\Fl_n^\mm}(Y))\\
        &=\dim(Y)\\
        &= \sum_{i=1}^{\mm}  \dim(X_{c_i}^{d_i})\\
        &=\sum_{i=1}^{\mm}  \dim(  \tilde{P}_{c_i}^{d_i}).
    \end{align*}
     This is only possible when the polytopes $\{\tilde{P}_{c_i}^{d_{i}}\}$ are \textit{affinely independent}\footnote{By \textbf{affinely independent}, we mean that if we translate each of the polytopes to contain the origin and then take the linear spans of the points in each translated polytope, then these two linear subspaces will be linearly independent, i.e., the dimension of the sum of those two spaces will be the sum of their dimensions.} from each other. 
     
     Let $a_1,\dots,a_{\mm},b_1,\dots, b_{\mm}$ be permutations in $S_n$ such that $c_i\leq a_i\leq b_i\leq d_i$.  Since $\{\tilde{P}_{c_i}^{d_{i}}\}$ are affinely independent, the sub-polytopes $\{\tilde{P}_{a_i}^{b_i}\}$ are affinely independent. Note that $T$-invariant subvarieties of toric varieties with respect to $T$ are toric. Since $X_{a_i}^{b_i}$ is a $T$-invariant subvariety of $ X_{c_i}^{d_{i}}$, and $X_{c_i}^{d_{i}}$ is toric, we find that $X_{a_i}^{b_i}$ is toric.

     Since $\{\tilde{P}_{a_i}^{b_i}\}$ is an affinely independent collection of toric twisted Bruhat interval polytopes, we conclude
     \begin{align*}
        \dim(\mu_{(\Fl_n)^\mm}(\prod_{i=1}^{\mm} X_{a_i}^{b_i}))&=\dim(\sum_{i=1}^{\mm} \tilde{P}_{a_i}^{b_i})\\
        &=\sum_{i=1}^{\mm} \dim(\tilde{P}_{a_i}^{b_i})\\
        &=\sum_{i=1}^{\mm} \dim(X_{a_i}^{b_i})\\
        &=\dim(\prod_{i=1}^{\mm} X_{a_i}^{b_i})).
     \end{align*}
     Thus $\prod_{i=1}^{\mm} X_{a_i}^{b_i}$ is a $T$-orbit closure contained in $Y$.

    Suppose that $Z=\overline{T\cdot z}$ for some $z\in Y$. Let $\pi_i$ be the map from $(\Fl_n)^{\mm}$ to the $i$th $\Fl_n$ factor. Then $\pi_i(Z)=\overline{T\cdot \pi_i(z)}\subseteq X_{c_i}^{d_i}.$ By \cite[Theorem 4.1]{TW15} and the orbit-cone correspondence, the $T$-orbit closures contained in $X_{c_i}^{d_{i}}$ are exactly $X_{a_i}^{b_i}$ with $c_i\leq a_i\leq b_i \leq d_i$. Thus, there must be some $a_i,b_i$ satisfying those inequalities such that $\pi_i(Z)=X_{a_i}^{b_i}.$ Hence,
     \begin{align*}
         Z&\subseteq \pi_1(Z)\times  \cdots \pi_{\mm}(Z)\\
         &=  X_{a_1}^{b_1}\times \cdots X_{a_{\mm}}^{b_{\mm}}.
     \end{align*} 
   By \cref{lem:MinkowskiMachine}, $\dim(\mu_{(\Fl_n)^{\mm}}(Z))=\dim(\mu_{(\Fl_n)^{\mm}}(X_{a_1}^{b_1}\times \cdots X_{a_{\mm}}^{b_{\mm}})).$
    Since $Z$ and $X_{a_1}^{b_1}\times \cdots X_{a_{\mm}}^{b_{\mm}}$ are both $T$-orbit closures, this means $\dim(Z)=\dim(X_{a_1}^{b_1}\times \cdots X_{a_{\mm}}^{b_{\mm}}).$ Hence, $Z=X_{a_1}^{b_1}\times \cdots X_{a_{\mm}}^{b_{\mm}}.$ We conclude that every $T$-orbit closure contained in $Y$ is a product of sub-Richardson varieties. 
\end{proof}

\subsubsection{The spackle poset and the face poset of the subdivision}\label{subsubsection:spackle} We will now define the \textit{spackle poset}. We will prove in \cref{thm:poset-isos} that the spackle poset is isomorphic to both the poset of torus-orbit closures in the special fibre of the degeneration of \cref{degen of brick}, and to the face poset of the subdivision of the associated bulky brick polytope induced by this degeneration.

\begin{defn}\label{defn:spackle}
    Let $Q=(\q_1,\dots, \q_{\mm})$ be a sequence of permutations. Define the \textbf{spackle poset} as
    \[\mathbb{S}(Q)=\{[\atom_1,\boson_1]\times \cdots \times [\atom_{\mm},\boson_{\mm}]: \exists u_{\bullet}\in L(Q) \text{ s.t. }   u_{i}\leq \atom_{i}\leq \boson_i\leq u_{i-1}\q_i \text{ for } i=1,\dots, \mm\}\]
    with partial order given by inclusion. 
    Define 
    \[\mathbb{S}_{i}(Q)=\{[\atom_1,\boson_1]\times \cdots \times [\atom_{\mm},\boson_{\mm}]\in \mathbb{S}(Q): \ell(\boson_1)-\ell(\atom_1)+\cdots +\ell(\boson_{\mm})-\ell(\atom_{\mm}) = i\}.\] 
    We will identify each minimal element $[\atom_1,\atom_1]\times \cdots \times [\atom_{\mm},\atom_{\mm}] \in \mathbb{S}_0(Q)$ of the poset with the corresponding sequence $(\atom_1,\atom_2,\cdots, \atom_{\mm}).$
\end{defn}

Before proving \cref{thm:poset-isos}, we will recall \cite[Theorem 3.1]{AllenHammer}, and record it as \cref{prop: allen hammer} below. This proposition states that when $\Brick^Q$ has finitely many $T$-orbits, the poset of torus-orbit closures in $\mrm{Brick}^Q(0)$ is isomorphic to the face poset of the induced subdivision of $\BB(Q)$. For use in the following proposition, recall the definitions of \textit{torus-equivariant embedded degeneration} and \textit{semitoric degeneration} from \cref{defn:torus-equivariant-embedded-degeneration}. 

\begin{thm}(\cite[Theorem 3.1]{AllenHammer})\label{prop: allen hammer}
    Let $T$ be a torus, let $Y$ be a projective $T$-variety, and let $X\subseteq Y$ be $T$-invariant, irreducible, and with finitely many $T$-orbits. Let $\mcal{F}$ be a $T$-invariant embedded (inside $Y$) degeneration $X\leadsto \bigcup X_i$, where the $X_i$ are the irreducible components of the special fiber, possibly with multiplicities $(m_i)$.

    Then each $X_i$ has an open dense $T$-orbit, so the degeneration is semitoric. The open $T$-orbit in $X_i$ may have a larger pointwise stabilizer $\mrm{Stab}_T(X_i)$ than it does in $X$, but $\mrm{Stab}_T(X_i)/\mrm{Stab}_T(X)$ has finite order.

    Let $f:Y\to \P^{r}$ be a $T$-equivariant projective embedding of $Y$.
    Then the polytopes $\mu_{\P^r}(X_i)$ are the maximal cells of a subdivision of $\mu_{\P^r}(X)$. The map $Z\mapsto \mu_{\P^r}(Z)$ is an order isomorphism from the poset of $T$-orbit closures in $\bigcup X_i$ to the face poset of the subdivision.
\end{thm}

We will now prove the main result of the subsection, \cref{thm:poset-isos}.

\begin{thm}\label{thm:poset-isos}
        Let $Q=(\q_1,\dots, \q_{\mm})$ be a sequence of permutations with $\dem(Q)=w_0$ such that $\Brick^Q$ has finitely many $T$-orbits.  
        \begin{enumerate}
            \item The degeneration  $\Brick^Q\leadsto \Brick^Q(0)$ from  \cref{degen of brick} is semi-toric so that that each $C_u$ is toric and $\BB(Q)$ is subdivided by their moment polytopes: \[\BB(Q)=\bigcup_{u\in L(Q)} \mu_{\Fl_n^{\mm}}(C_u).\]
             \item The poset of $T$-orbit closures of $\Brick^Q(0)$ is isomorphic to the face poset of the subdivision $\BB(Q)=\bigcup_{u\in L(Q)} \mu_{\Fl_n^{\mm}}(C_u)$ via the moment map $\mu_{\Fl_n^{\mm}}$.
            \item The spackle poset $\mathbb{S}(Q)$ is isomorphic  to the poset of $T$-orbit closures of $\Brick^Q(0)$ via the map $\kappa$ defined
        \[[\atom_1,\boson_1]\times \cdots \times [\atom_{\mm},\boson_{\mm}] \mapsto X_\id^\id \times X_{\atom_1}^{\boson_1}\times \cdots \times X_{\atom_{\mm}}^{\boson_{\mm}}\] for $[\atom_1,\boson_1]\times \cdots \times [\atom_{\mm},\boson_{\mm}]$ an element of the spackle poset. 
            \item The subdivision $\bigcup_{u\in L(Q)} \mu_{\Fl_n^{\mm}}(C_u)$ is a mixed subdivision with respect to the Minkowski decomposition $\BB(Q)=\sum_{i=1}^{\mm} \tilde{P}_{w_0\dem(\q_{i+1},\dots, \q_{\mm})^{-1},\dem(\q_{1},\dots, \q_{i})}$ from \cref{thm:Minkowski}. 
        \end{enumerate}
\end{thm}
\begin{proof}
    The degeneration $\Brick^Q\leadsto \Brick^Q(0)=\bigcup_{u\in L(Q)} C_u$ is a $T$-invariant embedded in $(\Fl_n)^{\mm+1}$, so by \cref{prop: allen hammer}: it is a semitoric degeneration, the polytopes $\mu_{\Fl_n^{\mm}}(C_u)$ are the maximal cells of a subdivision of $\BB(Q)=\mu_{(\Fl_n)^{\mm}}(\Brick^Q)$, and the face poset of that subdivision is isomorphic to the poset of $T$-orbit closures of $\Brick^Q(0)$ via the moment map. It remains to prove that the spackle poset is isomorphic to the poset of $T$-orbit closures in $\Brick^Q(0).$

     By the definition of the spackle poset, there exists $u\in L(Q)$ such that $u_i\leq \atom_i\leq \boson_i\leq u_{i-1}\q_i.$ 
     Hence, 
     \begin{align*}
         \kappa([\atom_1,\boson_1]\times \cdots [\atom_{\mm},\boson_{\mm}])&=X_e^e\times X_{\atom_1}^{\boson_1}\times \cdots \times X_{\atom_{\mm}}^{\boson_{\mm}}\\
         &\subseteq X_\id^\id \times X_{u_1}^{u_0\q_1}\times \cdots \times X_{u_{\mm}}^{u_{\mm-1}\q_{\mm}}\\
         &=C_u\\
         &\subseteq \Brick^Q(0).
     \end{align*} 
    We observe that: $\kappa$ is injective, everything in the image of $\kappa$ is a product of sub-Richardson varieties of some $C_u$, every product of sub-Richardson varieties of some $C_u$ is in the image of $\kappa$, the $T$-orbit closures of  $C_u$ are exactly the products of sub-Richardson varieties of $C_u$ (by \cref{lem: prod the rich}), and every $T$-orbit closure of $\Brick^Q(0)$ is contained in some $C_u$. Together, these facts prove that $\kappa$ is a bijection from the spackle poset to the poset of $T$-orbit closures in $\Brick^Q(0).$ Since $X_{a}^{b}\subseteq X_{c}^d$ if and only if $[a,b]\subseteq [c,d],$ we conclude that $\kappa$ is an order embedding and then by bijectivity, an order isomorphism. Therefore, the poset of $T$-orbit closures in $\Brick^Q(0)$ is isomorphic to the spackle poset.
    
    By \cite{LMP21}, subintervals of a toric Bruhat interval correspond to faces of the (twisted) Bruhat interval polytope. Applying \cref{proj is Richardson} and \cref{subrichardson2}, we find that $\mu_{\Fl_n}(\pi_i(C_u))=\tilde{P}_{u_i}^{u_{i-1}\q_i}$ is a face of $\mu_{\Fl_n}(\pi_{i}(\Brick^Q)).$ Hence the polytopes $\mu_{(\Fl_n)^{\mm+1}}(C_u)=\sum_{i=1}^{\mm+1}\mu_{\Fl_n}(\pi_i(C_u))$ are Minkowski cells of the sum $\sum_{i=1}^{\mm+1} \pi_i(\Brick^Q)$. Therefore, $\bigcup_{u\in L(Q)} \mu_{\Fl_n^{\mm}}(C_u)$ is a mixed subdivision with respect to the Minkowski decomposition of $\BB(Q)$ in \cref{thm:Minkowski}.
     \end{proof}

     \begin{cor}
           Let $Q=(\q_1,\dots, \q_{\mm})$ be a sequence of permutations with $\dem(Q)=w_0$ such that $\Brick^Q$ has finitely many $T$-orbits. The vertices of the subdivision correspond to the elements of $\mbb{S}_0(Q)$ and the vertex corresponding to $\mbb{a}\in \mbb{S}_0(Q)$ appears at 
           \[\sum_{i=1}^{\mm} \mbf{a}_i\left(\sum_{i=1}^n \omega_i\right).  \]
           
           For $u_{\bullet}\in L(Q)$, the subset of $\mbb{S}_0(Q)$ corresponding to the vertices of $C_{u_{\bullet}}$ is $V_{u_\bullet}=\{(\mbb{a}_1,\mbb{a}_2,\dots, \mbb{a}_\mm): u_i\leq \mbb{a}_i\leq u_{i-1}\q_i \text{ for } i=1,\dots, \mm \}.$ 
        Hence,
           \[\mu_{\Fl_n^{\mm}}(C_u) = \conv_{\mbf{a}\in V_{u_\bullet}}\left(\sum_{i=1}^{\mm} \mbf{a}_i\left(\sum_{i=1}^n \omega_i\right)   \right).  \]
     \end{cor}

     \subsection{Volumes of the full-dimensional cells}\label{subsection:volumes}

The main result of this subsection is Proposition \ref{prop:component-volume}, which is a formula for the volumes of the full-dimensional cells in the subdivision of the brick polytope of \cref{thm:poset-isos} in cases where the associated brick variety is a \textit{normal toric variety}.

    \subsubsection{The Hilbert polynomial and Ehrhart polynomial}
  
    Let $X$ be a projective toric variety with respect to the torus $T=(\mbb{C}^\times)^n$, and let $f\colon X\hookrightarrow \mbb{P}^m$ be a $T$-equivariant embedding of $X$ into $\mbb{P}^m$. Recall the Hilbert polynomial $H_{\mcal{O}_X}(x)\in\mbb{Q}[x]$ of $X$, as defined \cref{sub-sub:Hilbert}, with respect to its structure sheaf $\mcal{O}_X:=\mcal{O}_{\mbb{P}^m}|_X$. We will use the notation $H_f(x):=H_{\mcal{O}_X}(x)$. The leading term of $H_{f}(x)$ equals $\tfrac{\mrm{deg}_f(X)}{(\mrm{dim}_{\mbb{C}}(X))!}$, where $\mrm{deg}_f(X)$ is the degree of $X$ with respect to the embedding $f$, and $\mrm{dim}_{\mbb{C}}(X)$ is the dimension of $X$.

    Let $P$ be the moment polytope of $X$ with respect to the embedding $f$, and recall that $P$ sits inside $\mbb{R}^n$. Define a function
	\[e_f\colon \mbb{N}\to\mbb{N},\quad e_f(i):=|(iP)\cap \mbb{Z}^n|,\]
    where $iP$ is the \textit{dilation} of the polytope $P$ by a factor $i$ (i.e., if $P$ is the convex hull of points $v_1,\dotsc,v_t$ in $\mbb{R}^n$, then $iP$ is defined to be the convex hull of the points $i\cdot v_1,\dotsc, i\cdot v_t$ in $\mbb{R}^n$). There is a polynomial $E_f(x)\in\mbb{Q}[x]$, called the \textbf{Ehrhart polynomial} of $X$ with respect to the embedding $f$, such that for $d$ sufficiently large, $E_f(d)=e_f(d)$. The leading term of $E_f(x)$ is $\frac{\mrm{Vol}(P)}{(\mrm{dim}_{\mbb{C}}(X))!}$, where $\mrm{Vol}(P)$ is the volume of $P$.

     The theorem below says that the Ehrhart polynomial coincides with the Hilbert polynomial in the setting above, when the variety $X$ is \textit{normal}. Recall that an algebraic variety is \textbf{normal} if all of the local ring at every point of the variety is an integrally closed domain.
     
	\begin{thm}(see, e.g., \cite[Theorem 13.11]{S96})\label{theo:equal-polys}
		Suppose $X$ is a \textit{normal}, projective toric variety with respect to a torus $T$, and $f\colon X\hookrightarrow \mbb{P}^m$ is a $T$-equivariant embedding of $X$ into $\mbb{P}^m$. Then the polynomials $H_f(x)$ and $E_f(x)$ are equal.
	\end{thm}

    \subsubsection{Saturated chains and the degree of a Richardson variety}
	Recall the weight lattice $\Lambda:=\mrm{Hom}(T,\mbb{C}^\times)\simeq \mbb{Z}^n$ of $T$. A weight $\lambda=(\lambda_1,\lambda_2,\dotsc,\lambda_n)\in\Lambda$ is called \textbf{regular} if $\lambda_1>\lambda_2>\dotsb>\lambda_n$. For any regular weight $\lambda\in\Lambda$, there is an embedding $f_\lambda\colon \Fl_n\hookrightarrow \mbb{P}(V_\lambda)$, where $V_\lambda$ is the irreducible representation of $\mrm{GL}_n$ with highest weight $\lambda$. 
	Define the regular weight $\rho:=\omega_1+\omega_2+\dotsb+\omega_n$. The resulting embedding $f_\rho:\Fl_n\hookrightarrow \mbb{P}(V_\rho)$ is precisely the composition
	\[\Fl_n\hookrightarrow \mrm{Gr}(1,n)\times \mrm{Gr}(2,n)\times\dotsb\times \mrm{Gr}(n,n)\hookrightarrow \mbb{P}^{{n\choose 1}-1}\times \mbb{P}^{{n\choose 2}-1}\times\dotsb\times \mbb{P}^{{n\choose n}-1}\hookrightarrow  \mbb{P},\]
	where the first map is the natural inclusion of the flag variety $\Fl_n$ into a product of Grassmannians, the second map is a product of Pl\"ucker embeddings, and the third map is induced by the Segre embedding. The \textbf{$\lambda$-degree} of $X_u^v$ is defined as the degree of $X_u^v$ with respect to the projective embedding $f_\lambda$. We will take the \textbf{degree} $\mrm{deg}(X_u^v)$ of $X_u^v$ to be the $\rho$-degree of $X_u^v$.	
	
    We will denote the transposition in $S_n$ that swaps $i$ and $j$, with $i<j$, by $r_{i,j}$. Given a covering relation $u\precdot u r_{i,j}$ in the strong Bruhat order on $S_n$, its \textbf{Chevalley multiplicity} with respect to $\lambda$ is  
	\[m_\lambda(u\precdot u r_{i,j}):=\lambda_i-\lambda_{j}.\]
	A \textbf{saturated chain} in $S_n$ is a sequence $C=(u_0\precdot u_1\precdot \dotsb\precdot u_l)$ with all $u_i\in S_n$. The \textbf{Chevalley multiplicity} of a saturated chain $C=(u_0\precdot u_1\precdot \dotsb\precdot u_l)$ in $S_n$ with respect to $\lambda$ is defined as 
	\[m_\lambda(C):=m_\lambda(u_0\precdot u_1)m_\lambda(u_1\precdot u_2)\dotsb m_\lambda(u_{l-1}\precdot u_l).\]

	\begin{thm}\label{thm:volume-of-sat-chains}(\cite[Proposition 3.3]{ATZ24}, also \cite[\S4.8]{BL00}, \cite[Proposition 4.2]{PS09}, \cite{S76})
		 The $\lambda$-degree of $X_u^v$ is 
		\[\mrm{deg}_\lambda(X_u^v)=\sum_{C}m_\lambda(C),\]
		where the sum runs over all saturated chains $C=(u=u_0\precdot u_1\precdot \dotsb\precdot u_l=v)$ in $S_n$.
	\end{thm}

    \subsubsection{The volume formula}
	
	For the remainder of this subsection, we will assume that a subvariety $V$ of $\Fl_n$ is embedded inside a projective space via the restriction of the embedding $f_\rho$ to $V$. The notations $H_{V}(x)$ and $E_V(x)$ will henceforth mean the Hilbert polynomial and Ehrhart polynomial of $V$, respectively, with respect to the embedding $f_\rho$. 

    Fix $Q=(\q_1,\dots, \q_{\mm})$, a sequence of permutations such that $\dem(Q)=w_0$ and such that $\mrm{Brick}^Q$ is a normal, projective toric variety with respect to $T$. For example, if $Q$ is a sequence of simple transpositions satisfying the conditions of \cite[Theorem 2.9]{E16}, then $\mrm{Brick}^Q$ is a smooth (hence normal), projective toric variety. By \cref{thm:degen}, there is a semi-toric degeneration of $\mrm{Brick}^Q$ to the union $\bigcup_{u\in L(Q)}C_u$ with $C_u=X_{u_0}^{u_0}\times X_{u_1}^{u_0 \q_1}\times X_{u_2}^{u_{1}\q_2}\times\dotsb\times X_{u_{\mm}}^{u_{\mm-1}\q_{\mm}}$, and by \cref{thm:poset-isos}, there is a subdivision of the bulky brick polytope $\BB(Q)$ given by $\bigcup_{u\in L(Q)} \mu_{\Fl_n^{\mm}}(C_u)$. We will henceforth assume that a component $C_u=X_{u_0}^{u_0}\times X_{u_1}^{u_0 \q_1}\times X_{u_2}^{u_{1}\q_2}\times\dotsb\times X_{u_{\mm}}^{u_{\mm-1}\q_{\mm}}$ is embedded inside projective space by precomposing the product of embeddings $f_\rho$ by a Segre embedding.
	
	\begin{prop}\label{prop:component-volume}
        Assume $Q=(\q_1,\dots, \q_{\mm})$, a sequence of permutations such that $\dem(Q)=w_0$ and such that $\mrm{Brick}^Q$ is a normal, projective toric variety with respect to $T$. Then the volume of $\mu_{\Fl_n^{m+1}}(C_u)$ is
        \[\mrm{Vol}(\mu_{\Fl_n^{m+1}}(C_u))=\left(\prod_{i=1}^m\frac{\mrm{deg}(X_{u_i}^{u_{i-1}q_i})}{(l(u_{i-1}q_i)-l(u_i))!}\right)\left(\left(\sum_{i=1}^m(l(u_{i-1}q_i)-l(u_i))\right)!\right).\]
        The degree $\mrm{deg}(X_{u_i}^{u_{i-1}q_i})$ has an explicit expression in terms of saturated chains, given in \cref{thm:volume-of-sat-chains}. Finally, the volume of $\BB(Q)$ is 
        \[\mrm{Vol}(\BB(Q))=\sum_{u\in L(Q)}\mrm{Vol}(\mu_{\Fl_n^{m+1}}(C_u)).\]
	\end{prop}
	\begin{proof}
		As $C_u=X_{u_1}^{u_{0}q_1}\times\dotsb\times X_{u_m}^{u_{m-1}q_m}$, we have $H_{C_u}(x)=H_{X_{u_1}^{u_{0}q_1}}(x)\dotsb H_{X_{u_m}^{u_{m-1}q_m}}(x)$. As $C_u$ is toric, normal, and reduced, it follows from \cref{theo:equal-polys} that $H_{C_u}(x)$ equals $E_{C_u}(x)$. In particular, the leading term of $H_{C_u}(x)$ equals the leading term of $E_{C_u}(x)$. The leading term of $H_{C_u}(x)$ equals the product over the leading terms of the $H_{X_{u_i}^{u_{i-1}q_i}}(x)$, and the leading term of $H_{X_{u_i}^{u_{i-1}q_i}}(x)$ is \[\frac{\mrm{deg}(X_{u_i}^{u_{i-1}q_i})}{(\mrm{dim}(X_{u_i}^{u_{i-1}q_i}))!}=\frac{\mrm{deg}(X_{u_i}^{u_{i-1}q_i})}{(l(u_{i-1}q_i)-l(u_i))!}.\] The result follows from the fact that the leading term of the Ehrhart polynomial $E_{C_u}(x)$ is 
        \[\frac{\mrm{Vol}(\mu_{\Fl_n^{m+1}}(C_u))}{(\mrm{dim}(C_u))!}=\frac{\mrm{Vol}(\mu_{\Fl_n^{m+1}}(C_u))}{(\sum_{i=1}^m(l(u_{i-1}q_i)-l(u_i)))!}.\]
	\end{proof}

\begin{rmk}
    When the bulky brick polytope $\BB(Q)$ is the Loday associahedron, we expect there is a different, simpler formula for the volume of $\BB(Q)$ than the one given in \cref{prop:component-volume}, which can be proven using techniques from \cite[\S 10]{qdivdiff}.
\end{rmk}

\subsection{Cubulating the associahedron}\label{subsection:Lehmer-complex}

Call the face poset of the unit cube subdivision of the $1\times 2\times \dotsb \times (n-1)$ box the \textit{Lehmer complex}. We will show that when the brick polytope is the Loday associahedron, the spackle poset is combinatorially isomorphic to the Lehmer complex, completing the proof of \cref{thm:main-result}.

We will establish the notation used in this subsection in \cref{subsubsection:notation}. Technical results on brickly sequences will be proven in \cref{subsubsection:technical}, and the spackle poset for the word corresponding to the Loday associahedron will be analyzed in \cref{subsubsection:spackle-poset}. In \cref{subsubsection:witch}, the explicit isomorphism between the spackle poset and Lehmer complex will be proven as \cref{lem:spackle-poset-iso}, completing the proof of \cref{thm:main-result}.

\subsubsection{Notation}\label{subsubsection:notation} To begin, we recall the choice of word $Q$ in $S_n$ such that $\BB(Q)$ is a translation of the Loday associahedron. Let $\sigma_1=s_1s_2\dots s_{n-1}$ and $\sigma_i=s_1s_2\dots s_{n-i+1}$ for $i\geq 2.$ Let $Q$ be the concatenation $(\sigma_1,\sigma_2,\dots, \sigma_{n}).$ By \cref{prop:is Loday}, $\BB(Q)$ is the same as the Loday realization of the associahedron up to translation and reflection.
 Let $Q_i=(s_1,s_2,\dots, s_{n-i+1})$ for $i\geq 2$ and set $Q_1=Q_2$. 
 Since each $\sigma_i$ is a product of distinct simple reflections, the $T$-equivariant projective embedding of $\Brick^Q$ into $\prod_{i=1}^{n}\prod_{j\in [n-1]:s_j\leq \sigma_i}\Gr(j,n)$ is the same as the projective embedding of $\Brick^{Q_1+\cdots+Q_n}$. The brick variety $\Brick^{Q_1+\cdots+Q_n}$ is smooth and is the toric variety of the Loday associahedron, due to \cite{E16}, so the same is true of $\Brick^Q$. In particular, $\Brick^Q$ has finitely many $T$-orbits so we can apply \cref{prop: allen hammer} to understand the face poset of the subdivision of $\BB(Q)$ as the spackle poset for $Q$.

We fix the word $Q$ above for the rest of this section, and below we record notation which will be used for the rest of this section.

 \begin{nota}\label{notation}
 We fix the following notation:
 \begin{itemize}
     \item Set $\sigma_0:=e$, and $\sigma_1:=s_1s_2\dots s_{n-1}$, and $\sigma_i:=s_1s_2\dots s_{n-i+1}$ for $i\geq 2.$ 
     \item Set $\alpha_0=\alpha_1:=e$, and $\alpha_i:=\sigma_2\sigma_3\dotsb \sigma_i$ for $i=2,\dotsc,n$.
     \item Set $Q:=(\sigma_1,\dotsc,\sigma_n)$.
 \end{itemize}
 \end{nota}

\subsubsection{Technical results on brickly sequences and permutations}\label{subsubsection:technical} Before proceeding to prove our main results, we will prove three technical results, \cref{lem: look the interval}, \cref{lem:passover}, and \cref{lem:alpha interval}, involving permutations and brickly sequences which will be used to prove \cref{lem:spackle-poset-iso}.

\begin{lem}\label{lem: look the interval}
    Let $\sigma_i$, $\alpha_i$, and $Q$ be as in \cref{notation}. If $(u_0,u_1,\dots, u_{|Q|})\in L(Q),$ then 
    $[u_i,u_{i-1}\sigma_i]\subseteq [\alpha_i, \alpha_{i+1}]$ for all $i\geq 0$.

\end{lem}
\begin{proof}
    For $i=0,$ we have $u_0=e$ and $\alpha_0=\alpha_1=e$, so that $e=u_0\in [\alpha_0,\alpha_1]=[e,e]$ and $u_0\sigma_{1}=\sigma_1 \in [e,\sigma_1]= [\alpha_1,\alpha_2]$. 
    
    For $i>0,$ let $x_i=w_0\dem(\sigma_{i+1},\sigma_{i+2},\dots,\sigma_{n})^{-1}$ and $y_i=\dem(\sigma_{1},\sigma_{2},\dots \sigma_{i}).$ Since $\sigma_2\dots \sigma_n = w_0,$ we have $x_1=e=\alpha_1$ and $y_1=\sigma_1=\sigma_2=\alpha_2.$
    For $i\geq 2,$ $\sigma_i$ is the piece of the $c$-sorting word for $w_0$ coming from the $i$th copy of $C=(s_1,s_2,\cdots , s_{n-1})$ in $C^{\infty}$.
    Thus Lemma \ref{lem:easydem} says that $y_i=\dem(\sigma_1,\sigma_2,\dots, \sigma_i)=\sigma_2\sigma_3\cdots \sigma_i \sigma_{i+1} = \alpha_{i+1}$ and $x_i=w_0\dem(\sigma_{i+1},\sigma_{i+2},\dots,\sigma_{n})^{-1} = \sigma_2 \sigma_{3} \cdots \sigma_i = \alpha_i.$ By Lemma \ref{subrichardson}, $[u_i,u_{i-1}\sigma_i] \subseteq [x_i,y_i] = [\alpha_i,\alpha_{i+1}].$
\end{proof}

\begin{lem}\label{lem:passover}
Let $\sigma_i$, $\alpha_i$, and $Q$ be as in \cref{notation}. Then,
\[\{ \alpha_i \cdot  s: s\leq \sigma_{i+1}\} = \{s\cdot \alpha_i: s\leq s_is_{i+1}\dots s_{n-1}\}.\] In particular, for $s=s_{j_1}\dotsb s_{j_{k}}\leq \sigma_{i+1}$, we have $\alpha_i s = s_{j_1+i-1}s_{j_2+i-1}\cdots s_{j_k+i-1} \alpha_i.$ Consequently, if $\alpha_is=s'\alpha_i$ and $\alpha_i t =t'\alpha_i$, then $s\leq t$ if and only if $s'\leq t'$. 
\end{lem}
\begin{proof}

Let $v=\sigma_{i+1}\sigma_{i+2}\cdots \sigma_{n}.$ Then $\alpha_iv=w_0.$ Embed $S_{n-i+1}$ into $S_{n}$ in the maximally sensible way, identifying $s_j$ in $S_{n-i+1}$ with $s_j$ in $S_n$. Not only is $v\in S_{n-i+1},$ it is the longest permutation of $S_{n+1-i}$ because it is equal to $(s_1s_2\cdots s_{n-i})(s_1s_2\cdots s_{n-i-1})\cdots (s_1s_2)(s_1).$ One consequence of this is that $v=v^{-1}.$

For any $\pi \in S_n$, 
\begin{align*}
    \alpha_i \pi &= \alpha_i \pi (v v^{-1})\alpha_i^{-1}\alpha_i \\
    &= \alpha_i \pi v (\alpha_i v)^{-1} \alpha_i\\
    &= \alpha_i \pi v (w_0)^{-1} \alpha_i\\
    &= (w_0 v^{-1} \pi v w_0^{-1}) \alpha_i\\
    &=(w_0    v    \pi v^{-1} w_0^{-1}) \alpha_i.
\end{align*}

The simple transposition $s_j$ has cycle notation $(j\quad j+1)$ so $w_0 s_j w_0^{-1}$ has cycle notation $(n-j+1\quad n-j).$ Hence $w_0 s_j w_0^{-1} = s_{n-j}$. Moreover $w_0 s_{j_1}s_{j_2}\cdots s_{j_k} w_0^{-1} = s_{n-j_1}s_{n-j_2}\cdots s_{n-j_k}$ for any product of simple reflections $s_{j_1}s_{j_2}\cdots s_{j_k}$. Since $v$ is the longest permutation of  $S_{n+1-i}$, this means that $vs_jv^{-1}=s_{n+1-i-j}$ for $j<n-i$ and thus $v s_{j_1}s_{j_2}\cdots s_{j_k} v^{-1} = s_{n+1-i-j_1}s_{n-j_2}\cdots s_{n-j_k}$ for any product of simple reflections $s_{j_1}s_{j_2}\cdots s_{j_k}$ with $j_{\ell}\leq n-i$ for each $\ell$. Carefully composing conjugations coaxes us into comprehending that $w_0 v s_{j_1}s_{j_2}\cdots s_{j_k} v^{-1}w_0^{-1} = s_{j_1+i-1}s_{j_2+i-1}\cdots s_{j_k+i-1}$ for any product of simple reflections $s_{j_1}s_{j_2}\cdots s_{j_k}$ with $j_{\ell}\leq n-i$ for each $\ell$. 

Restricting our attention to $s=s_{j_1}s_{j_2}\cdots s_{j_k}\leq \sigma_{i+1},$ we find $\alpha_i s = s_{j_1+i-1}s_{j_2+i-1}\cdots s_{j_k+i-1} \alpha_i$ and $(w_0    v    \sigma_{i+1} v^{-1} w_0^{-1}) =s_is_{i+1}\cdots s_{n-1}$. Hence, $\{ \alpha_i \cdot  s: s\leq \sigma_{i+1}\} = \{s\cdot \alpha_i: s\leq s_is_{i+1}\dots s_{n-1}\}.$
\end{proof}

\begin{lem}\label{lem:alpha interval}
    Let $\sigma_i$, $\alpha_i$, and $Q$ be as in \cref{notation}.
   The following hold:
   \begin{enumerate}
    \item For $s\leq \sigma_{i+1}$, we have $\alpha_i s=\alpha_i\star s$ and for $s\leq s_is_{i+1}\dots s_{n-1},$ we have $s \alpha_i = s\star \alpha_i$. 
       \item $[\alpha_i,\alpha_{i+1}] = \{ \alpha_i \cdot  s: s\leq \sigma_{i+1}\} = \{s\cdot \alpha_i: s\leq s_is_{i+1}\dots s_{n-1}\}.$ 
       \item For $s\leq \sigma_{i+1}$ and $s'\leq \sigma_{i+1}$, we have $\alpha_i s \leq \alpha_i s'$ if and only if $s\leq s'.$ 
       \item For $s\leq s_is_{i+1}\dots s_{n-1}$ and  $s'\leq s_is_{i+1}\dots s_{n-1}$, we have $s\alpha_i\leq s'\alpha_i$ if and only if $s\leq s'$.
   \end{enumerate}
\end{lem}

\begin{proof}

We will begin by proving (1). Let $s\leq s_1s_2\dots s_{n-i}.$ Since $s$ may be viewed as living inside $S_{n-i+1},$ there exists some $r\in S_{n-i+1}$ such that $s\cdot r$ is the longest permutation of $S_{n-i+1}$ and $\ell(s)+\ell(r)=\ell(s\cdot r).$ When we include $S_{n-i+1}$ into  $S_{n},$ the longest permutation of $S_{n-i+1}$ may be written $\sigma_{i+1}\sigma_{i+2}\dots \sigma_{n}$. Including $r$ into $S_n$ as well, we have $sr=\sigma_{i+1}\sigma_{i+2}\dots \sigma_{n}.$ Thus $\sigma_2\sigma_3\cdots \sigma_i \cdot s\cdot r$ is a reduced expression for the longest permutation of $S_{n}.$ We conclude that $\alpha_i s = \sigma_2\sigma_3\cdots \sigma_i s$ is a reduced expression when $s\leq \sigma_{i+1}$. Applying  \cref{lem:passover}, we also conclude that $s'\alpha_i=s'\sigma_2 \sigma_3\dotsb \sigma_i$ is a reduced expression when $s'\leq s_i s_{i+1}\dotsb s_{n-1}$. This completes the proof of (1). 

We will now explain how (2) follows from (1). Part (1) tells us that $\alpha_i\leq \alpha_i s\leq \alpha_i \sigma_{i+1}=\alpha_{i+1}$ for all $s\leq \sigma_{i+1}.$ In particular, $\{ \alpha_i s: s\leq \sigma_{i+1}\}\subseteq [\alpha_i , \alpha_{i+1}]$. It is shown in \cite[Theorem 1.3]{LMP21} that a Bruhat interval polytope $\tilde{P}_{u,v}$ is combinatorially equivalent to a cube if and only if the Bruhat interval $[u,v]$ is a Boolean lattice. This observation, combined with \cref{Loday mink sum}, says that the interval $ [\alpha_i , \alpha_{i}\sigma_{i+1}]$ is a Boolean lattice, and thus has $2^{\ell(\sigma_{i+1})}$ elements -- the same cardinality as $\{\alpha_i s: s\leq \sigma_{i+1}\}$. Therefore, $\{ \alpha_i s: s\leq \sigma_{i+1}\}= [\alpha_i , \alpha_{i+1}]$. Applying \cref{lem:passover} gives the second equality.

We will next explain why (3) is true. When $s\leq s'\leq \sigma_{i+1},$ the fact that $\alpha_i s$ and $\alpha_i s'$ are both reduced expressions implies that $\alpha_i s \leq \alpha_{i}s'.$ Thus there is a bijective monotone map from the poset of subwords of  $\sigma_{i+1}$ to the Bruhat interval $[\alpha_i,\alpha_{i+1}].$ Since the poset of subwords of $\sigma_{i+1}$ forms a boolean lattice under containment and the Bruhat interval $[\alpha_i,\alpha_{i+1}]$ is a Boolean lattice of the same rank, our bijective monotone map must in fact be an order isomorphism. So we find that for $s\leq \sigma_{i+1}$ and $s'\leq \sigma_{i+1}$, we have $\alpha_i s \leq \alpha_{i}s'$ if and only if $s\leq s'.$ 

Finally, (4) follows from (3) by applying \cref{lem:passover}.
\qedhere

\end{proof}

\subsubsection{The spackle poset for the word $Q$}\label{subsubsection:spackle-poset}

We will now study the spackle poset $\mbb{S}(Q)$ with respect to the word $Q$ we fixed at the start of this subsection. In \cref{pieces} and \cref{lem: the vertices}, we will give explicit descriptions for $\mbb{S}_{n-1}(Q)$ and $\mbb{S}_0(Q)$, respectively, which will help us prove that that the poset $\mbb{S}(Q)$ is isomorphic to the face poset of the standard cubulation (which we will call the \textit{Lehmer complex} in \cref{subsubsection:witch}).

In \cref{pieces} below, note that it follows from \cref{thm:poset-isos} that $\mbb{S}_{n-1}=L(Q)$.

\begin{lem}\label{pieces}
Let $\sigma_i$, $\alpha_i$, and $Q$ be as in \cref{notation}. Then\footnote{Below, we are slightly abusing notation. A brickly sequence $(\id,v_1\alpha_1,\dots, v_{n}\alpha_{n})$ in $L(Q)$ by definition begins at the identity. The element of $\mbb{S}_{n-1}(Q)$ corresponding to this brickly sequence is obtained by omitting the first component: $(v_1\alpha_1,\dots, v_{n}\alpha_{n})$.}, 
\begin{align*}
    \mbb{S}_{n-1}(Q)&=L(Q)\\
    &=\{(\id,v_1\alpha_1,\dots, v_{n}\alpha_{n}):v_i\in S_n, v_{i+1}\leq v_{i}, v_i\leq s_{i+1}s_{i+2}\dots s_{n-1} \text{ for } i\geq 1\}.
\end{align*}
\end{lem} 
\begin{proof}
Let $(u_0,u_1,\dots, u_{n})\in L(Q)$ be a brickly sequence for $Q.$ Then $u_0=e$, so let $i>0$.

We know from \cref{lem: look the interval} that $u_{i}\in [\alpha_{i},\alpha_{i+1}]$ and $u_i\sigma_{i+1}\in [\alpha_{i+1},\alpha_{i+2}]$. Thus, by \cref{lem:alpha interval}, there exist $v_i\leq s_{i}s_{i+1}\cdots s_{n-1}$ and $v_i'\leq s_{i+1}\cdots s_{n-1}$ such that $u_i=v_i\alpha_i$ and $u_i\sigma_{i+1}=v_i'\alpha_{i+1}.$ Observing that $v_i'\alpha_{i+1}=u_i\sigma_{i+1}=v_i\alpha_i\sigma_{i+1}=v_i\alpha_{i+1}$, we find that $v_i=v_i'.$ In particular, $v_i\leq s_{i+1}\cdots s_{n-1}$ for all $i\geq 1$.

We know from the definition of brickly sequences that $u_{i+1}\leq u_i\sigma_{i+1}$. Thus, \[v_{i+1}\alpha_{i+1}=u_{i+1}\leq u_{i}\sigma_{i+1}= v_i\alpha_i\sigma_{i+1}=v_i\alpha_{i+1}.\] Using our newfound knowledge that $v_{i}\leq s_{i+1}\cdots s_{n-1}$ and $v_{i+1}\leq s_{i+1}\cdots s_{n-1},$ we apply \cref{lem:alpha interval} to conclude that $v_{i+1}\leq v_i$ for each $i\geq 1$.

Summarizing what we have shown thus far, we have that for $u_{\bullet}\in L(Q)$, there exists (uniquely) a  sequence of permutations $v_1,\dots, v_n$ with $v_i\leq s_{i+1}\cdots s_{n-1}$ such that $u_i= v_i\alpha_{i}$ and $v_{i+1}\leq v_i$. This proves that 
\[L(Q)\subseteq \{(\id,v_1\alpha_1,\dots, v_{n}\alpha_{n}):v_i\in S_n, v_{i+1}\leq v_{i}, v_i\leq s_{i+1}s_{i+2}\dots s_{n-1} \text{ for } i\geq 1\}.\] 

Now suppose we have a sequence of permutations $v_1,\dots, v_n$ with $v_i\leq s_{i+1}\cdots s_{n-1}$ such that $v_{i+1}\leq v_i,$ and $v_{i}\leq s_{i+1}\cdots s_{n-1}.$ Set $u_0=\id$ and $u_i= v_i\alpha_{i}$ for $i=1,\dots, n$.
Then $u_i\sigma_{i+1}=v_i\alpha_{i}\sigma_{i+1}=v_i\alpha_{i+1}$  and by \cref{lem:alpha interval}, each of the three expressions is reduced. 
 Since $v_{i+1}\leq v_i$, we have $u_{i+1}=v_{i+1}\alpha_{i+1}\leq v_i\alpha_{i+1}=u_i\sigma_{i+1}$ by \cref{lem:alpha interval}.
Finally, the requirement $v_i\leq s_{i+1}\cdots s_{n-1}$ forces $v_{n-1}=v_n=1$ so $u_{n}=\alpha_n = w_0.$ Thus $u_\bullet$ is a brickly sequence.
\end{proof}

At this point it is not too difficult to explain why $\mbb{S}_{n-1}(Q)$ has $(n-1)!$ elements\footnote{To see the $(n-1)!$, set $v_0=s_1s_2\cdots s_{n-1}$ and let $d_i$ be such that $s_{i}\leq v_{d_i-1}$ but $s_i\not \leq v_{d_i}$. This exists uniquely because $s_{1}s_2\cdots s_{n-1}=v_0\geq v_{i}\geq v_{n-1}=\id$. Because $v_i\leq s_{i+1}\cdots s_{n-1},$ we have $d_i\leq i$. So $(d_1,\cdots, d_{n-1})\in [1] \times [2] \times [3] \times \cdots \times [n-1]$. There are $(n-1)!$ such sequences and each of them can be realized. So that is why there are $(n-1)!$ pieces.}, which is the same  as the number of maximal cells in the face poset of the standard cubulation. We can also see that each is combinatorially a cube. From \cref{thm:smooth-toric} and \cref{Loday mink sum} we learned that

\begin{lem}\label{lem: the vertices}
    Let $\sigma_i$, $\alpha_i$, and $Q$ be as in \cref{notation}. In this setting, we have $$\mbb{S}_0(Q)=\{(x_1\alpha_1,\dots, x_n\alpha_n)\in (S_n)^n:x_i\leq x_{i-1} \text{ for } i\geq 2 \text{ and }  x_i\leq s_is_{i+1}\dots s_{n-1} \text{ for } i\geq 1\}.$$
\end{lem}
\begin{proof}
    Let $(w_1,w_2,\dots, w_n)\in \mbb{S}_0(Q).$ By the definition of $\mbb{S}_0(Q),$ there exists some $u_{\bullet}\in L(Q)$ such that $u_i\leq w_i \leq u_{i-1}\sigma_i$ for $i=1,\dots, n.$ By \cref{pieces}, there exist $v_1,v_2,\dots, v_n$ such that, for $i\geq 1$, we have $v_{i+1}\leq v_{i}$, and $v_i\leq s_{i+1}s_{i+2}\dotsb s_{n-1}$, and $u_i=v_i\alpha_i$. From \cref{lem: look the interval} we know that $[u_i,u_{i-1}\sigma_i]\subseteq [\alpha_i,\alpha_{i+1}]$ so $w_i\in [\alpha_i,\alpha_{i+1}].$ Therefore, by \cref{lem:alpha interval}, there exists  $(x_1,\dots, x_n) \in (S_n)^n$ so that $w_i=x_i\alpha_i$ and $x_i\leq s_is_{i+1}\dots s_{n-1}$ for each $i$.

    Since $x_i\alpha_i=w_i\geq u_i=v_i\alpha_i,$ we get that $x_i\geq v_i$ by \cref{lem:alpha interval}. Since $x_i\alpha_i=w_i\leq u_{i-1}\sigma_i=v_{i-1}\alpha_{i-1}\sigma_i=v_{i-1}\alpha_{i},$ we get that $x_i\leq v_{i-1}$ by \cref{lem:alpha interval}. Hence, $x_i\leq v_{i-1}\leq x_{i-1}$ for all $i$. This proves 
    \[\mbb{S}_0(Q)\subseteq \{(x_1\alpha_1,\dots, x_n\alpha_n)\in (S_n)^n:x_i\leq x_{i-1} \text{ for } i\geq 2 \text{ and }  x_i\leq s_is_{i+1}\dots s_{n-1} \text{ for } i\geq 1\}.\]

    To prove the reverse inclusion, suppose we have some sequence $(x_1,\dots,x_n)\in (S_n)^n$ such that $x_i\leq x_{i-1}$ and $x_i\leq s_is_{i+1}\cdots s_{n-1}$ for each $i$. Letting $w_i=x_i\alpha_i,$ we want to show that $(w_1,w_2,\dots, w_n)\in \mbb{S}_0(Q).$ 

    Let $v_i=s_i\dem(s_i,x_i)$; in other words, if $x_i$ begins with $s_i$ then we remove it, and otherwise we let $x_i=v_i$. This construction guarantees that $v_i\leq s_{i+1}s_{i+2}\cdots s_{n-1}$. Since $x_i\leq x_{i-1}$ we will have $v_i\leq v_{i-1}$. Letting $u_0=1$ and $u_i=v_i\alpha_i$ for $i>0$ we have that $(u_0,u_1,\dots, u_n)\in L(Q)$ by Lemma \ref{pieces}.

     When we view $x_{i-1}$ as a subword of $s_{i-1}\cdots s_{n-1}$, we obtain $v_{i-1}$ as a subword of $x_{i-1}$ by removing $s_{i-1}$ if it is there. Then viewing $x_i$ as a subword of $x_{i-1}$ which doesn't use $s_{i-1},$ it is clear that $x_i\leq v_{i-1}$. Since $v_{i-1}\alpha_i \in [\alpha_i,\alpha_{i+1}],$ we may apply \cref{lem:alpha interval} to get $w_i=x_i\alpha_i\leq v_{i-1}\alpha_i = u_{i-1}\sigma_i.$  Since $v_i\leq x_i$ we also have that $u_i=v_i\alpha_i\leq x_i\alpha_i=w_i.$ Hence $w_i\in [u_i,u_{i-1}\sigma_i].$ Therefore, $(w_1,w_2,\dots, w_n)\in \mbb{S}_0(Q)$ by the definition of the spackle poset. We conclude that 
\[\mbb{S}_0(Q)= \{(x_1\alpha_1,\dots, x_n\alpha_n)\in (S_n)^n:x_i\leq x_{i-1} \text{ for } i\geq 2 \text{ and }  x_i\leq s_is_{i+1}\dots s_{n-1} \text{ for } i\geq 1\}.\]
\end{proof}

\subsubsection{The Lehmer complex, the witch, and the spackle poset}\label{subsubsection:witch}

We now have all the tools needed to prove \cref{lem:spackle-poset-iso}. Let $\sigma_i$, $\alpha_i$, and $Q$ be as in \cref{notation}.

We begin by defining the \textit{Lehmer complex} below.

\begin{defn}\label{defn:Lehmer-complex}
    Define the Lehmer complex $\LL^n$ as the face poset of the unit cube subdivision of the box with corners $(0,\dots,0)$ and $(1,2,\dots, n-1)$. In particular:
        \begin{itemize}
        \item the vertices are $[0,1]\times [0,2] \times [0,3] \times \cdots \times [0,n-1] \subseteq \mathbb{N}^{n-1}$,
        \item faces of dimension $\geq0$ are indexed by pairs of vertices $\gimel,\daleth$ such that every coordinate of $\daleth-\gimel$ is either 0 or 1. The face indexed by such a $(\gimel,\daleth)$ is the unit hypercube with $\gimel$ and $\daleth$ at opposite corners (i.e. the face with vertices $\{\gimel_1,\daleth_1\}\times \cdots \times \{\gimel_{n-1},\daleth_{n-1}\}$.) We will frequently identify $(\gimel,\daleth)$ with the face it indexes.
    \end{itemize}
    Let $\LL^n_{i}$ be the set of $i$ dimensional faces of $\LL^n.$ 
\end{defn}

\begin{egg}
    Consider the unit cube subdivision of the box with corners $(0,0)$ and $(1,2)$, pictured as the Lehmer complex of \cref{fig:lehmerEx}. The following examples should illustrate how pairs of vertices of the Lehmer complex index faces in this specific case. 
    \begin{itemize}
        \item The face indexed by the pair $(\gimel, \daleth)=((0,1),(0,2))$ has vertices \[\{\gimel_1,\daleth_1\}\times \{\gimel_2,\daleth_2\}=\{0,0\}\times\{1,2\}=\{(0,1),(0,2)\}.\]
        \item The face indexed by the pair $(\gimel, \daleth)=((0,1),(1,2))$ has vertices \[\{\gimel_1,\daleth_1\}\times \{\gimel_2,\daleth_2\}=\{0,1\}\times\{1,2\}=\{(0,1),(0,2),(1,1),(1,2)\}.\]
    \end{itemize}
\end{egg}

\begin{SCfigure}
    \centering
    \caption{A picture of the faces of the Lehmer complex $\mbb{L}^3$. The vertices are labeled by elements of $[0,1]\times [0,2]$. Each face is labeled $(\gimel, \daleth)$ where $\gimel$ is its bottom-leftmost vertex, $\daleth$ is its top-rightmost vertex.}
    \includegraphics[width=0.36\linewidth]{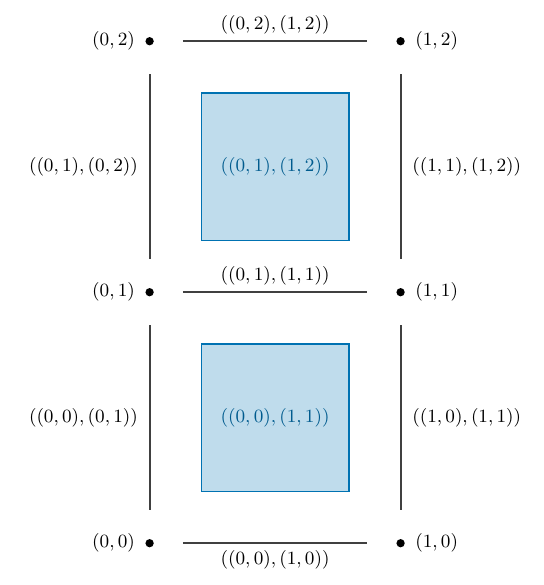}
    \label{fig:lehmerEx}
\end{SCfigure}

In \cref{subsection:cubulation}, we defined a map $\reversemathwitch$ from $\mbb{L}_0^n$ to $(S_n)^n$. Now that we have the notation $\mbb{S}_0(Q)$, which is the image of $\reversemathwitch$, we will define this map again, in \cref{witch} below. Using $\reversemathwitch$, we will define a map $\Upsilon$ in \cref{defn:extension-of-witch} from $\mbb{S}(Q)$ to $(\mbb{L}_0^n)^2$, which we will show in \cref{lem:spackle-poset-iso} defines an order isomorphism onto its image. Its image can be identified with the Lehmer complex $\mbb{L}^n$, completing the proof of \cref{thm:main-result}.

\begin{defn}\label{witch}
    Define $\reversemathwitch: \LL^n_0\to \mbb{S}_0(Q)$ as the map sending $(\daleth_1,\daleth_2,\dots, \daleth_{n-1})\in \LL^n_0$ to $(x_1\alpha_1, x_2\alpha_2,\dots, x_n\alpha_n)$ with 
 \[x_j= \prod_{i=1}^{n} \left(\begin{cases}
        s_i & j\leq \daleth_i\\
        \id & \text{otherwise.}
    \end{cases}\right) \]
    Define $\mathwitch=\reversemathwitch^{-1}$. 
\end{defn}
\cref{lem: the vertices} showed that $\reversemathwitch$ (and thus $\mathwitch$) are well defined bijections, and we will show in \cref{lem: preserve prec} that $\reversemathwitch$ is in fact an isomorphism of posets, where the poset structure on $\mbb{S}_0(Q)$ is defined in \cref{defn:partial-order} and the poset structure on $\mbb{L}_0^n$ defined in \cref{defn:partial-order}. We can more explicitly describe $\mathwitch$ as the map sending $(x_1\alpha_1, x_2\alpha_2,\dots, x_n\alpha_n)\in \mbb{S}_0(Q)$ to $(\daleth_1,\daleth_2,\dots, \daleth_{n-1})$ where $\daleth_i$ is such that $\{j:s_i\leq x_j\}=\{1,2,\dots, \daleth_{j}\}$ if and only if $j\leq \daleth_i$ (when $\{j:s_i\leq x_j\}=0$ we will have $\daleth_i=0$).

In \cref{defn:partial-order} below, we will define a partial order on $\mbb{S}_0(Q)$, and we will explain in \cref{lem: spackle prec} that this partial order on $\mbb{S}_0(Q)$ contains all of the information about the poset structure on $\mbb{S}(Q)$. In \cref{lem: Lehmer prec}, we will show that the poset structure on $\mbb{L}^n$ is determined entirely by the poset structure on $\mbb{L}_0^n$. Knowing that the poset structures on $\mbb{S}(Q)$ and $\mbb{L}^n$ boil down to the poset structures on $\mbb{S}_0(Q)$ and $\mbb{L}_0^n$ is an important observation towards lifting the map $\mathwitch$ from $\mbb{S}_0(Q)\to\mbb{L}_0^n$ (which, as we will show in \cref{lem: preserve prec} is an isomorphism of posets) to a map $\Upsilon$ from $\mbb{S}(Q)$ to $\mbb{L}^n$, as we will see shortly. 

\begin{defn}\label{defn:partial-order}
Define $\prec_{\mathbb{S}_0(Q)}$ as the partial order on $\mbb{S}_0(Q)$ where $\atom \prec_{\mathbb{S}_0(Q)} \boson$ for $\atom,\boson\in \mathbb{S}_0(Q)$ iff $\atom_i\leq \boson_i$ for all $i$. Similarly, define $\prec_{\LL^n_{0}}$ as the partial order on $\LL^n_{0}$ where $\gimel \prec_{\LL^n_{0}} \daleth$ for $\gimel,\daleth \in \LL^n_{0}$ iff $\gimel_i\leq \daleth_i$ for all $i$.
\end{defn}

\begin{lem}\label{lem: spackle prec}
    Let $\pumpkin=[\atom_1,\boson_1]\times \cdots \times [\atom_n,\boson_n] \in \mbb{S}(Q)$ and let $\pumpkin'=[\atom'_1,\boson'_1]\times \cdots \times [\atom'_n,\boson'_n] \in \mbb{S}(Q).$ Then $\pumpkin\subseteq \pumpkin'$ if and only if  $\atom'\prec_{\mbb{S}_0(Q)} \atom \prec_{\mbb{S}_0(Q)} \boson \prec_{\mbb{S}_0(Q)} \boson'.$
\end{lem}
\begin{proof}
    This follows from the observation that $[\atom_i,\boson_i]\subseteq [\atom_i',\boson_i']$ if and only if $\atom'_i\leq \atom_i\leq \boson_i\leq \boson'_i.$ 
\end{proof}

\begin{lem}\label{lem: Lehmer prec}
    For $(\gimel,\daleth)\in \LL^n$ and $(\gimel',\daleth')\in \LL^n$, we have $(\gimel,\daleth)\subseteq (\gimel',\daleth')$ (that is, the face of $\LL^n$ indexed by $(\gimel,\daleth)$ is contained in the face indexed by $(\gimel',\daleth')$) if and only if $\gimel'\prec_{\LL^n_{0}} \gimel  \prec_{\LL^n_{0}} \daleth \prec_{\LL^n_{0}} \daleth'.$
\end{lem}
\begin{proof}
    We find that $\{\gimel_i,\daleth_i\}\subseteq \{\gimel'_i,\daleth'_i\}$ if and only if $\gimel'_i\leq \gimel_i \leq \daleth_i \leq \daleth'_i$ because    
    $\daleth_i-\gimel_i\in \{0,1\}$ and $\daleth'_i-\gimel'_i\in \{0,1\}$. Hence $\{\gimel_i,\daleth_i\}\subseteq \{\gimel'_i,\daleth'_i\}$ for all $i$ if and only if $\gimel'_i\leq \gimel_i \leq \daleth_i \leq \daleth'_i$ for all $i$. The former statement is equivalent to $(\gimel,\daleth) \subseteq (\gimel', \daleth')$ and the latter statement is equivalent to $\gimel'\prec_{\LL^n_{0}} \gimel \prec_{\LL^n_{0}} \daleth \prec_{\LL^n_{0}} \daleth'.$
    \end{proof}

We will now show that $\reversemathwitch$ is an isomorphism of posets.

\begin{lem}\label{lem: preserve prec}
    For $\atom,\boson\in \mathbb{S}_0(Q)$, $\atom \prec_{\mathbb{S}_0(Q)} \boson$ if and only if $\mathwitch(\atom) \prec_{\LL^n_{0}} \mathwitch(\boson).$ Equivalently, for $\gimel, \daleth \in \LL^n_{0},$ $\gimel \prec_{\LL^n_{0}} \daleth$ if and only if $\reversemathwitch(\gimel) \prec_{\mbb{S}_0(Q)} \reversemathwitch(\daleth).$
\end{lem}
\begin{proof}
    Let $\atom=(\atom_1,\atom_2,\dots, \atom_n)\in \mbb{S}_0(Q)$ and $\boson=(\boson_1,\boson_2,\dots, \boson_n)\in \mathbb{S}_0(Q)$.  There exist $x_{\bullet},y_{\bullet}$ such that 
\begin{itemize}
    \item $\atom_i=x_i\alpha_i, \boson_i=y_i\alpha_i,$
    \item $x_i\leq x_{i-1}, y_i\leq y_{i-1}$
    \item and $x_i\leq s_is_{i+1}\dots s_{n-1}, y_i\leq s_is_{i+1}\dots s_{n-1}.$
\end{itemize}
By \cref{lem:alpha interval} and \cref{lem: the vertices}, $\atom_i\leq \boson_i $ if and only if $ x_i\leq y_{i}$. The construction of $\mathwitch$ implies that $x_i\leq y_i$ for all $i$ if and only if $(\mathwitch(\atom))_i \leq (\mathwitch(\boson))_i$ for all $i$. Hence, $\atom \prec_{\mathbb{S}_0(Q)} \boson$ if and only if $\mathwitch(\atom) \prec_{\LL^n_{0}} \mathwitch(\boson).$
\end{proof}

In \cref{defn:extension-of-witch} below, we will define a map $\Upsilon$ from $\mbb{S}(Q)$ to $(\mbb{L}_0^n)^2$, which we will show in \cref{lem: cubejection and more} has image $\mbb{L}^n$ and lifts the map $\mathwitch\colon \mbb{S}_0(Q)\to \mbb{L}_0^n$. In \cref{lem:spackle-poset-iso} we will prove the main result of this section, that $\Upsilon$ defines an isomorphism of posets from $\mbb{S}(Q)$ to $\mbb{L}^n$.
 
\begin{defn}\label{defn:extension-of-witch}
    Define $\assumption:\mbb{S}(Q)\to (\LL^n_{0})^2$ as the map sending $[\atom_1,\boson_1]\times \cdots \times [\atom_n,\boson_n] \in \mathbb{S}(Q)$ to $(\mathwitch(\atom),\mathwitch(\boson)).$
\end{defn}
We will soon prove that the image of $\assumption$ is $\LL^n$ and in fact $\assumption$ is a bijection between the spackle poset $\mbb{S}(Q)$ and the Lehmer complex $\LL^n$.

\begin{lem}\label{lem: cubejection and more}
     The image of $\assumption$ is $\LL^n$ and the restriction of $\assumption$ to $\mbb{S}_{n-1}(Q)$ is a bijection between $\mbb{S}_{n-1}(Q)$ and $\LL_{n-1}^{n}.$ 
\end{lem}
\begin{proof}
    Let $u\in L(Q)$.
   Let $\conductivity=(u_1,u_2,\dots,u_n)$ and $\dipole=(u_0\sigma_1,u_1\sigma_2,\dots, u_{n-1}\sigma_n).$ It follows from the definition of the spackle poset that $\conductivity\in \mbb{S}_0(Q)$ and $\dipole\in \mbb{S}_0(Q).$ 
    Let $x_i=u_i\alpha_i^{-1}$ so that $\conductivity=(u_1,\dots,u_n)=(x_1\alpha_1,\dots, x_n\alpha_n).$ Then $x_{i-1}=u_{i-1}\sigma_i\sigma_i^{-1}\alpha_{i-1}^{-1}=u_{i-1}\sigma_i \alpha_{i}^{-1}$ so \[\dipole=(\sigma_1,u_1\sigma_2,\dots, u_{n-1}\sigma_n) = (\alpha_1,x_1\alpha_2,x_2\alpha_3,\dots, x_{n-1}\alpha_n).\]
    Hence, $(\mathwitch(\dipole))_i=(\mathwitch(\conductivity))_i+1$. Therefore, $(\mathwitch(\dipole))_i - (\mathwitch(\conductivity))_i =1\in \{0,1\}.$ This shows that $(\mathwitch(\conductivity),\mathwitch(\dipole))\in \LL^n$. Observing that $\sum_{i=1}^{n-1} (\mathwitch(\dipole))_i - (\mathwitch(\conductivity))_i =n-1,$ we learn that in fact $(\mathwitch(\conductivity),\mathwitch(\dipole))\in \LL_{n-1}^{n}.$ Hence, the restriction of $\assumption$ to $\mathbb{S}_{n-1}(Q)$ has image contained in $\LL_{n-1}^{n}$. That is, $\assumption$ takes maximal cubes of the spackle poset $\mathbb{S}(Q)$ to maximal cubes of the Lehmer complex. Since $\mathwitch$ is injective, $\assumption$ is as well.
We know $|\mathbb{S}_{n-1}(Q)|=|\LL_{n-1}^{n}|=(n-1)!$ so $\assumption|_{\mathbb{S}_{n-1}(Q)}$ is bijective.

Now let $\pumpkin\in \mbb{S}(Q)$ and without loss of generality suppose that $\pumpkin \subseteq [u_1,\sigma_1]\times [u_2,u_1\sigma_2]\times \cdots \times [u_n,u_{n-1}\sigma_n]$ (this is true for some $u\in L(Q)$ so it might as well be the general $u$ we have been considering). Write $\pumpkin=[\atom_1,\boson_1]\times \cdots \times [\atom_n, \boson_n].$ We have that $\conductivity\prec_{\mbb{S}_{0}(Q)} \atom \prec_{\mbb{S}_{0}(Q)} \boson \prec_{\mbb{S}_{0}(Q)} \dipole.$ Hence, by \cref{lem: preserve prec}, $\mathwitch(\conductivity)\prec_{\LL_{0}^n}\mathwitch( \atom) \prec_{\LL_{0}^n}\mathwitch( \boson) \prec_{\LL_{0}^n} \mathwitch(\dipole).$ In particular, $(\mathwitch(\conductivity))_i\leq (\mathwitch( \atom))_i\leq  (\mathwitch( \boson))_i \leq (\mathwitch(\dipole))_i$. Since $(\mathwitch(\dipole))_i-(\mathwitch(\conductivity))_i=1,$ we conclude $(\mathwitch( \boson))_i-(\mathwitch( \atom))_i\in \{0,1\}$ and thus $\assumption(\pumpkin)=(\mathwitch(\atom),\mathwitch(\boson))\in \LL^n.$  Therefore, the image of $\assumption$ is contained in $\LL^n.$\qedhere

\end{proof}

We will now prove the main result of this section, and hence \cref{thm:main-result}.

\begin{thm}\label{lem:spackle-poset-iso}
   The spackle poset $\mbb{S}(Q)$ is isomorphic to the Lehmer complex $\LL^n$ (partially ordered by face containment) and the explicit isomorphism is given by $\assumption$. 
\end{thm}
\begin{proof}

    Let $(\gimel,\daleth)\in \LL^n.$ There exists some $n-1$ dimensional face  $\mathcloud\in \LL_{n-1}^{n}$ of the Lehmer complex such that $(\gimel,\daleth)\subseteq \mathcloud$.   By \cref{lem: cubejection and more}, $\mathcloud =\assumption(\conductivity,\dipole)$ where $\conductivity_i=u_i$ and $\dipole_i=u_{i-1}\sigma_i$ for some $u\in L(Q)$. Since $(\gimel,\daleth)\in \LL^n$ is such that $(\gimel,\daleth)\subseteq \mathcloud=(\mathwitch(\conductivity),\mathwitch(\dipole))$, \cref{lem: Lehmer prec} tells us that $\mathwitch(\conductivity)\prec_{\LL_{0}^n} \gimel \prec_{\LL_{0}^n} \daleth \prec_{\LL_{0}^n} \mathwitch(\dipole).$ Applying \cref{lem: preserve prec}, \[\conductivity=\reversemathwitch(\mathwitch(\conductivity))\prec_{\mbb{S}_0(Q)} \reversemathwitch(\gimel) \prec_{\mbb{S}_0(Q)} \reversemathwitch(\daleth) \prec_{\mbb{S}_0(Q)} \reversemathwitch(\mathwitch(\dipole))=\dipole.\]
    This means that $u_i\leq \reversemathwitch(\gimel)_i \leq  \reversemathwitch(\daleth)_i \leq u_{i-1}\sigma_i$ for all $i$. Letting $\pumpkin=[(\reversemathwitch(\gimel))_1,(\reversemathwitch(\daleth))_1]\times \cdots \times [(\reversemathwitch(\gimel))_n,(\reversemathwitch(\daleth))_n],$
    we get that $\pumpkin\in \mbb{S}(Q)$ and $\assumption(\pumpkin)=(\gimel,\daleth).$ We conclude that $\assumption$ surjects onto $\LL^n$.
    
    Since $\mathwitch$ is bijective, it follows that $\assumption$ is injective and therefore bijective as well. Taken together with \cref{lem: spackle prec}, \cref{lem: Lehmer prec}, and \cref{lem: preserve prec}, this tells us that $\assumption$ is an order isomorphism between the spackle poset $\mbb{S}(Q)$ and the Lehmer complex $\LL^n$.
\end{proof}

\printbibliography

@article {GU25,
    AUTHOR = {Gandhi, Raj and Udell, Gabe},
     TITLE = {A degeneration of the brick variety and a mixed subdivision of
              the associahedron into cubes},
   JOURNAL = {S\'{e}m. Lothar. Combin.},
  FJOURNAL = {S\'{e}minaire Lotharingien de Combinatoire},
    VOLUME = {93B},
      YEAR = {2025},
     PAGES = {Art. 91, 12},
      ISSN = {1286-4889},
   MRCLASS = {14M25 (52B05 52B20)},
  MRNUMBER = {4972552},
}

@misc{qdivdiff,
	title={Quasisymmetric divided differences},
	author={Nadeau, Philippe and Spink, Hunter and Tewari, Vasu},
    journal={arXiv preprint arXiv:2406.01510},
	year={2026}
}

@book {CG10,
    AUTHOR = {Chriss, Neil and Ginzburg, Victor},
     TITLE = {Representation theory and complex geometry},
    SERIES = {{Modern Birkh\"{a}user Classics}},
      NOTE = {Reprint of the 1997 edition},
 PUBLISHER = {{Birkh\"{a}user Boston, Ltd., Boston, MA}},
      YEAR = {2010},
     PAGES = {x+495},
      ISBN = {978-0-8176-4937-1},
   MRCLASS = {22E47 (17B35 20C08 20G05)},
  MRNUMBER = {2838836},
       DOI = {10.1007/978-0-8176-4938-8},
       URL = {https://doi.org/10.1007/978-0-8176-4938-8},
}

@article {DS19,
    AUTHOR = {Da Silva, Sergio},
     TITLE = {Strict {B}ott-{S}amelson resolutions of {S}chubert varieties},
   JOURNAL = {Exp. Math.},
  FJOURNAL = {Experimental Mathematics},
    VOLUME = {28},
      YEAR = {2019},
    NUMBER = {3},
     PAGES = {313--321},
      ISSN = {1058-6458,1944-950X},
   MRCLASS = {14M15 (14E15)},
  MRNUMBER = {3985836},
MRREVIEWER = {Edward\ Lowell\ Richmond},
       DOI = {10.1080/10586458.2017.1398695},
       URL = {https://doi.org/10.1080/10586458.2017.1398695},
}

@article {T07,
    AUTHOR = {Tenner, Bridget Eileen},
     TITLE = {Pattern avoidance and the {B}ruhat order},
   JOURNAL = {J. Combin. Theory Ser. A},
  FJOURNAL = {Journal of Combinatorial Theory. Series A},
    VOLUME = {114},
      YEAR = {2007},
    NUMBER = {5},
     PAGES = {888--905},
      ISSN = {0097-3165},
   MRCLASS = {05E15 (06A07 20F55)},
  MRNUMBER = {2333139},
MRREVIEWER = {Alessandro Conflitti},
       DOI = {10.1016/j.jcta.2006.10.003},
       URL = {https://doi.org/10.1016/j.jcta.2006.10.003},
}

@book{BB05,
    AUTHOR = {Bj\"{o}rner, Anders and Brenti, Francesco},
     TITLE = {Combinatorics of {C}oxeter groups},
    SERIES = {Graduate Texts in Mathematics},
    VOLUME = {231},
 PUBLISHER = {Springer, New York},
      YEAR = {2005},
     PAGES = {xiv+363},
      ISBN = {978-3540-442387; 3-540-44238-3},
   MRCLASS = {05-01 (05E15 20F55)},
  MRNUMBER = {2133266},
MRREVIEWER = {Jian-yi Shi},
}

@incollection {BA16,
    AUTHOR = {Abe, Hiraku and Billey, Sara},
     TITLE = {Consequences of the {L}akshmibai-{S}andhya theorem: the
              ubiquity of permutation patterns in {S}chubert calculus and
              related geometry},
 BOOKTITLE = {Schubert calculus---{O}saka 2012},
    SERIES = {Adv. Stud. Pure Math.},
    VOLUME = {71},
     PAGES = {1--52},
 PUBLISHER = {Math. Soc. Japan, [Tokyo]},
      YEAR = {2016},
   MRCLASS = {14M15 (05E10 14J17 14N15 20F55)},
  MRNUMBER = {3644818},
MRREVIEWER = {Mikhail Mazin},
       DOI = {10.2969/aspm/07110001},
       URL = {https://doi.org/10.2969/aspm/07110001},
}

@article {G01,
    AUTHOR = {Gaussent, St\'{e}phane},
     TITLE = {The fibre of the {B}ott-{S}amelson resolution},
   JOURNAL = {Indag. Math. (N.S.)},
  FJOURNAL = {Koninklijke Nederlandse Akademie van Wetenschappen.
              Indagationes Mathematicae. New Series},
    VOLUME = {12},
      YEAR = {2001},
    NUMBER = {4},
     PAGES = {453--468},
      ISSN = {0019-3577},
   MRCLASS = {14M15 (20E42 20F55)},
  MRNUMBER = {1908873},
MRREVIEWER = {Ivan Arzhantsev},
       DOI = {10.1016/S0019-3577(01)80034-3},
       URL = {https://doi.org/10.1016/S0019-3577(01)80034-3},
}

@article {BW01,
    AUTHOR = {Billey, Sara C. and Warrington, Gregory S.},
     TITLE = {Kazhdan-{L}usztig polynomials for 321-hexagon-avoiding
              permutations},
   JOURNAL = {J. Algebraic Combin.},
  FJOURNAL = {Journal of Algebraic Combinatorics. An International Journal},
    VOLUME = {13},
      YEAR = {2001},
    NUMBER = {2},
     PAGES = {111--136},
      ISSN = {0925-9899},
   MRCLASS = {05E15 (14M15 20F55)},
  MRNUMBER = {1826948},
MRREVIEWER = {Michael S. Kleber},
       DOI = {10.1023/A:1011279130416},
       URL = {https://doi.org/10.1023/A:1011279130416},
}

@article {D90,
    AUTHOR = {Deodhar, Vinay V.},
     TITLE = {A combinatorial setting for questions in {K}azhdan-{L}usztig
              theory},
   JOURNAL = {Geom. Dedicata},
  FJOURNAL = {Geometriae Dedicata},
    VOLUME = {36},
      YEAR = {1990},
    NUMBER = {1},
     PAGES = {95--119},
      ISSN = {0046-5755},
   MRCLASS = {20H15},
  MRNUMBER = {1065215},
MRREVIEWER = {Bhama Srinivasan},
       DOI = {10.1007/BF00181467},
       URL = {https://doi.org/10.1007/BF00181467},
}

@incollection {B05,
    AUTHOR = {Brion, Michel},
     TITLE = {Lectures on the geometry of flag varieties},
 BOOKTITLE = {Topics in cohomological studies of algebraic varieties},
    SERIES = {Trends Math.},
     PAGES = {33--85},
 PUBLISHER = {{Birkh\"{a}user, Basel}},
      YEAR = {2005},
   MRCLASS = {14M15 (14C35 19E08)},
  MRNUMBER = {2143072},
MRREVIEWER = {Christian Ohn},
       DOI = {10.1007/3-7643-7342-3\_2},
       URL = {https://doi.org/10.1007/3-7643-7342-3_2},
}

@book {S96,
    AUTHOR = {Sturmfels, Bernd},
     TITLE = {Gr\"{o}bner bases and convex polytopes},
    SERIES = {University Lecture Series},
    VOLUME = {8},
 PUBLISHER = {American Mathematical Society, Providence, RI},
      YEAR = {1996},
     PAGES = {xii+162},
      ISBN = {0-8218-0487-1},
   MRCLASS = {13P10 (14M25 52B20)},
  MRNUMBER = {1363949},
MRREVIEWER = {P. Schenzel},
       DOI = {10.1090/ulect/008},
       URL = {https://doi.org/10.1090/ulect/008},
}

@misc{ATZ24,
      AUTHOR = {An, Serena and Tung, Katherine and Zhang, Yuchong},
      EPRINT = {2412.02051},
  EPRINTTYPE = {arxiv},
       TITLE = {Postnikov-{S}tanley polynomials are {L}orentzian},
        YEAR = {2024}
}

@book {BL00,
    AUTHOR = {Billey, Sara and Lakshmibai, V.},
     TITLE = {Singular loci of {S}chubert varieties},
    SERIES = {Progress in Mathematics},
    VOLUME = {182},
 PUBLISHER = {{Birkh\"{a}user Boston, Inc., Boston, MA}},
      YEAR = {2000},
     PAGES = {xii+251},
      ISBN = {0-8176-4092-4},
   MRCLASS = {14M15 (14L35 20F55 20G05)},
  MRNUMBER = {1782635},
MRREVIEWER = {Michel Brion},
       DOI = {10.1007/978-1-4612-1324-6},
       URL = {https://doi.org/10.1007/978-1-4612-1324-6},
}

@article {PS09,
    AUTHOR = {Postnikov, Alexander and Stanley, Richard P.},
     TITLE = {Chains in the {B}ruhat order},
   JOURNAL = {J. Algebraic Combin.},
  FJOURNAL = {Journal of Algebraic Combinatorics. An International Journal},
    VOLUME = {29},
      YEAR = {2009},
    NUMBER = {2},
     PAGES = {133--174},
      ISSN = {0925-9899},
   MRCLASS = {14M15 (05A05 05E10 05E15)},
  MRNUMBER = {2475632},
       DOI = {10.1007/s10801-008-0125-4},
       URL = {https://doi.org/10.1007/s10801-008-0125-4},
}

@incollection {S76,
    AUTHOR = {Stanley, Richard P.},
     TITLE = {Some combinatorial aspects of the {S}chubert calculus},
 BOOKTITLE = {Combinatoire et repr\'{e}sentation du groupe sym\'{e}trique ({A}ctes
              {T}able {R}onde {CNRS}, {U}niv. {L}ouis-{P}asteur
              {S}trasbourg, {S}trasbourg, 1976)},
    SERIES = {Lecture Notes in Math., Vol. 579},
     PAGES = {217--251},
 PUBLISHER = {Springer, Berlin-New York},
      YEAR = {1977},
   MRCLASS = {05-02 (05A19 14M15 14N10)},
  MRNUMBER = {465880},
MRREVIEWER = {S. L. Kleiman},
}

@book {B91,
    AUTHOR = {Borel, Armand},
     TITLE = {Linear algebraic groups},
    SERIES = {Graduate Texts in Mathematics},
    VOLUME = {126},
   EDITION = {Second},
 PUBLISHER = {Springer-Verlag, New York},
      YEAR = {1991},
     PAGES = {xii+288},
      ISBN = {0-387-97370-2},
   MRCLASS = {20-01 (20Gxx)},
  MRNUMBER = {1102012},
MRREVIEWER = {F.\ D.\ Veldkamp},
       DOI = {10.1007/978-1-4612-0941-6},
       URL = {https://doi.org/10.1007/978-1-4612-0941-6},
}

@book {P08,
    AUTHOR = {Perrin, Daniel},
     TITLE = {Algebraic geometry},
    SERIES = {Universitext},
      NOTE = {An introduction,
              Translated from the 1995 French original by Catriona Maclean},
 PUBLISHER = {{Springer-Verlag London, Ltd., London; EDP Sciences, Les Ulis}},
      YEAR = {2008},
     PAGES = {xii+258},
      ISBN = {978-1-84800-055-1; 978-2-7598-0048-3},
   MRCLASS = {14-01},
  MRNUMBER = {2372337},
       DOI = {10.1007/978-1-84800-056-8},
       URL = {https://doi.org/10.1007/978-1-84800-056-8},
}

@article {KnutsonMiller,
    AUTHOR = {Knutson, Allen and Miller, Ezra},
     TITLE = {Gr\"obner geometry of {S}chubert polynomials},
   JOURNAL = {Ann. of Math. (2)},
  FJOURNAL = {Annals of Mathematics. Second Series},
    VOLUME = {161},
      YEAR = {2005},
    NUMBER = {3},
     PAGES = {1245--1318},
      ISSN = {0003-486X,1939-8980},
   MRCLASS = {05E15 (13C40 13F55 13P10 14M15 14N15)},
  MRNUMBER = {2180402},
MRREVIEWER = {Harry\ Tamvakis},
       DOI = {10.4007/annals.2005.161.1245},
       URL = {https://doi-org.proxy.library.cornell.edu/10.4007/annals.2005.161.1245},
}

@article {KM04,
    AUTHOR = {Knutson, Allen and Miller, Ezra},
     TITLE = {Subword complexes in {C}oxeter groups},
   JOURNAL = {Adv. Math.},
  FJOURNAL = {Advances in Mathematics},
    VOLUME = {184},
      YEAR = {2004},
    NUMBER = {1},
     PAGES = {161--176},
      ISSN = {0001-8708},
   MRCLASS = {20F55 (05E05 05E15 05E25 13F55)},
  MRNUMBER = {2047852},
MRREVIEWER = {Jian-yi Shi},
       DOI = {10.1016/S0001-8708(03)00142-7},
       URL = {https://doi.org/10.1016/S0001-8708(03)00142-7},
}

@article {CJP14,
    AUTHOR = {Ceballos, Cesar and Labb\'{e}, Jean-Philippe and Stump,
              Christian},
     TITLE = {Subword complexes, cluster complexes, and generalized
              multi-associahedra},
   JOURNAL = {J. Algebraic Combin.},
  FJOURNAL = {Journal of Algebraic Combinatorics. An International Journal},
    VOLUME = {39},
      YEAR = {2014},
    NUMBER = {1},
     PAGES = {17--51},
      ISSN = {0925-9899,1572-9192},
   MRCLASS = {05E45 (52B05)},
  MRNUMBER = {3144391},
MRREVIEWER = {Micha\l \ Adamaszek},
       DOI = {10.1007/s10801-013-0437-x}
}

@article{Sturmfels,
    AUTHOR = {Sturmfels, Bernd},
     TITLE = {Gr\"obner bases of toric varieties},
   JOURNAL = {Tohoku Math. J. (2)},
  FJOURNAL = {The Tohoku Mathematical Journal. Second Series},
    VOLUME = {43},
      YEAR = {1991},
    NUMBER = {2},
     PAGES = {249--261},
      ISSN = {0040-8735,2186-585X},
   MRCLASS = {14M25 (13P10 52B20)},
  MRNUMBER = {1104431},
MRREVIEWER = {Michael\ M.\ Kapranov},
       DOI = {10.2748/tmj/1178227496},
       URL = {https://doi-org.proxy.library.cornell.edu/10.2748/tmj/1178227496},
}

@incollection {K82,
    AUTHOR = {Konarski, Jerzy},
     TITLE = {Properties of projective orbits of actions of affine algebraic
              groups},
 BOOKTITLE = {Group actions and vector fields ({V}ancouver, {B}.{C}., 1981)},
    SERIES = {Lecture Notes in Math.},
    VOLUME = {956},
     PAGES = {79--91},
 PUBLISHER = {Springer, Berlin},
      YEAR = {1982},
      ISBN = {3-540-11946-9},
   MRCLASS = {14L30 (20G15)},
  MRNUMBER = {704988},
MRREVIEWER = {E.\ Aky\i ld\i z},
       DOI = {10.1007/BFb0101510},
       URL = {https://doi.org/10.1007/BFb0101510},
}

@article {KMS06,
    AUTHOR = {Knutson, Allen and Miller, Ezra and Shimozono, Mark},
     TITLE = {Four positive formulae for type {$A$} quiver polynomials},
   JOURNAL = {Invent. Math.},
  FJOURNAL = {Inventiones Mathematicae},
    VOLUME = {166},
      YEAR = {2006},
    NUMBER = {2},
     PAGES = {229--325},
      ISSN = {0020-9910},
   MRCLASS = {14M15 (05E15)},
  MRNUMBER = {2249801},
MRREVIEWER = {Harry Tamvakis},
       DOI = {10.1007/s00222-006-0505-0},
       URL = {https://doi.org/10.1007/s00222-006-0505-0},
}

@article {flagPositroids,
    AUTHOR = {Boretsky, Jonathan and Eur, Christopher and Williams, Lauren},
     TITLE = {Polyhedral and tropical geometry of flag positroids},
   JOURNAL = {Algebra Number Theory},
  FJOURNAL = {Algebra \& Number Theory},
    VOLUME = {18},
      YEAR = {2024},
    NUMBER = {7},
     PAGES = {1333--1374},
      ISSN = {1937-0652,1944-7833},
   MRCLASS = {05E14 (14T15)},
  MRNUMBER = {4757308},
       DOI = {10.2140/ant.2024.18.1333},
       URL = {https://doi.org/10.2140/ant.2024.18.1333},
}

@article {TW15,
    AUTHOR = {Tsukerman, E. and Williams, L.},
     TITLE = {Bruhat interval polytopes},
   JOURNAL = {Adv. Math.},
  FJOURNAL = {Advances in Mathematics},
    VOLUME = {285},
      YEAR = {2015},
     PAGES = {766--810},
      ISSN = {0001-8708},
   MRCLASS = {52B40 (05A05 06A07 20F55 52B05)},
  MRNUMBER = {3406515},
MRREVIEWER = {Joseph Kung},
       DOI = {10.1016/j.aim.2015.07.030},
       URL = {https://doi.org/10.1016/j.aim.2015.07.030},
}

@article {He,
    AUTHOR = {He, Xuhua},
     TITLE = {A subalgebra of 0-{H}ecke algebra},
   JOURNAL = {J. Algebra},
  FJOURNAL = {Journal of Algebra},
    VOLUME = {322},
      YEAR = {2009},
    NUMBER = {11},
     PAGES = {4030--4039},
      ISSN = {0021-8693,1090-266X},
   MRCLASS = {20C08},
  MRNUMBER = {2556136},
MRREVIEWER = {Thomas\ Pietraho},
       DOI = {10.1016/j.jalgebra.2009.04.003},
       URL = {https://doi-org.proxy.library.cornell.edu/10.1016/j.jalgebra.2009.04.003},
}

@article {LOYY24,
    AUTHOR = {Li, Tina and Oh, Suho and Richmond, Edward and Yan, Grace and
              You, Kimberly},
     TITLE = {Demazure product of permutations and hopping},
   JOURNAL = {Electron. J. Combin.},
  FJOURNAL = {Electronic Journal of Combinatorics},
    VOLUME = {31},
      YEAR = {2024},
    NUMBER = {1},
     PAGES = {Paper No. 1.48, 15},
      ISSN = {1077-8926},
   MRCLASS = {20F55 (05E16 20B30)},
  MRNUMBER = {4707769},
MRREVIEWER = {Jiayuan\ Wang},
       DOI = {10.37236/11983},
       URL = {https://doi.org/10.37236/11983},
}

@article {bricksort,
    AUTHOR = {Pilaud, Vincent and Santos, Francisco},
     TITLE = {The brick polytope of a sorting network},
   JOURNAL = {European J. Combin.},
  FJOURNAL = {European Journal of Combinatorics},
    VOLUME = {33},
      YEAR = {2012},
    NUMBER = {4},
     PAGES = {632--662},
      ISSN = {0195-6698,1095-9971},
   MRCLASS = {52B12 (05C75 52C30)},
  MRNUMBER = {2864447},
       DOI = {10.1016/j.ejc.2011.12.003},
       URL = {https://doi.org/10.1016/j.ejc.2011.12.003},
}

@article {E16,
    AUTHOR = {Escobar, Laura},
     TITLE = {Brick manifolds and toric varieties of brick polytopes},
   JOURNAL = {Electron. J. Combin.},
  FJOURNAL = {Electronic Journal of Combinatorics},
    VOLUME = {23},
      YEAR = {2016},
    NUMBER = {2},
     PAGES = {Paper 2.25, 18},
   MRCLASS = {14M15 (05E99 14M25)},
  MRNUMBER = {3512647},
MRREVIEWER = {Justin Brown},
       DOI = {10.37236/5038},
       URL = {https://doi.org/10.37236/5038},
}

@article {PS15,
    AUTHOR = {Pilaud, Vincent and Stump, Christian},
     TITLE = {Brick polytopes of spherical subword complexes and generalized
              associahedra},
   JOURNAL = {Adv. Math.},
  FJOURNAL = {Advances in Mathematics},
    VOLUME = {276},
      YEAR = {2015},
     PAGES = {1--61},
      ISSN = {0001-8708,1090-2082},
   MRCLASS = {05E45 (05E15 05E30 20F55 52B12)},
  MRNUMBER = {3327085},
MRREVIEWER = {Jian-yi\ Shi},
       DOI = {10.1016/j.aim.2015.02.012},
       URL = {https://doi.org/10.1016/j.aim.2015.02.012},
}

@article {inclusions,
    AUTHOR = {Gasharov, V. and Reiner, V.},
     TITLE = {Cohomology of smooth {S}chubert varieties in partial flag
              manifolds},
   JOURNAL = {J. London Math. Soc. (2)},
  FJOURNAL = {Journal of the London Mathematical Society. Second Series},
    VOLUME = {66},
      YEAR = {2002},
    NUMBER = {3},
     PAGES = {550--562},
      ISSN = {0024-6107,1469-7750},
   MRCLASS = {14M15 (05E15 14F25)},
  MRNUMBER = {1934291},
MRREVIEWER = {Raika\ Dehy},
       DOI = {10.1112/S0024610702003605},
       URL = {https://doi-org.proxy.library.cornell.edu/10.1112/S0024610702003605},
}

@book {Ziegler,
    AUTHOR = {Ziegler, G\"{u}nter M.},
     TITLE = {Lectures on polytopes},
    SERIES = {Graduate Texts in Mathematics},
    VOLUME = {152},
 PUBLISHER = {Springer-Verlag, New York},
      YEAR = {1995},
     PAGES = {x+370},
      ISBN = {0-387-94365-X},
   MRCLASS = {52Bxx},
  MRNUMBER = {1311028},
MRREVIEWER = {Margaret\ M.\ Bayer},
       DOI = {10.1007/978-1-4613-8431-1},
       URL = {https://doi.org/10.1007/978-1-4613-8431-1},
}

@article {postnikov,
    AUTHOR = {Postnikov, Alexander},
     TITLE = {Permutohedra, associahedra, and beyond},
   JOURNAL = {Int. Math. Res. Not. IMRN},
  FJOURNAL = {International Mathematics Research Notices. IMRN},
      YEAR = {2009},
    NUMBER = {6},
     PAGES = {1026--1106},
      ISSN = {1073-7928,1687-0247},
   MRCLASS = {05E30},
  MRNUMBER = {2487491},
       DOI = {10.1093/imrn/rnn153},
       URL = {https://doi-org.proxy.library.cornell.edu/10.1093/imrn/rnn153},
}

@article {Bilen,
    AUTHOR = {Can, Mahir Bilen},
     TITLE = {The cross-section of a spherical double cone},
   JOURNAL = {Adv. in Appl. Math.},
  FJOURNAL = {Advances in Applied Mathematics},
    VOLUME = {101},
      YEAR = {2018},
     PAGES = {215--231},
      ISSN = {0196-8858,1090-2074},
   MRCLASS = {06A07 (14M15 20G20)},
  MRNUMBER = {3857557},
       DOI = {10.1016/j.aam.2018.08.002},
       URL = {https://doi.org/10.1016/j.aam.2018.08.002},
}

@article {LMP21,
    AUTHOR = {Lee, Eunjeong and Masuda, Mikiya and Park, Seonjeong},
     TITLE = {Toric {B}ruhat interval polytopes},
   JOURNAL = {J. Combin. Theory Ser. A},
  FJOURNAL = {Journal of Combinatorial Theory. Series A},
    VOLUME = {179},
      YEAR = {2021},
     PAGES = {Paper No. 105387, 41},
      ISSN = {0097-3165},
   MRCLASS = {52B05 (14M25)},
  MRNUMBER = {4190574},
MRREVIEWER = {Margherita Barile},
       DOI = {10.1016/j.jcta.2020.105387},
       URL = {https://doi.org/10.1016/j.jcta.2020.105387},
}

@article {smallResolutions,
    AUTHOR = {Larson, Scott Joseph},
     TITLE = {Decompositions of {S}chubert varieties and small resolutions},
   JOURNAL = {Comm. Algebra},
  FJOURNAL = {Communications in Algebra},
    VOLUME = {52},
      YEAR = {2024},
    NUMBER = {5},
     PAGES = {2101--2126},
      ISSN = {0092-7872,1532-4125},
   MRCLASS = {22E47 (14M15)},
  MRNUMBER = {4716812},
MRREVIEWER = {Stephen\ Slebarski},
       DOI = {10.1080/00927872.2023.2281602},
       URL = {https://doi-org.proxy.library.cornell.edu/10.1080/00927872.2023.2281602},
}

@article {silv,
    AUTHOR = {Da Silva, Sergio},
     TITLE = {Strict {B}ott-{S}amelson resolutions of {S}chubert varieties},
   JOURNAL = {Exp. Math.},
  FJOURNAL = {Experimental Mathematics},
    VOLUME = {28},
      YEAR = {2019},
    NUMBER = {3},
     PAGES = {313--321},
      ISSN = {1058-6458,1944-950X},
   MRCLASS = {14M15 (14E15)},
  MRNUMBER = {3985836},
       DOI = {10.1080/10586458.2017.1398695},
       URL = {https://doi-org.proxy.library.cornell.edu/10.1080/10586458.2017.1398695},
}

@incollection {B03,
    AUTHOR = {Brion, Michel},
     TITLE = {Multiplicity-free subvarieties of flag varieties},
 BOOKTITLE = {Commutative algebra ({G}renoble/{L}yon, 2001)},
    SERIES = {Contemp. Math.},
    VOLUME = {331},
     PAGES = {13--23},
 PUBLISHER = {Amer. Math. Soc., Providence, RI},
      YEAR = {2003},
   MRCLASS = {14M15 (14L30 14N15)},
  MRNUMBER = {2011763},
MRREVIEWER = {Christian Ohn},
       DOI = {10.1090/conm/331/05900},
       URL = {https://doi.org/10.1090/conm/331/05900},
}

@book {EH16,
    AUTHOR = {Eisenbud, David and Harris, Joe},
     TITLE = {3264 and all that---a second course in algebraic geometry},
 PUBLISHER = {Cambridge University Press, Cambridge},
      YEAR = {2016},
     PAGES = {xiv+616},
      ISBN = {978-1-107-60272-4; 978-1-107-01708-5},
   MRCLASS = {14-01 (14C15 14M15 14N10)},
  MRNUMBER = {3617981},
MRREVIEWER = {Arnaud\ Beauville},
       DOI = {10.1017/CBO9781139062046},
       URL = {https://doi.org/10.1017/CBO9781139062046},
}

@book {F98,
    AUTHOR = {Fulton, William},
     TITLE = {Intersection theory},
    SERIES = {Ergebnisse der Mathematik und ihrer Grenzgebiete. 3. Folge. A
              Series of Modern Surveys in Mathematics [Results in
              Mathematics and Related Areas. 3rd Series. A Series of Modern
              Surveys in Mathematics]},
    VOLUME = {2},
   EDITION = {Second},
 PUBLISHER = {Springer-Verlag, Berlin},
      YEAR = {1998},
     PAGES = {xiv+470},
      ISBN = {3-540-62046-X; 0-387-98549-2},
   MRCLASS = {14C17 (14-02)},
  MRNUMBER = {1644323},
       DOI = {10.1007/978-1-4612-1700-8},
       URL = {https://doi.org/10.1007/978-1-4612-1700-8},
}

@book{V,
    AUTHOR = {Vakil, Ravi},
     TITLE = {The rising sea---foundations of algebraic geometry},
 PUBLISHER = {Princeton University Press, Princeton, NJ},
      YEAR = {[2025] \copyright 2025},
     PAGES = {xxiii+662},
      ISBN = {978-0-691-26866-8; 978-0-691-26867-5; 978-0-691-26868-2},
   MRCLASS = {Prelim},
  MRNUMBER = {4942570},
}

@book {H77,
    AUTHOR = {Hartshorne, Robin},
     TITLE = {Algebraic geometry},
    SERIES = {Graduate Texts in Mathematics, No. 52},
 PUBLISHER = {Springer-Verlag, New York-Heidelberg},
      YEAR = {1977},
     PAGES = {xvi+496},
      ISBN = {0-387-90244-9},
   MRCLASS = {14-01},
  MRNUMBER = {463157},
MRREVIEWER = {Robert Speiser},
}

@incollection {B02,
    AUTHOR = {Brion, Michel},
     TITLE = {Positivity in the {G}rothendieck group of complex flag
              varieties},
      NOTE = {Special issue in celebration of Claudio Procesi's 60th
              birthday},
   JOURNAL = {J. Algebra},
  FJOURNAL = {Journal of Algebra},
    VOLUME = {258},
      YEAR = {2002},
    NUMBER = {1},
     PAGES = {137--159},
      ISSN = {0021-8693},
   MRCLASS = {14C35 (14C17 14M15 19E15)},
  MRNUMBER = {1958901},
MRREVIEWER = {Mikhail Kogan},
       DOI = {10.1016/S0021-8693(02)00505-7},
       URL = {https://doi.org/10.1016/S0021-8693(02)00505-7},
}

@article {R92,
    AUTHOR = {Richardson, R. W.},
     TITLE = {Intersections of double cosets in algebraic groups},
   JOURNAL = {Indag. Math. (N.S.)},
  FJOURNAL = {Koninklijke Nederlandse Akademie van Wetenschappen.
              Indagationes Mathematicae. New Series},
    VOLUME = {3},
      YEAR = {1992},
    NUMBER = {1},
     PAGES = {69--77},
      ISSN = {0019-3577,1872-6100},
   MRCLASS = {20G15 (14L40)},
  MRNUMBER = {1157520},
MRREVIEWER = {James\ E.\ Humphreys},
       DOI = {10.1016/0019-3577(92)90028-J},
       URL = {https://doi.org/10.1016/0019-3577(92)90028-J},
}

@book {AndersonFulton,
    AUTHOR = {Anderson, David and Fulton, William},
     TITLE = {Equivariant cohomology in algebraic geometry},
    SERIES = {Cambridge Studies in Advanced Mathematics},
    VOLUME = {210},
 PUBLISHER = {Cambridge University Press, Cambridge},
      YEAR = {2024},
     PAGES = {xv+446},
      ISBN = {978-1-00-934998-7},
   MRCLASS = {14L30 (05E14 14-02 14F43 14Mxx 20G05 55N91)},
  MRNUMBER = {4655919},
MRREVIEWER = {Michael\ Orin\ Joyce},
}

@book{ShniderSternberg,
  author    = {Shnider, Steven and Sternberg, Shlomo},
  title     = {Quantum Groups: From Coalgebras to Drinfeld Algebras},
  series    = {Graduate Texts in Mathematical Physics},
  publisher = {International Press},
  year      = {1994}
}

@article {PP25,
    AUTHOR = {Pilaud, Vincent and Poliakova, Daria},
     TITLE = {Hochschild polytopes},
   JOURNAL = {Math. Ann.},
  FJOURNAL = {Mathematische Annalen},
    VOLUME = {392},
      YEAR = {2025},
    NUMBER = {2},
     PAGES = {2395--2441},
      ISSN = {0025-5831,1432-1807},
   MRCLASS = {52B05 (05C30)},
  MRNUMBER = {4906324},
MRREVIEWER = {Mehmetcik\ Pamuk},
       DOI = {10.1007/s00208-025-03120-x},
       URL = {https://doi.org/10.1007/s00208-025-03120-x},
}

@incollection {Lehmer,
    AUTHOR = {Lehmer, D. H.},
     TITLE = {Teaching combinatorial tricks to a computer},
 BOOKTITLE = {Proc. {S}ympos. {A}ppl. {M}ath., {V}ol. 10},
     PAGES = {179--193},
 PUBLISHER = {Amer. Math. Soc., Providence, RI},
      YEAR = {1960},
   MRCLASS = {68.00},
  MRNUMBER = {113289},
MRREVIEWER = {A.\ D.\ Booth},
}

@article {SU04,
    AUTHOR = {Saneblidze, Samson and Umble, Ronald},
     TITLE = {Diagonals on the permutahedra, multiplihedra and associahedra},
   JOURNAL = {Homology Homotopy Appl.},
  FJOURNAL = {Homology, Homotopy and Applications},
    VOLUME = {6},
      YEAR = {2004},
    NUMBER = {1},
     PAGES = {363--411},
      ISSN = {1532-0081,1532-0073},
   MRCLASS = {55U05 (05A18 52B05 55P35 81T30)},
  MRNUMBER = {2118493},
MRREVIEWER = {Alberto\ Cavicchioli},
       URL = {http://projecteuclid.org/euclid.hha/1139839559},
}

@misc{AllenHammer,
      title={Permutahedra, {L}usztig varieties, degenerations, and subdivisions}, 
      author={Allen Knutson and Mario Sanchez and Melissa Sherman-Bennett},
      year={2026},
      eprint={2607.02701},
      archivePrefix={arXiv},
      primaryClass={math.AG},
      url={https://arxiv.org/abs/2607.02701}, 
}

@misc{woo,
      title={Catalan numbers and {S}chubert polynomials for $w=1(n+1)... 2$}, 
      author={Alexander Woo},
      year={2004},
      eprint={math/0407160},
      archivePrefix={arXiv},
      primaryClass={math.CO},
      url={https://arxiv.org/abs/math/0407160}, 
}

@misc{cubical,
      title={Extended Abstract: Canonical join complex and cubical coordinates for all framing lattices}, 
      author={Jonah Berggren and Clément Chenevière},
      year={2026},
      eprint={2605.15319},
      archivePrefix={arXiv},
      primaryClass={math.CO},
      url={https://arxiv.org/abs/2605.15319}, 
}

@article {Karola,
    AUTHOR = {M\'esz\'aros, Karola},
     TITLE = {Pipe dream complexes and triangulations of root polytopes
              belong together},
   JOURNAL = {SIAM J. Discrete Math.},
  FJOURNAL = {SIAM Journal on Discrete Mathematics},
    VOLUME = {30},
      YEAR = {2016},
    NUMBER = {1},
     PAGES = {100--111},
      ISSN = {0895-4801,1095-7146},
   MRCLASS = {05E45 (52B12)},
  MRNUMBER = {3448341},
MRREVIEWER = {Gabriel\ K.\ Cunningham},
       DOI = {10.1137/15M1014802},
       URL = {https://doi-org.proxy.library.cornell.edu/10.1137/15M1014802},
}

@book {T51,
    AUTHOR = {Tamari, Dov},
     TITLE = {Mono\"{\i}des pr\'{e}ordonn\'{e}s et cha\^{i}nes de {M}alcev},
      NOTE = {Th\`ese},
 PUBLISHER = {Universit\'{e} de Paris, Paris},
      YEAR = {1951},
     PAGES = {iv+81 pp. (mimeographed)},
   MRCLASS = {20.0X},
  MRNUMBER = {51833},
MRREVIEWER = {D.\ C.\ Murdoch},
}

@misc{stacks-project,
    shorthand    = {Stacks},
    author       = {The {Stacks Project Authors}},
    title        = {\textit{Stacks Project}},
    howpublished = {\url{https://stacks.math.columbia.edu}},
    year         = {2018},
  }

@article {S63,
    AUTHOR = {Stasheff, James Dillon},
     TITLE = {Homotopy associativity of {$H$}-spaces. {I}, {II}},
      NOTE = {{\bf 108} (1963), 275-292; ibid},
   JOURNAL = {Trans. Amer. Math. Soc.},
  FJOURNAL = {Transactions of the American Mathematical Society},
    VOLUME = {108},
      YEAR = {1963},
     PAGES = {293--312},
      ISSN = {0002-9947,1088-6850},
   MRCLASS = {55.40},
  MRNUMBER = {158400},
MRREVIEWER = {A.\ H.\ Clark},
       DOI = {10.1090/s0002-9947-1963-0158400-5},
       URL = {https://doi.org/10.1090/s0002-9947-1963-0158400-5},
}

@misc {H84,
    AUTHOR = {Haiman, Mark},
     TITLE = {Constructing the associahedron},
     NOTE = {Unpublished manuscript, MIT 1984, 11 pages},
       URL = {http:
//math.berkeley.edu/~mhaiman/ftp/assoc/manuscript.pdf},
}

@article {L89,
    AUTHOR = {Lee, Carl W.},
     TITLE = {The associahedron and triangulations of the {$n$}-gon},
   JOURNAL = {European J. Combin.},
  FJOURNAL = {European Journal of Combinatorics},
    VOLUME = {10},
      YEAR = {1989},
    NUMBER = {6},
     PAGES = {551--560},
      ISSN = {0195-6698,1095-9971},
   MRCLASS = {52A25 (51M20)},
  MRNUMBER = {1022776},
MRREVIEWER = {J.\ M.\ Wills},
       DOI = {10.1016/S0195-6698(89)80072-1},
       URL = {https://doi.org/10.1016/S0195-6698(89)80072-1},
}

@article {L04,
    AUTHOR = {Loday, Jean-Louis},
     TITLE = {Realization of the {S}tasheff polytope},
   JOURNAL = {Arch. Math. (Basel)},
  FJOURNAL = {Archiv der Mathematik},
    VOLUME = {83},
      YEAR = {2004},
    NUMBER = {3},
     PAGES = {267--278},
      ISSN = {0003-889X,1420-8938},
   MRCLASS = {52B11 (55P47 55P48)},
  MRNUMBER = {2108555},
MRREVIEWER = {Richard\ John\ Steiner},
       DOI = {10.1007/s00013-004-1026-y},
       URL = {https://doi.org/10.1007/s00013-004-1026-y},
}

@article{HHMP,
  title={The volume polynomial of regular semisimple Hessenberg varieties and the Gelfand—Zetlin polytope},
  author={Harada, Megumi and Horiguchi, Tatsuya and Masuda, Mikiya and Park, Seonjeong},
  journal={Proceedings of the Steklov Institute of Mathematics},
  volume={305},
  number={1},
  pages={318--344},
  year={2019},
  publisher={Springer}
}

@article {Kostant-Toda,
    AUTHOR = {Kodama, Yuji and Williams, Lauren},
     TITLE = {The full {K}ostant-{T}oda hierarchy on the positive flag
              variety},
   JOURNAL = {Comm. Math. Phys.},
  FJOURNAL = {Communications in Mathematical Physics},
    VOLUME = {335},
      YEAR = {2015},
    NUMBER = {1},
     PAGES = {247--283},
      ISSN = {0010-3616,1432-0916},
   MRCLASS = {17B80 (14M15 22E46 52B12)},
  MRNUMBER = {3314504},
MRREVIEWER = {Rutwig\ Campoamor-Stursberg},
       DOI = {10.1007/s00220-014-2203-x},
       URL = {https://doi.org/10.1007/s00220-014-2203-x},
}

@phdthesis{LauraThesis,
  title={Brick varieties and toric matrix {S}chubert varieties},
  author={Vega, Laura Escobar},
  year={2015},
  school={Cornell University}
}

@incollection {E19,
    AUTHOR = {Escobar, Laura},
     TITLE = {A brief survey about moment polytopes of subvarieties of
              products of {G}rassmanians},
 BOOKTITLE = {Algebraic and geometric combinatorics on lattice polytopes},
     PAGES = {200--216},
 PUBLISHER = {World Sci. Publ., Hackensack, NJ},
      YEAR = {2019},
   MRCLASS = {14M15 (14N15 52B20)},
  MRNUMBER = {3971694},
}

@article{benedetti2024lattice,
  title={Lattice path matroids and quotients},
  author={Benedetti-Vel{\'a}squez, Carolina and Knauer, Kolja},
  journal={Combinatorica},
  volume={44},
  number={3},
  pages={621--650},
  year={2024},
  publisher={Springer}
}

@misc{MarioMV,
      title={Mirkovi\'c-{V}ilonen Polytopes from Combinatorics}, 
      author={Mario Sanchez},
      year={2023},
      eprint={2311.16979},
      archivePrefix={arXiv},
      primaryClass={math.CO},
      url={https://arxiv.org/abs/2311.16979}, 
}

@incollection {Ch14,
    AUTHOR = {Chabert, Jean-Luc},
     TITLE = {Integer-valued polynomials: looking for regular bases (a
              survey)},
 BOOKTITLE = {Commutative algebra},
     PAGES = {83--111},
 PUBLISHER = {Springer, New York},
      YEAR = {2014},
   MRCLASS = {13F20 (11B65 11R21 11S05)},
  MRNUMBER = {3330213},
MRREVIEWER = {Keith Peter Johnson},
}

@article {CGGLSS25,
    AUTHOR = {Casals, Roger and Gorsky, Eugene and Gorsky, Mikhail and Le,
              Ian and Shen, Linhui and Simental, Jos\'{e}},
     TITLE = {Cluster structures on braid varieties},
   JOURNAL = {J. Amer. Math. Soc.},
  FJOURNAL = {Journal of the American Mathematical Society},
    VOLUME = {38},
      YEAR = {2025},
    NUMBER = {2},
     PAGES = {369--479},
      ISSN = {0894-0347},
   MRCLASS = {13F60 (14M15 20F36)},
  MRNUMBER = {4868947},
       DOI = {10.1090/jams/1048},
       URL = {https://doi.org/10.1090/jams/1048},
}

@article {Cambrian,
    AUTHOR = {Reading, Nathan},
     TITLE = {Cambrian lattices},
   JOURNAL = {Adv. Math.},
  FJOURNAL = {Advances in Mathematics},
    VOLUME = {205},
      YEAR = {2006},
    NUMBER = {2},
     PAGES = {313--353},
      ISSN = {0001-8708,1090-2082},
   MRCLASS = {05E25 (06B10 20F55 52C07)},
  MRNUMBER = {2258260},
MRREVIEWER = {Axel\ Hultman},
       DOI = {10.1016/j.aim.2005.07.010},
       URL = {https://doi-org.proxy.library.cornell.edu/10.1016/j.aim.2005.07.010},
}

@article {HLT,
    AUTHOR = {Hohlweg, Christophe and Lange, Carsten E. M. C. and Thomas,
              Hugh},
     TITLE = {Permutahedra and generalized associahedra},
   JOURNAL = {Adv. Math.},
  FJOURNAL = {Advances in Mathematics},
    VOLUME = {226},
      YEAR = {2011},
    NUMBER = {1},
     PAGES = {608--640},
      ISSN = {0001-8708,1090-2082},
   MRCLASS = {20F55 (05E15 52B11)},
  MRNUMBER = {2735770},
MRREVIEWER = {Chi\ Kin\ Mak},
       DOI = {10.1016/j.aim.2010.07.005},
       URL = {https://doi-org.proxy.library.cornell.edu/10.1016/j.aim.2010.07.005},
}

@book {BrionKumar,
    AUTHOR = {Brion, Michel and Kumar, Shrawan},
     TITLE = {Frobenius splitting methods in geometry and representation
              theory},
    SERIES = {Progress in Mathematics},
    VOLUME = {231},
 PUBLISHER = {{Birkh\"auser Boston, Inc., Boston, MA}},
      YEAR = {2005},
     PAGES = {x+250},
      ISBN = {0-8176-4191-2},
   MRCLASS = {14M15 (13A35 14C05 17B10 20G05)},
  MRNUMBER = {2107324},
MRREVIEWER = {Vikram\ B.\ Mehta},
}

@article{quasiflag,
	title = {The quasisymmetric flag variety: {A} toric complex on noncrossing partitions},
	volume = {14},
	issn = {2050-5086},
	shorttitle = {The quasisymmetric flag variety},
	url = {https://www.cambridge.org/core/product/identifier/S2050508626100341/type/journal_article},
	doi = {10.1017/fmp.2026.10034},
	language = {en},
	urldate = {2026-07-27},
	journal = {Forum of Mathematics, Pi},
	author = {Bergeron, Nantel and Gagnon, Lucas and Nadeau, Philippe and Spink, Hunter and Tewari, Vasu},
	year = {2026},
	pages = {e18},
}

@misc{bergeron_coxeter_2026,
	title = {The {Coxeter} {Flag} {Variety}},
	url = {http://arxiv.org/abs/2601.23111},
	doi = {10.48550/arXiv.2601.23111},
	urldate = {2026-07-27},
	publisher = {arXiv},
	author = {Bergeron, Nantel and Gagnon, Lucas and Spink, Hunter and Tewari, Vasu},
	month = jan,
	year = {2026},
	note = {arXiv:2601.23111 [math.AG]},
}

@misc{nadeau_geometry_2024,
	title = {The geometry of quasisymmetric coinvariants},
	url = {http://arxiv.org/abs/2410.12643},
	doi = {10.48550/arXiv.2410.12643},
	urldate = {2026-07-27},
	publisher = {arXiv},
	author = {Nadeau, Philippe and Spink, Hunter and Tewari, Vasu},
	month = oct,
	year = {2024},
	note = {arXiv:2410.12643 [math.AG]},
}

@incollection {BrionProcesi,
    AUTHOR = {Brion, Michel and Procesi, Claudio},
     TITLE = {Action d'un tore dans une vari\'et\'e{} projective},
 BOOKTITLE = {Operator algebras, unitary representations, enveloping
              algebras, and invariant theory ({P}aris, 1989)},
    SERIES = {Progr. Math.},
    VOLUME = {92},
     PAGES = {509--539},
 PUBLISHER = {{Birkh\"auser Boston, Boston, MA}},
      YEAR = {1990},
      ISBN = {0-8176-3489-4},
   MRCLASS = {14L30 (17B35 52B20)},
  MRNUMBER = {1103602},
MRREVIEWER = {Anthony\ Joseph},
}

@article{Perrin, 
title={Small resolutions of minuscule Schubert varieties}, 
volume={143},
DOI={10.1112/S0010437X07002734},
number={5},
journal={Compositio Mathematica},
author={Perrin, Nicolas},
year={2007},
pages={1255–1312}}

@article{lian,
  title={The HHMP decomposition of the permutohedron and degenerations of torus orbits in flag varieties},
  author={Lian, Carl},
  journal={International Mathematics Research Notices},
  volume={2024},
  number={20},
  pages={13380--13399},
  year={2024},
  publisher={Oxford University Press}
}

@article{gorsky2026unexpected,
  title={Unexpected toric Richardson varieties},
  author={Gorsky, Eugene and Kim, Soyeon and Sherman-Bennett, Melissa},
  journal={arXiv preprint arXiv:2603.29260},
  year={2026}
}

@article{gaetz2026poset,
  title={Poset topology, moves, and Bruhat interval polytope lattices},
  author={Gaetz, Christian and Hersh, Patricia},
  journal={Bulletin of the London Mathematical Society},
  volume={58},
  number={3},
  pages={e70309},
  year={2026},
  publisher={Wiley Online Library}
}

@article{gaetz2023one,
  title={One-skeleton posets of Bruhat interval polytopes},
  author={Gaetz, Christian},
  journal={Advances in Mathematics},
  volume={430},
  pages={109216},
  year={2023},
  publisher={Elsevier}
}

\end{document}